\documentclass[11pt,letterpaper]{amsart}
\usepackage{xcolor}
\usepackage{amsmath} 
\usepackage{amsthm}
\usepackage{amssymb}
\makeatletter 
\usepackage{tikz-cd}
\usepackage{tikz,float,caption}
\usetikzlibrary{arrows.meta}
\usetikzlibrary{decorations.pathmorphing}
\usetikzlibrary{decorations.pathreplacing}
\usetikzlibrary{patterns} 
\usetikzlibrary{arrows.meta,calc,decorations.markings,patterns,cd,patterns.meta}

\usepackage[bb=ams,scr=euler]{mathalfa}

\usepackage{bbding} 

\usepackage[colorlinks=true,linkcolor=blue, citecolor=blue]{hyperref}
\usepackage{tikz,float,caption}
\usetikzlibrary{arrows.meta,calc,decorations.markings,patterns,cd,patterns.meta}

\usepackage{enumitem}
\setenumerate{
  label=(\roman*),
}

\newtheorem{theorem}{Theorem}[section]

\newtheorem{corollary}[theorem]{Corollary}

\newtheorem{prop}[theorem]{Proposition}
\newtheorem{lemma}[theorem]{Lemma}
\newtheorem{definition}[theorem]{Definition}
\newtheorem{remark}[theorem]{Remark}

\newtheorem{question}[theorem]{Question}

\newtheorem{alphatheorem}{Theorem}

\newcommand{\abs}[1]{\left|#1\right|}

\renewcommand{\d}{\mathrm{d}}
\renewcommand{\H}{\mathbb{H}}
\newcommand{\A}{\mathbb{A}}
\newcommand{\id}{\mathrm{id}}
\newcommand{\Fl}{\mathcal{F}}

\newcommand{\kh}{\mathrm{kh}}

\newcommand{\FD}{\mathrm{FD}}
\newcommand{\eps}{\epsilon}

\newcommand{\pr}{\mathrm{pr}}
\newcommand{\R}{\mathbb{R}}
\newcommand{\D}{\mathbb{D}}
\newcommand{\C}{\mathbb{C}}

\newcommand{\la}{\lambda}

\def\@secnumfont{\bfseries}
\renewcommand\section{\@startsection{section}{1}{0pt}{-3.5ex \@plus -1ex \@minus -.2ex}{2.3ex \@plus.2ex}{\centering\itshape}}

\renewcommand{\subsection}{\@startsection{subsection}{2}%
  \z@{.5\linespacing\@plus.7\linespacing}{-.5em}%
  {\normalfont\itshape}}

\renewcommand{\subsubsection}{\@startsection{subsubsection}{3}%
  \z@{.5\linespacing\@plus.7\linespacing}{-.5em}%
  {\normalfont\itshape}}

\renewcommand{\u}[1]{\underline{#1}}

\newcommand{\Z}{\mathbb{Z}}

\renewcommand{\t}[1]{\widetilde{#1}}

\newcommand{\Tr}{\mathrm{T}}
\newcommand{\m}{\mathrm{m}}
\newcommand{\p}{\mathrm{p}}
\newcommand{\pss}{\mathrm{PSS}}
\newcommand{\Ham}{\mathrm{Ham}}
\newcommand{\Symp}{\mathrm{Symp}}

\newcommand{\CF}{\mathrm{CF}}
\newcommand{\HF}{\mathrm{HF}}
\newcommand{\QH}{\mathrm{QH}}
\renewcommand{\H}{\mathrm{H}}
\renewcommand{\QH}{\mathrm{QH}}
\newcommand{\fix}{\mathrm{Fix}}
\newcommand{\Hof}{\mathrm{Hof}}
\newcommand{\Fix}{\mathrm{Fix}}

\newcommand{\K}{\mathbb{K}}
\newcommand{\om}{\omega}

\title[]{Invariant sets and spectral rigidity of Hamiltonian diffeomorphisms}
\author{Habib Alizadeh and Egor Shelukhin}
\makeatother

\begin{document}
\begin{abstract}
We solve a well-known question of Polterovich from 2002 regarding the rigidity of Hamiltonian diffeomorphisms in Hofer's metric in dimension two: every non-trivial Hamiltonian diffeomorphism of a surface of genus at least one has all positive iterations separated from the identity in Hofer's metric. In higher dimensions, we make progress on the case of autonomous Hamiltonian flows: we prove Hofer non-recurrence for symplectically hyperbolic manifolds, and completely characterize Hofer recurrence for functions of the moment polytope on monotone toric manifolds. These results apply equally well to Viterbo's spectral metric, settling cases of the $\gamma$-rigidity conjecture, and therefore in many cases also to the $C^0$-metric.

Our approach relies on the philosophy that symplectically visible invariant sets govern symplectic non-recurrence phenomena. In dimension two, we use invariant annuli provided by low-dimensional dynamics, and apply methods of $C^0$ symplectic topology to approximately invariant Lagrangian submanifolds. For symplectically hyperbolic manifolds, the invariant sets are provided by sublevel and superlevel sets of autonomous Hamiltonians, while in the toric case, they comprise Lagrangian torus fibers. Across the arguments, we use quantitative Lagrangian Floer theory and its relation to notions of classical dynamics. \end{abstract}

\maketitle
\setcounter{tocdepth}{1} 

\tableofcontents

\section{Introduction}\phantomsection\label{sec:introduction}
Invariant sets of diffeomorphisms have been an important notion in classical dynamics, hyperbolic dynamics in particular. For instance, due to a celebrated result of Katok \cite{katok-89}, \cite[Suppl. S.]{KH-book}, hyperbolic invariant sets capture the topological entropy of surface diffeomorphisms. Recently the reflection of such sets in Floer theory was studied in \cite{cgg-barcode-entropy, cgg-generic-behaviour, cgg-growth-of-floer-barcode} among other works. In this paper, we study the implications of invariant sets, which are not in general hyperbolic to elliptic dynamics, understood from the metric perspective in symplectic topology. 

Namely, we study the following {\em rigidity} question regarding Hamiltonian diffeomorphisms of symplectic manifolds. Given a non-trivial Hamiltonian diffeomorphism $\phi$ of a symplectic manifold $M,$ when can the iterations $\phi^k$ of $\phi$ have an accumulation point at the identity? In other words, according to different terminologies, when is $\phi$ {\em recurrent} or {\em rigid} or an {\em approximate identity}? Of course this question depends on the choice of topology on the group $\Ham(M,\om)$ of Hamiltonian diffeomorphisms. Interesting examples of recurrent mappings were extensively studied in the literature: see \cite{gg-aid} for a discussion and \cite{leroux-seyfaddini, joksimovic-seyfaddini, gg-pseudo-rotations, Bramham-15} for investigations in symplectic topology. Our main interest in this paper is showing that recurrence in symplectic dynamics is generally scarce.

A well-known question of Polterovich going back to \cite{Polterovich-02} and motivated by the observation that iterations of a given $\phi$ in $\Ham(M,\om)$ cannot {\em converge} to the identity in any non-degenerate norm $\nu$ as, for all $k\in\Z,$ \[\max\{\nu(\phi^k),\nu(\phi^{k+1})\} \geq \frac{1}{2}\nu(\phi),\] asserts that on a closed {\em symplectically aspherical} symplectic manifold, namely $\langle [\om],A\rangle = 0, \langle c_1(TM),A\rangle = 0$ for all classes $A \in \pi_2(M),$ the iterations $\phi^k$ must be separated from the identity in Hofer's celebrated metric \cite{hofer-metric,polterovich-lagrangian,lalonde-mcduff}, see also \cite{polterovich-book}. We recall that this metric is defined as $d_{\Hof}(\phi,\psi) = ||\psi \phi^{-1}||_{\Hof}$ for the Hofer norm given by \[||\theta||_{\Hof} = \inf_{\phi^1_H = \theta} \int_0^1 (\max_M H_t - \min_M H_t )\,dt,\] where $H \in C^\infty(\R/\Z \times M, \R)$ is a time-dependent Hamiltonian and $(\phi^t_H)_{t \in [0,1]}$ is the flow of the vector field $X^t_H$ satisfying $i_{X^t_H} \om = - dH_t,$ $H_t(x) = H(t,x)$ for all $(t,x) \in \R/\Z \times M.$ 
Note that recurrence is a typical feature of maps which can be inscribed in an action of a compact group. For symplectically aspherical manifolds, Polterovich proved \cite{Polterovich-02} that there are no non-trivial actions of compact Lie groups by Hamiltonian diffeomorphisms, whence by the smooth Hilbert-Smith conjecture \cite{bochner-montgomery} there are no continuous such compact group actions either. This provides additional motivation for the above question.

The question also makes sense for the spectral metric \cite{MR1157321, schwarz-00, MR2103018, oh-duke}, discussed in \S\ref{sec:spectral_invariants}, a Floer-theoretic lower bound for Hofer's metric, which depends on a choice of coefficients for Floer theory, and for various metrics of a more classical nature: the $C^0$-metric for example. We note that it follows from \cite{Polterovich-02,Sugimoto-23, gg-aid} that the same question has a rather quick solution in the $C^1$ metric, and in dimension two classical dynamics provides a solution in the $C^0$ metric \cite{kp-recurrent}. However, Polterovich's question and its variant for the spectral metric posed by \c{C}ineli-Ginzburg-G\"urel \cite{cgg-generic-behaviour}, known also as the {\em $\gamma$-rigidity conjecture}, have hitherto eluded a complete resolution even in dimension two. 

Partial results were obtained in \cite{cgg-generic-behaviour} proving the conjecture for the spectral norm for Hamiltonian diffeomorphisms belonging to a $C^{\infty}$-open and dense set and in \cite{SW-unpublished}, see also \cite{S-HS}, it was shown to hold along {\em prime} iterations for all closed manifolds $(M,\om)$ with $\langle [\om],A\rangle \langle c_1(TM),A\rangle \leq 0$ for all $A \in \pi_2(M),$ where the norm is Hofer's norm, or spectral norm over the integers. Furthermore, for surfaces of genus at least one, non-recurrence of {\em strongly non-degenerate} Hamiltonian diffeomorphisms is known for the spectral and Hofer norms \cite[Cor. 3.4]{cgg-generic-behaviour}. However, none of the methods so far appears to  generalize naturally to a resolution of the full question, even for symplectically aspherical manifolds.

We make new progress on this question in dimension $2$ as well as in higher dimensions.

\subsection*{Surfaces} Our first main, and most definitive, result is to resolve Polterovich's question, and in its version for the spectral norm with arbitrary coefficients, in dimension two.   

\begin{alphatheorem}\phantomsection\label{thm:main_theorem_introduction}
Let $(\Sigma,\sigma)$ be a closed symplectic surface of genus at least one. For every Hamiltonian diffeomorphism $\phi \in \Ham(\Sigma,\sigma),$ $\phi \neq \id$ there exists $\delta > 0$ such that \[\gamma(\phi^k) \geq \delta\] for all integers $k \geq 1.$ In particular $||\phi^k||_{\Hof} \geq \delta$ for all $k \geq 1.$
\end{alphatheorem}

The proof of this theorem appears in \S\ref{sec:proof_main_theorem}.

\begin{remark}
A similar result holds for non-identity compactly supported Hamiltonian diffeomorphisms of compact surfaces with boundary except the disk. This is an immediate consequence of Theorem \ref{thm:main_theorem_introduction} and a phenomenon called locality of spectral invariants that holds in particular for Liouville domains with contact type and incompressible boundary inside aspherical symplectic manifolds \cite{MR4587314}; see also \cite{MR3548467, alizadeh-shelukhin-locality-26}.
\end{remark}

The remaining case in dimension $2$ for Polterovich's question is the sphere $S^{2}$. First examples of Hamiltonian diffeomorphisms $\phi$ for which $\gamma(\phi^{n_{k}})$ accumulates at zero for some sequence $\{n_{k}\}$ are given by rational rotations. Moreover, irrational rotations and in fact all pseudo-rotations are known to satisfy $\liminf_{k\geq 1}\gamma(\phi^{k}) = 0$; see \cite{gg-pseudo-rotations}. If $h_{top}(\phi)>0$ then one could use the hyperbolic invariant sets provided by Katok \cite{katok-89} to bound $\gamma(\phi^{k})$ from below uniformly in $k$; see \cite{cgg-generic-behaviour}. Therefore, the remaining cases are the Hamiltonian diffeomorphisms of $S^{2}$ that are not torsion, have at least $3$ periodic points and $h_{top}(\phi) = 0$. This is exactly the setting where the structure theorem of Franks and Handel \cite{franks-handel-zero-ent}, see also \cite{LeCalvez-Tal-forcing-theory-18}, apply and one can find many $\phi$-invariant open topological annuli. These annuli can possibly be used to uniformly bound $\gamma(\phi^{k})$ in $k$; see \S\ref{sec:proof_main_theorem} where the proofs are based on this idea and are successfully applied in the case of higher genus surfaces; such invariant annuli in higher genus surfaces are provided by Theorem \ref{thm:up_coming_work_of_LCT}. However, our methods in \S\ref{sec:proof_main_theorem} do not seem to apply directly to the case of the sphere. Therefore it would be interesting to investigate the following question:

\begin{question}
Suppose $\phi\in \Ham(S^{2})$ is not torsion, has at least $3$ periodic points and $h_{top}(\phi) = 0$. Does there exist a constant $\delta>0$ so that $\gamma(\phi^{k}) \geq \delta$ for all $k\geq 1$ ? 
\end{question}

\subsection*{Higher dimensions}

We also make new progress towards a resolution of the  question of Polterovich regarding Hofer's metric and of the $\gamma$-rigidity conjecture in higher dimensions, where we primarily work with autonomous Hamiltonian flows. 

To state our first result, we recall that a closed symplectic manifold $(M,\om)$ is called {\em symplectically hyperbolic}, a notion first introduced by Gromov in K\"ahler geometry \cite{gromov-kahlerhyp} and then studied by Polterovich \cite{Polterovich-02} in symplectic topology for dynamical purposes, if the lift $\widetilde{\om}$ of $\om$ to the universal cover $\t{M}$ of $M$ is exact and has a primitive $\la,$ $\widetilde{\om} = d\t{\la},$ whose norm $||\t{\la}||$ with respect to a metric $\t{g}$ lifted from a Riemannian metric $g$ on $M$ is a bounded function on $\t{M}$; see \S\ref{sec:autonomous_ham_higher_dim} for more details. For example, products of Riemann surfaces of genera at least two are symplectically hyperbolic.

In Section \ref{sec:autonomous_ham_higher_dim} we prove the following result as Theorem \ref{mainthm:auto_on_symp_hyperb}.

\begin{alphatheorem}\phantomsection\label{thm: hyp}
Let $(M,\om)$ be a closed symplectically hyperbolic manifold and $H \in C^\infty(M,\R)$ a non-constant function. Then there exists $\delta>0$ such that \[\gamma(\phi^t_H) \geq \delta \] for all $t \geq 1,$ where $\gamma$ is the spectral norm taken with arbitrary coefficients. In particular, $||\phi^t_H ||_{\Hof} \geq \delta$ for all $t \geq 1.$      
\end{alphatheorem}

It would be interesting to extend this result to arbitrary symplectically aspherical manifolds, which we plan to investigate in the future. This would likely prove the full conjecture by a stabilization construction; however, our methods do not appear to apply to this case directly at the moment.

Beyond the aspherical case, we note that $H \in C^\infty(M,\R)$ appearing as functions of integrals in an integrable Hamiltonian system generally have many invariant sets. For instance, if $(M,\om)$ is a symplectic toric manifold, with $T^n = (\R/\Z)^n$ acting on $M$ with moment map $\mu: M \to \R^n,$ every torus fiber $L_{\eta} := \mu^{-1}(\eta)$ is an invariant set of the flow of every Hamiltonian $H = h \circ \mu$ on $M$ given by a smooth function $h$ on $\R^n.$

Note that if $h(v) = \langle \alpha, v \rangle + b$ for $\alpha \in (\R^n)^\vee$ and $b \in \R$ is an affine function, then $(\phi^t_H)$ for $t \geq 1$ is not separated from the identity. Indeed, this flow generates the action of a closed subgroup $S$ of $T^n,$ whose identity component is a subtorus $T.$ Hence there is a sequence $t_k \in [1,\infty),$ $t_k \to \infty$ as $k \to \infty,$ with $\phi^{t_k}$ corresponding to the action of $\tau_k \in T$ with $\tau_k \to 0$ in $T.$ Therefore $\phi^{t_k} \to \id$ in $C^\infty$ topology and hence in all the norms we have discussed. We prove under the additional assumption that $(M,\om)$ is monotone, that is $[\om]|_{\pi_2(M)} = \kappa \cdot c_1(TM)|_{\pi_2(M)},$ for $\kappa > 0,$ that the above $h$ is the only possible such example; see Theorem \ref{thm:full_int_case_proj_space1} in Section \ref{sec:autonomous_ham_higher_dim}. This result resolves a greatly generalized version of the last part of \cite[Problem 7.9]{cggm-topics}.

\begin{alphatheorem}\phantomsection\label{thm: mon tor}
Let $(M,\om)$ be a closed monotone symplectic toric manifold with moment map $\mu: M \to \R^n,$ and $H = h\circ \mu \in C^\infty(M,\R)$ is a non-affine function of the moment map, that is $\nabla h$ is not constant on the image polytope of $\mu$. Then there exists $\delta>0$ such that \[\gamma(\phi^t_H) \geq \delta \] for all $t \geq 1,$ where $\gamma$ is the spectral norm taken with arbitrary coefficients. In particular, $||\phi^t_H ||_{\Hof} \geq \delta$ for all $t \geq 1.$     \end{alphatheorem}

While our exact methods do not appear to extend far beyond the monotone case, we expect that the same statement should apply to arbitrary closed symplectic toric manifolds, which we will investigate in a future work.

Finally we remark that the results on the spectral norm equally apply to the $C^{0}$-norm in the cases where the spectral norm is known to be $C^{0}$-continuous. This is shown for surfaces by Seyfaddini \cite{MR3123670}, for all symplectically aspherical manifolds by Buhovsky-Humilière-Seyfaddini \cite{MR4263685}, for negative monotone symplectic manifolds by Kawamoto \cite{MR4504915}, and for complex projective spaces by Shelukhin \cite{MR4480149}.

\subsection*{Outlines of the proofs}
We briefly outline the proofs of our main results and highlight the new ideas involved.

\subsubsection*{Theorem \ref{thm:main_theorem_introduction}}

The proof of this result consists of a number of steps. First, we rely on a result of \c{C}ineli-Ginzburg-G\"urel \cite[Theorem 3.1, Example 3.2]{cgg-generic-behaviour} stating that the theorem holds if $\phi$ has positive topological entropy.

Therefore we assume that $\phi$ has zero topological entropy. In this case, a result of \cite{GLCPT-25} produces an open $\phi$-invariant topological annulus $A$ in $\Sigma.$ 

If $A$ is essential in $\Sigma,$ a variation on a classical argument of Lalonde, McDuff, and Polterovich, which is however not an immediate generalization, lets us prove that $\gamma(\phi^k) \to \infty$ as $k \to \infty$ (see Theorem \ref{mainthm:ess_smth_inv_ann}). The proof involves a covering argument for the spectral norm as well as a delicate argument related to rotation numbers and residual finiteness of surface groups.

The main challenge of the proof is the case where $A$ is inessential. Our approach to this case is twofold: a solution of the case of a smooth invariant annulus, using Floer theory and ergodic theory, and a reduction of the general case to this one by means of $C^0$ symplectic topology and continuity arguments both on the Floer side and on the classical dynamics side. We now describe these two steps in slightly greater detail.

First, we solve the case where $A$ is the interior of a closed embedded annulus with smooth boundary (see Theorem \ref{mainthm:smth_cont_ann_case}). The proof proceeds by first showing, in the spirit of \cite[Theorem E]{kislev-shelukhin-21}, that if $\gamma(\phi^k)$ is much smaller than the areas of the disks bounded by the boundary circles $L, L'$ of $A,$ then $L,L'$ have Hamiltonian chords of a flow generating $\phi^k$ with equal index, say index zero, and actions which differ by roughly $\gamma(\phi^k).$ On the other hand, by a small perturbation we may assume that $\phi,$ and hence $\phi^k,$ admits a contractible fixed point on each of $L, L'.$ We then use ergodic theory and index calculations to relate the difference of actions of such chords to the integral of the rotation number of $\phi^k$ inside the annulus (see Theorem \ref{mainthm:diff_act_chords_vs_ave_rot_num}). Finally, a theorem of Franks and Handel \cite[Proposition 2.6]{franks-handel-zero-ent} (see also \cite{franks1988generalizations}) together with a direct inspection of our formula implies that the action difference is non-zero and grows with the iterations of $\phi.$ This provides a contradiction which finishes the argument. 

Second, for a general open invariant annulus $A,$ by work of Franks and Handel \cite{franks-handel-zero-ent}, passing to a sufficiently large prime iteration of $\phi,$ we may reduce consideration to the ``regular" case where the complement of $A$ consists of two connected components with non-empty interior; see \S\ref{sec:non_smooth_non_reg_non_ess_inv_ann}. We then construct, by methods of $C^0$ symplectic topology explained in great details in \S\ref{sec:fragmentation_c0_approximation}, two families of Lagrangian circles $L_p, L'_p$ in $A$ converging to the boundary components of its prime-end compactification, and constituting the boundary components of a nested sequence of sub-annuli $A_p$ of $A,$ as well as a sequence of Hamiltonian diffeomorphisms $\psi_p$ supported in $A$ with $\psi_p(A_p) = \phi(A_p)$ for all $p \geq 1$ and $\gamma(\psi_p) \to 0$ as $p \to \infty$ while $\psi_p$ tends to the identity in the $C^0$ norm on the prime-end compactification, in a suitable sense; this is done in Theorem \ref{mainthm:approx_smth_inv_annulus} by an elaborate variation on the arguments of \cite{MR2884036} incorporating the theory of prime end compactification. 
Moreover by a small perturbation, we may assume that $\phi_p = \psi_p^{-1}\phi$ has a contractible fixed point on $L_p, L'_p$; see Lemma \ref{lem:pert_to_make_fixed_points}. Therefore the smooth case applies to $\phi_p,$ and we prove that it does so uniformly in $p$ by suitable continuity statements for the $\gamma$-norm and the total rotation number. The latter is proved in Lemma \ref{lem:mean_rot_vec_non_zero_forallp} which uses a deep result of Franks and Handel \cite[Theorem 1.4 (1)]{franks-handel-zero-ent} on continuity of the rotation number function; see also the work of Le Calvez and Tal \cite{lecalvez-tal-22}. This provides a contradiction for $p$ sufficiently large. Curiously, this argument appears to be one of the first applications of $C^0$ symplectic topology to smooth symplectic dynamics.





\subsubsection*{Theorem \ref{thm: hyp}}

We modify $H$ near the invariant sublevel set $X$ of $\phi^t = \phi^t_H$ near its minimum and superlevel set $Y$ near its maximum, which are independent of $t,$  in such a way that it becomes constant in those sets without changing the spectral norm of $\phi^t.$ Then we modify $\phi^t$ again in these regions so that it displaces a fixed ball, again without changing the spectral norm. The result follows by an application of Usher's energy-capacity inequality for the spectral norm \cite{usher-bound-depth-13}. 

The possibility to make these modifications relies on action estimates for $t$-periodic orbits of $H$ in the invariant sets $X$ and in $Y$ of the Hamiltonian flow, which are made possible by the symplectic hyperbolicity condition. See Theorem \ref{mainthm:auto_on_symp_hyperb} and Proposition \ref{lem:just_const}.

\subsubsection*{Theorem \ref{thm: mon tor}}

The proof consists in a combination of the argument for the smooth case of Theorem \ref{thm:main_theorem_introduction} producing two Lagrangian fibers with chords of the same index and action difference $D$ at most $\gamma(\phi^k)$, together with an explicit calculation of these actions in terms of the function $h.$ We choose one of the Lagrangians to be the monotone toric fiber, and the other one to be the fiber over a point near which $\nabla h$ is not constant and, by a bit of number theory, arranged so that the difference $D$ is bounded from below uniformly in $k.$ This finishes the proof. See Theorem \ref{thm:full_int_case_sphere} for the case of $S^{2}$ and Theorem \ref{thm:full_int_case_proj_space1} for general monotone toric manifolds.

Finally, we also prove a partial result, Theorem \ref{thm:comp_int_sys}, of a similar kind, which applies to completely integrable Hamiltonian systems beyond symplectic toric manifolds. The methods are based on a truncated version of the module action of Hamiltonian Floer homology on the Lagrangian Floer homology developed in \cite{albers-10,biran-cornea-09,MR4879080, cant-zhang-24}.

\subsection*{Organization of the paper}

In Section \ref{sec:lag_ham_floer_theory} we recall preliminaries on Lagrangian and Hamiltonian Floer theory, which we use along the whole paper. Section \ref{sec:indices} discussing indices and Section \ref{sec:module_action} discussing the module action of Hamiltonian Floer theory on Lagrangian Floer theory in small action windows are particularly important. Sections \ref{sec:zero_entropy_franks_handel} and \ref{sec:theory_prime_ends} present preliminaries on low-dimensional dynamics: the theory of Franks and Handel regarding diffeomorphisms of the sphere of zero topological entropy, a discussion of mean rotation vector and its implications to the existence of fixed points in particular, and a brief presentation of the theory of prime end compactification.
Section \ref{sec:fragmentation_c0_approximation} proves necessary $C^0$ approximation and fragmentation results in order to produce approximately invariant Lagrangian circles. The proof of Theorem \ref{thm:main_theorem_introduction} is contained in Section \ref{sec:proof_main_theorem}. The proofs of the higher-dimensional results, Theorems \ref{thm: hyp} and \ref{thm: mon tor} appear in Section \ref{sec:autonomous_ham_higher_dim}. 

\section*{Acknowledgments}
We thank Leonid Polterovich for numerous discussions on related subjects and for suggesting to investigate invariant sets as a unifying principle after we obtained preliminary results in this project. We also thank Dan Cristofaro-Gardiner and Abror Pirnapasov for a fruitful collaboration on a different question in low-dimensional symplectic dynamics which has improved our understanding of classical techniques related to entropy zero Hamiltonian diffeomorphisms. Finally, we thank the Institute for Advanced Study for its wonderful research atmosphere during the second semester of 2025 when work on this project began. 
H.A. was partially supported by National Key R\&D Program of China Grant No. 2023YFA1010500 and NSFC Grants No. 12301081. E.S. was partially supported by an NSERC Discovery grant, by the Courtois chair in fundamental research, by a Sloan research fellowship, by an FRQNT teams grant, and, at the IAS, by the National Science Foundation under Grant No. DMS-1926686 during the academic year 2024-2025 as Von Neumann fellow and by an anonymous donor during the academic year 2026-2027.

\section{Lagrangian and Hamiltonian Floer theory}\phantomsection\label{sec:lag_ham_floer_theory}
Here, we briefly describe filtered Lagrangian Floer homology, one of the main ingredients in the proof the main theorem; for a detailed account of the constructions we refer to \cite{kislev-shelukhin-21, MR3850107, MR3524783}. The ambient symplectic manifold $(M,\omega)$ is always closed and monotone, $[\omega] = \lambda_{M} c_{1}$ as classes in $\H^{2}(M,\R)$ for some $\lambda_{M} > 0$, and the Lagrangian submanifolds are assumed to be monotone, $[\omega] = \lambda_{L}\mu_{L}$ as classes in $\H^{2}(M,L, \R)$ for some $\lambda_{L} > 0$ where $\mu_{L}$ is the Maslov class; see \S\ref{sec:indices}. It is easy to check that $\lambda_{M} = 2\lambda_{L}$. For $H: M \times S^{1} \rightarrow \R$ the Hamilton's equation is $\omega(X_{H}, -) = -dH$ and the generated Hamiltonian flow is denoted by $\phi_{H,t}$ and the time-one map by $\phi_{H}$.

\subsection{Hamiltonian Floer homology}\phantomsection\label{sec:abs_ham_floer_theory}
Fix $H: M \times S^{1} \rightarrow \R$. Let $LM$ be the space of contractible free loops in $M$ and let $\t{L}M$ be its universal cover. Consider $c: \pi_{1}(LM) \rightarrow N\Z$ be the homomorphism defined by $c(\alpha) = c_{1}(A_{\alpha})$ where $A_{\alpha}\in \H_{2}(M,\Z)$ is determined by $\alpha$, see \S\ref{sec:indices}; here $N > 0$ is the minimal positive number so that the map $c$ is surjective. Consider the following covering of $LM$,
$$\t{L}M_{c_{1}}:= \t{L}M \times_{c} N\Z;$$
one can identify this covering with the space of pairs $(x,\xi)$ where $x$ is a free loop in $M$ and $\xi: \D \rightarrow M$ is a map with $\xi|_{\partial\bar\D} = x$, where $(x,\xi)$ is identified with $(y,\eta)$ if and only if $x = y$ and $c_{1}(\bar{\xi}\# \eta) = 0$ with $\bar{\xi}$ being $\xi$ with reverse orientation. Define the action functional $\mathcal{A}_{H}: \t{L}M_{c_{1}} \rightarrow \R$ by
$$\mathcal{A}_{H}(x,\xi) = \int_{S^{1}}H_{t}(x(t))dt - \int_{\D}\xi^{*}\omega.$$
The set critical points $\t{O}(H)$ of $\mathcal{A}_{H}$ is consist of elements $[x,\xi]$ where $\dot{x}(t) = X_{H}(x(t))$. Assume $H$ is non-degenerate, the graph $\phi_{H}$ intersects the diagonal $\Delta\subset M\times M$ transversely. The Floer complex $\CF(H)$ is a free $\Z/2\Z$-module generated by $\t{O}(H)$. It is graded by the Conley-Zehnder index \S\ref{sec:indices}. To define the Floer differential one picks a generic family of compatible almost complex structures $J = \{J_{t}\}_{t\in [0,1]}$ and defines $d[x,\xi] = \sum n([x,\xi], [y,\eta])[y,\eta]$ where the coefficient is the number of isolated solutions to the following asymptotic boundary value problem, which can be thought of as a formal negative gradient equation for $\mathcal{A}_{H}$,
\begin{enumerate}
\item $\partial_{s}u(s,t) + J_{t}(\partial_{t}u(s,t) - X_{H}(u(s,t))) = 0$,
\item $u(-\infty, t) = x(t), u(+\infty, t) = y(t),$
\item $\xi \# u = \eta$.
\end{enumerate}
The Floer complex is denoted by $\CF(H,J)$ and its homology is denoted by $\HF(H,J)$. One can take completion of $\CF(H,J)$ in the direction of decreasing actions to turn it into a free graded module of finite rank over the Novikov field $\K[q^{-1}, q]]$ where $\K = \Z/2\Z$ and $\deg q = -2N$. For any real number $t$ we define the filtered Floer homology $\HF^{t}(H,J)$ to be the homology of the Floer complex $\CF^{t}(H,J)$ that is consist of all capped periodic orbits $[x,\xi]$ whose action is smaller than $t$. Similarly we define $\CF^{[a,b]}(H,J)$ for any interval $[a,b]$ with $a,b\not\in \mathrm{Spec}(H)$ where the latter is the set of critical values of $\mathcal{A}_{H}$. It is well-known that the graded filtered Floer homology only depends on the homotopy class of the path $\{\phi_{H,t}\}_{t\in [0,1]} \subset \t{\Ham}(M,\omega)$, the latter is the universal cover of $\Ham(M,\omega)$, and so we will denote the filtered Floer homology for $(H,J)$ in the interval $[a,b]$ by $\HF^{[a,b]}([\phi_{H,t}])$. In some cases, for instance \cite{schwarz-00} when $\omega = c_{1} = 0$, the homology only depends on the time-one map $\phi_{H}$ in which case we denote it by $\HF^{[a,b]}(\phi_{H})$ where we additionally assume that $H$ is normalized to have zero mean.

One can extend the definition to all possibly degenerate Hamiltonians $H'$ by letting $\HF^{[a,b]}(H') =\lim_{n\to \infty} \HF^{[a,b]}(H_{n}, J)$ where the limit is over any sequence of non-degenerate Hamiltonians $H_{n}$ converging to $H$ in $C^{\infty}$-topology and $J$ is a generic family of compatible almost complex structures, and the limit is via continuation maps, 
$$c_{n}: \HF^{[a,b]}(H_{n}, J) \rightarrow \HF^{[a,b]}(H_{n+1}, J);$$ 
the latter is defined by counting isolated solutions to Floer's equation for generic $s$-dependent Hamiltonians $G_{n,s}$ which coincide with $H_{n}$ on $s \leq 0$ and with $H_{n+1}$ on $s \geq 1$.
\subsubsection{Pair-of-pants product and product of units}\phantomsection\label{sec:pair_of_pants_unit_times_unit} The Floer homology group has a product structure called the pair of pants product. In general, given Hamiltonian functions $H$ and $G$ and a family of almost complex structures $J = \{J_{z}\}_{z\in \Sigma_{3}}$, where $\Sigma_{3} = \C\backslash\{0,1\}$, there is chain map,
$$\p\p: \CF(H,J) \otimes \CF(G,J) \rightarrow \CF(H\# G, J),$$
that is defined by $\p\p([x,u], [y,v]) = \sum n([x,u], [y,v], [z,w])[z,w]$ where $[x,u]\in \t{O}(H), [y,v]\in \t{O}(G)$ and the sum is over $[z,w]\in \t{O}(H\#G)$ and the coefficient is the number of isolated solutions to the following asymptotic boundary value problem,
\begin{enumerate}
\item $f:\Sigma_{3} \rightarrow M$ such that $(df - X_{\mathfrak{a}}\circ f)^{(0,1)} = 0$.
\item $f\circ \eps_{1,-}(-\infty,t) = x(t), f\circ \eps_{2,-}(-\infty,t) = y(t), f\circ \eps_{+}(\infty,t) = z(t)$.
\item $w = u\#v\#f$.
\end{enumerate}
Here $\mathfrak{a}\in \Omega^{1}(\Sigma_{3}, C^{\infty}(M))$ is a Hamiltonian connection one form with zero curvature so that on the image of the coordinate chart $\eps_{1,-}, \eps_{2,-}, \eps_{+}$ it is of the form $H_{t}dt, G_{t}dt, (H\#G)_{t}dt$ respectively and $J$ is a generic family of compatible almost complex structures that are $s$-independent near the punctures, and finally,
\begin{enumerate}
\item $\eps_{1,-}, \eps_{2,-}: (-\infty, 0) \times S^{1} \rightarrow \Sigma_{3}$ are holomorphic cylindrical neighborhoods of the negative punctures.
\item $\eps_{+}: (0, +\infty) \times S^{1} \rightarrow \Sigma_{3}$ is a holomorphic cylindrical neighborhood of the positive puncture.
\end{enumerate}
Since the curvature of the connection one form $\mathfrak{a}$ vanishes one can show that if there exists a solution as above then,
$$E(f) = \mathcal{A}_{H}([x,u]) + \mathcal{A}_{G}([y,v]) - \mathcal{A}_{H\#G}([z,w]) \geq 0.$$
The moduli space of solutions $f: \Sigma_{3} \rightarrow M$ as above are smooth manifolds of finite dimension for generic data $(H,G,J)$ and the dimension equals,
$$\mathrm{CZ}([x,u]) + \mathrm{CZ}([y,v]) - \mathrm{CZ}([z,w]) - n.$$
Hence the pair-of-pants product defined as above induces the following product structure on the corresponding Floer homologies,
$$\p\p: \HF_{k}(H) \otimes \HF_{l}(G) \rightarrow \HF_{k+l-n}(H\#G).$$
Through standard analysis in Floer theory one can show that the product $\p\p$ corresponds to the quantum product on quantum homology via $\pss$-isomorphism \S\ref{sec:pss_isomorphisms} and hence we have the identity, 
$$\p\p(\pss_{H}([M]), \pss_{G}([M])) = \pss_{H\#G}([M]),$$
that will be crucially used in the proof of the main theorem.
\subsection{Lagrangian Floer homology}\phantomsection\label{sec:lag_floer_theory}
Fix a Hamiltonian $H$ and a monotone closed Lagrangian submanifold $L$. Let $PL$ be the space of free paths from $L$ to itself that are contractible relative to $L$. Let $\mu: \pi_{1}(PL) \rightarrow N_{L}\Z$ be defined by $\mu(\alpha) = \mu_{L}(D_{\alpha})$ where $D_{\alpha}\in \H_{2}(M,L,\Z)$ is the class determined by $\alpha$ and $\mu_{L}$ is the Maslov class \S\ref{sec:indices}; here $N_{L}$ is the smallest positive number for which $\mu$ is surjective. Let $\t{P}L$ be the universal cover of $PL$ and define
$$\t{P}L_{\mu} = \t{P}L \times_{\mu_{L}} N_{L}\Z.$$
One can identify $\t{P}L_{\mu}$ with the space of pairs $(x, \xi)$ where $x$ is a path from $L$ to itself and $\xi: \D \cap \mathbb{H} \rightarrow M$ so that $\xi(\partial\D \cap \mathbb{H})\subset L$ and $\xi|_{\D \cap \R} = x$, here $\mathbb{H} = \{z\in \C : \mathrm{Im}(z) \geq 0\}$, $\D = \{|z|\leq 1\}$, and $(x,\xi)$ is identified with $(y,\eta)$ if and only if $x = y$ and $\mu_{L}(\bar{\xi}\#\eta) = 0$ where $\bar{\xi}$ is $\xi$ with reverse orientations. The action functional $\mathcal{A}_{L,H}: \t{P}L_{\mu} \rightarrow \R$ is given by
$$\mathcal{A}_{L,H}([x,\xi]) = \int_{0}^{1}H_{t}(x(t)) - \int_{\D\cap \mathbb{H}} \xi^{*}\omega,$$
which is well-defined by monotonicity assumption. The critical points of $\mathcal{A}_{L,H}$ are the pairs $[x,\xi]$ where $x:[0,1] \rightarrow M$ so that $x(0),x(1)\in L$ and $\dot{x}(t) = X_{H}(x(t))$ and the set of critical points is denoted by $\t{O}(H,L)$. Assume that the pair $(H,L)$ is non-degenerate, i.e., $\phi_{H}(L)$ intersects $L$ transversely. Then, the Lagrangian Floer complex $\CF(L, H)$ is a free $\Z/2\Z$-module generated by critical points of $\mathcal{A}_{L,H}$. To define the differential one picks a generic family of compatible almost complex structures $J = J_{t}$ and defines $d[x,\xi] = \sum n([x,\xi], [y,\eta])[y,\eta]$ where the coefficient is the number of isolated solutions of the following asymptotic boundary value problem, 
\begin{enumerate}
\item $\partial_{s}u(s,t) + J_{t}(\partial_{t}u(s,t) - X_{H}(u(s,t))) = 0$,
\item $u(-\infty, t) = x(t), u(+\infty, t) = y(t),$
\item $u(s,0), u(s,1) \in L$ for all $s\in\R$,
\item $\xi \# u = \eta$;
\end{enumerate}
this can be thought of a formal negative gradient flow equation for $\mathcal{A}_{L,H}$. The Lagrangian Floer complex is denoted by $\CF(L,H; J)$ and its homology is denoted by $\HF(L, H; J)$. It is graded by the Maslov index \S\ref{sec:indices}. One can again consider the completion of the complex in the direction of decreasing actions and turn it into a finite rank over the Novikov field $\K[p^{-1}, p]]$ where $\deg p = -N_{L}$ and $\K = \Z/2\Z$. For every real number $t$ we define the filtered Lagrangian Floer homology $\HF^{t}(L,H; J)$ to be the homology of the filtered complex $\CF^{t}(L,H; J)$ that is consist of capped chords $[x,\xi]$ with action smaller than $t$. Similarly for an interval $[a,b]$ with $a,b\not\in \mathrm{Spec}(L, H)$ we define $\HF^{[a,b]}(L, H; J)$; here $\mathrm{Spec}(L,H)$ is the set of critical values of $\mathcal{A}_{H}$. It is well-known that the graded filtered Lagrangian Floer homology depends only on the homotopy class of the path $\{\phi_{H,t}\}_{t\in [0,1]} \in \t{\Ham}(M,\omega)$ and so we denote the filtered Lagrangian Floer homology by $\HF^{[a,b]}(L, [\phi_{H,t}])$. In some particular cases, for example when $\omega = c_{1} = 0$, the graded filtered Lagrangian Floer homologies will only depend on the time-one map $\phi_{H}$ as long as $H$ is mean normalized;
see \cite[\S3.1 and \S4.3]{schwarz-00}, \cite[Corollary 3.16]{hu-lalonde-10} and \cite[\S2.9]{kislev-shelukhin-21} for details. In the latter cases we denote the filtered homology by $\HF^{[a,b]}(L,\phi_{H})$.

It is possible to define the graded filtered Lagrangian Floer homology for degenerate pairs $(L,H)$. One picks a sequence of non-degenerate pairs $(L, H_{n})$ where $H_{n} \to H$ in $C^{\infty}$-topology and one defines, 
$$\HF^{[a,b]}(L,H; J) = \lim_{n\to \infty}\HF^{[a,b]}(L,H_{n}; J),$$ 
where $J$ is a generic family of compatible almost complex structures and the limit is via continuation maps 
$$c_{n}: \HF^{[a,b]}(L,H_{n}; J) \rightarrow \HF^{[a,b]}(L,H_{n+1}; J).$$ 
The latter are defined by counting isolated solutions to Floer's equation as above for a generic $s$-dependent Hamiltonians $G_{n,s}$ which coincide with $H_{n}$ on $s \leq 0$ and with $H_{n+1}$ on $s \geq 1$. The continuation map $c_{n}$ shift the action intervals by $|H_{n} - H_{n+1}|_{C^{0}}$ which for large enough $n$ will be small enough.

\subsubsection{Lagrangian intersection Floer homology}\phantomsection\label{sec:anchored_lag_floer_hom}
Fix a reference point $w\in M$ and a Lagrangian subspace $\Lambda_{w}\subset T_{w}M$. We follow \cite{MR3524783, MR4480149} and define an anchored Lagrangian brane $\underline{L} = (L, \alpha, \lambda)$ to be consist of a closed monotone Lagrangian submanifold $L$, a path $\alpha: [0,1] \rightarrow M$ with $\alpha(0)= w$ and $\alpha(1) = z\in L$ for some point $z\in L$, and $\Lambda$ is a section of the Lagrangian Grassmannian bundle $\alpha^{*}\mathrm{Lag}(M,\omega) \rightarrow [0,1]$ so that $\Lambda(0) = \Lambda_{w}$. Given two anchored Lagrangians $\u{L}$ and $\u{L}'$ the path,
$$\alpha_{\u{L},\u{L}'} = \overline{\alpha} \# \alpha' = \begin{cases}
\alpha(1 - 2t), t\in [0,1/2 ],\\
\alpha'(2t - 1), t\in [1/2,1],
\end{cases}$$
determines a homotopy class in the space of paths $P(L,L')$ from $L$ to $L'$; we denote by $P(\u{L}, \u{L}')$ the space of paths from $L$ to $L'$ that are homotopic to $\alpha_{\u{L}, \u{L}'}$. Assume that $N_{L} = N_{L'}$, $\lambda_{L} = \lambda_{L'}$, and the counts of $J$-holomorphic disks with boundary on $L$ and Maslov index $2$ and passing through a generic point in $L$ is the same as that of $L'$, for a generic almost complex structure $J$; this is the case when $L' = \phi_{H}(L)$ for some Hamiltonian $H$. Let $\t{P}(\u{L}, \u{L}')$ to be the universal cover of $\t{P}(\u{L}, \u{L}')$ and define
$$\t{P}(\u{L}, \u{L}')_{\mu} = \t{P}(\u{L}, \u{L}') \times_{\mu_{\u{L},\u{L}'}} N_{L,L'}\Z,$$
where $N_{L,L'} = N_{L} = N_{L'}$ and $\mu_{\u{L}, \u{L}'}: \pi_{1}(P(\u{L}, \u{L}')) \rightarrow N_{L,L'}\Z$ is defined by $\mu_{\u{L}, \u{L}'}(\alpha) = \mu_{L, L'}(C_{\alpha})$ where $C_{\alpha}\in \H_{2}(M, L \cup L', \Z)$ is represented by $\alpha$ and $\mu_{L,L'}$ is defined in \S\ref{sec:indices}. The space $\t{P}(\u{L}, \u{L}')_{\mu}$ can be identified with the space of pairs $(x, \xi)$ where $x$ is a path from $L$ to $L'$ and $\xi: [0,1]_{s} \times [0,1]_{t} \rightarrow M$ so that $\xi(\{t=0\})\subset L, \xi(\{t=1\})\subset L', \xi(0,t) =\alpha_{\u{L}, \u{L}'}(t)$ and $\xi(1,t)=x(t)$, see Figure \ref{fig:capped_int_point}, and $(x, \xi)$ is identified with $(y,\eta)$ if $x = y$ and $\mu_{L, L'}(\bar{\xi}\#\eta) = 0$.
\begin{figure}[ht]
  \centering
  \begin{tikzpicture}
   \draw (0,0) -- (2,0) -- (2,1) -- (0,1) -- cycle;
   \node at (-0.2, 0.5) {\tiny{$\alpha$}};
   \node at (2.2, 0.5) {\tiny{$x$}};
   \node at (1, 1.2) {\tiny{$L'$}};
   \node at (1, -0.2) {\tiny{$L$}};
   \node at (1, 0.5) {\tiny{$\xi$}};
  \end{tikzpicture}
  \caption{A capped path from $L$ to $L'$ where $\alpha = \alpha_{\u{L}, \u{L}'}.$}
  \phantomsection\label{fig:capped_int_point}
\end{figure}
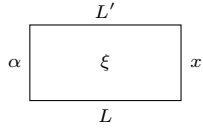
The symplectic action functional $\mathcal{A}_{\u{L}, \u{L}'}: \t{P}(\u{L}, \u{L}')_{\mu} \rightarrow \R$ is defined by 
$$\mathcal{A}_{\u{L}, \u{L}'}([x,\xi]) = -\int_{[0,1]^{2}} \xi^{*}\omega.$$
The critical point of $\mathcal{A}_{\u{L}, \u{L}'}$ are the pairs $[x,\xi]$ where $x\equiv p\in L \cap L'$. Assume that $L$ and $L'$ intersect transversely. The Lagrangian intersection Floer complex is denoted by $\CF(\u{L},\u{L}')$ and is a free $\Z/2\Z$-module generated by the capped intersection points. For the differential one picks a generic family of compatible almost complex structures $J = \{J_{t}\}_{t\in [0,1]}$ and one defines $d[p, \xi] = \sum n([p,\xi], [q,\eta])[q,\eta]$ where the coefficient is the number of isolated solutions to the following asymptotic boundary value problem,
\begin{enumerate}
\item $\partial_{s}u(s,t) + J_{t}\partial_{t}u(s,t) = 0$,
\item $u(-\infty, t) = p, u(+\infty, t) =q,$
\item $u(s,0)\in L, u(s,1) \in L'$ for all $s\in\R$,
\item $\xi \# u = \eta$.
\end{enumerate}
The Lagrangian intersection Floer complex is denoted by $\CF(\u{L}, \u{L}'; J)$ and its homology is denoted by $\HF(\u{L}, \u{L}'; J)$. This homology is also graded by the Maslov indices \S\ref{sec:indices}. The completion of the complex with decreasing actions turns it into a finite rank module over the Novikov field $\K[q^{-1}, q]]$ with $\deg q = -N_{L,L'}$. Similarly one filters the complex as before and denotes the filtered homology by $\HF^{[a,b]}(\u{L}, \u{L}'; J)$. The graded filtered Floer homologies do not depend on the choice of $J$ and so $J$ will be omitted from the notation.

One can define the homology for non-transverse anchored Lagrangians $\u{L}, \u{L}'$ by picking a sequence of Hamiltonians $H_{n}$ so that $L$ and $\phi_{H_{n}}^{-1}(L')$ intersect transversely and $|H_{n}|_{C^{\infty}} \to 0$ as $n \to \infty$ and one defines $\HF^{[a,b]}(\u{L}, \u{L}'; J) = \lim_{n\to \infty}\HF^{[a,b]}(\u{L}, \u{L}'; H_{n}, J)$ where $J$ is some generic family of compatible almost complex structures and the limit is via continuation maps and the latter homology is defined as follows. The generators of $\CF(\u{L}, \u{L}'; H,J)$ are the pairs $(x, \xi)$ where $x$ is a Hamiltonian chord of $H$ from $L$ to $L'$ and $\xi: [0,1]_{s} \times [0,1]_{t} \rightarrow M$ is a capping of $x$, i.e., $\xi(0,t) = \alpha_{\u{L}, \u{L}'}(t), \xi(1,t) = x(t), \xi(s,0)\subset L$ and $\xi(s,1) \subset L'$; see Figure \ref{fig:capped_int_point}. The complex is filtered by the action functional $\mathcal{A}_{\u{L},\u{L}',H}: \t{P}(\u{L},\u{L}'; H) \rightarrow \R$ defined by
$$\mathcal{A}_{\u{L},\u{L}', H}([x,\xi]) = \int_{0}^{1}H_{t}(x(t)) - \int_{[0,1]^{2}} \xi^{*}\omega,$$
where the domain consist of pairs $(x,\xi)$ up to an equivalence relation similar to above. The continuation maps are defined similarly.

For simplicity, we denote $CF(\u L, H) = CF(\u L, \u L, H).$ It is easy to see by observing that $\alpha_{\u L, \u L}$ together with the anchor data over it is homotopic to the trivial path at $q\in L$ and $T_q L$ over it, that for an arbitrary choice of anchored Lagrangian $\u L$ with underlying Lagrangian $L,$ \[CF(\u L, H) = CF(L, H)\] in the sense of being naturally chain-isomorphic preserving action and index. To be more precise, pick $(x,u)\in \CF(\u L, H)$ where $u: [0,1]\times [0,1] \rightarrow M$, $u(1,t) = x(t), u(0,t) = \alpha_{\u L, \u L}(t)$ and $u(s,0),u(s,1)\subset L$. Define $v:[0,1]\times [0,1]\rightarrow M$ by $v(s,t) = \alpha_{\u L, \u L}(st)$. Consider the map
$$\mathrm{S}: \CF(\u L, H)\rightarrow \CF(L, H); (x,u)\mapsto (x, v \# u).$$
Then, it is obvious that the transformation $\mathrm{S}$ preserves both actions and indices; see \S\ref{sec:indices} for the definition of the index of a capped chord.
\subsubsection{The transform map}\phantomsection\label{sec:transform_map}
We now describe a transform map that induces an isomorphism between the two types of Lagrangian Floer homologies introduced above. Let $\u{L}$ be an anchored Lagrangian with $w = q\in L$ and constant anchor $\alpha \equiv q, \Lambda(t) \equiv T_{q}L$ for some point $q\in L$, and $H$ be a Hamiltonian with flow $\phi_{H,t}$. We define
$$H_{*}\u{L}:= (\phi_{H}(L), \phi_{H,t}(\alpha(t)), d\phi_{H,t}(\Lambda(t))).$$
The transform map is a chain map,
$$\Tr: \CF(\u{L},H) \rightarrow \CF(\u{L}, \bar{H}_{*}\u{L}),$$
that on capped chord $(x, \xi)$ is defined as follows,
$$\Tr(x,\xi) = (x(0), \eta);\ \eta(s,t) := \phi_{H,t}^{-1}(\xi(s,t)).$$
\begin{figure}[ht]
    \centering
    \begin{tikzpicture}
    \draw (0,0) -- (2,0) -- (2,1) -- (0,1) -- cycle;
    \node at (-0.2, 0.5) {\tiny{$q$}};
    \node at (2.6, 0.5) {\tiny{$\phi_{H,t}(p)$}};
    \node at (1, 1.2) {\tiny{$L$}};
    \node at (1, -0.2) {\tiny{$L$}};
    \node at (1, 0.5) {\tiny{$\xi$}};
    
    \draw[->] (3.5,0.5) -- (4.5,0.5); 
    \node at (4, 0.7) {\tiny{$\Tr$}};
    \end{tikzpicture}
    \hspace{0.3cm}
    \begin{tikzpicture}
    \centering
    \draw (0,0) -- (2,0) -- (2,1) -- (0,1) -- cycle;
    \node at (-0.6, 0.5) {\tiny{$\phi_{H,t}^{-1}(q)$}};
    \node at (2.2, 0.5) {\tiny{$p$}};
    \node at (1, 1.3) {\tiny{$\phi_{H}^{-1}(L)$}};
    \node at (1, -0.2) {\tiny{$L$}};
    \node at (1, 0.5) {\tiny{$\eta$}};
    \end{tikzpicture}
    \caption{The transform map acting on capped chords.}
  \phantomsection\label{fig:transform_map}
\end{figure}
The transform map $\Tr$ shifts the actions as follows, for brevity let $\varphi_{t} =\phi_{H,t}$,
\begin{equation}\phantomsection\label{eq:action_shift_int_chord}
\begin{aligned}
\mathcal{A}_{\u{L}, \bar{H}_{*}\u{L}}(x(0),\eta)=\mathcal{A}_{\u{L}, H}(x,\xi) - \int_{0}^{1}H_{t}(q)dt.
\end{aligned}
\end{equation}
Under the transformation $\Tr$ the Conley-Zehnder indices are preserved. To see this, for a given capped chord $(\varphi_{t}(p), u)$ and its corresponding capped intersection point $(p, v)$, it is enough to find some symplectic trivializations over $u$ and $v$ so that the indices coincide. We choose a trivialization $\Phi: \R^{2} \rightarrow T_{p}\Sigma$ and extend it over $\varphi_{t}(p)$ by $d\varphi_{t}(p)\circ \Phi$, and further extend it to the domain of $u$ along straight lines by parallel transport and interpolations as in \cite[\S2]{MR3674984} and denote the trivialization by $\Phi_{u}$. Then, for $v$ we consider the trivialization $d\varphi_{t}^{-1} \circ \Phi_{u}$. Then it is easy to show that the Conley-Zehnder index of $(\varphi_{t}(p), u)$ and $(p, v)$ are the same with respect to the particular trivializations described above. Then, the must have the same index with respect to any trivializations; see \S\ref{sec:indices}.

\subsection{Naturality transformation}\phantomsection\label{sec:naturality_transformation} Let $\u{L} = (L,\alpha,\Lambda)$ be an anchored Lagrangian with constant anchor $\alpha \equiv q$ and $H$ be a Hamiltonian function generating the Hamiltonian flow $\{\phi_{t}\}$. Suppose $G: M \times \R \rightarrow \R$ generates a loop of Hamiltonian diffeomorphisms $\{\psi_{t}\}$. Choose a fixed capping $w$ for the loop $\psi_{t}(q)$. There is a naturality transformation associated with $G$ and $w$ as follows,
$$\mathfrak{n}_{G,w}:\CF(\u{L}, H) \rightarrow \CF(\u{L}, G\#H),$$
that takes a capped Hamiltonian chord $[x, u]$ to the Hamiltonian chord $[y,v]$ where $y(t) = \psi_{t}(x(t))$ and $v = (G_{*}u) \# w$ where $(G_{*}u)(s,t) = \psi_{t}(u(s,t))$. One can readily check the following identity on actions,
$$\mathcal{A}_{G\#H}([y,v]) = \mathcal{A}_{H}([x,u]) + \mathcal{A}_{G}([\{\psi_{t}(q)\},w]).$$
If the Hamiltonian loop $\{\psi_{t}\}$ is contractible and $G$ is mean normalized then one can show that the latter term is zero \cite{schwarz-00} and so the naturality transformation will preserve the actions. It is also easy to check that in this case the Conley-Zehnder index of the Hamiltonian chords are also preserved.
\subsection{Maslov indices}\phantomsection\label{sec:indices}
For a detailed exposition of the indices see \cite{conley-zehnder-83, salamon-zehnder-92, robbin-salamon-maslov-for-paths-93, MR3674984}. For a path of unitary matrices $U: [0,1] \rightarrow U(n)$ so that $\det_{\C}U(0) = \det_{\C}U(1)$ we define the Maslov index $\mu(U)$ of $U$ to be the degree of the map $\det_{\C}(U): S^{1} \rightarrow S^{1}$.

For a symplectic manifold $(M,\omega)$ one can define the first Chern class $c_{1}\in \H^{2}(M,\Z)$ by defining its value on $[\Sigma]$ where $\Sigma$ is an embedded closed surface in $M$ to be the Maslov index of the loop of symplectic linear maps $\Phi: C \rightarrow \mathrm{Sp}(2n)$ where $C$ is an splitting of $\Sigma = \Sigma_{1}\cup_{C}\Sigma_{2}$ and $\Phi = \varphi_{2}|_{C} \circ \varphi_{1}^{-1}|_{C}$ where $\varphi_{i}: TM|_{\Sigma_{i}} \rightarrow \Sigma_{i} \times \R^{2n}$ are some symplectic trivializations; see \cite[\S2.7]{MR3674984}.

Let $L \subset M$ be a closed Lagrangian submanifold. The Maslov class $\mu_{L}\in \H^{2}_{D}(M,L,\Z)$ is defined as follows. For a disk map $w: \D \rightarrow M$ with boundary on $L$, the Maslov index of $w$ is defined to be the Maslov index of the loop of unitary matrices $U(t)^{2}$ where $U(t) = X(t) + iY(t)$ and the rows of $[X^{t}\ Y^{t}]$ form an orthonormal basis for $\Phi(T_{w(e^{2\pi it})}L)$; here $\Phi: TM|_{w} \rightarrow \D\times \R^{2n}$ is a symplectic trivialization. For a cylinder map $w: [0,1]_{s} \times S^{1}_{t} \rightarrow M$ with $w(0,t)\in L_{1}, w(1,t)\in L_{2}$ for all $t$ and some fixed Lagrangian submanifolds $L_{1}$ and $L_{2}$, to define the Maslov index $\mu_{L_{1},L_{2}}(w)$, one picks a symplectic trivialization $\Phi: TM|_{w} \rightarrow [0,1] \times S^{1} \times \R^{2n}$ and defines $\mu_{L_{1},L_{2}}(w) = \mu(U_{1}(t)^{2}) - \mu(U_{2}(t)^{2})$ where $U_{1}(t),U_{2}(t)$ are defined similarly for $L_{1}, L_{2}$ respectively.

Fix a Hamiltonian $H$ with Hamiltonian flow $\phi_{H,t}$. The Conley-Zehnder index $\mathrm{CZ}([x,u])$ of a non-degenerate capped Hamiltonian orbit $[x,u]$ is so that $-\mathrm{CZ}([x,u])$ is equal to the Robbin-Salamon index of the path of Lagrangians 
$$\{L_{t}:= \{\mathrm{gr}\big(\Phi(e^{2\pi it})^{-1}d\phi_{H,t}(x(0))\Phi(1)\big)\}_{t\in [0,1]}$$ 
inside $\bar{\R}^{2n} \times \R^{2n}$ where $\Phi: \D \times \R^{2n} \rightarrow u^{*}TM$ is a symplectic trivialization. The Robbin-Salamon index of a path of Lagrangians $\{L_{t}\}_{t\in [0,1]}$ in $\R^{2n}$, relative to a fixed Lagrangian subspace $V\subset \R^{2n}$ and a Lagrangian complement $W$ of it, is defined by
$$\mu_{RS}(\{L_{t}\}) := \frac{1}{2}\mathrm{sign}(Q_{0}) + \sum_{t_{*}\in\text{cross}} \mathrm{sign}(Q_{t_{*}}) + \frac{1}{2}\mathrm{sign}(Q_{1}),$$
where the sum is over all times $t_{*}\in [0,1]$ so that $L_{t_{*}} \cap V \neq 0$ and $Q_{t_{*}}: L_{t_{*}}\cap V \rightarrow \R$ is the crossing form $Q_{t_{*}}(v):=\partial_{t}|_{t = t_{*}}\omega(v, w_{t})$ with $w_{t}$ being the unique vector in $W$ so that $v + w_{t}\in L_{t}$, and $\mathrm{sign}(Q)$ is the signature of the quadratic form $Q$. The convention for the Conley-Zehnder index of Hamiltonian orbits is chosen so that the Conley-Zehnder index of a critical point $z$ of a $C^{2}$-small Morse function with the trivial capping is equal to $\mathrm{ind}(z)- n$ where $\mathrm{ind}(z)$ is the Morse index of $z$. 

Fix an anchored Lagrangian $\u{L} = (L, \alpha, \Lambda)$ where the anchor is constant $\alpha \equiv q$ and $\Lambda(t)\equiv T_{q}L$ for some point $q\in L$. Let $(x,u)$ be capped Hamiltonian chord from $L$ to itself. The Conley-Zehnder index of this capped Hamiltonian chord is denoted by $\mathrm{CZ}([x,u])$ and is defined by the Robbin-Salamon index \cite{robbin-salamon-maslov-for-paths-93} of the loop of Lagrangians $L_{t} \subset \R^{2n}$ obtained by concatenating the following paths of Lagrangian subspaces where $\Phi: u^{*}TM \rightarrow [0,1]_{s} \times [0,1]_{t} \times \R^{2n}$ is a symplectic trivialization,
\begin{enumerate}
\item $\Phi\big(T_{u(0,t)}L\big), \Phi\big(d\phi_{H,t}(T_{p}L)\big), L_{\mathrm{can},t}, \Phi\big(T_{u(1,1-t)}L\big),$
\end{enumerate}
where $p = u(0,1)\in L$ and $L_{\mathrm{can},t}$ is a canonical "short" path of Lagrangians from $\Phi\big(d\phi_{H}(T_{p}L)\big)$ to $\Phi\big(T_{\phi_{H}(p)}L\big)$; in dimension $2$ it is just an anti-clockwise rotation of the line, see \cite{MR3220941} for exact definition. As an example, which we shall use later, let $F:W\rightarrow \R$ be the lift of a $C^{2}$-small Morse function $f: L\rightarrow \R$ to a small Weinstein neighborhood $W$ of $L$. Then, for a critical pint $p\in L$ of $f$ the Conley-Zehnder index of $(p, u)$ where $u\equiv p$ is a constant capping is equal to $n-\mathrm{ind}(p)$ where the latter is the Morse index of $p$. To see this one observes that $d\phi_{F,t}(T_{p}L)$ followed by $L_{\mathrm{can},t}$ has two crossing times with crossing forms $Q_{0}, Q_{1}:T_{p}L \rightarrow \R$,
$$Q_{0}(\xi) = \xi^{T}\mathrm{Hess}(f;p)\xi, Q_{1}(\xi) = |\xi|^{2},$$
hence the Maslov index is equal to $(n - 2\mathrm{ind}(p))/2 + n/2 = n - \mathrm{ind}(p)$. 

Let $\u{L}, \u{L}'$ be two anchored Lagrangians with connecting anchor $\alpha_{\u{L}, \u{L}'} = \bar{\alpha} \# \alpha'$. The Conley-Zehnder index of the capped intersection point $(p,v)$ is denoted by $\mathrm{CZ}([p,v])$ and is defined to be the Robbin-Salamon index of the following paths of Lagrangian subspaces where $\Phi: v^{*}TM \rightarrow [0,1]_{s} \times [0,1]_{t} \times \R^{2n}$ is a symplectic trivialization,
\begin{enumerate}
\item $\Phi\big(T_{u(0,t)}L\big), \Phi(1,t)(T_{p}L), L_{\mathrm{can},t}, \Phi\big(T_{u(1,1-t)}L'\big), \Phi\big(\bar{\Lambda}_{\u{L}, \u{L}'}\big),$
\end{enumerate}
where $L_{\mathrm{can},t}$ is the canonical short path from $\Phi(1,1)\big(T_{p}L\big)$ to $\Phi(1,1)\big(T_{p}L'\big)$ and $\Lambda_{\u{L}, \u{L}'} =\bar{\Lambda} \#\Lambda'$ is the concatenation of $\bar{\Lambda}$ with $\Lambda'$. As an example, let $F:W \rightarrow \R$ be lift of a $C^{2}$-small Morse function $f: L \rightarrow \R$ and let $\u{L}$ be the anchored Lagrangian with constant $\alpha \equiv p\in \mathrm{Crit}(f)$ and constant $\Lambda \equiv T_{p}L$. Let also $\u{L}' = \bar{F}_{*}\u{L}$. Then, the index of $(p,v)$ where $v \equiv p$ is the constant capping is equal to $n-\mathrm{ind}(p)$. The proof is similar to the above analysis; note that the connecting anchor up to homotopy is $\Lambda_{\u{L}, \u{L}'}(t) = d\phi_{F,t}^{-1}(T_{p}L)$ and $\bar{\Lambda}_{\u{L},\u{L}'}$ contributes to the index.

\subsection{Fukaya's trick and the $\hbar$ invariant}\phantomsection\label{sec:minimal_enr_inv}
Let $L\subset M$ be a closed Lagrangian submanifold inside a closed symplectic manifold $(M,\omega)$. Fix a Riemannian metric $g$ on $M$. Let $J$ be any almost complex structure on $M$, not necessarily taming $\omega$. We define
$$\hbar(L,J):=\inf \left\{E(u):
     \begin{aligned}
       & \bar{\partial}_{J}u = 0 \text{ where,}\\
       &u:\C P^{1}\rightarrow M \text{ or }\\
       &u:(\D,\partial \D) \rightarrow (M,L)\\
       &\text{is non-constant}
     \end{aligned}
\right\},$$
where $E(u) = \int_{\D}\abs{du}_{g}^{2}d\mathrm{vol}_{\D}$. If $J$ tames $\omega$ and $g = g_{\omega, J}$ then $E(u) = \omega(u)$. It is well-known \cite[Proposition 4.1.4]{mcduff-salamon-12}, via mean-value inequality, that $\hbar(L,J)>0$; one can also show this using Gromov compactness analysis in the case where $J$ tames $\omega$.
We write $\hbar(L,J,\om)$ if we want to emphasize the dependence on $\om.$

The following lemma will be used in \S\ref{sec:comp_integ_ham_sys} and the proof of it uses a well-known trick of Fukaya; e.g., see \cite[\S2.1]{MR4822930}.


\begin{lemma}\phantomsection\label{lem:fukaya_trick}
Let $\{L_{t}\}_{t\in [0,1]}$ be a family of closed Lagrangian submanifolds inside a closed symplectic manifold $(M,\omega)$ that are all isotopic to $L_{0}$, i.e., $L_{t} = f_{t}(L_{0})$ where $f_{t}:M\rightarrow M$ is a smooth family of diffeomorphisms. For every $t_0 \in [0,1]$ there exists $\eps > 0$ such that there exists a family of almost complex structures $\{J_{t}\}_{t\in (t_0-\eps, t_0+\eps) \cap [0,1]}$ such that $J_{t}$ tames $\omega$ for all $t\in (t_0-\eps, t_0+\eps) \cap [0,1]$ and 
$$\inf_{t\in (t_0-\eps, t_0+\eps) \cap [0,1]}\hbar(L_{t}, J_{t})>0.$$ 
\end{lemma}


\begin{proof}
Let $J_{t_0}$ be an almost complex structure which tames $\omega.$ Let $h_t = f_{t}f_{t_0}^{-1}.$ It is clear that $J_{t}:= (h_t)_{*}J_{t_0}$ tames $\omega$ for all $t\in [0,1] \cap (t_0-2\eps, t_0+2\eps)$ for $\eps>0$ small enough. It remains to show that 
$$\inf_{t\in [0,1]\cap (t_0-\eps,t_0+\eps)}\hbar(L_{t}, J_{t}) > 0.$$ 
Set $\om_t = h_t^*\om.$ One observes that $\hbar(L_{t}, J_{t}, \om) = \hbar(L_{t_0}, J_{t_0}, \om_t)$ for all $t\in [0,1] \cap (t_0-\eps, t_0+\eps).$ The result now follows from the mean value property of holomorphic disks with fixed Lagrangian boundary condition $L_{0}$, see \cite[Lemma 4.3.1]{mcduff-salamon-12}, and the fact the constants involved in the mean value inequality depend continuously on $\om_t$; see \cite[Remark 4.3.2]{mcduff-salamon-12}. Indeed, only the comparison constant between the Frauenfelder metric of $(L_{t_0}, J_{t_0})$ and the metric induced by $(\om_t, J_{t_0})$ plays a role.
\end{proof}

\begin{remark}\label{rem:C2 top}
In fact by the above proof we get that \[\hbar(L_t, J_t) \geq \delta(L_t,J_t) > 0\] for every $t \in(t_0-\eps, t_0+\eps) \cap [0,1],$ where for a Lagrangian $L$ and tame almost complex structure $J,$ $\delta(L, J) > 0$ is a constant which depends continuously in the $C^2$ topology on $J$ and moreover \[\delta(t_0,\eps) = \inf_{t\in (t_0-\eps, t_0+\eps) \cap [0,1]} \delta(L_t, J_t) > 0.\]    
\end{remark}

\subsection{Module action}\phantomsection\label{sec:module_action}
Here we describe the module action of the Hamiltonian Floer homology on the Lagrangian Floer homology. Let $L\subset M$ be a closed Lagrangian and let $H$ be a Hamiltonian. The module action is constructed using a map,
$$\m: \CF(L,G; J) \otimes \CF(H,J) \rightarrow \CF(L,H \# G; J),$$
that is defined by $\m([x,u],[y,v]) = \sum n([x,u], [y,v], [z,w])[z,w]$ where the coefficient is the number of isolated solutions to the following asymptotic boundary value problem,
\begin{enumerate}
\item $f: \Sigma_{2;1} \rightarrow M$ satisfying $(df - X_{\mathfrak{a}}\circ f)^{(0,1)} = 0$.
\item $f\circ \epsilon_{-}(-\infty, t) = x(t)$, $f\circ \epsilon_{+}(+\infty, t) = z(t)$ and $f\circ \epsilon_{0}(-\infty, t) = y(t)$.
\item $f(\partial \bar{\Sigma}_{2;1}\backslash\Gamma) \subset L$ where $\Gamma$ is the set of punctures.
\item $w = u \# v \# f.$
\end{enumerate}
Here $\Sigma_{2;1}$ is a disk, as a Riemann surface, with two boundary punctures and one interior puncture, $\mathfrak{a}\in \Omega^{1}(\Sigma_{2;1}, C^{\infty}(M))$ is a Hamiltonian connection one form with zero curvature so that on the image of $\epsilon_{-}, \epsilon_{0}, \epsilon_{+}$ it is of the form $G_{t}dt, H_{t}dt, (H\# G)_{t}dt$ respectively, $J = \{J_{z}\}_{z\in \Sigma_{2;1}}$ is a generic family of compatible almost complex structures that are $s$-independent near the punctures, and,
\begin{enumerate}
\item $\epsilon_{-}: (-\infty, 0) \times [0,1] \rightarrow \Sigma_{2;1}$ is a holomorphic strip-like neighborhood of one of the boundary punctures, the negative puncture.
\item $\epsilon_{+}: (0, + \infty) \times [0,1] \rightarrow \Sigma_{2;1}$ is a holomorphic strip-like neighborhood of the other boundary puncture, the positive puncture.
\item $\epsilon_{0}: (-\infty, 0) \times S^{1} \rightarrow \Sigma_{2;1}$ is a holomorphic cylindrical neighborhood of the interior puncture.
\end{enumerate}
The module action by a cycle $\alpha \in \CF(H,J)$ is defined to be $\m(-, \alpha)$, however, in order for it to be well-defined one has to restrict oneself to a small action window to prevent the disk-bubbling phenomenon. Let $\hbar = \hbar(L,J)$ be as in \S\ref{sec:minimal_enr_inv}. Let $[a,b]$ be an interval with $b - a < \hbar.$ Then, the module action, 
$$\m(- \otimes \alpha): \HF^{[a,b]}(L,G) \rightarrow \HF^{[a,b]+\mathcal{A}_{H}(\alpha)}(L, H\#G),$$
is well-defined. The moduli space of maps $f: \Sigma_{2;1} \rightarrow M$ satisfying the above is a smooth manifold for generic choices of $H,G,J$ and the dimension is computed by
$$\mathrm{CZ}([z,w]) - \mathrm{CZ}([x,u]) - \big(n - \mathrm{CZ}([y,v])\big);$$
See \cite{kislev-shelukhin-21, MR4879080} and references therein for a more detailed exposition. It follows that the module action by $\alpha\in \CF_{k}(H,J)$, 
$$\m(- \otimes \alpha): \HF_{*}^{I}(L,G) \rightarrow \HF_{*+n-k}^{J}(L,H\#G)$$ 
shifts the degree by $n-k$. By the energy identity,
$$E(f) = \mathcal{A}_{L,G}([x,u]) + \mathcal{A}_{H}([y,v]) - \mathcal{A}_{L,H\#G}([z,w]),$$
the following holds for all chains $\xi\in \CF(L,G; J), \alpha\in \CF(H,J)$,
$$\mathcal{A}_{L,H\#G}(\m(\xi \otimes \alpha)) \leq \mathcal{A}_{L,G}(\xi) + \mathcal{A}_{H}(\alpha);$$
so, the module action $\m(- \otimes \alpha)$ shifts the actions by at most $\mathcal{A}_{H}(\alpha)$.

\subsubsection{Double module action}\phantomsection\label{sec:twin_tower_map}
Let $\alpha \in \CF_{k}(H, J)$ and $\beta\in \CF_{l}(\bar{H}, J)$ be two cycles with $A = \mathcal{A}_{H}(\alpha)$ and $B= \mathcal{A}_{H}(\beta)$. Let $[a,b]$ be an interval so that $b - a < \hbar$ so that the module actions $\m(- \otimes \alpha)$ and $\m(- \otimes \beta)$ are well-defined. Then, the double module action is the composition of the latter two actions, 
$$\m_{2}: \HF_{*}^{[a,b]}(L,G) \rightarrow \HF_{*+2n-k-l}^{[a,b]+A+B}(L,G),$$
defined by $\m_{2}(\xi):= \m(\m(\xi \otimes \alpha) \otimes \beta),$ where $[a,b]+c:=[a+c, b+c]$. Standard gluing analysis in Floer theory, Floer-Gromov compactness analysis \cite{floer-morese-th-lag-int, floer-symp-fix-pt} and Deligne-Mumford compactification \cite[Chapter 9]{MR2441780} imply that,
\begin{equation}\phantomsection\label{eq:module_act_vs_pp}
\m_{2}(\xi) = \m(\xi \otimes \mathrm{pp}(\alpha \otimes \beta)),
\end{equation}
where $\mathrm{pp}: \HF_{k}(H) \otimes \HF_{l}(\bar{H}) \rightarrow \HF_{k+l-n}(0)$ is the pair of pants product; see \S\ref{sec:pair_of_pants_unit_times_unit}. Let $\zeta$ be a representative of the fundamental class $[M]\in \QH(M)$ and let $c = \mathcal{A}_{0}(\pss_{0}(\zeta))$ where $\pss_{0}$ is the PSS-isomorphism for zero Hamiltonian; see \S\ref{sec:pss_isomorphisms}. Consider the module action, 
$$\m(-\otimes \pss_{0}(\zeta)): \HF^{I}(L,0) \rightarrow \HF^{I + c}(L,0),$$ 
where $I$ is a small interval with $0\in I\cap (I +c)$ and $|I| < \hbar$. Then,
\begin{prop}\phantomsection\label{prop:pss_modul_incl_comm_diag}
The module action $\m(-\otimes \pss_{0}(\zeta))$ equals the map $\HF^{I}(L,0) \rightarrow \HF^{I+c}(L,0)$ induced by inclusion. Moreover, the following diagram commutes,
\[
\begin{tikzcd}[row sep=3em, column sep=5em]
\HF_{n-*}^{I}(L,0) \arrow[r, "\m(-\otimes \pss_{0}(\zeta))"] & \HF_{n-*}^{I+c}(L,0) \\
\H_{*}(L) \arrow[r, "id"] \arrow[u, "\pss_{L,0,I}"] & \H_{*}(L) \arrow[u, "\pss_{L,0,I+c}", swap]
\end{tikzcd}
\]
\end{prop}
\begin{proof}
The proof is the standard gluing and compactness analysis, we refer to \cite[\S5]{kislev-shelukhin-21} for further details of the proof.
\end{proof}
\subsection{Lagrangian and Hamiltonian PSS isomorphisms}\phantomsection\label{sec:pss_isomorphisms}
For a closed monotone symplectic manifold $(M,\omega)$ there is an isomorphism between quantum homology $\QH(M)$, see \cite[Chapter 11]{mcduff-salamon-12}, and Floer homology $\HF(H)$ for any smooth Hamiltonian $H$, the isomorphism is called PSS-isomorphism \cite{MR1432464} and it shifts degrees by $-n$ with our conventions; Conley-Zehnder index of a critical point of $C^{2}$-small Morse function is equal to the Morse index minus $n$. The $\pss$ maps is roughly defined as follows. One picks a Morse function $f:M\rightarrow \R$ and for every critical point $z\in \mathrm{Crit}(f)$ one defines $\pss(z) = \sum n(z,[x,u])[x,u]$ where the coefficient is the number of isolated solutions to the following asymptotic boundary value problem,
\begin{enumerate}
\item $w:\R\times S^{1} \rightarrow M$ satisfying $(dw - X_{\mathfrak{a}}\circ w)^{(0,1)} = 0$ where $\mathfrak{a}\in \Omega^{1}(\R\times S^{1}, C^{\infty}(M))$ is a Hamiltonian connection one form that is of the form $H_{t}dt$ on $s \geq 1$ and vanishes on $s\leq 0$.
\item $w(-\infty,t) \in W^{u}(z)$ and $w(+\infty, t) =x(t)$ where $W^{u}(z)$ is the unstable manifold of $z$; note that by removable singularity theorem $w$ extends to $\C$ if we identify $\R\times S^{1}$ with $\C^{\times}$ via $(s,t)\mapsto e^{s + 2\pi it}$.
\item $w = u$ as capping of $x$.
\end{enumerate}
Then, the isomorphism will be a map,
$$\pss_{H}: (\mathrm{HM}(f)\otimes \K[q^{-1},q]])_{*} \rightarrow \HF_{* - n}(H),$$ 
an index $k$ critical point will be mapped to capped orbits with index $k - n$ since the module spaces of solutions $w$ as above will be a smooth manifold, for generic data $f,g,H,J$ where $g,J$ are a generic metric and a generic family of almost complex structures, and of dimension $\mathrm{ind}(z) - n - \mathrm{CZ}([x,u])$. 

Let $w: \C\rightarrow M$ be a $\pss$-plane as above with asymptotic $[x,u]$. Then, the energy of $w$ is computed by the following identity,
$$E(w) = -\mathcal{A}_{H}([x,u]) + \int_{w}\tau,$$
where $\tau = d\mathfrak{a} - \{\mathfrak{a}, \mathfrak{a}\}$ is the curvature term which is equal to $\partial_{s}H_{s,t} ds \wedge dt$ when $\mathfrak{a} = H_{s,t}dt$. Picking the connection form $\mathfrak{a}$ in the latter form we obtain,
$$0 \leq E(w) \leq -\mathcal{A}_{H}([x,u]) + \int_{0}^{1} \max H_{t} dt,$$
which implies $\mathcal{A}_{H}([x,u]) \leq |H|$ where the latter is the Hofer norm defined by the integral above. So, given a Hamiltonian $H$ with $|H|\leq \epsilon$ the "$\pss$ representatives" of the class $\pss_{H}([M])$ have actions smaller than $\epsilon$.

\subsubsection{Lagrangian quantum homology and Lagrangian $\pss$}\phantomsection\label{sec:lag_quant_hom_lag_pss}
For a detailed construction of the Lagrangian quantum homology we refer the reader to \cite{biran-cornea-09, biran-cornear-lag-quan-hom-09}. We quickly sketch the construction here. Given a monotone closed Lagrangian submanifold with minimal Maslov number at least $2$, one picks a Morse function $f:L\rightarrow \R$, a generic metric $\rho$ on $L$ and a generic compatible almost complex structure $J$ on $M$. One considers the complex $\mathrm{PC}(L,f,\rho,J)$ that is a free $\K[p^{-1},p]]$-module generated by the critical points of $f$ where $\deg p = -N_{L}$. The differential is $\K[p^{-1}, p]]$-linear and for a critical point $z\in \mathrm{Crit}(f)$ one defines $dz = \sum n(z,w)p^{-\mu_{L}(A)/N_{L}}w$ where the some is over $w\in \mathrm{Crit}(f)$ and $A\in \pi_{2}(M,L)$ and the coefficient is the number of isolated solutions, up to some natural identifications, to the following problem,
\begin{enumerate}
\item $(u_{1}, \dots, u_{k})$ where $u_{j}:\D \rightarrow M$ is a non-constant $J$-holomorphic disk with boundary on $L$.
\item For all $2\leq j\leq k$ one has $u_{j}(-1,0)\in W^{u}(u_{j-1}(1,0))$, $u_{1}(-1,0)\in W^{u}(z)$ and $w\in W^{u}(u_{k}(1,0))$.
\item $[u_{1}] + \dots + [u_{k}] = A.$
\end{enumerate}
The homology of the above complex is denoted by $\QH(L)$ and it coincides with $\H(L)$ whenever $L$ is weakly exact. There is a $\pss$ isomorphism from $\QH(L)$ to $\HF(L,H)$ for every Hamiltonian $H$; see \cite{MR2424172, biran-cornear-lag-quan-hom-09, MR2122744}. It is roughly defined as follows. For a critical point $z\in \mathrm{Crit}(f)$ one defines $\pss(z) = \sum n(z, [x,\xi])[x,\xi]p^{-\mu_{L}(A)/N_{L}}$ where the sum is over $[x,\xi]\in \mathrm{Crit}(\mathcal{A}_{L,H})$ and $A\in \pi_{2}(M,L)$ and the coefficient is the number of isolated solutions of the following problem, 
\begin{enumerate}
\item $(u_{1}, \dots, u_{k}, w)$ where $u_{j}:\D \rightarrow M$ is a non-constant $J$-holomorphic disk with boundary on $L$ and $w: \R \times [0,1] \rightarrow M$ satisfies $(dw - X_{\mathfrak{a}}\circ w)^{(0,1)} = 0$ and $w(s,0), w(s,1)\in L$ for all $s$ where $\mathfrak{a}\in \Omega^{1}(\R \times [0,1], C^{\infty}(M))$ is a Hamiltonian connection one form that vanishes on $s\leq 0$ and coincides with $H_{t}dt$ on $s\geq 1$.
\item For all $2\leq j\leq k$ one has $u_{j}(-1,0)\in W^{u}(u_{j-1}(1,0))$, $u_{1}(-1,0)\in W^{u}(z)$, $w(-\infty,t)\in W^{u}(u_{k}(1,0))$, and $w(+\infty, t) = x(t)$.
\item $[u_{1}] + \dots + [u_{k}] = A$ and $u_{1}\#\dots\#u_{k}\#w = \xi$.
\end{enumerate}
This will define a degree preserving isomorphism, 
$$\pss_{L,H}: \QH_{*}(L) \rightarrow \HF_{n-*}(L,H);$$ 
it preserves the degree since the virtual dimension of the above configurations is equal to $\mathrm{ind}(z) + \mu_{L}(A) + \mathrm{CZ}([x,\xi])-n$. 

One can show that if $[x, \xi]$ appears in the image $\pss_{L,H}(\zeta)$, where $[\zeta] = [L]\in \QH(L)$, and $(u_{1}, \dots, u_{k},w)$ is a (pearl) trajectory from $\zeta$ to $[x,\xi]$ then,
$$E(w) = -\mathcal{A}_{L,H}([x,\xi]) + \int_{w} \tau,$$
where $\tau = d\mathfrak{a} - \{\mathfrak{a}, \mathfrak{a}\}$ is the curvature form that is equal to $\partial_{s}H_{s,t}ds\wedge dt$ when $\mathfrak{a} = H_{s,t}dt$. In the latter case, one observes that $\mathcal{A}_{L,H}([x,\xi]) \leq |H|$. 

For small enough intervals $I =[a,b]$ so that $ab < 0$ and $b-a < \hbar(L)$ there is a $\pss$-type map $\pss_{L,0,I}: \H_{*}(L) \rightarrow \HF_{n-*}^{I}(L,0)$ that is injective. Indeed, one counts isolated solutions $w: \R \times [0,1] \rightarrow M$ of $(dw - X_{\mathfrak{a}}\circ w)^{(0,1)} = 0$ with correct asymptotics and Lagrangian boundaries as before, but since we have restricted ourselves to a small action window and the Hamiltonian is zero, by the energy identity above one can prevent disk-bubbling and so the map $\pss_{L,0,I}$ will be well-defined.

\subsection{Hamiltonian spectral invariants}\phantomsection\label{sec:spectral_invariants}
Given a class $a\in \HF(H,J)$ we define its spectral invariant,
$$c(a, H):= \inf_{t}\{a\in \mathrm{Im}\big(\HF^{(-\infty,t)}(H) \rightarrow \HF(H)\big)\},$$
where the map is induced by the inclusion of the sub-complex $\CF^{(-\infty,t)}(H,J)$. For a quantum homology class $a\in \QH(M)$ one simply defines $c(a, H):= c(\pss_{H}(a), H)$. These invariants satisfy properties listed below:
\begin{enumerate}
\item $c(\alpha;H)\in \mathrm{Spec}(H)$ for every $\alpha\in \QH(M)$ and Hamiltonian $H$.
\item $c(\alpha * \beta, H \# G) \leq c(\alpha, H) + c(\beta, G)$.
\item $|c(\alpha; H) - c(\alpha;K)| \leq |H - K|$, where $|G|$ is the Hofer norm of $G$.
\item $c(q \alpha; H) = c(x; H) - \rho$.
\end{enumerate}
Here $q$ is the quantum variable in $\K[q^{-1},q]]$, and $\omega(\pi_{2}(M)) = \rho\Z$; $\rho > 0$ is called the rationality constant. The spectral pseudo-norm of a Hamiltonian $H$ is defined by 
$$\gamma(H) := c([M], H) + c([M], \bar{H}).$$ 
When the ambient symplectic manifold $(M,\omega)$ is symplectically aspherical, for instance oriented closed surfaces of genus at least $1$, the spectral pseudo-norm $\gamma(H)$ only depends on the time-one map $\phi_{H}$ of the Hamiltonian flow generated by $H$, hence, for a Hamiltonian diffeomorphism $\phi$ one can define $\gamma(\phi) := \gamma(H)$ where $H$ is any Hamiltonian with $\phi_{H} = \phi$. We refer to \cite{MR1157321, schwarz-00, MR2103018, MR3850107} for more details. 

An important property of the spectral invariants that will be useful later is called the energy-capacity inequality stating that if $H$ and $G$ are two Hamiltonian diffeomorphisms of a closed symplectic manifold generating $\phi_{H}$ and $\phi_{G}$ so that $\phi_{H}(\mathrm{supp}(G)) \cap \mathrm{supp}(G) = \emptyset$ then one has,
$$c([M], G) \leq \gamma(H).$$
In particular, for a given open set $U \subset M$ one has, 
$$c(U):=\sup \{\gamma(G) :\mathrm{supp}(G)\Subset U\} \leq 2\gamma(H).$$
\begin{remark}\phantomsection\label{rem:spec_cap_of_disks_in_high_genus_surf}
In the particular case of $B\subset \Sigma$ being an open disk and $\Sigma$ a closed oriented surface of genus at least $1$, one can show that $c(B) = \mathrm{area}(B)$; see \cite{MR5018404} and \cite{MR4587314}. Hence, in this case if $\phi_{H}(B)\cap B = \emptyset$ then one has $\gamma(\phi_{H}) \geq \mathrm{area}(B)$.
\end{remark}
A Poincaré duality property holds for the Hamiltonian spectral invariants \cite{MR1979584}, see \cite{MR3850107} in the Lagrangian case, which works as follows: for a class $\alpha \in \QH(M)$,
$$c(\alpha, \bar{H}) = -\inf\{c(\beta, H) : \Pi(\alpha, \beta) \neq 0\},$$
where the infimum is over quantum homology classes $\beta$ that pair non-trivial with $\alpha$ where the pairing $\Pi(\alpha,\beta)$ is defined to be the coefficient of the point class in the quantum cup product $\alpha * \beta$; see \cite[\S2.6]{MR1979584} for details. In particular, whenever $\omega = 0$ on sphere classes, $\QH(M) = \H(M)$ and one proves the identity $c([M], \bar{H}) = -c([\mathrm{pt}], H)$. 

We will use the following elementary property of the $\gamma$-norm, which in fact applies to every conjugation-invariant norm on a group.

\begin{lemma}\label{lemma:biinv}
For all $f,g \in \Ham(M,\om),$ and $k \in \mathbb Z,$ \[\gamma((fg)^k) \leq k \gamma(f) + \gamma(g^k).\]      
\end{lemma}

\begin{proof}
This is immediate by the triangle inequality and conjugation invariance applied to the identity \[(fg)^k = (fgf^{-1}) \ldots (f^k g f^{-k}) f^k.\]   
\end{proof}

\subsection{Persistence modules and barcodes}\phantomsection\label{sec:pers_mod_and_barcodes}
For a detailed exposition of the material presented in this section we refer to \cite{polterovich-shelukhin-16, polterovich-shelukhin-Stojisavljevic-17, kislev-shelukhin-21}. Here we briefly recall some standard notions in the theory of persistence modules, in particular the Usher's boundary depth \cite{usher-bound-depth-13}. A persistence module over a field $\mathbb{K}$ is a pair $(V,\pi)$ where $V = (V_{t})_{t\in \mathbb{R}}$ is a family of finite dimensional vector spaces and $\pi_{st}: V_{s} \rightarrow V_{t}$ are linear maps for every $s \leq t$ that satisfy the following properties,
\begin{enumerate}
\item for $t \ll 0$, $V_{t} = 0$ and for $s,t \gg 0$ the map $\pi_{st}$ is isomorphism.
\item For all $s\leq t\leq r$, $\pi_{sr} = \pi_{tr} \circ \pi_{st}$ and $\pi_{ss} = id$.
\item $\forall r\in \mathbb{R}, \exists \epsilon > 0$ so that $\forall\ r - \epsilon < s \leq t \leq r$, $\pi_{st}$ is an isomorphism.
\item For almost all $r$ there is a neighborhood $U$ of $r$ so that $\forall\ s,t \in U$, $\pi_{st}$ is an isomorphism.
\end{enumerate}
The finite set of numbers $r$ that do not satisfy the last condition is called the spectrum of $V$ and is denoted by $\mathrm{Spec}(V)$. A family of morphisms $A = (A_{t})_{t\in \mathbb{R}}$ is called a persistence morphism $A: (V,\pi) \rightarrow (V',\pi')$ if for all $t$, $A_{t}: V_{t} \rightarrow V_{t}'$ is a linear morphism and they are compatible with the persistence morphisms of $V,V'$, i.e., $A_{s}\circ \pi_{st} = \pi_{st}'\circ A_{t}.$ An elementary persistence module is denoted by $V(I)$ where $I = (a,b]$ or $(a, +\infty)$ is an interval and $V(I)_{t} = \mathbb{K}$ for $t\in I$, $V(I)_{t} = 0$ for $t\notin I$ and $\pi_{st} = id$ for $s,t\in I$ and zero otherwise. The structure theorem for persistence modules \cite{Zomorodian-Carlsson-05} states that for any persistence module $(V,\pi)$ there is a unique collection of pairwise distinct intervals $I_{i}$ of the form $(a_{i}, b_{i}]$ or $(a_{i}, +\infty)$ for $a_{i}, b_{i}\in \mathrm{Spec}(V)$ with positive multiplicities $m_{i}$ so that 
$$V \cong \oplus_{i} V(I_{i})^{m_{i}}.$$
The multi-set containing $m_{i}$ copies of the intervals $I_{i}$ is called the barcode of $V$ and is denoted by $\mathcal{B}(V)$. The length of the biggest bar is denoted by $\beta(V)$ and is called the boundary depth of $V$; see \cite{usher-bound-depth-13}. The latter will be heavily used in the subsequent section.

Given a Hamiltonian $H$ on a symplectic manifold one can construct a persistence module by defining $V_{t} = \HF^{(-\infty, t)}(H,J)$ and $\pi_{st}$ are the inclusion induced maps. We denote the boundary depth of $H$ by $\beta(H)$. It only depends on the time-one map of $H$, hence, for a Hamiltonian diffeomorphism $\phi$ we define its boundary depth to be the boundary depth of a Hamiltonian $H$ with $\phi_{H} = \phi$ and we denote it by $\beta(\phi)$. A key inequality \cite{kislev-shelukhin-21} that will be used often in the course of the proof of the main theorem is that 
\begin{equation}\phantomsection\label{eq:beta_smaller_gamma}
\beta(\phi) \leq \gamma(\phi),
\end{equation}
for all Hamiltonian diffeomorphisms $\phi$, where $\gamma(\phi)$ is the spectral norm introduced in \S\ref{sec:spectral_invariants}.

\section{A theorem of Franks and Handel on zero entropy diffeomorphisms}\phantomsection\label{sec:zero_entropy_franks_handel}
In this section we recollect a few well-known results about the dynamics of area-preserving homeomorphisms of surfaces including the celebrated result of Franks and Handel \cite{franks-handel-zero-ent} on zero entropy area-preserving diffeomorphisms of $S^{2}$ which we shall crucially use in the proof of our main theorem.

\subsection{Birkhoff's ergodic theorem}\phantomsection\label{sec:birkhoff_erg_thm} Let $(X, \mathcal{B}, \mu)$ be a finite measure space and let $T: X \rightarrow X$ be a measure preserving transformation, i.e., $\mu(T^{-1}(B)) = \mu(B)$ for every measurable subset $B\in \mathcal{B}$. Let $f: X \rightarrow \R$ be an integrable function, $f\in L^{1}(X, \mu)$. The following is known as the (pointwise) Birkhoff's ergodic theorem. The limit below exists for almost every $x\in X$,
$$f^{\infty}(x):= \lim_{n \to \infty}\frac{f(x) + f(T(x)) + \dots + f(T^{n-1}(x))}{n}.$$
Moreover, the function $f^{\infty}$ is measurable and $T$-invariant and,
$$\int_{X}f^{\infty}d\mu = \int_{X}fd\mu.$$
We refer to \cite{MR648108, mirzakhani-ergodic-theory-14} and \cite{MR682312, MR832433, MR1567747} for further details.
\subsection{Action, rotation and winding numbers}\phantomsection\label{sec:asymp_action_def}
Let $\phi: \D \rightarrow \D$ be a compactly supported area-preserving diffeomorphism. Let $\omega = d\lambda$ where $\lambda$ is some smooth one form which we later fix it to be $\lambda = r^{2}d\theta/2$ where $(r,\theta)$ is the polar coordinate on the disk. There is a smooth function $a_{\phi, \lambda}: \D \rightarrow \R$ so that,
$$\phi^{*}\lambda - \lambda = da_{\phi,\lambda} \text{ and } a_{\phi, \lambda}(x) = 0, x\in \partial \D.$$
It is called the action function and enjoys the following properties:  
\begin{enumerate}
\item $a_{\phi, \lambda + df} = a_{\phi, \lambda} + f\circ \phi - f.$
\item $a_{\phi\psi, \lambda} = a_{\phi, \lambda} \circ \psi + a_{\psi, \lambda}.$
\item $a_{\phi, \psi^{*}\lambda} = a_{\psi\phi, \lambda} - a_{\psi, \lambda}.$
\item $-a_{\phi, (\phi^{-1})^{*}\lambda} = a_{\phi^{-1}, \lambda} = - a_{\phi, \lambda} \circ \phi^{-1}.$
\end{enumerate}
Using the convention $\omega(X, -) = -dH$ for Hamilton's equation and letting $H_{t}$, which is normalized by being zero on the boundary $\partial \D$, generate $\{\phi_{t}\}$ with $\phi_{1} = \phi$ we have
\begin{equation}\phantomsection\label{eq:two_acts_compared}
a_{\phi, \lambda}(x) = -\int_{0}^{1}H_{t}\circ \phi_{t}(x)\d t + \int_{\{\phi_{t}(x)\}}\lambda;
\end{equation}
see \cite[Proposition 2.1.4]{bechara2023asymptotic}. If $\phi$ is not compactly supported, then, one first picks a symplectic isotopy $\{\phi_{t}\}$ with time-one map $\phi$ and defines $a_{\{\phi_{t}\}, \lambda}$ similarly by
$$\phi^{*}\lambda - \lambda = da_{\{\phi_{t}\}, \lambda} \text{ and } a_{\{\phi_{t}\}, \lambda}(x) = \int_{\{\phi_{t}(x)\}}\lambda, x\in \partial \D.$$
The existence of $a_{\{\phi_{t}\}, \lambda}$ is not immediate: one has to prove that if the second equality, the normalization condition, holds for a point $x\in \partial \D$, then it holds for all other points $x'\in \partial \D$. Similar properties as above hold plus:
\begin{enumerate}[resume]
\item The value $a_{\{\phi_{t}\}, \lambda}(x)$ is independent of $\lambda$ for all $x\in \fix(\phi_{1})$.
\item $a_{\{\phi_{t}\}, \lambda}$ depends only on the homotopy class of the isotopy $\{\phi_{t}\}$ with fixed end points.
\end{enumerate}
The asymptotic mean action of a point $x\in \D$ is defined by
$$a^{\infty}_{\{\phi_{t}\},\lambda}(x) := \lim_{n\to \infty} \frac{a_{\{\phi_{t}^{n}\}, \lambda}(x)}{n} = \lim_{n\to \infty} \frac{1}{n}\sum_{j=0}^{n-1} a_{\{\phi_{t}\},\lambda}(\phi_{1}^{j}(x)),$$
if it exists. Using the above properties $(\mathrm{i})$ and $(\mathrm{ii})$ of the action functions one observes that this limit does not depend on the primitive $\lambda$ and so we are free to denote it by $a^{\infty}_{\{\phi_{t}\}}(x)$. By Birkhoff's ergodic theorem \S\ref{sec:birkhoff_erg_thm} one can prove that the limit exists for almost every point $x\in \D$ and moreover by the same theorem one shows,
$$\int_{\D}a^{\infty}_{\{\phi_{t}\}, \lambda} \omega= \int_{\D}a_{\{\phi_{t}\}, \lambda} \omega.$$

\subsubsection{Poincaré rotation number}\phantomsection\label{sec:rotation_numb_def}
Given a diffeomorphisms $\phi: \R/\Z \rightarrow \R/\Z$ and its lift $\t{\phi}: \R \rightarrow \R$ we define the rotation number of $x\in \R/\Z$ by
$$\rho_{\t{\phi}}(x) = \lim_{n\to \infty} \frac{\t{\phi}^{n}(\t{x}) - \t{x}}{n},$$
where $\t{x}\in \R$ is some lift of $x$. The above limit exists and is independent of $x$, see \cite{poincare-85}, hence we define the rotation number of $\t{\phi}$ to be the limit above and denote it by $\rho(\t{\phi})$. One can also define $\rho(\phi) \in \R/\Z$ by the limit modulo $\Z$. One way to pick a lift $\t{\phi}$ for $\phi$ is to choose an isotopy $\{\phi_{t}\}$, if it exists, with $\phi_{0} = id$ and $\phi_{1} = \phi$. Then, one can define similarly the rotation number $\rho(\{\phi_{t}\})$.
\subsubsection{Rotation number on an annulus}\phantomsection\label{sec:rot_numb_ann}
Let $\phi: A\rightarrow A$ be a symplectic map on the annulus $A = \R/\Z \times [0,1]$ that is isotopic to the identity via $\{\phi_{t}\}$. The isotopy gives an isotopy of symplectic maps $\{\t{\phi}_{t}\}$ on $\R \times [0,1]$ lifting $\{\phi_{t}\}$. We denote $\t{\phi}_{1}$ by $\t{\phi}$. Alternatively, if $\phi$ is not isotopic to the identity one would just pick a lift $\t{\phi}$ of $\phi$. For a point $x\in A$ we say that $\rho_{\t{\phi}}(x)$ is a rotation number of $x$ with respect to $\t{\phi}$ if for every sequence of numbers $n_{k} \to +\infty$ the following limit exists and equals $\rho_{\t{\phi}}(x)$:
$$\lim_{k\to \infty}\frac{\pi_{1}(\t{\phi}^{n_{k}}(\tilde{x})) - \pi_{1}(\tilde{x})}{n_{k}} = \rho_{\t{\phi}}(x),$$
for every lift $\t{x}$ of $x$ where $\pi_{1}: \R \times [0,1] \rightarrow \R$ is the projection to the first factor. See \cite[Corollary 1.6]{franks-handel-zero-ent} and the discussion in \S\ref{sec:franks_handel_thm_subsec} for well-definedness and continuity of the rotation numbers on closed annuli where the map $\phi$ is area-preserving and has zero topological entropy in some particular cases.

The work of Le Calvez-Tal \cite[Theorem A]{lecalvez-tal-22} ensures the existence of rotation numbers of points with non-empty omega limit sets for zero entropy homeomorphisms of open or closed annuli. A point $x\in A$ in an open annulus $A$ is called an escaping point if $\{\phi^{n}(x)\}_{n \geq 0}$ is not included in any compact subset of $A$. We denote by $\mathrm{ne}_{+}(\phi)$ the complement of the set of escaping points of $\phi$ and by $\mathrm{ne}_{-}(\phi)$ the complement of the set of escaping points of $\phi^{-1}$ and finally we set $\mathrm{ne}(\phi) = \mathrm{ne}_{+}(\phi)\cup \mathrm{ne}_{-}(\phi)$. For a point $x\in \mathrm{ne}_{+}(\phi)$ we say that $\rho_{\t{\phi}}(x)$ is a rotation number of $x$ with respect to $\t{\phi}$ if for every sequence $n_{k} \to +\infty$ for which $\phi^{n_{k}}(x)\in K\subset A$ for some compact set $K$ the following limit exists and equals $\rho_{\t{\phi}}(x)$:
$$\lim_{k \to \infty}\frac{\pi_{1}(\t{\phi}^{n_{k}}(\tilde{x})) - \pi_{1}(\tilde{x})}{n_{k}} = \rho_{\t{\phi}}(x);$$
for every lift $\t{x}\in \R \times (0,1)$ of $x$. A point $x\in A$ is called non-wandering if for every neighborhood $U$ of $x$ there is an iteration $n$ so that $\phi^{n}(U) \cap U \neq \emptyset$. The set of such points is denoted by $\Omega(\phi)$. For a point $x\in A$ we denote by $\omega_{\phi}(x)$ the set of all accumulation points of the sequence $\{\phi^{n}(x)\}_{n \geq 0}$ and we set $\alpha_{\phi}(x) = \omega_{\phi^{-1}}(x)$. The following was shown in \cite[Theorem A]{lecalvez-tal-22}.

\begin{theorem}\phantomsection\label{thm:rot_numb_ann_defined}
Suppose $\phi:A \rightarrow A$ be a symplectic map on an open annuls $A$ isotopic to the identity and $\t{\phi}: \R\times \R \rightarrow \R\times \R$ is a choice of lift for $\phi$. Assume $\phi$ has no topological horseshoe. Then the following hold.
\begin{enumerate}
\item For every $x\in \mathrm{ne}_{+}(\phi)$ the rotation number $\rho_{\t{\phi}}(x)$ is well-defined.
\item For $x\in \mathrm{ne}_{+}(\phi)$ and $y\in \omega_{\phi}(x)$ we have $\rho_{\t{\phi}}(x) = \rho_{\t{\phi}}(y)$.
\item If $x\in \mathrm{ne}_{+}(\phi) \cap \mathrm{ne}_{-}(\phi)\cap \Omega(\phi)$ then $\rho_{\t{\phi}^{-1}}(x) = -\rho_{\t{\phi}}(x)$.
\item The map $\rho_{\t{\phi}^{\pm1}}:\Omega(\phi) \cap \mathrm{ne}(\phi) \rightarrow \R$ defined below is continuous,
$$\rho_{\t{\phi}^{\pm1}}(x) =
\begin{cases}
\rho_{\t{\phi}}(x) & \text{if } x \in \Omega(\phi) \cap \mathrm{ne}_{+}(\phi),\\[2mm]
-\rho_{\t{\phi}^{-1}}(x) & \text{if } x \in \Omega(\phi) \cap \mathrm{ne}_{-}(\phi).
\end{cases}$$
\end{enumerate}
\end{theorem}

\begin{remark}\phantomsection\label{rem:rot_numb_of_zero_ent_on_closed_ann}
It is a well-known result of Katok \cite{katok-89} that positive topological entropy forces the existence of a topological horseshoe \cite{smale-68}. The converse does not necessarily hold. The relation of the topological entropy and existence of topological horseshoes is not clear \cite{MR2339286, MR616561}. 

However, it was shown in \cite[Theorem L]{LeCalvez-Tal-forcing-theory-18} that if $\phi: A \rightarrow A$ is a \emph{homeomorphism} of an open annulus $A$, so that its induced map on the sphere obtained by collapsing the ends to two points has zero topological entropy, then the rotation number of all bi-recurrent points in $A$ are well-defined and the rotation function $\rho_{\t{\phi}}:A \rightarrow \R$ is continuous on the set of bi-recurrent points; this result was previously proved for diffeomorphisms by Handel in an unpublished work. 

This in particular implies that if $\phi: \bar{A} \rightarrow \bar{A}$ is an area-preserving homeomorphism of a closed annulus $\bar{A}$ with zero topological entropy the rotation function is defined everywhere and it is continuous. See a similar result of Franks and Handel on smooth maps of closed annuli in Remark \ref{rem:rot_def_for_zero_ent_closed_ann}.

\end{remark}
The following lemma will be important in \S\ref{sec:proof_main_theorem}.
\begin{lemma}\phantomsection\label{lem:rotation_number_conjugation}
Let $A_{1},A_{2}$ be a two closed annuli and $\Gamma: A_{1} \rightarrow A_{2}$ be a homeomorphism. Let $\phi: A_{2} \rightarrow A_{2}$ be a homeomorphism and define $\phi_{1}:= \Gamma^{-1}\phi_{2}\Gamma$. Then, for every point $p_{1}\in A_{1}$ and $p_{2} = \Gamma(p_{1})$ we have $\rho_{\t{\phi}_{1}}(p_{1}) = \rho_{\t{\phi}_{2}}(p_{2}),$ subject to both sides being well-defined, where $\t{\phi}_{2}$ is a lift of $\phi_{2}$ to $\t{A}_{2}$, $\t{\phi}_{1} = \t{\Gamma}^{-1}\t{\phi}_{2}\t{\Gamma}$ and $\t{\Gamma}$ is a lift of $\Gamma$.
\end{lemma}
\begin{proof}
Pick a lift $\t{p}_{1}$ of $p_{1}$ and let $\t{p}_{2} = \t{\Gamma}(\t{p}_{1})$, then,
$$\begin{aligned}
\rho_{\t{\phi}_{1}}(p_{1}) &= \lim_{n\to \infty} \frac{y(\t{\phi}_{1}^{n}(\t{p}_{1})) - y(\t{p}_{1})}{n}\\
&= \lim_{n\to \infty} \frac{y(\t{\Gamma}^{-1}\t{\phi}_{2}^{n}(\t{p}_{2})) - y(\t{p}_{1})}{n}\\
&=\lim_{n\to \infty} \frac{y(\t{\phi}_{2}^{n}(\t{p}_{2})) - y(\t{p}_{2})}{n} = \rho_{\t{\phi}_{2}}(p_{2}),\\
\end{aligned}$$
where the second last equality is due to the fact that $y(\t{\Gamma}^{-1}\t{\phi}_{2}^{n}(\t{p}_{2})) - y(\t{p}_{1})$ stays in a finite distance from $y(\t{\phi}_{2}^{n}(\t{p}_{2})) - y(\t{p}_{2})$ independent of $n$.
\end{proof}
\subsubsection{Winding number}
Let $\{\phi_{t}\}$ be a symplectic isotopy of $\D$, not necessarily compactly supported. The winding number for a pair of distinct points $x,y\in \D$ is, 
$$W_{\{\phi_{t}\}}(x,y) := (\theta(1) - \theta(0))/2\pi,$$
where $\theta(t)$ is the angular coordinate of the point $\phi_{t}(y) - \phi_{t}(x) \in \D\backslash\{0\}.$ It is not hard to show that the winding number depends only on the homotopy class of the isotopy $\{\phi_{t}\}$ with fixed ends. In the case of compactly supported isotopies, there is only one homotopy class of isotopies for a fixed time-one map and it would be safe to write $W_{\phi}(x,y)$.

The asymptotic winding number of a pair of points $x,y\in \D$ with respect to a symplectic isotopy $\{\phi_{t}\}$ is defined by
$$W^{\infty}_{\{\phi_{t}\}}(x,y) = \lim_{n\to \infty} \frac{W_{\{\phi_{t}^{n}\}}(x,y)}{n}.$$
Birkhoff's ergodic theorem \S\ref{sec:birkhoff_erg_thm} implies that the limit exists for almost every point $x\in \D$. Indeed, one applies the ergodic theorem to the function $W_{\{\phi_{t}\}}(-,-): \D \times \D \rightarrow \R$, which is known to be integrable, see for instance \cite{shelukhin-integ-wind-15}, and the measure preserving map $\phi_{1} \times \phi_{1}$, and uses,
$$\frac{W_{\{\phi_{t}^{n}\}}(x,y)}{n} = \frac{1}{n}\sum_{k=0}^{n-1} W_{\{\phi_{t}\}}(\phi_{1}^{k}(x),\phi_{1}^{k}(y)).$$
Therefore, the function $W^{\infty}_{\{\phi_{t}\}}(-, -)$ is integrable which implies $W^{\infty}_{\{\phi_{t}\}}(x, -)$ is an integrable function for almost every point $x\in \D$; see \cite[\S2.3.2]{bechara2023asymptotic}.

\subsubsection{Action-winding-rotation identities}\phantomsection\label{sec:awr_identities}
Let $\{\phi_{t}\}$ be a symplectic isotopy on $\D$. Let $b\in \partial \D$ and $x\in \mathrm{int}(\D)$ be any points. Then, 
\begin{equation}\phantomsection\label{eq:act_wind_rot_bound}
a^{\infty}_{\{\phi_{t}\}}(b) = \pi W^{\infty}_{\{\phi_{t}\}}(x,b)= \rho(\{\phi_{t}\})/2.
\end{equation}
Here is how to prove it. Write $\phi_{t}(b) = (1, \theta(t))$. If $X_{t}$ is the generating vector field of the isotopy $\{\phi_{t}\}$ then on the boundary it is of the form $X_{t}(\phi_{t}(b)) = \dot{\theta}(t)\partial_{\theta}$. So,
$$a_{\{\phi_{t}\}, \lambda}(b) = \int_{\{\phi_{t}(b)\}}\lambda = \int_{\{\phi_{t}(b)\}}r^{2}d\theta/2 = 1/2\int_{0}^{1}d\theta(\dot{\theta}(t)\partial_{\theta})dt = \frac{\t{\theta}(1) - \t{\theta}(0)}{2},$$
where $\t{\theta}(t)$ is some lift of $\theta(t)$. The latter, denoted by $w_{\{\phi_{t}\}}(b)$, satisfies,
$$\abs{\pi W_{\{\phi_{t}\}}(x,b) - w_{\{\phi_{t}\}}(b)} \leq \pi/{2}.$$
Applying this estimate to the isotopy $\phi_{t}^{n}$ and dividing by $n$ and taking $n \to \infty$ one obtains $a^{\infty}_{\{\phi_{t}\}}(b) = \pi W^{\infty}_{\{\phi_{t}\}}(x,b)$. For the other equality in \eqref{eq:act_wind_rot_bound}, one observes that $\lim w_{\{\phi_{t}^{n}\}}(b)/n$ computes the rotation number $\rho_{\{\phi_{t}\}}(b)/2$. For a point $x\in \mathrm{int}(\D)$ one can show \cite[Theorem 4.3.2]{bechara2023asymptotic},
$$a^{\infty}_{\{\phi_{t}\}}(x) = \int_{\D} W^{\infty}_{\{\phi_{t}\}}(x,\xi)\omega(\xi);$$
this is consistent with \eqref{eq:act_wind_rot_bound} as the term on the right becomes $\pi W^{\infty}_{\{\phi_{t}\}}(x,0)$ for any point $x\in \partial\bar\D$; see also \cite[Theorem 4.1]{MR5026815}.

\subsection{Flux and mean rotation vector}\phantomsection\label{sec:flux_average_rot_vec}
In this section, following \cite[\S2]{franks-handle-pp-ham-03} and \cite{franks-handle-rot-vector-96}, we prove a key identity that relates the flux of a symplectic isotopy on a closed oriented surface to the mean rotation vector. Let $\Sigma$ be a closed oriented surface or a closed annulus in which case we denote it by $A$. First, recall that the flux homomorphism $\Fl: \t{\Symp}_{0}(\Sigma) \rightarrow \H^{1}(\Sigma, \R)$ is defined by 
$$\Fl([\{\psi_{t}\}])(\gamma) = \int u_{\gamma}^{*}\omega,$$
where $u_{\gamma}(s,t) = \psi_{t}(\gamma(s)).$ In the case of the annulus $A$ the flux is valued in $\H^{1}(A, \R)$ and on the generator of $\H_{1}(A, \partial A, \R)$ it is equal to,
$$\Fl([\{\psi_{t}\}])(\mathrm{gen}) = \int_{C(\gamma_{0})} \t{\omega},$$
where $C(\gamma_{0})$ is the region in $\t{A} = [0,1] \times \R$ enclosed by $\gamma_{0} = [0,1] \times \{0\}$ and $\t{\psi}_{1}(\gamma_{0})$ where $\t{\psi}_{1}$ is the lift of $\psi_{1}$ given by the isotopy $\{\psi_{t}\}$. We now define the rotation vector associated to $[\{\psi_{t}\}]$. Let $x\in \Sigma$ be a point and consider the path $\{\psi_{t}(x)\}_{t\in [0,1]}$. Turn this path into a loop by connecting $\psi_{1}(x)$ to $x$ by a "short" path, for instance, in a fundamental domain of the universal cover of $\Sigma$, connect both ends of $\{\t{\psi}_{t}(x)\}$ to the origin; see \cite[\S2]{franks-handle-pp-ham-03} for a more precise definition. Denote the obtained closed loop by $\{\psi^{c}_{t}(x)\}$. Define $\delta(x, \{\psi_{t}\})\in \H_{1}(\Sigma, \R)$ to be the homology class $[\{\psi^{c}_{t}(x)\}]$. We define the rotation vector of $x$ with respect to the isotopy $\{\psi_{t}\}$ by
$$\rho_{\{\psi_{t}\}}(x): = \lim_{n\to \infty} \frac{1}{n}\sum_{k=1}^{n}\delta(\psi_{1}^{k}(x), \{\psi_t\}).$$ 
To be precise this is defined by how it pairs with the elements $\alpha \in \H^{1}(\Sigma, \R)$,
\begin{equation}\phantomsection\label{eq:def_rot_vec_via_int_alpha}
\langle \rho_{\{\psi_{t}\}}(x), \alpha\rangle = \lim_{n\to \infty} \frac{1}{n}\sum_{k=1}^{n}\int_{\{\psi_{t}(\psi_{1}^{k}(x))\}}\alpha.
\end{equation}
In the case of the annulus $A$, the rotation vector, or the rotation number, is defined similarly and is valued in $\H_{1}(A, \partial A, \R) = \R$. It coincides with the definition in \S\ref{sec:rotation_numb_def}. The mean rotation vector of $\{\psi_{t}\}$ is defined by
$$\overline{\rho}(\{\psi_{t}\}):= \int_{\Sigma}\rho_{\{\psi_{t}\}}(x)\omega.$$
The following proposition is proved in \cite[Proposition 2.11]{franks-handle-pp-ham-03}
\begin{prop}
Let $[\{\psi_{t}\}]\in \t{\Symp}_{0}(\Sigma, \omega)$. Then for every $\gamma\in \H_{1}(\Sigma, \R)$,
$$\overline{\rho}([\{\psi_{t}\}]) \cap \gamma = \frac{1}{\text{area}(\Sigma)}\Fl([\{\psi_{t}\}])(\gamma).$$
\end{prop}
Similar equation holds in the case of the annulus \cite[Lemma 3.2]{qu-xia-24}. In the theorem below, we have identified $\H^{1}(A,\R)$ and $\H_{1}(A,\partial A, \R)$ with $\R$ and $\rho_{\{\psi_{t}\}}(x)$ is understood as evaluated on the unique one-form $dy$, i.e.,
$$\rho_{\{\psi_{t}\}}(x) = \lim_{n\to \infty} \frac{1}{n}\sum_{k=1}^{n}\int_{\{\psi_{t}(\psi_{1}^{k}(x))\}} dy = \lim_{n\to \infty} \frac{\pi_{1}(\t{\psi}_{1}^{n}(\t{x})) - \pi_{1}(\t{x})}{n},$$
which coincides with the classical rotation number defined in \S\ref{sec:rot_numb_ann}. 
\begin{lemma}
Consider the closed one-form $dy$ on $\t{A} = [0,1]_{x} \times \R_{y}$. For any area-preserving isotopy $\psi_{t}: A \rightarrow A$, where the symplectic form on $A$ is induced by $\t{\omega} = dx\wedge dy$, one has,
\begin{equation}\phantomsection\label{eq:flux_vs_mean_rot}
\Fl([\{\psi_{t}\}]) = \int_{A}\rho_{\{\psi_{t}\}}(x)\omega.
\end{equation}
\end{lemma}
\begin{proof}
The following identity follows from Birkhoff's ergodic theorem \S\ref{sec:birkhoff_erg_thm} applied to the measure preserving map $T = \psi$ and the function $f(x) = \int_{\{\psi_{t}(x)\}}dy$, see \eqref{eq:def_rot_vec_via_int_alpha},
$$\int_{A}\rho_{\{\psi_{t}\}}(x)\omega = \int_{A}\int_{\{\psi_{t}(x)\}}dy \wedge \omega(x).$$
To prove \eqref{eq:flux_vs_mean_rot} we proceed as follows,
$$\begin{aligned}
\int_{A}\int_{\{\psi_{t}(x)\}}dy \wedge \omega(x) &= \int_{\mathrm{FD}} \big( y(\t{\psi}_{1}(\t{x})) - y(\t{x})\big) \wedge \t{\omega}(x)\\
&=\int_{\t{\psi}_{1}(\mathrm{FD})}ydx\wedge dy - \int_{\mathrm{FD}}ydx\wedge dy\\
&= \int_{\t{\psi}_{1}(\mathrm{FD})}ydx\wedge dy - 1/2,
\end{aligned}$$
where $\mathrm{FD} \subset \t{A}$ is the fundamental domain $[0,1] \times [0,1]$; note that one gets $-1/2$ because the orientation on the boundary of $\mathrm{FD}$. The first term is computed as follows, let $\gamma_{j} \subset \mathrm{FD}$ be the curve $[0,1]\times \{j\}$ for $j= 0,1$,
$$\begin{aligned}
\int_{\t{\psi}_{1}(\mathrm{FD})}ydx\wedge dy &=-\int_{\partial\t{\psi}_{1}(\mathrm{FD})}y^{2} dx/2 = \int_{\t{\psi}_{1}(\gamma_{1})}y^{2}dx/2 - \int_{\t{\psi}_{1}(\gamma_{0})}y^{2}dx/2\\
&=\int_{\t{\psi}_{1}(\gamma_{0})}(y+1)^{2}dx/2 - \int_{\t{\psi}_{1}(\gamma_{0})}y^{2}dx/2\\
&=\int_{\t{\psi}_{1}(\gamma_{0})}ydx + 1/2.
\end{aligned}$$
Therefore, we have so far the following identity, 
$$\int_{A}\rho_{\{\psi_{t}\}}(x)\omega = \int_{\t{\psi}_{1}(\gamma_{0})}ydx.$$
We now compute the flux of $\psi_{t}$,
$$\Fl([\{\psi_{t}\}]) = \int_{C(\gamma_{0})}\t{\omega} = \int_{\t{\psi}_{1}(\gamma_{0})}ydx -\int_{\gamma_{0}}ydx = \int_{\t{\psi}_{1}(\gamma_{0})}ydx,$$
which finishes the proof.
\end{proof}

\subsection{A theorem of Franks and Handel}\phantomsection\label{sec:franks_handel_thm_subsec}
In this section, we recall the Franks-Handel \cite{franks-handel-zero-ent} structure theorem for zero-entropy area-preserving diffeomorphisms of $S^{2}$ with infinite order and at least $3$ periodic points. Later, this was upgraded by Le Calvez and Tal \cite[Theorem M]{lecalvez-tal-22} for orientation-preserving homeomorphisms of $S^{2}$ with zero entropy and no wandering points; see also \cite[Theorem G]{lecalvez-tal-22} for a version of a structure theorem for homeomorphisms with no topological horseshoes, which is a weaker assumption than zero entropy \cite{katok-89}, but possibly with wandering points. 

We denote by $h_{top}(\phi)$ the topological entropy of $\phi: S^{2} \rightarrow S^{2}$ for which there are more than one standard definitions \cite{MR1043272}. We define it as follows. For an $\eps >0$ and $n\in \Z^{+}$ a subset $E \subset X$ is called $(n,\eps)$-separated if for every distinct points $x,y\in E$ there exists $0\leq j<n$ so that $d(\phi^{j}(x), \phi^{j}(y)) >\eps.$ Defining $r(\eps, n, \phi)$ to be the maximal cardinality of an $(\eps, n)$-separated subset we define
$$h_{top}(\phi):= \lim_{\eps\to 0}\limsup_{n\geq 1}\frac{\log(r(\eps,n,\phi))}{n}.$$

The $\omega$-limit set of a point $x\in \D$ is the set of accumulation points of the sequence $\{\phi^{k}(x)\}_{k\geq 1}$ and is denoted by $\omega(x)$, and finally $\alpha$-limit set of $x$ is defined to be the set of accumulation points of the sequence $\{\phi^{-k}(x)\}_{k\geq 1}$ and is denoted by $\alpha(x)$. The word ``\emph{frontier}" of an open annulus $A_{\alpha}$ in the following theorem will be referring to $\bar{A}_{\alpha}\backslash A_{\alpha}$. The following is proved in \cite[Theorem 1.2, Proposition 5.1]{franks-handel-zero-ent}.

\begin{theorem}\phantomsection\label{thm:franks_handel_theorem}
Let $\phi : S^{2} \to S^{2}$ be an area-preserving diffeomorphism with zero topological entropy, infinite order and at least $3$ periodic points. Then, there exists a family of pairwise disjoint invariant open sets $\mathcal{A}(\phi) = \{A_{\alpha}\}_{\alpha}$ whose union is dense in the complement of $\mathrm{Fix}(\phi)$, and moreover,
\begin{enumerate}
    \item the sets $A_{\alpha}$ are the maximal fixed-point-free invariant open annuli.
    \item the $\alpha$-limit set of a point $z \notin \cup A_{\alpha}$ is contained in a single connected component of $\mathrm{Fix}(\phi)$; the same holds for the $\omega$-limit set of $z$.
    \item Let $C$ be a connected component of the frontier of $A_{\alpha}$. Then, the connected components of $\mathrm{Fix}(\phi)$ that contain $\alpha(z)$ and $\omega(z)$ are independent of $z \in C$.
\end{enumerate}
\end{theorem}
The elements of $\mathcal{A}(\phi)$ are the components of the set of weakly free disk recurrent points of $f$, see \cite[Definition 1.1]{franks-handel-zero-ent} which is an open dense subset of $S^{2}\backslash \Fix(\phi)$ containing all the bi-recurrent points.

\begin{remark}\phantomsection\label{rem:franks_handel_for_homeos}
The above theorem holds also for area-preserving homeomorphisms of $S^{2}$ that are smooth in the complement of finitely many fixed points; see \cite[Theorem 1.2]{franks-handel-zero-ent}. One can apply this result to certain area-preserving diffeomorphism of the disk or an annulus because by collapsing the boundaries to points, one obtains an area-preserving homeomorphism of $S^{2}$ that is smooth outside of finitely many fixed points.
\end{remark}
We know that if $\phi$ has zero topological entropy then so do all of its iterations. Hence, for every $q$ we can apply Theorem \ref{thm:franks_handel_theorem} to $\phi^{q}$ and find a collection of disjoint invariant open regular annuli that are dense in the complement of $\mathrm{Fix}(\phi^{q})$ and satisfy the extra properties of Theorem \ref{thm:franks_handel_theorem}; we denote this collection by $\mathcal{A}(\phi,q)$. Franks and Handel \cite{franks-handel-zero-ent} showed that these annuli satisfy the following additional properties.
\begin{theorem}\phantomsection\label{thm:regopen}
The annuli $\mathcal{A}(\phi,q)$ given by Theorem \ref{thm:franks_handel_theorem} for $\phi^{q}$ satisfy
\begin{enumerate}
\item each $A \in \mathcal{A}(\phi,q)$ is the interior of its closure inside $\Sigma\backslash \Fix(\phi^{q})$.
\item The map $\phi$ permutes the elements of $\mathcal{A}(\phi,q)$.
\item Each $A\in \mathcal{A}(\phi,q)$ is a subset of some $A' \in \mathcal{A}(\phi,1)$.
\item $\phi_{c}$ has a fixed point on components of $\partial U_{c}$ that correspond to frontiers with more than one points of $U\in \mathcal{A}(\phi,1)$.
\end{enumerate}
\end{theorem}
\begin{proof}
See \cite[Lemma 14.5]{franks-handel-zero-ent}, \cite[Lemma 15.1]{franks-handel-zero-ent}, \cite[Lemma 15.3]{franks-handel-zero-ent} and \cite[Lemma 13.12 (4)]{franks-handel-zero-ent} for items $(i),(ii),(iii), (iv)$ respectively.
\end{proof}
The following theorem \cite[Theorem 1.4 (1)]{franks-handel-zero-ent} is at the heart of the proof of our main theorem.

\begin{theorem}\phantomsection\label{thm:franks_handel_rotation_num_is_def}
Let $\phi: S^{2} \rightarrow S^{2}$ be an area-preserving diffeomorphism with zero entropy, infinite order and at least $3$ periodic points. Let $U\in \mathcal{A}(\phi)$ be an invariant open annulus from Theorem \ref{thm:franks_handel_theorem} and let $U_{c}$ be the prime end compactification of $U$ as in \S\ref{sec:theory_prime_ends}. Letting $\phi_{c}: U_{c} \rightarrow U_{c}$ be the extension of $\phi$ to the compactification, then, the rotation number of every point $x\in U_{c}$ with respect to any lift $\t{\phi}_{c}$ is defined and the function $\rho_{\t{\phi}_{c}}(-): U_{c} \rightarrow \R$ is continuous.
\end{theorem}
\begin{remark}\phantomsection\label{rem:rot_def_for_zero_ent_closed_ann}
If $\phi$ is an area-preserving homeomorphism of a closed annulus $A$ that is smooth in the complement of finitely many interior fixed points and has zero entropy then the rotation number $\rho_{\t{\phi}}(x)$ is defined for every point $x\in A$ and $\rho_{\t{\phi}}:A \rightarrow \R$ is continuous; see \cite[Theorem 1.5]{franks-handel-zero-ent}. See also \S\ref{sec:rot_numb_ann}.
\end{remark}
The following lemmas will also be useful in the subsequent sections.
\begin{lemma}\phantomsection\label{lem:zero_ave_rot_num_fixed_point}
Let $\phi: A \rightarrow A$ be an area-preserving homeomorphism of the closed annulus $A$ isotopic to the identity. If the mean rotation number,
$$\frac{1}{\mathrm{area}(A)}\int_{A}\rho_{\t{\phi}}(x)\omega,$$
is an integer for some, and hence for all, lifts $\t{\phi}$ of $\phi$ to the universal cover $\t{A}$ of $A$, then $\phi$ must have fixed point in the interior of $A$.
\end{lemma}
\begin{proof}
See \cite[Proposition 2.6]{franks-handel-zero-ent} for a detailed proof.
\end{proof}
In the following lemma the rotation number $\rho_{\phi,V}$ is defined via the extension $\phi_{c}: V_{c} \rightarrow V_{c}$ to the prime end compactification of $V$ and $\rho_{\phi,U}$ is similarly defined using the extension $\phi_{c}: U_{c} \rightarrow U_{c}$ of $\phi$ to the compactification $U_{c}$. A bi-recurrent point $x$ is a point so that for every neighborhood $U_{x}$ of $x$ there exist $n,-m \geq 1$ so that $\phi^{n}(x), \phi^{m}(x) \in U_{x}$, i.e., $x\in \omega(x) \cap \alpha(x)$.
\begin{lemma}\phantomsection\label{lem:rot_num_int_ann_vs_ext_ann}
Suppose $V\in \mathcal{A}(\phi,q)$ is a proper and essential subset of $U\in \mathcal{A}(\phi)$. Then, there is a full measure subset $W \subset V$ containing $V \cap \mathrm{Per}(\phi)$ so that for every $x\in W$ one has, 
$$[\rho_{\phi,V}(x)] = [\rho_{\phi, U}(x)]\in \R/\Z,$$ 
where both sides are well-defined.
\end{lemma}
\begin{proof}
The set $W \subset V$ consisting of all the bi-recurrent points inside $V$ has full measure. Since $V$ is essential in $U$ and $\phi$ is area-preserving it must preserve $V$. For a point $x\in W\subset V$, we fix a small open disk $B_{x}$ with $x\in B_{x}$ and $n_{i}\to \infty$ a sequence of numbers so that $\phi^{n_{i}}(x)\in B_{x}$ for all $i\geq 1$. Hence, for any given point $x\in W$ and its lift $\t{x}\in \t{V}$ to the universal cover $\t{V}$ of $V$ we can join $\t{x}$ to $\t{\phi}^{n_{i}}(\t{x})$ by an arc $\t{\beta_{i}} \subset \t{V}$ where $\t{\phi}$ is a lift of $\phi$ to $\t{V}$. Let $\beta_{i} = \pi \circ \t{\beta}_{i}$. One can connect the two ends of $\beta_{i}$, which by the choice of $n_{i}$ both lie inside $B_{x}$ by an arc $\alpha_{i}$ who's length is bounded uniformly in $i$ and obtain a loop $\gamma_{i}$; this is where we use bi-recurrence since otherwise one may not be able to connect the ends by a curve with uniformly bounded length. Then, the rotation number is the limit of the homology class of $\gamma_{i}$ divided by $n_{i}$; this does not depend on the lift $\t{\phi}$ modulo $\Z$. One can use the same arcs to define the rotation number inside $U$ as the latter limit since they define the same homology classes inside $\mathrm{H}_{1}(U,\Z)$. This finishes the proof. See also \cite[Lemma 15.7]{franks-handel-zero-ent}.
\end{proof}
\begin{remark}\phantomsection\label{rem:equality_of_rot_num_nested_ann}
The proof of \cite[Lemma 15.7]{franks-handel-zero-ent} that was presented above works for any pair of annuli $A_{1} \Subset A_{2}$ and an area-preserving diffeomorphism $\phi$ on $A_{2}$ such that $A_{1}$ is essential in $A_{2}$ and is preserved by $\phi$.
\end{remark}

\begin{remark}\label{rem:rest_are_isotop_to_id}
Throughout the subsequent sections we often work with the restriction of a Hamiltonian diffeomorphism $\phi$ of a closed oriented surface $\Sigma$ of genus at least one to an invariant \emph{inessential} open topological annulus $A$. To freely apply results of \cite{franks1988generalizations, franks-handel-zero-ent} to $\phi|_{A}$ we need to confirm that the prime extension $(\phi|_{A})_{c}$ of $\phi|_{A}$ to the prime end compactification $A_{c}$ of $A$, see Lemma \ref{lem:carath_ext_of_maps} and Remark \ref{rem:prime_end_on_annulus}, is isotopic to the identity. By an immediate topological reason we know that $\phi$ preserves the boundary components of $A$ and so $(\phi|_{A})_{c}$ preserves the boundary components of $A_{c}$. Moreover, $(\phi|_{A})_{c}$ is an orientation preserving homeomorphism of the closed annulus $A_{c}$. So $(\phi|_{A})_{c}$ induces identity on the homology group $\H(A_{c}, \partial A_{c}, \Z)$ and it is \emph{homotopic} to the identity, which is immediate by linear interpolation in the universal cover. Now it follows from a theorem of Baer \cite{MR1581099,MR1581156} and Epstein \cite{MR214087}, see also \cite[Theorem 1.12]{MR2850125}, that $(\phi|_{A})_{c}$ is isotopic to the identity via homeomorphisms. 
\end{remark}

\subsection{Invariant annulus for a diffeomorphism of a surface of genus $\geq 1$}

For the purpose of proving Theorem \ref{thm:main_theorem_introduction}, the structure theorem of Franks and Handel \cite{franks-handel-zero-ent} from \S\ref{sec:franks_handel_thm_subsec} will be enough for us provided that the following holds:

\begin{theorem}\phantomsection\label{thm:up_coming_work_of_LCT}
Let $\Sigma$ be a closed oriented surface of genus at least $1$. For every given non-identity Hamiltonian diffeomorphism $\phi\in \Ham(\Sigma)$ with zero topological entropy there exists an open annulus $A\subset \Sigma$ that is $\phi$-invariant and the rotation interval $\rho_{\t{\phi}}(A)$ is non-trivial, where $\t{\phi}$ is a lift of $\phi|_{A}$ to the universal cover $\t{A}$ of $A$.
\end{theorem}
The authors were informed by Fabio Armando Tal that the above theorem was proved by Guih\'eneuf-Le Calvez-Passeggi-Tal in their upcoming work \cite{GLCPT-25}, as a special case of more general results. 

To make the paper a bit more self-contained we briefly explain how to find such an invariant annulus in the case of the torus $\Sigma =T^{2}$ under a mild assumption. For definitions of the notions used in the following discussion we refer to \cite{MR3193751, MR3820002}. Suppose $\phi\in \Ham(T^{2})$ with zero entropy. Moreover, assume $\Fix(\phi^{k})$ is not fully essential for all $k\geq1$; this is of course not a necessary assumption but it simplifies the following. Up to passing to an iteration of $\phi$, it follows from \cite[Theorem 2.1]{Guelman-Koropecki-Tal-14} and \cite[Proposition 1.4(5)]{MR3193751} that either $\Fix(\phi)$ is inessential or $\phi$ is annular; the latter means there is a direction $v\in \Z^{2}$ and a lift $\t{\phi}$ of $\phi$ to $\R^{2}$ that has bounded deviation in the direction of $v$. If $\phi$ is annular, with a bit of refining the arguments in the proof of \cite[Theorem A]{Guelman-Koropecki-Tal-14}, in particular using \cite[Proposition 2.17]{MR4650154} and 
\cite[Corollary 2.7]{Guelman-Koropecki-Tal-14}, one can find an essential invariant annulus $A$ with non-trivial rotation interval. Assume $\Fix(\phi)$ is inessential. By \cite[Theorem 70]{lecalvez-tal-forcing-theory} one of the following holds,
\begin{enumerate}
\item The rotation set $\rho(\t{\phi})$ is the origin and there exists $C>0$ so that $\abs{\t{\phi}^{n}(z) - z} \leq C$ for all $z\in \R^{2}$ and $n\in \Z$.
\item The rotation set $\rho(\t{\phi})$ is a segment with rational slope, the origin is in $\mathrm{int}(\rho(\t{\phi}))$, and there exists an essential $\phi$-invariant annulus in $T$.
\item The rotation set $\rho(\t{\phi})$ has non-empty interior, the origin is in $\mathrm{int}(\rho(\t{\phi}))$.
\end{enumerate}
The third item is excluded by the zero entropy assumption on $\phi$; see \cite{MR1101087} and \cite[Theorem G]{lecalvez-tal-forcing-theory}. The second item is similar to the annular case as shown in \cite[Theorem A]{Guelman-Koropecki-Tal-14}. Assume the first item and let $\t{\phi}$ be the lift of $\phi$ with bounded deviation. Clearly $\t{\phi}$ is the Hamiltonian lift of $\phi$. Take a bounded disk $D\subset \R^{2}$ containing a non-periodic point $x$ of $\t{\phi}$; such a point exists by \cite{MR1507223}. Let $D_{\infty}$ be the connected component of the orbit of $D$ under $\t{\phi}$ containing $x$. After ``filling" $D_{\infty}$, see \cite[\S1.4]{MR3193751}, one has an $\t{\phi}$-invariant open topological disk containing $x$. Find one other such disk $D'_{\infty}$ disjoint from $D_{\infty}$, not necessarily containing a non-periodic point, and another one $D''_{\infty}$ that contains both $D_{\infty}$ and $D'_{\infty}$. Now one can apply the following theorem of Franks and Handel to $\t{\phi}|_{D''_{\infty}}$ to find the desired invariant annulus for $\phi$, after possibly lifting $\phi$ to a finite cover of $T^{2}$; this is no loss of generality for the proof of Theorem \ref{thm:main_theorem_introduction}; see Corollary \ref{cor:spec_norm_covering_spc}. This finishes the poof of the existence of an invariant annulus with non-trivial rotation interval.

\section{Theory of prime ends}\phantomsection\label{sec:theory_prime_ends}
This section mostly follows Milnor's book \cite{MR2193309} and partially \cite{mather-top-proof-82, MR614728} and, for the sake of self-contained-ness, explains all the necessary background on the theory of prime ends first introduced by Carathéodory \cite{caratheodory-original-paper}. First we will give a precise definition of the \emph{prime end compactification} for which we require a few lemmas. This is a certain compactification of open simply connected domains in $\R^{2}$ equipped with a topology, the prime end topology, with respect to which, the prime end compactification will be homeomorphic to a closed disk; see Figure \ref{fig:prime_end_comp}. 
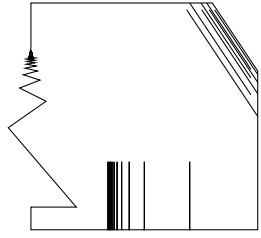
\begin{figure}[ht]
\centering
\begin{tikzpicture}[scale=0.3]
    \draw
        (-5, 5)
        -- (3, 5)
        -- (5, 2)
        -- (5, -5)
        -- (2, -5)
        \foreach \n in {1, 2, ..., 10} {
            -- ({-2 + 4/\n}, -5)
            -- ({-2 + 4/\n}, -2)
            -- ({-2 + 4/\n}, -5)
        }
        -- (-2, -5)
        -- (-5, -5)
        -- (-5, -4)
        \foreach \n in {1, 2, ..., 60} {
            -- ({-5 - cos(\n*180) * 2 / \n}, {3 - 7/\n})
        }
        -- (-5, 3)
        -- cycle;
    
    \foreach \x in {2.75, 2.5, 2} {
        \draw (\x, 5) -- (4.7, {-2.05 + 1.5*\x});
    }    
    \foreach \y in {1.85, 1.5,1, 0} {
        \draw (5, \y) -- ({1.867 + 2*\y/3}, 4.7);
    }
\end{tikzpicture}
\caption{An open simply connected domain.}
\label{fig:prime_end_comp}
\end{figure}
\begin{lemma}\phantomsection\label{lem:almost_all_lin_fin_len}
Let $f: \D \rightarrow U$ be a conformal diffeomorphism onto an open bounded subset $U \subset \C$. Then, for almost every point $\theta\in [0,2\pi)$ the length of the curve $f\circ \gamma_{\theta}$ is finite where $\gamma_{\theta}(r) = (r, \theta)$ and $r\in [0,1)$.
\end{lemma}
\begin{proof}
This follows from the fact that,
\begin{equation}\phantomsection\label{eq:where_conformal_is_used}
\begin{aligned}
\int_{0}^{2\pi}\mathrm{length}(f\circ \gamma_{\theta})^{2}d\theta &= \int_{0}^{2\pi} \big(\int_{0}^{1/2}|\partial_{r}f|dr + \int_{1/2}^{1}|\partial_{r}f|dr\big)^{2}d\theta\\
&\leq \int_{0}^{2\pi}4\big(\int_{1/2}^{1}\abs{\partial_{r}f}dr\big)^{2}d\theta=:X,
\end{aligned}
\end{equation}
where in the second line we have assumed that $\int_{0}^{1/2}\abs{\partial_{r}f}dr$ is smaller than $\int_{1/2}^{1}\abs{\partial_{r}f}dr$ since other wise we would be done because the former integral is finite. We now apply Cauchy-Schwartz inequality to get,
\begin{equation}
\begin{aligned}
X&\leq 4\int_{0}^{2\pi}\big(\int_{1/2}^{1}\abs{\partial_{r}f}^{2}rdr \int_{1/2}^{1}\frac{1}{r}dr\big)d\theta\\
&= 4\log 2\int_{0}^{2\pi}\int_{1/2}^{1}\abs{\partial_{r}f}^{2}rdrd\theta\\
&=2\log(2) \mathrm{area}(f(\D(1)\backslash\D(1/2))) < \infty,
\end{aligned}
\end{equation}
where the last line is the energy identity for the holomorphic function $f$.
\end{proof}
\begin{lemma}\phantomsection\label{thm:fatou_reisz}
Let $U$ be a simply connected bounded open subset of $\C$ and let $f: \D \rightarrow U$ be the Riemann map. For almost every point $x\in \partial\overline{\D}$ the length of the curve $f\circ \gamma_{x}\subset U$ is finite where $\gamma_{x}: (0,1) \rightarrow \D$ is defined by $\gamma_{x}(r) = rx$. In particular, the limit $\lim f(\gamma_{x}(r))$ exists as $r \to 1$ for almost every $x\in \partial \overline{\D}$. Moreover, fixing $x\in \partial\overline{\D}$, for almost every $y\in \partial\overline{\D}$ one has,
$$\lim_{r \to 1} f(\gamma_{y}(r)) \neq \lim_{r \to 1} f(\gamma_{x}(r)).$$
\end{lemma}
\begin{proof}
The first statement immediately follows from Lemma \ref{lem:almost_all_lin_fin_len}. For the second statement, it is enough to show that the set of all boundary points $y\in \partial \overline{\D}$ for which $\lim_{r\to 1} f(ry) = 0$ has Lebesgue measure zero. Assume the contrary, so the latter set has some non-zero Lebesgue measure $e > 0$. Define
$$E(\epsilon, \delta):= \{y\in \partial \overline{\D} : |f(ry)| < \epsilon,\ \forall r\in (1-\delta, 1)\}.$$ 
Hence, for all small enough $\epsilon, \delta$ we have Lebesgue measure of $E(\epsilon, \delta) > e/2$. Up to multiplying $f$ by a constant we may assume that $|f| < 1$ and so $\log(|f|) < 0$ everywhere on $\D$. For all small enough $\epsilon,\delta$ and any $r$ close enough to $1$ we have
$$\int_{\partial \overline{\D}} \log|f(ry)|dy \leq \log(\epsilon)E(\epsilon, \delta) \leq \log(\epsilon)\ e/2,$$
which by taking $\epsilon \to 0$ implies that the left hand side diverges to $-\infty$ as $r \to 1$ which is a contradiction unless $f \equiv 0$ which is not the case here; the latter contradiction is with Jensen's inequality which asserts that the average on the left hand side is monotonically increasing as $r$ increases.
\end{proof}

\begin{lemma}\phantomsection\label{lem:shrinking_nested_neighbs}
Let $f: \D \rightarrow U$ be the Riemann map as before. Every point $x\in \partial \overline{\D}$ possesses a nested sequence of neighborhoods $N_{1}\supseteq N_{2} \supseteq N_{3} \supseteq \dots$ in $\overline{\D}$ so that,
\begin{enumerate}
\item The boundaries $\overline{N}_{k} \cap \overline{\D \backslash N_{k}}$ consists of an open curve $\gamma_{k} \subset \D$ with $\overline{\gamma}_{k}$ having exactly two end points lying on $\partial \overline{\D}$.
\item The boundary pieces $\overline{\gamma}_{k}$ are pair-wise disjoint.
\item Both $\gamma_{k}$ and $f\circ \gamma_{k}$ have bounded lengths and their length converge to zero as $k \to \infty$. 
\item Each curve $\overline{f\circ \gamma_{k}}$ has exactly two distinct points on the boundary $\partial \overline{U}$ and the curves $\overline{f \circ \gamma}_{k}$ are pair-wise disjoint.
\end{enumerate}
\end{lemma}
\begin{proof}
This readily follows from Lemma \ref{lem:almost_all_lin_fin_len} and Lemma \ref{thm:fatou_reisz} above; one can pick $N_{k} \subset \overline{\D}$ to be the region enclosed by two generic small radial lines near the boundary $\partial \D$ and a segment on some circle $|z| = c_{k}$ where $c_{k} \to 1$ as $k \to \infty$.
\end{proof}

\begin{lemma}\phantomsection\label{lem:transverse_cut_corresp}
Suppose $\gamma: [0,1) \rightarrow U$ be an embedded curve with $\lim_{r \to 1}\gamma(r)$ exists and is equal to a point in $\partial \overline{U}$. Then, the curve $f^{-1}\circ \gamma$ converges to a point $\partial\overline{\D}$ as $r\to 1$. Moreover, if two such curves $\gamma_{1}, \gamma_{2}$ converge to distinct points on $\partial\overline{U}$, then, the curves $f^{-1}\circ \gamma_{1}$ and $f^{-1}\circ \gamma_{2}$ converge to distinct points on $\partial\overline{\D}$.
\end{lemma}
\begin{proof}
Suppose $\gamma(r) \to p\in \partial \overline{U}$ as $r \to 1$. If $f^{-1}\circ \gamma$ does not converge to one single point then one can find a subset $E$ of $\partial \overline{\D}$ with non-zero Lebesgue measure so that for each $y\in E$ the curves $r\mapsto f(ry)$ converge to $p$ which is in contradiction with Lemma \ref{thm:fatou_reisz}. For the second part, if $\gamma_{i}(r) \to p_{i}$ for some $p_{1} \neq p_{2}\in \partial \overline{U}$ and $f^{-1}\circ \gamma_{1}(r), f^{-1}\circ \gamma_{2}(r) \to p\in \partial \overline{\D}$ for some single point $p\in \partial \overline{\D}$ then one arrives at a contradiction by constructing "shrinking" neighborhoods around $p$ satisfying properties of Lemma \ref{lem:shrinking_nested_neighbs}.
\end{proof}
A subset $\Gamma \subset \overline{U}$ is called a \emph{transverse cut} if it is the image of an injective continuous map $\gamma: [0,1] \rightarrow \overline{U}$ so that either it is a simple closed curve that does not intersect $\partial \overline{U}$ or it intersects $\partial \overline{U}$ only at $\gamma(0), \gamma(1)$. Any such curve decomposes $U$ into two connected components; this can be seen by collapsing the boundary of $U$ into a point and observe that $\Gamma$ becomes a Jordan curve in $S^{2}$. Fix a point $q\in U$ and for any transverse cut $\Gamma$ define $N_{\Gamma}$ to be the connected component of $U \backslash \Gamma$ that does not contain $q$.
\begin{definition}\phantomsection\label{def:prime_ends}
A chain in $U$ is a sequence of pair-wise disjoint transverse cuts $\Gamma_{1}, \Gamma_{2}, \dots, $ so that $N_{\Gamma_{1}} \supseteq N_{\Gamma_{2}} \supseteq \dots,$ and the diameter of $\Gamma_{i} \to 0$ as $i \to \infty$. For two chains $\{\Gamma_{i}\}$ and $\{\Gamma_{i}'\}$ we say $\{\Gamma_{i}'\}$ is smaller than $\{\Gamma_{i}\}$ and write $\{\Gamma_{i}'\}\leq \{\Gamma_{i}\}$ if for every $i$ there exists $j$ so that $N_{\Gamma_{j}'} \subset N_{\Gamma_{i}}$. We say that $\{\Gamma_{i}\}$ and $\{\Gamma_{i}'\}$ are equivalent if $\{\Gamma_{i}'\} \leq \{\Gamma_{i}\} \leq \{\Gamma_{i}'\}$. An equivalence class of chains $[\{\Gamma_{i}\}]$ is called a \emph{prime end} if for every other chain $\{\Gamma_{i}'\}\leq \{\Gamma_{i}\}$ one has $[\{\Gamma_{i}'\}] = [\{\Gamma_{i}\}]$. See Figure \ref{fig:non_prime_trans_cut} for a non-prime end.
\end{definition}
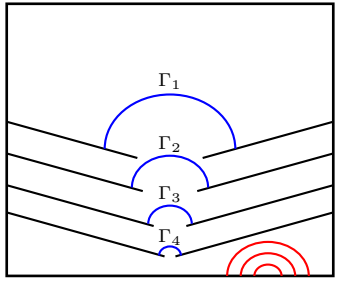
\begin{figure}
\centering
\begin{tikzpicture}[scale=0.6, xscale = 1.2, line width=0.8pt]

    \draw[line width=1pt] (0,0) rectangle (6,6);

    \draw[red] (4.55, 0) arc[start angle=180, end angle=0, radius=0.25];
    \draw[red] (4.30, 0) arc[start angle=180, end angle=0, radius=0.50];
    \draw[red] (4.05, 0) arc[start angle=180, end angle=0, radius=0.75];

    \draw (0, 3.4) -- (2.4, 2.6);               
    \draw (3.6, 2.6) -- (6, 3.4);               
    \draw[blue] (1.8, 2.8) arc[start angle=180, end angle=0, radius=1.2]; 
    \node at (3.0, 4.3) {\tiny{$\Gamma_1$}};

    \draw (0, 2.7) -- (2.5, 1.87);              
    \draw (3.5, 1.87) -- (6, 2.7);              
    \draw[blue] (2.3, 1.95) arc[start angle=180, end angle=0, radius=0.7]; 
    \node at (3.0, 2.9) {\tiny{$\Gamma_2$}};

    \draw (0, 2.0) -- (2.7, 1.1);               
    \draw (3.3, 1.1) -- (6, 2.0);               
    \draw[blue] (2.6, 1.15) arc[start angle=180, end angle=0, radius=0.4]; 
    \node at (3.0, 1.8) {\tiny{$\Gamma_3$}};

    \draw (0, 1.4) -- (2.9, 0.43);              
    \draw (3.1, 0.43) -- (6, 1.4);              
    \draw[blue] (2.8, 0.45) arc[start angle=180, end angle=0, radius=0.2]; 
    \node at (3.0, 0.85) {\tiny{$\Gamma_4$}};
\end{tikzpicture}
\caption{The blue transverse cuts $\{\Gamma_{i}\}$ define a \emph{non-prime end} since the red transverse cuts are smaller but not equivalent to $\{\Gamma_{i}\}$; one should imagine infinitely many diagonal segments whose ends inside $U$ converge to a point on $\partial \bar{U}$.}
\label{fig:non_prime_trans_cut}
\end{figure}
Two prime ends $[\{\Gamma_{i}\}]$ and $[\{\Gamma_{i}'\}]$ are disjoint if there exist $i,j$ so that $N_{\Gamma_{i}} \cap N_{\Gamma_{j}'} = \emptyset.$ The \emph{impression} of a prime end $p = [\{\Gamma_{i}\}]$ is $\mathrm{I}(p):= \cap_{i\geq 1} N_{\Gamma_{i}}.$
\begin{lemma}\phantomsection\label{lem:imp_is_subset_boundary}
The impression $\mathrm{I}(p)$ of any prime end $p$ is either a point in $U$ or a connected compact subset of $\partial\overline{U}$.
\end{lemma}
\begin{proof}
Let $p = [\{\Gamma_{i}\}]$ be a prime end and suppose that $\mathrm{I}(p)\not\subset \partial \bar{U}$. Let $z\in U \cap \mathrm{I}(p)$, so $z\in N_{\Gamma_{k}}$ for all $k\geq 1$. Let $D_{k}\subset N_{\Gamma_{k}}$ a closed topological disk containing $z$ and let $\Gamma_{k}' := \partial D_{k}$ so that the diameter of $D_{k}$ goes to zero as $k\to \infty$. Then it is clear that $\{\Gamma_{i}'\}\leq \{\Gamma_{i}\}$ and since $\{\Gamma_{i}\}$ is a prime end we conclude that $p =\{\Gamma_{i}\} =\{\Gamma_{i}'\}$ and hence $\mathrm{I}(\{\Gamma_{i}\}) =\mathrm{I}(\{\Gamma_{i}'\}) = \{z\}$. This finishes the proof.
\end{proof}
Terminology. From now on we call an end $\{\Gamma_{i}\}$ prime if $\mathrm{I}(\{\Gamma_{i}\}) \subset \partial \bar{U}$.

Let $U \subset \C$ a simply connected bounded open set. The \emph{Carathéodory} completion of $U$ is denoted by $U_{c} = U \sqcup \{\text{prime ends of $U$}\}$ with the topology generated by the following basis,
\begin{enumerate}
\item open subsets of $U$.
\item For every transverse cut $\Gamma$ we add the basis element,
\begin{equation}\phantomsection\label{eq:prime_end_neighborhoods}
\mathcal{N}_{\Gamma} := N_{\Gamma} \sqcup \{[\{\Gamma_{i}\}] : \exists j \text{ s.t. } N_{\Gamma_{j}} = N_{\Gamma}\}.
\end{equation}
\end{enumerate}
\begin{remark}
The identity map $id: \D \rightarrow \D$ extends uniquely to a homeomorphism $id_{c}: \overline{\D} \rightarrow \D_{c}.$
\end{remark}
The following is a standard result in prime end theory and is the main result of this section which is due to Carathéodory \cite{caratheodory-original-paper},
\begin{theorem}\phantomsection\label{thm:ext_to_closure_vs_carath_compl}
Let $f: U\rightarrow \D$ be a Riemann map from a simply connected bounded open set $U\subset \C$ to the standard disk $\D \subset \C$. Then, $f$ extends uniquely to a homeomorphism,
$$f_{c}: U_{c} \rightarrow \D_{c}=\bar{\D}.$$
\end{theorem}
\begin{proof}
It follows from Lemma \ref{lem:transverse_cut_corresp} that any chain $\{\Gamma_{i}\}$ in $U$ uniquely corresponds to a chain in $\D$. It is clear that this descends to a well-defined injective map, denoted by $f_{c}$, on the prime ends of $U$, i.e., one defines $f_{c}([\{\Gamma_{j}\}]):= [\{\overline{f\circ \Gamma_{j}|_{(0,1)}}\}]$. Given a prime end of $\D_{c} = \overline{\D}$ it will have a representative of the form constructed in Lemma \ref{lem:shrinking_nested_neighbs}, hence, the correspondence $f_{c}$ is surjective. Let $\mathcal{N}_{\Gamma}$ be a basis element corresponding to a transverse cut $\Gamma \subset \overline{\D}$. We need to show that $f_{c}^{-1}(\mathcal{N}_{\Gamma})$ is an open set. Pick a point $[\{\Gamma_{i}\}]\in f_{c}^{-1}(\mathcal{N}_{\Gamma})$ which means $[\{\overline{f\circ \Gamma_{i}|_{(0,1)}}\}] \in \mathcal{N}_{\Gamma}$ and hence there exists $k$ so that $\overline{f \circ \Gamma_{k}|_{(0,1)}} = \Gamma$. The latter implies that $\mathcal{N}_{\Gamma_{k}} \subset f_{c}^{-1}(\mathcal{N}_{\Gamma})$ which is an open neighborhood of $[\{\Gamma_{i}\}]$.
\end{proof}
\begin{lemma}\phantomsection\label{lem:carath_ext_of_maps}
Given two bounded simply connected open sets $U_{1}$ and $U_{2}$ and a smooth diffeomorphism $\phi: U_{1} \rightarrow U_{2}$ with bounded derivative, there is a unique extension $\phi_{c}: U_{1,c} \rightarrow U_{2,c}$ of $\phi$ to a homeomorphism to the prime end compactification.
\end{lemma}
\begin{proof}
The proof is similar to that of Theorem \ref{thm:ext_to_closure_vs_carath_compl} except that it is easier under the bounded derivative assumption. To define the extension for a prime end $[\{\eta_{i}\}]$ we define $\phi_{c}([\{\eta_{i}\}]) = [\{\phi(\eta_{i})\}]$ which is easily checked to be well-defined and bijective since the derivative of $\phi$ is bounded. To check that it is a homeomorphism let $\mathcal{N}_{\eta}$ be a basis of the prime end topology on $\widehat{U}_{2}$ and let $[\{\xi_{i}\}]\in \phi^{-1}(\mathcal{N}_{\eta})$. We claim that $\phi^{-1}(\mathcal{N}_{\eta}) = \mathcal{N}_{\phi^{-1}(\eta)}$. If $[\{\xi_{i}\}] \in \phi^{-1}(\mathcal{N}_{\eta})$ then there is some $J$ so that $\phi(\xi_{j}) = \eta$ and so $[\{\xi_{i}\}] \in \mathcal{N}_{\xi_{j}} = \mathcal{N}_{\phi^{-1}(\eta)}$, and if $[\{\xi_{j}\}]\in \mathcal{N}_{\phi^{-1}(\eta)}$ then there is some $J$ with $\xi_{j} = \phi^{-1}(\eta)$ and so $\phi([\{\xi_{i}\}]) \in \mathcal{N}_{\eta}$ which implies $[\{\xi_{i}\}] \in \phi^{-1}(\mathcal{N}_{\eta})$. So the map $
\phi_{c}$ defined above is a bijective continuous map with continuous inverse and hence it is a homeomorphism that obviously extends $\phi$.
\end{proof}

\begin{remark}\label{rem:prime_end_on_annulus}
One can define the notion of prime ends and the prime end topology for all open subsets $U$ of the plane $\C$ similar to Definition \ref{def:prime_ends}. In particular, if $U$ is an open \emph{bounded} and \emph{regular} topological annulus in $\C$ then the prime end compactification $U_{c}$ of $U$ defined similarly as above will be homeomorphic to a closed annulus. Moreover, if $f: \A\rightarrow U$ is a Riemann map where $\A =\{z\in \C: a <\abs{z}<b\}$ for some numbers $b > a> 0$ then it uniquely extends to a homeomorphism $f_{c}: \bar{\A} = \A_{c} \rightarrow U_{c}$. A similar statement holds for all diffeomorphisms $\phi: \A \rightarrow U$ with bounded derivatives. The proofs are similar to the above. See for instance the Corollary in Page $401$ of \cite{MR543339}, see also \cite{mather-top-proof-82}.
\end{remark}


\section{Approximately invariant Lagrangians}\phantomsection\label{sec:fragmentation_c0_approximation}
In this section we prove that one can perturb a symplectic map of an open disk, which extends to a homeomorphism of the closed disk, by a $C^{0}$-small symplectic diffeomorphism so that the perturbed map preserves a circle $\Lambda_{\eps}$, defined below, near the boundary of the disk and the closer the circle is to the boundary one can pick the $C^{0}$-norm of the perturbation smaller. The construction of this section is a main step in the proof of the main theorem of the paper in the non-smooth cases; see \S\ref{sec:non_smooth_non_ess_inv_ann}. Fix a symplectic form $\Omega$ on the open disk $\D$ so that the area of $\D$ with respect to $\Omega$ is finite; note that $\Omega$ may not even be extendable to the boundary of the disk. Let $\phi: \D \rightarrow \D$ be a smooth map that preserves $\Omega$ and it extends to a homeomorphism $\phi_{c}: \bar{\D} \rightarrow \bar{\D}$. For every $0 < \epsilon < 1$ let $\Lambda_{\epsilon}:= \{|z| = 1 - \epsilon\}.$ 
\begin{remark}\phantomsection\label{rem:adjustment_from_disk_to_annulus}
The results presented in this section for the open disk $\D$ will also be valid for the open annulus $\A$; see Remark \ref{rem:prime_end_on_annulus}. The only difference is that the perturbation maps are apriori only symplectic diffeomorphisms which in the case of $\D$ are automatically Hamiltonian diffeomorphisms. In general, not every symplectic diffeomorphism of the annulus $\A$ is Hamiltonian. However, we will use the analysis of this section only in \S\ref{sec:non_smooth_inv_ann} where the open annulus $A \subset \Sigma$ is a contractible annulus and the perturbation maps will be Hamiltonian as diffeomorphisms of the surface $\Sigma$.
\end{remark}

\subsection{Khanevsky's smooth perturbation with an invariant Lagrangian}\phantomsection\label{sec:smooth_pert_inv_lag}
The goal of this section is to construct smooth diffeomorphisms $\phi_{\mathrm{kh},n}: \D \rightarrow \D$ compactly supported in the disk $\D$ so that $\phi_{\mathrm{kh}, n}(\Lambda_{\epsilon_{n}}) = \phi(\Lambda_{\epsilon_{n}})$ for some sequence $\eps_{n}\to 0$ so that, 
$$|\phi_{\mathrm{kh}, n}|_{C^{0}} \to 0 \text{ as } n \to \infty.$$
In addition, the map $\phi_{\mathrm{kh}, n}$ will be supported in a small neighborhood, 
$$V_{\epsilon_{n}}' = \{\frac{(1-R_{\epsilon_{n}})^{2}}{1 - r_{\eps_{n}}} < |z| < 1 - r_{\epsilon_{n}}\},$$ 
of $\Lambda_{\epsilon_{n}}$ where $r_{\epsilon_{n}}, R_{\epsilon_{n}}$ are some numbers converging to zero as $n \to \infty$. The Khanevsky map $\phi_{\kh,n}$ will be defined to be the composition of three shift maps, 
$$\phi_{\kh, n} := (sh_{n,3} \circ sh_{n, 2} \circ sh_{n,1})^{-1},$$ 
that are constructed in the subsequent sections; see also \S\ref{sec:khanevsky_map}.

\subsubsection{The first shift map}\phantomsection\label{sec:shift_map_first}
Since $\phi$ extends to a homeomorphism of $\bar{\D}$ it follows that it is uniformly continuous and so for every $\eps$ there are numbers $0 < r_{\eps} <\eps< R_{\eps} < 1$ so that $r_{\eps},R_{\eps} \to 0$ as $\eps \to 0$ and moreover
$$\phi(\Lambda_{\eps})\subset V_{\eps} = \{1- R_{\eps} < |z| < 1 - r_{\eps}\}.$$
We construct a $C^{0}$-small diffeomorphism $sh_{\eps,1}: \D \rightarrow \D,$ compactly supported near $\Lambda_{\epsilon}$ so that $sh_{\eps,1}(\phi(\Lambda_{\epsilon})) \cap \Lambda_{\epsilon} = \emptyset$. We define $sh_{\eps,1}$ to be the time-one map of the flow of the following vector field,
$$X_{sh,1}(\theta,r) = \alpha_{\eps}(r)\log(\frac{1 - R_{\epsilon}}{1 - r_{\epsilon}})r\partial_{r},$$
where $\alpha_{\eps}: \R \rightarrow \R$ is a smooth monotone cut off function that equals $1$ on the entire interval $[\frac{(1 - R_{\epsilon})^{2}}{1 - r_{\epsilon}}, 1 - r_{\epsilon}]$ except near a small enough neighborhood of its boundary where $\alpha_{\eps}$ decays and eventually vanishes on an even smaller neighborhood of its boundary. We require in particular that $\alpha_{\eps}(r)$ is equal to $1$ on $\frac{1 - R_{\epsilon}}{1 - r_{\epsilon}}\phi(\Lambda_{\epsilon})$. Then, on the curve $\phi(\Lambda_{\epsilon})$ one has
$$sh_{\eps,1}(\theta,r) = (\theta, \frac{1 - R_{\epsilon}}{1 - r_{\epsilon}}r).$$
It is clear that the map $sh_{\eps,1}$ is supported in
\begin{equation}\phantomsection\label{eq:the_bigger_neigh}
V_{\epsilon}' = \{(\theta,r) : \frac{(1 - R_{\epsilon})^{2}}{1 - r_{\epsilon}} < r < 1 - r_{\epsilon}\},
\end{equation}
and since it preserves the radial lines one has $|sh_{\eps,1}|_{C^{0}} \leq  1 - \frac{(1 - R_{\epsilon})^{2}}{1 - r_{\epsilon}}$ which converges to zero as $\epsilon \to 0$. It is now clear that $sh_{\eps,1}(\phi(\Lambda_{\epsilon})) \subset \D(1 - R_{\epsilon})$ and so it does not intersect $\Lambda_{\epsilon}$. See Figure \ref{fig:shift_map_one} for an illustration of the first shift map.
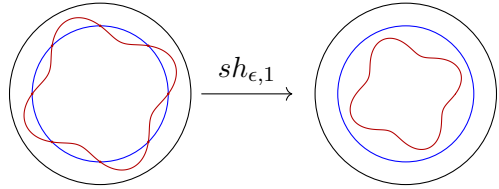
\begin{figure}[ht]
\centering
\begin{tikzpicture}[line cap=round, scale = 0.3]
  \draw (0,0) circle[radius=4];

  \draw[blue] (0,0) circle[radius=3];

  \draw[red!70!black]
    plot[variable=\t, domain=0:360, samples=361, smooth]
    (\t : {3 + 0.55*sin(4*\t)});

  \foreach \a in {0,45,...,315}
    \fill[red!70!black] (\a:3) circle[radius=1.6pt];
\end{tikzpicture}
\begin{tikzpicture}[line cap=round, scale = 0.3]
  \draw[->] (-9,0) -- (-5,0);
  \node at (-7,1) {$sh_{\eps,1}$};
  \draw (0,0) circle[radius=4];
  \draw[blue] (0,0) circle[radius=3];

  \draw[red!70!black]
    plot[variable=\t, domain=0:360, samples=361, smooth]
    (\t : {0.73*(3 + 0.55*sin(4*\t))});
\end{tikzpicture}
\caption{The first shift map $sh_{\eps,1}$. The blue circle is $\Lambda_{\eps}$.}
\label{fig:shift_map_one}
\end{figure}
\subsubsection{The second shift map}\phantomsection\label{sec:shift_map_second}
Fix a small $\eps > 0$ and $\Lambda_{\eps}$, and let $\phi, sh_{\eps,1}$ be as in the previous section. Denote $\Lambda_{\eps}' = sh_{\eps,1}(\phi(\Lambda_{\eps}))$. So we have $\Lambda_{\eps}' \cap \Lambda_{\eps} = \emptyset$ and $\Lambda_{\eps}' \subset V_{\eps}'$. Let $N > 0$ be an integer and $\bar{x} = (x_{1}, \dots, x_{N})$ be some distinct points clockwise ordered on $\Lambda_{\eps}$. Suppose that they are in generic positions so that the radial lines $\ell_{i} = \{rx_{i}: r\leq 1\}$ intersect $\Lambda_{\eps}'$ transversely for all $i = 1,\dots, N$. Let $x_{i}'\in \Lambda_{\eps}'$ be the intersection point of $\ell_{i} \cap \Lambda_{\eps}'$ closest to $x_{i}$. Pick arbitrary small neighborhoods $\xi_{i}$ of $x_{i}$ inside $\Lambda_{\epsilon}$ for $i=1,\dots,N$ so that they are pairwise disjoint and $\Lambda_{\epsilon}' \cap P_{\eps,i}$ is a disjoint union of small arcs $c_{i1} \sqcup c_{i2} \sqcup \dots \sqcup c_{ik_{i}}$ where
$$P_{\eps,i} := \{(\theta,r) : (\theta,1-\epsilon)\in \xi_{i}, \theta\in [0,2\pi), r\leq 1 - r_{\epsilon}\}.$$ 
Denote by $c_{i}$ the connected component $c_{ij_{i}}$ that is the closest to $\Lambda_{\epsilon}$. It is the component that contains $x_{i}'$. Then, $c_{i}$ is the graph of a smooth $C^{0}$-small function on $\xi_{i}$. Fix some strictly smaller connected arcs $c_{i}' \subset c_{i}$ containing $x_{i}'$. One can find smooth diffeomorphisms $g_{i}$ compactly supported in $P_{\eps,i} \backslash \D(\frac{(1-R_{\eps})^{2}}{1-r_{\eps}})$ so that $g_{i}(c'_{i}) \subset \Lambda_{\epsilon}$, moreover, if one lets $G = g_{1}g_{2}\dots g_{N}$ then $G(\Lambda_{\epsilon}') \subset \overline{\D}(1 - \epsilon)$. We now explicitly construct the diffeomorphisms $g_{i}$. Let $\rho_{i}: \xi_{i} \rightarrow [\frac{(1 - R_{\epsilon})^{2}}{1 -r_{\eps}}, 1-r_{\eps}]$ be a smooth function so that,
$$c_{i} = \{\big(\theta, \rho_{i}(\theta, 1-\epsilon)\big) : (\theta, 1-\epsilon)\in \xi_{i}\}.$$
Then, consider the following vector field, 
$$X_{sh,2}^{i}(\theta,r) = \big((1 - \epsilon) - \rho_{i}(\theta,1-\epsilon)\big)\partial_{r}(\theta,r).$$
Let $U$ be a small tubular neighborhood $U$ of the boundary of $P_{\eps,i} \backslash \D(\frac{(1-R_{\eps})^{2}}{1-r_{\eps}})$ and $U \subset U'$ be a slightly large neighborhood. Let $\beta_{\eps,i}$ be a smooth monotone cut-off function that vanishes $U$ and equals $1$ on $P_{\eps, i}\backslash U'$. Let $g_{i,t}$ be the flow of $\beta_{\eps,i} X^{i}_{sh,2}$. It is then clear that the time-one maps $g_{i}:= g_{i,1}$ for $i = 1,\dots, N$ satisfy the required properties. We denote the map $G$ above by $sh_{\bar{x},\eps,2}$. Since the supports of $g_{i}$'s are disjoint and have diameter converging to zero with $\epsilon \to 0$ we conclude that $|sh_{\bar{x}, \eps,2}|_{C^{0}} \to 0$ as $\eps \to 0$. Hence, letting $\Lambda_{\bar{x},\eps}'' := sh_{\bar{x}, \eps, 2}(\Lambda_{\eps}')$ we have $\Lambda_{\bar{x},\eps}'' \subset \bar{\D}(1 - \eps)$ and it intersects $\Lambda_{\eps}$ on some small neighborhoods of the points $x_{i}$ for $i = 1,\dots,N$. Removing $\{x_{i}\}_{i=1}^{N}$ from $\Lambda_{\bar{x},\eps}''$ splits it into $N$ disjoint open arcs which we denote by $\eta_{1} \sqcup \dots \sqcup \eta_{N}$ where $\eta_{i}$ is the arc whose closure connects $x_{i}$ to $x_{i+1}$ where $x_{N+1}:= x_{1}$; we sometimes denote $\eta_{i} = [x_{i},x_{i}]_{\Lambda_{\bar{x},\eps}''}$, the oriented arc in $\Lambda_{\bar{x},\eps}''$ connecting $x_{i}$ to $x_{i+1}$. We also denote by $D_{\bar{x},\eps,i}$ the topological disk bounded by $\eta_{i} =[x_{i}, x_{i+1}]_{\Lambda_{\bar{x},\eps}''}$ and $[x_{i}, x_{i+1}]_{\Lambda_{\eps}}$ where the latter is the oriented arc in $\Lambda_{\eps}$ connecting $x_{i}$ to $x_{i+1}$. See Figure \ref{fig:shift_map_two} for an illustration of the second shift map.
\begin{figure}[ht]
\centering
\begin{tikzpicture}[line cap=round, scale = 0.3]
  \draw (0,0) circle[radius=4];
  \draw[blue] (0,0) circle[radius=3];

  \draw[red!70!black]
    plot[variable=\t, domain=0:360, samples=361, smooth]
    (\t : {0.73*(3 + 0.55*sin(4*\t))});
\end{tikzpicture}
\begin{tikzpicture}[line cap=round, scale = 0.3]
  \draw[->] (-9,0) -- (-5,0);
  \node at (-7,1) {$sh_{\eps,2}$};
  \draw (0,0) circle[radius=4];
  \draw[blue] (0,0) circle[radius=3];

  \draw[red!70!black]
    plot[variable=\t, domain=0:360, samples=361, smooth]
    (\t : {0.85*(3 + 0.55*sin(4*\t))});
\end{tikzpicture}
\caption{The second shift map $sh_{\eps,2}$. The blue circle is $\Lambda_{\eps}$.}
\label{fig:shift_map_two}
\end{figure}
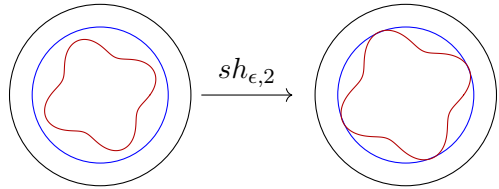
\subsubsection{The third shift map}\phantomsection\label{sec:shift_map_third}
For every $\eps > 0$ the construction of $sh_{\eps,1}$ from \S\ref{sec:shift_map_first} depends on the data $(r_{\eps},R_{\eps}, \alpha_{\eps})$. The construction of $sh_{\bar{x},\eps,2}$ depends on the choices of $(r_{\eps}, R_{\eps}, \xi_{i},c_{i}',\beta_{\eps,i})$. We pick a sequence of numbers $\{\eps_{k}\}_{k\geq 1}$ so that $\eps_{k} \to 0$ as $k \to \infty$ and we let $\Lambda_{k}:= \Lambda_{\eps_{k}}$. For every $k$ we pick the numbers $r_{k} := r_{\eps_{k}}$ and $R_{k} := R_{\eps_{k}}$ as in \S\ref{sec:shift_map_first} which converge to zero as $k\to \infty$. Then we fix a choice of cut off functions $\alpha_{k} := \alpha_{\eps_{k}}$ as in \S\ref{sec:shift_map_first}. So far these data determine a sequence of diffeomorphisms $sh_{k,1}: = sh_{\eps_{k},1}$, whose $C^{0}$-norm converge to zero with $k$, and $\Lambda_{k}' := sh_{k,1}(\Lambda_{k})$. We now pick a sequence of tuples $\bar{x}_{k} = (x_{k,1}, x_{k,2}, \dots, x_{k,k})$ consisting of points in $\Lambda_{k}$ that are generic and equidistributed, i.e., the radial lines $l_{k,j} =\{rx_{k,j}: r \leq 1\}$ intersect $\Lambda_{k}'$ transversely and $\{x_{k,j}\}_{j=1}^{k}$ become dense in $\Lambda_{k}$ as $k \to \infty$. We now pick the data $(\xi_{k,j}, c_{k,j}', \beta_{k,j})_{j=1}^{k}$ where $\xi_{k,i}$ are arbitrary small in particular $\ell(\xi_{i}) \to 0$ as $ k\to \infty$, and we construct $sh_{k,2} = sh_{\bar{x}_{k}, \eps_{k},2}$ as in \S\ref{sec:shift_map_second} using the above data. Hence, we have a fixed sequence of maps $sh_{k,2},sh_{k,1}$ and curves $\Lambda_{k}, \Lambda_{k}'$ and $\Lambda_{k}'' := sh_{k,2}(\Lambda_{k}')$ so that, 
$$|sh_{k,1}|_{C^{0}}, |sh_{k,2}|_{C^{0}} \to 0 \text{ as } k \to \infty,$$ 
and for every $k$ one has $\Lambda_{k}, \Lambda_{k}', \Lambda_{k}'' \subset V_{k}'$ where $V_{k}' = V_{\eps_{k}}'$ is as in \eqref{eq:the_bigger_neigh} and finally $\Lambda_{k}'' \subset \bar{\D}(1- \eps_{k})$ and it intersects $\Lambda_{k}$ at small neighborhoods of the points $\{x_{k,i}\}_{i=1}^{k}$. Removing the points $\{x_{k,j}\}_{j=1}^{k}$ from $\Lambda_{k}''$ one obtains $k$ disjoint arcs $\eta_{k,1} \sqcup \eta_{k,2} \sqcup \dots \sqcup \eta_{k,k}$ and we denote by $D_{k,j}$ the topological disk enclosed by $[x_{k,j}, x_{k,j+1}]_{\Lambda_{k}} \cup [x_{k,j}, x_{k,j+1}]_{\Lambda_{k}''}$ where $x_{k,k+1} := x_{k,1}$. In the following lemma,
$$x_{k,j} = (\theta_{k,j}, r_{k,j}),$$
is the coordinate representation of $x_{k,j}$ for every $k \geq 1$ and $1\leq j\leq k$.
\begin{lemma}\phantomsection\label{lem:diam_of_remaining_arcs}
For any $\delta > 0$ there exists an integer $N> 0$ such that for every $k \geq N$ and $1\leq j\leq k$, and for every point $p =(\theta,r)\in \eta_{k,j}$, we have $$\theta \in [\theta_{k,j} -\delta, \theta_{k,j+1} + \delta].$$
\end{lemma}
\begin{proof}
Suppose the contrary and assume that there is a positive number $\delta >0$ and a sequence of numbers $k_{n} \to \infty$ as $n \to \infty$ so that the conclusion does not hold for $k_{n}$, i.e., for every $n$ there exists $1\leq j_{n} \leq k_{n}$ and a point $p_{n} =(\theta_{n},r_{n})\in \eta_{k_{n},j_{n}}$ so that,
$$\theta_{n}\not\in [\theta_{k_{n}, j_{n}} - \delta, \theta_{k_{k}, j_{n}+1} + \delta].$$
Up to passing to a subsequence we can assume that $p_{n} \to p\in \partial\bar{\D}$ for some point $p$ on the boundary. Given that $\{x_{k_{n},j}\}_{j=1}^{k_{n}}$ are becoming dense in $\Lambda_{k_{n}}$ we conclude that $d(x_{k_{n},j_{n}}, x_{k_{n}, j_{n}+1}) \to 0$ and so up to passing to a subsequence we assume that $x_{k_{n}, j_{n}}, x_{k_{n}, j_{n}+1} \to x\in \partial{\bar{\D}}$ for some point $x\in \partial\bar{\D}$. Writing $x = (\theta_{x}, 1)$ and $p = (\theta_{p},1)$ it is clear by construction that $|\theta_{p}-\theta_{x}| \geq \delta > 0$, in particular $x\neq p$. Let,
$$f_{n} = sh_{k_{n},2}\circ sh_{k_{n},1} \circ \phi,$$ 
and $\xi_{n} = f_{n}^{-1}(\eta_{k_{n},j_{n}})$, $y_{n} = f_{n}^{-1}(x_{k_{n},j_{n}}), y_{n}' = f_{n}^{-1}(x_{k_{n},j_{n}+1})$ and $q_{n} = f_{n}^{-1}(p_{n})$. Then $\xi_{n}$ will be a connected arc with ends $y_{n},y_{n}'$ that contains $q_{n}$ and the ends $y_{n},y_{n}'$ converge to some point $y = \phi^{-1}(x)\in \partial\bar{\D}$ while $q_{n}$ converge to some point $q = \phi^{-1}(p)\in \partial \bar{\D}$ and $y \neq q$. So up to passing to subsequences either $\ell(\xi_{n}) \to 0$ or $\ell(\xi_{n}) \to 2\pi$ as $n \to \infty$, the first case is not possible as the limit arc $\xi_{\infty}:= \lim \xi_{n}$ contains two distinct points $y$ and $q$. So we must have $\ell(\xi_{n}) \to 2\pi$. This implies $\ell(\xi_{n}^{c}) \to 0$ as $n \to \infty$ where $\xi_{n}^{c} := \Lambda_{k_{n}}\backslash \xi_{n}$. This is a contradiction because $\xi_{n}^{c}$ contains all the points $f_{n}^{-1}(x_{k_{n},j})$ where $j \neq j_{n}$ and by uniform continuity of $f_{n}$ one would conclude that $d(x_{k_{n}, j}, x_{k_{n}, j'}) \to 0$ for all $j,j' \neq j_{n}$ and this is in contradiction with the fact that $\{x_{k_{n},j}\}_{j=1}^{k_{n}}$ are becoming dense in $\Lambda_{k_{n}}$ by construction.
\end{proof}
Hence, by passing to a subsequence and relabeling the sequence, we can assume that for every $n\geq 1$ and $1 \leq j \leq n$ we have
$$D_{n,j} \subset \{(\theta,r): \theta\in [\theta_{n,j}-\delta_{n}, \theta_{n,j+1} + \delta_{n}], r\in [\frac{(1 -R_{n})^{2}}{1 - r_{n}}, 1-r_{n}]\},$$
where $\delta_{n} >0$ are some numbers that converge to zero as $n \to\infty$. So, in particular we conclude that $\mathrm{diam}(D_{n,i}) \to 0$ as $n \to \infty$. The disks $\{D_{n,i}\}_{i=1}^{n}$ are pairwise disjoint. One can now construct smooth diffeomorphisms $h_{n,i}$ supported in a slightly larger neighborhood of $D_{n,i}$, still pairwise disjoint, so that $h_{n,i}(\eta_{n,i}) = [x_{n,i}, x_{n,i+1}]$ for all $i = 1, \dots, n$. We define
$$sh_{n,3}:= sh_{\bar{x}_{n}, \eps_{n}, 3} := h_{n,1} \circ h_{n,2} \circ \dots \circ h_{n,n},$$
which satisfies $|sh_{n,3}|_{C^{0}} \to 0$ as $n \to \infty$ and $sh_{n,3}(\Lambda_{n}'') = \Lambda_{n}$. See Figure \ref{fig:shift_map_three} for an illustration of the third shift map.
\begin{figure}[ht]
\centering
\begin{tikzpicture}[line cap=round, scale = 0.3]
  \draw (0,0) circle[radius=4];
  \draw[blue] (0,0) circle[radius=3];

  \draw[red!70!black]
    plot[variable=\t, domain=0:360, samples=361, smooth]
    (\t : {0.85*(3 + 0.55*sin(4*\t))});
\end{tikzpicture}
\begin{tikzpicture}[line cap=round, scale = 0.3]
  \draw[->] (-9,0) -- (-5,0);
  \node at (-7,1) {$sh_{\eps,3}$};
  \draw (0,0) circle[radius=4];
  \draw[red!70!black] (0,0) circle[radius=3];
\end{tikzpicture}
\caption{The third shift map $sh_{\eps,3}$. The blue circle is $\Lambda_{\eps}$.}
\label{fig:shift_map_three}
\end{figure}
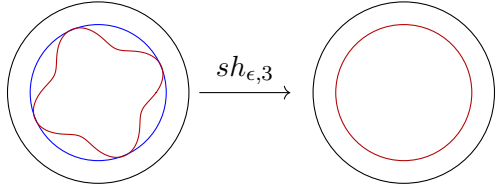
\subsubsection{The Khanevsky map}\phantomsection\label{sec:khanevsky_map}
Finally for every $n \geq 1$ for the data $\mathrm{d}_{n} = (\eps_{n}, \bar{x}_{n})$ introduced in \S\ref{sec:shift_map_third} we define
$$\phi_{\kh,n}^{-1}:= sh_{n, 3} \circ sh_{n,2} \circ sh_{n,1}.$$
Since the norms $|sh_{n,3}|_{C^{0}}, |sh_{n,2}|_{C^{0}}, |sh_{n, 1}|_{C^{0}} \to 0$ as $n\to \infty$, see \S\ref{sec:shift_map_first}, \S\ref{sec:shift_map_second} and \S\ref{sec:shift_map_third}, up to passing to a subsequence and relabeling the sequence we can assume that for every $n\geq 1$,
$$\begin{aligned}
|\phi_{\kh,n}|_{C^{0}} &\leq |sh_{n,1}^{-1}|_{C^{0}} + |sh_{n,2}^{-1}|_{C^{0}} + |sh_{n, 3}^{-1}|_{C^{0}} \leq 3/n.
\end{aligned}$$
Moreover, for every $n \geq 1$ we have $\phi_{\kh,n}(\Lambda_{n}) = \phi(\Lambda_{n})$.

\subsection{A symplectic perturbation with an invariant Lagrangian}\phantomsection\label{sec:smooth_pert_inv_lag_symp} 
In this section we use the techniques of Lefeuvre-Sikorav, see the work of Entov-Polterovich-Py \cite{MR2884036} and Leclaire \cite{ian_thesis}, to show that one can perturb $\phi_{n}:= \phi_{\mathrm{kh}, n}$, from the previous section, with ($C^{0}$-small) diffeomorphisms so that the resulting map will preserve $\Omega$ and still map $\Lambda_{n}$ to $\phi(\Lambda_{n})$. We will construct a sequence of smooth compactly supported diffeomorphisms $\psi_{p}$, by perturbing the maps $\phi_{n}$, so that for some large enough $n_{p}$ depending on $p$,
$$\psi_{p}^{*}\Omega = \Omega,\ |\psi_{p}|_{C^{0}(D_{p} \cup \ell_{0})}\leq 1/p,\ \psi_{p} (\Lambda_{{n_{p}}}) = \phi(\Lambda_{{n_{p}}}),$$
where $D_{p}$ is the disk with boundary $\Lambda_{n_{p}}$ and $\ell_{0} = \bar\D\cap \R^{+}$, moreover, $\mathrm{supp}(\psi_{p})\Subset U_{p} = \{1- 1/p < |z| < 1\}$. 
\begin{figure}[ht]
\begin{tikzpicture}[scale = 0.5]
    \def\Rad{5}      
    \def\m{4}      
    
    \pgfmathsetmacro{\closeRadius}{\Rad*0.96}
    
    \pgfmathsetmacro{\innerRadius}{\Rad*(\m-1)/\m}
    
    \pgfmathsetmacro{\sectorAngle}{360/\m}
    
    \fill[red!20, even odd rule] (0,0) circle (\Rad) (0,0) circle (\closeRadius -0.05);
    
    \fill[green!20, even odd rule] (0,0) circle (\closeRadius -0.05) (0,0) circle (\innerRadius);
    
    \foreach \k in {1,...,\numexpr\m-1} {
        \pgfmathsetmacro{\radius}{\Rad*\k/\m}
        \draw[thick] (0,0) circle (\radius);
    }
    
    \draw[thick] (0,0) circle (\Rad);
    
    \draw[blue, thick] (0,0) circle (\closeRadius -0.05);
    
    \foreach \k in {0,...,\numexpr\m-1} {
        \pgfmathsetmacro{\angle}{360*\k/\m}
        \draw[thick] (0,0) -- (\angle:\Rad);
    }
    
    \draw[green, thick] (0:\innerRadius) -- (0:\Rad);
    
    \pgfmathsetmacro{\smallCircleRadius}{\innerRadius + 0.3}
    \foreach \k in {1,...,\m} {
        \pgfmathsetmacro{\midAngle}{360*(\k-0.5)/\m}
        \draw[thick] (\midAngle:\smallCircleRadius) circle (0.15);
    }
    
    \node at (45:\Rad+1) {$\sigma_{1,n}''$};
    \node[rotate=215] at (45:{\Rad+0.2}) {$\rightsquigarrow$};
    \node at (20:\Rad+1.2) {$\sigma_{1,n}'$};
    \node[rotate=200, scale=1.7] at (20:{\Rad}) {$\rightsquigarrow$};
    
    \pgfmathsetmacro{\firstMidAngle}{360*0.5/\m}
    \node at (61:\Rad-0.6) {\tiny{$B_1$}};
    \node[rotate=-70] at (53:{\Rad-0.8}) {$\rightsquigarrow$};
\end{tikzpicture}
\caption{The partition $K$. The blue line is $\Lambda_{n_{p}}$. The map $\psi_{p}$ is supported in the green and the red regions. The $C^{0}$-norm of $\psi_{p}$ converges to zero as $p\to \infty$ in the green region and on the green horizontal segment; the norm is not controlled in the red region except near the boundary. The small circles are $B_{i}$'s and they contain the supports of $\eta_{i}$ from \S\ref{sec:the_functions_eta_n}.}
\label{fig:the_partition}
\end{figure}
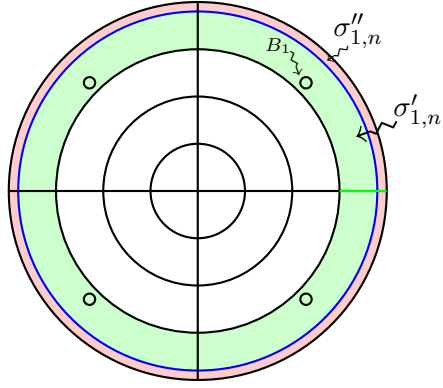
In the following we outline the proof up to some technicalities that will be explained in details in the subsequent sections. Fix a partition $K$ of $\bar{\D}$ into many small cells using concentric circles and radial lines. Let $\sigma_{1}, \sigma_{2}, \dots, \sigma_{m}$ be the cells adjacent to $\partial \bar{\D}$. Fix one small disk  $B_{i} \subset \mathrm{int}(\sigma_{i})$ for $i = 1,\dots, m$ and denote $B = B_{1} \sqcup B_{2}\sqcup \dots \sqcup B_{m}$; see Figure \ref{fig:the_partition}.  Then observe that for each $\sigma \in K$, $$|\mathrm{area}_{\Omega}(\phi_{n}(\sigma)) - \mathrm{area}_{\Omega}(\sigma)| \to 0 \text{ as } n \to \infty;$$
this is because $|\phi_{n}|_{C^{0}} \to 0$ as $n \to \infty$ and area of the set of one dimensional cells $K^{1}$ of $K$ equals zero. Let $N = N(K,B) > 0$ large enough which will be determined throughout the section and restrict the following discussion to all $n \geq N$. In particular, $N$ should be large enough so that, 
$$|\phi_{n}|_{C^{0}} \leq \max_{\sigma \in K}\mathrm{diam}(\sigma), \forall n \geq N.$$
Let $\Omega_{n}:= \phi_{n}^{*}\Omega$. Picking $N$ large enough we can assume that $V_{n}' \cap B = \emptyset,$ where $V_{n}'$ is defined by \eqref{eq:the_bigger_neigh} for $\eps = \eps_{n}$, see \S\ref{sec:shift_map_first}, and it contains the support of $\phi_{n}$. We construct a smooth diffeomorphism $\phi_{n,1}$, see \S\ref{sec:const_of_phi_one}, that is supported in an arbitrarily small neighborhood of $K^{1} \cup \Lambda_{n}$, it is identity on $K^{1} \cup \Lambda_{n}$, $|\phi_{n,1}|_{C^{0}}$ is arbitrarily small and $\Omega_{n,1} :=\phi_{n,1}^{*}\Omega_{n} = \Omega$ on $K^{1} \cup \Lambda_{n}$. We pick $N$ large enough so that $B_{i} \subset \mathrm{int}(\phi_{n}\phi_{n,1}(\sigma_{i}))$ for all $i= 1,\dots,m$. Adding $\Lambda_{n}$ to the partition $K$ we obtain a new partition $\t{K}$ where each $\sigma_{i}$ is split into two cells $\sigma_{i,n}'$ and $\sigma_{i,n}''$ where the latter shares an edge with boundary $\partial \bar\D$; see Figure \ref{fig:the_partition}. There is a smooth function $\bar\eta_{n}: \D \rightarrow \R^{>0}$, see \S\ref{sec:the_functions_eta_n}, so that $|\bar\eta_{n} - 1|_{C^{0}} \to 0$ as $n \to \infty$ and,
$$\int_{\sigma_{i,n}'}\Omega_{n,1} = \int_{\sigma_{i,n}'}\bar\eta_{n}\Omega,\ \forall i \in \{1,\dots, m\}.$$
Moreover, $\bar\eta_{n}-1$ will be supported in $B_{1} \sqcup B_{2} \sqcup \dots \sqcup B_{m}$ which is away from the boundary $\partial\bar{\D}$ independent of $n$; note that the support of $\phi_{n}\phi_{n,1}$ only intersects $\sigma_{1} \cup\dots \cup \sigma_{m}$ for large enough $n$ so for every other $\sigma\in K$ that is not from $\{\sigma_{i}\}_{i=1}^{m}$ one has $\int_{\sigma}\Omega_{n,1} = \int_{\sigma}\bar\eta_{n}\Omega$. We find a primitive $\beta$, see \S\ref{sec:const_of_phi_two}, for the two form $\Omega_{n,1} - \bar\eta_{n}\Omega$ so that it vanishes on $\t{K}^{1}\cap \bar\D(1 - \eps_{n})$, on $\ell_{0}$, and near the boundary of the disk. Note that this is possible in particular because $\Omega_{n,1}$ and $\bar\eta_{n}\Omega$ have the same area on each $\sigma \in \t{K}\cap \bar\D(1-\eps_{n})$ and they coincide on $\t{K}^{1} = K^{1} \cup \Lambda_{n}$. Then, because, 
$$\begin{aligned}
\int_{\D} \Omega_{n,1} &= \int_{\bar\D(1-\eps_{n})} \Omega_{n,1}+ \int_{\D\backslash\bar\D(1-\eps_{n})}
\Omega_{n,1}\\
&= \int_{\bar\D(1-\eps_{n})} \bar\eta_{n}\Omega+ \int_{\phi_{n}\phi_{n,1}(\D\backslash\bar\D(1-\eps_{n}))}\Omega\\
&= \int_{\bar\D(1-\eps_{n})} \bar\eta_{n}\Omega+ \int_{\phi(\D\backslash\bar\D(1-\eps_{n}))}\Omega =\int_{\D}\bar\eta_{n}\Omega,
\end{aligned}$$
by Moser's argument there is a diffeomorphism $\phi_{n,2}$, see \S\ref{sec:const_of_phi_two}, of the disk so that $\Omega_{n,2}:= \phi_{n,2}^{*}\Omega_{n,1} = \bar\eta_{n}\Omega$ and $\phi_{n,2} = id$ on $\t{K}^{1}\cap \bar\D(1 - \eps_{n})$, on $\ell_{0}$, and near the boundary of the disk, hence, 
$$|\phi_{n,2}|_{C^{0}(\bar\D(1-\eps_{n}))} \leq \max_{\sigma\in K}\mathrm{diam}(\sigma), \phi_{n,2}|_{\ell_{0}} = id.$$
Moreover, one can find a smooth isotopy $\phi_{n,2}^{t}$ of the disk with time-one map $\phi_{n,2}$ and preserving $\Lambda_{n}$ for all time, so that it is $C^{0}$-small inside $\bar\D(1-\eps_{n})$, i.e., $|\phi_{n,2}^{t}|_{C^{0}(\bar\D(1-\eps_{n}))}$ is smaller than the diameter of cells of $K$ for all $t\in [0,1]$; see \S\ref{sec:the_functions_eta_n}. So we have
$$|\phi_{n}\phi_{n,1}\phi_{n,2}|_{C^{0}} \leq 3 \max_{\sigma \in K}\mathrm{diam}(\sigma),$$
and $(\phi_{n}\phi_{n,1}\phi_{n,2})^{*}\Omega = \bar\eta_{n}\Omega$ and $|\eta_{n}-1|_{C^{0}} \to 0$ as $n \to \infty$.

The two forms $\bar\eta_{n}\Omega$ and $\Omega$ coincide on a neighborhood of the boundary $\partial\bar\D$ that is independent of $n$. Let $\alpha_{n} := \bar\eta_{n} \Omega - \Omega$. Since integral of $\alpha$ over $\D$ is zero and $\alpha_{n}$ is compactly supported in $\D$, by the Poincaré lemma for compactly supported De Rham cohomology there is a primitive one form $\beta_{n,3}$ for $\alpha_{n}$ that is compactly supported in the disk $\D$ and its support can be chosen independent of $n$. Hence, by Moser's argument applied to $\Omega_{n,t} = \Omega + t(\bar\eta_{n}\Omega - \Omega)$ and solving $\Omega_{n,t}(X_{n,t}, -) = -\beta_{n,3}$ for $X_{n,t}$ one constructs a diffeomorphism $\phi_{n,3}$ so that $\phi_{n,3}^{*}(\bar\eta_{n}\Omega) =\Omega$, $\phi_{n,3} = id$ on a neighborhood of the boundary $\partial \D$ that is independent of $n$, and $|\phi_{n,3}|_{C^{0}} \to 0$ as $n \to \infty$; see \S\ref{sec:const_of_phi_three} for more details. Letting $\phi_{K,B,n}:= \phi_{n}\phi_{n,1}\phi_{n,2}\phi_{n,3}$ for some $n \geq N$ and $N$ large enough depending on $K,B$ as above, we have shown that for any partitioning $K$ and choices of disks $B_{1}, \dots, B_{m}$ as above there exists a smooth diffeomorphism $\phi_{K,B,n}$ of the disk $\D$ so that $\phi_{K,B,n}(\Lambda_{n}) = \phi(\Lambda_{n})$, $\phi_{K,B,n}^{*}\Omega = \Omega$ and,
$$
\begin{aligned}
|\phi_{K,B,n}|_{C^{0}(\bar\D(1-\eps_{n})\cup \ell_{0})} &= |\phi_{n}\phi_{n,1}\phi_{n,2}\phi_{n,3}|_{C^{0}(\bar\D(1-\eps_{n})\cup \ell_{0})}\\
&\leq |\phi_{n}|_{C^{0}} + |\phi_{n,1}|_{C^{0}} +|\phi_{n,2}|_{C^{0}(\bar\D(1-\eps_{n})\cup \ell_{0})} + |\phi_{n,3}|_{C^{0}}\\
&\leq  4\max_{\sigma\in K}\mathrm{diam}(\sigma).
\end{aligned}
$$
Let now $K_{p}, B_{p} = B_{p,1}\sqcup B_{p,2} \sqcup \dots \sqcup B_{p,m_{p}}$ be a sequence of partitions of the disk $\bar{\D}$ by concentric circles and radial lines and $B_{p,i} \subset \mathrm{int}(\sigma_{p,i})$ where $\{\sigma_{p,i}\}_{i=1}^{m_{p}}$ are the cells adjacent to $\partial \bar{\D}$ and so that,
$$\max_{\sigma\in K_{n}} \mathrm{diam}(\sigma) \leq 1/4p.$$
Then, letting $\psi_{p} = \phi_{K_{p}, B_{p}, n_{p}}$ for some $n_{p} \geq N_{p}$ where $N_{p}$ is large enough integer depending on $K_{p},B_{p}$, we have
$$\psi_{p}^{*}\Omega = \Omega,\ |\psi_{p}|_{C^{0}(\bar\D(1-\eps_{n_{p}}) \cup \ell_{0})}\leq 1/p,\ \psi_{p} (\Lambda_{n_{p}}) = \phi(\Lambda_{n_{p}}).$$
\subsubsection{Construction of \texorpdfstring{$\phi_{n,1}$}{}}\phantomsection\label{sec:const_of_phi_one}
Fix a partition $K$ of $\bar{\D}$ and a collection of small balls $B = B_{1}\sqcup \dots \sqcup B_{N}$ as in \S\ref{sec:smooth_pert_inv_lag_symp}. Let $\t{K}$ be the refined partitioning obtained by adding $\Lambda_{n}$ to $K$ and remove the vertices and edges on the boundary of the disk. First we pick small disjoint disks $D_{v}$ around each vertex $v\in \t{K}^{0}$. By Darboux's theorem we find diffeomorphisms $g_{v}$ compactly supported in $D_{v}$ so that $g_{v}^{*}\Omega_{n} = \Omega$ on some possibly smaller disks $D_{v}' \subset D_{v}$ for all $v\in \t{K}^{0}$. Let $G = \Pi_{v\in \t{K}^{0}}g_{v}$, it is clear that $|G|_{C^{0}}$ is smaller than the diameter of the disks $D_{v}$ which one can choose arbitrary small. For every $e\in \t{K}^{1}$ we pick a small tubular neighborhood $Q_{e}$ of $e \backslash \sqcup_{v}\frac{1}{2}D'_{v}$ so that they are pairwise disjoint. We then construct diffeomorphisms $h_{e}$ compactly supported in $Q_{e}$ with arbitrary small $C^{0}$-norm so that $h_{e}^{*}G^{*}\Omega_{n} = \Omega$ on $e$. We explain the construction of $h_{e}$ momentarily. Let $H = \Pi_{e\in \t{K}^{1}}h_{e}$. Then, $H^{*}G^{*}\Omega_{n} = \Omega$ on $\t{K}^{1}$. Defining $\phi_{n,1} := GH$ it is obvious that it satisfies the required properties above.

We now explain the construction of $h_{e}$. Let $\Omega_{n}' = G^{*}\Omega_{n}.$ Let $e\in \t{K}^{1}$ and $e' = e \backslash (\sqcup_{v}\frac{1}{2}D'_{v})$. We would like to construct a smooth diffeomorphism $h_{e}$ that is compactly supported in a small tubular neighborhood of $e$, it is identity on $e$, $|h_{e}|_{C^{0}}$ is arbitrarily small, and $h_{e}^{*}(\Omega_{n}') = \Omega$ on $e$. Let $r > 0$ be a small number and let $\chi: T_{r} \rightarrow \D$ be a smooth map from $T_{r} = [0,1] \times (-r,r)$ into the disk $\D$ whose image is a small neighborhood of $e$ satisfying $\chi([0,1]\times \{0\}) = e$ and so that a small neighborhood of $\partial T_{r} = \{0,1\}\times (-r,r)$ is mapped into $\cup_{v\in \t{K}^{0}}D'_{v}$ where $\Omega_{n}' = \Omega$. Let $\Omega_{n,\mathrm{loc}}' = \chi^{*}\Omega_{n}'$ and $\Omega_{\mathrm{loc}} = \chi^{*}\Omega$. Let $\Omega_{n,\mathrm{loc}}' = \rho\Omega_{\mathrm{loc}}$ for some smooth function $\rho: [0,1] \times (-r,r) \rightarrow \R$. To construct the map $h_{e}$ first we construct it in the local chart $T_{r}$. Consider the following autonomous vector field on $T_{r}$,
$$X = -\beta(t)\log(\rho(s,0))t\partial_{t},$$
where $\beta(t) = 1$ near $|t| <  r''$ for some $r'' \ll r$ and $\beta(t) = 0$ near $|t| > r'$ where $r'' \ll r' \ll r$. The time-one map of the flow of $X$ near $|t| < r''$ is,
\begin{equation}\phantomsection\label{eq:flow_aut_vf_rho}
h_{e,\mathrm{loc}}(s,t) = (s, e^{-\beta(t)\log(\rho(s,0))}t),
\end{equation}
which is equal to $(s,t)\mapsto (s,t/\rho(s,0))$. Since $h_{e,\mathrm{loc}}$ preserves the vertical lines $\{s= \mathrm{const}\}$ and is identity on $|t| > r'$ the $C^{0}$-norm of $h_{e,\mathrm{loc}}$ is smaller than $2r' \ll r$ which can be chosen arbitrarily small; see the proof of \cite[Lemma 6.5]{MR2884036} for a similar argument. It is easy to check that $h_{e,\mathrm{loc}}^{*}(\Omega_{n,\mathrm{loc}}) = \Omega_{\mathrm{loc}}$ on $[0,1] \times \{0\}$. Define $h_{e} = \chi h_{e,\mathrm{loc}}\chi^{-1}$. Then, we have the following on $e$,
$$\begin{aligned}
h_{e}^{*}(\Omega_{n}') &= (\chi^{-1})^{*}h_{e,\mathrm{loc}}^{*}\chi^{*}(\Omega_{n}')\\
&= (\chi^{-1})^{*}h_{e,\mathrm{loc}}^{*}(\Omega_{n,\mathrm{loc}}') \\
&= (\chi^{-1})^{*}\Omega_{\mathrm{loc}}= \Omega.
\end{aligned}$$
Moreover, $|h_{e}|_{C^{0}}$ can be made arbitrary small by picking $r'$ sufficiently small compare to $C^{1}$-norm of $\chi$. Note that near the boundary of $[0,1] \times (-r,r)$ we have $\rho = 1$ which implies that the map $h_{e,\mathrm{loc}} = id$ near the boundary, hence $h_{e}$ is compactly supported in the tubular neighborhood $Q_{e}:=\mathrm{Im}(T_{r})$.

\subsubsection{Construction of \texorpdfstring{$\bar\eta_{n}$}{}}\phantomsection\label{sec:the_functions_eta_n}
Fix the partition $K$ and the balls $B = B_{1} \sqcup \dots \sqcup B_{m}$ where $B_{i} \subset \mathrm{int}(\sigma_{i})$; see \S\ref{sec:smooth_pert_inv_lag_symp} for the notations. We define $\bar\eta_{n} = 1 + \sum_{i=1}^{m} \delta_{n,i}\eta_{i}$ where,
$$\delta_{n,i} = \int_{\sigma_{i,n}'}\Omega_{n,1} - \int_{\sigma_{i,n}'}\Omega,$$
and $\eta_{i}: \D \rightarrow \R$ is a fixed, independent of $n$, smooth function compactly supported in $B_{i}$ so that,
$$\int_{B_{i}}\eta_{i} \Omega = 1, \forall i\in \{1,2,\dots,m\}.$$
Since $|\phi_{n}\phi_{n,1}|_{C^{0}} \to 0$ as $n \to \infty$ it follows that $\delta_{n,i} \to 0$ hence it is clear that for large enough $n$ one has $\bar\eta_{n} > 0$ and $|\bar\eta_{n} - 1|_{C^{0}} \to 0$ as $n \to \infty$. Moreover, we have
$$\begin{aligned}
\int_{\sigma_{i,n}'}\bar\eta_{n}\Omega &= \int_{\sigma_{i,n}'}(1 + a_{n,i}\eta_{i})\Omega\\
&= \int_{\sigma_{i,n}'}\Omega + \bigg(\int_{\sigma_{i,n}'}\Omega_{n,1} -\int_{\sigma_{i,n}'}\Omega\bigg)\int_{B_{i}}\eta_{i}\Omega=\int_{\sigma_{i,n}'}\Omega_{n,1}.
\end{aligned}$$
For the rest of the cells $\sigma \in K\cap \bar\D(1 - \eps_{n})$ it is because $\eta_{n} = 1$ and $\phi_{n}\phi_{n,1} = id$.

\subsubsection{Construction of the primitive \texorpdfstring{$\beta_{n,2}$}{} and \texorpdfstring{$\phi_{n,2}$}{}}\phantomsection\label{sec:const_of_phi_two}
To construct the primitive $\beta_{n,2}$ we adjust the construction of Banyaga \cite{banyaga-moser-comp-mnflds} in his proof of Moser's theorem for compact manifolds as suggested in \cite[\S6]{MR2884036}. Let $\Delta$ be a smooth two form on the disk that is zero near the boundary and on $K^{1}$ for some partition $K$ given by radial lines $\{\ell_{i}\}_{i=0}^{m-1}$ with angles $\{\theta_{i}\}_{i=0}^{m-1}$ so that $\theta_{0} = 0$, and concentric circles $\{c_{i}\}_{i=1}^{m}$ with radii $r_{1}<r_{2}<\dots<r_{m}=1$, let $r_{0}=0$, see Figure \ref{fig:the_partition}, and the area of each cell $\sigma \in K$ is zero with respect to $\Delta$ except possibly the cells $\sigma_{1},\dots, \sigma_{m}$ that share an edge with the boundary $\partial\bar\D$. We show that there is a primitive $\beta$ for $\Delta$ that vanishes on $(K^{1}\cap\bar\D(r_{m-1})) \cup\ell_{0} \cup \mathcal{N}_{\partial\bar\D}$ where $\mathcal{N}_{\partial\bar\D}$ is a neighborhood of the boundary $\partial\bar\D$. Let $f(r,\theta): \D \rightarrow \R$ be a smooth function so that $\Delta(r,\theta) = f(r,\theta) dr\wedge d\theta.$ Define
$$\bar{f}_{i}(\theta):= \int_{r_{i-1}}^{r_{i}}f(s,\theta)ds; i=1,2,\dots,m.$$
Let $a_{i}(r):\R \rightarrow \R$ be a smooth function that is compactly supported in the interval $(r_{i-1}, r_{i})$ and $\int_{\R}a_{i}(s)ds = 1$ for all $i = 1,\dots,m$. Consider the following one form,
$$\beta(r,\theta):= \bigg[\int_{0}^{r}\big(f(s,\theta) - \sum_{i=1}^{m}a_{i}(s)\bar{f}_{i}(\theta)\big)ds\bigg]d\theta - \bigg[\sum_{i=1}^{m}a_{i}(r)\int_{0}^{\theta}\bar{f}_{i}(t)dt\bigg]dr.$$
We claim that $\beta$ is the desired primitive of $\Delta.$ The coefficient of $d\theta$ in $\beta$ vanishing on the radial lines $\ell_{i}$ immediately follows from the assumption that $\Delta$ and hence $f$ vanishes on $K^{1}$. The coefficient of $dr$ vanishes on radial lines inside $\bar\D(r_{m-1})$ because the terms $\int_{0}^{\theta_{j}}\bar{f}_{i}(t)dt$ for $1\leq i\leq m-1$ and $0\leq j\leq m-1$ are equal to the sum of the $\Delta$-areas of some cells from $K\cap \bar\D(r_{m-1})$ which are assumed to be zero; here we used the assumption that $\ell_{0} = \{\theta = 0\}$. To see that it vanishes on $\ell_{0}$ observe that for the fixed $\theta = \theta_{0} = 0$ not only the coefficient of $d\theta$ vanishes since $f(s,0)=0$ for all $s$, but also the coefficient of $dr$ obviously vanishes; note that for $\theta= \theta_{j}$ for $j \neq 0$ we do not know whether the coefficient of $dr$ vanishes because the areas of cells near the boundary are not necessarily zero. Next we show that $\beta$ vanishes on the concentric circles, i.e., $\beta(r_{i},\theta) = 0$ for every $\theta$ and all $i=1,\dots,m$. Fix some $j\in \{1,\dots, m\}$. Then,
$$\begin{aligned}
\beta(r_{j}, \theta) &= \bigg[\int_{0}^{r_{j}}f(s,\theta)ds - \sum_{i=1}^{j}\int_{r_{i-1}}^{r_{i}}a_{i}(s)\bar{f}_{i}(\theta)ds\bigg]d\theta-\bigg[\sum_{i=1}^{j}a_{i}(r_{j})\int_{0}^{\theta}\bar{f}_{i}(t)dt\bigg]dr\\
&=\bigg[\int_{0}^{r_{j}}f(s,\theta)ds - \sum_{i=1}^{j}\bar{f}_{i}(\theta)ds\bigg]d\theta = 0,
\end{aligned}$$
where the first equality is because $a_{i}(r) = 0$ for $r\leq r_{j}$ and $i\geq j+1$ and the second equality is by $\int_{\R}a_{i}(s)ds = 1$ and the fact that $a_{i}(r_{j}) = 0$ for all $i,j$. To see that $\beta$ vanishes near the boundary one observes that for $r$ close enough to $1$ one has $a_{i}(r)=0$ and $\int_{0}^{r}f(s,\theta)ds = \int_{0}^{1}f(s,\theta)ds$. It remains to show that $\beta$ is actually a primitive of $\Delta = f(r,\theta)dr\wedge d\theta$, this is a straightforward calculation.

To construct the desired $\beta_{n,2}$ let $\Delta_{n}:= \Omega_{n,1} - \bar\eta_{n}\Omega$ and let $K$ be a given partition by concentric circles and radial lines. Refine the partition $K$ to $\t{K}$ by adding the circle $\Lambda_{n}$ to the set of concentric circles defining $K$. By construction in \S\ref{sec:smooth_pert_inv_lag_symp} above we know that $\Delta_{n}$ vanishes on the one cells $\t{K}^{1}$ of $\t{K}$ and the areas of the cells of $\t{K}$ are zero with respect to $\Delta_{n}$ except possibly the cells $\sigma_{1,n}'', \dots, \sigma_{m,n}''$ that share an edge with the boundary $\partial\bar\D$; see Figure \ref{fig:the_partition}. Applying the above construction to $\Delta_{n}$ and $\t{K}$ we obtain the primitives $\beta_{n,2}$ that vanish on $(\t{K}^{1}\cap \bar\D(1 - \eps_{n})) \cup \ell_{0} \cup \mathcal{N}_{n,\partial\bar\D}$ where $\mathcal{N}_{n,\partial\bar\D}$ is some neighborhood of the boundary $\partial\bar\D$; recall that $\Lambda_{n}=\{|z| = 1-\eps_{n}\}$. Therefore, following Moser's trick, solve the equation $\Omega_{n,t}(X_{n,t}, -) = -\beta_{n,2}$ for $X_{n,t}$ and let $\phi_{n,2}$ be the time-one map of the flow $\phi_{n,2}^{t}$ of $X_{n,t}$, where $\Omega_{n,t} = \bar\eta_{n}\Omega + t\Delta_{n} = (1-t)\bar\eta_{n}\Omega + t\Omega_{n,1}$ which is a path of symplectic forms since for large enough $n$ we have $|\bar\eta_{n} - 1|_{C^{0}} \sim 0$ hence $\bar\eta_{n}\Omega$ is a symplectic form and so the convex linear combinations of $\bar\eta_{n}\Omega$ and $\Omega_{n,1}$ are symplectic forms. It is now clear that $\phi_{n,2}$ is the identity on $(K^{1}\cap\D(1-\eps_{n})) \cup \Lambda_{n} \cup \ell_{0}$ which implies that, 
$$|\phi_{n,2}|_{C^{0}(\bar\D(1-\eps_{n}) \cup\ell_{0})} \leq \max_{\sigma \in K}\mathrm{diam}(\sigma).$$
We point out two more properties of $\phi_{n,2}$ that will be important. First, there exist a smooth isotopy $\{\phi_{n,2}^{t}\}_{t\in[0,1]}$, the flow of $X_{n,t}$, starting from the identity and ending with $\phi_{n,2}$ so that,
$$|\phi_{n,2}^{t}|_{C^{0}(\bar\D(1-\eps_{n})\cup \ell_{0})} \leq \max_{\sigma \in K}\mathrm{diam}(\sigma), \forall t\in[0,1],$$
moreover, $\phi_{n,2}^{t} =id$ on $\Lambda_{n} = \partial\bar\D(1-\eps_{n})$ for all $t\in [0,1]$. Second, $\phi_{n,2}$ is compactly supported in $\D(r_{m})\backslash\bar\D(r_{m-1}) = \sigma_{1} \cup \dots \cup \sigma_{m}$, the cells of $K$ that share an edge with the boundary $\partial\bar\D$, this is because writing $\Delta_{n} = f_{n}(r,\theta)dr \wedge d\theta$ one has $f_{n} = 0$ on $\bar\D(r_{m-1})$; see the formula for $\beta =\beta_{n,2}$.

\subsubsection{Construction of \texorpdfstring{$\phi_{n,3}$}{}}\phantomsection\label{sec:const_of_phi_three}
Let $\alpha_{n} =\bar\eta_{n}\Omega - \Omega$. By construction $\alpha_{n} = 0$ outside of $B$ and it is closed in $\H^{2}_{c}(\D, \R) = \R$. Then, since the integral of $\alpha_{n}$ over $\D$ is zero by Poincaré lemma for compactly supported De Rham cohomology one finds $\beta_{n,3}\in \Omega^{1}_{c}(\D, \R)$ so that $\alpha_{n} = d\beta_{n,3}$ and $\beta_{n,3} = 0$ on a neighborhood $\mathcal{N}$ of the boundary of $\partial \D.$ The latter neighborhood $\mathcal{N}$ can be chosen to be independent of $n$ as one can restrict $\alpha_{n}$ to a smaller disk, independent of $n$, where it is still compactly supported and apply Poincaré lemma to the restrictions. We run Moser's argument for the path of symplectic forms $\Omega_{n,t} = \Omega + t(\bar\eta_{n}\Omega - \Omega)$ and solve the equations $\Omega_{n,t}(X_{n,t}, -) = -\beta_{n,3}$ for $X_{n,t}$. We define $\phi_{n,3}$ to be the time-one map of the flow of $X_{n,t}$. Clearly, $\phi_{n,3}^{*}(\bar\eta_{n}\Omega) = \Omega$. It remains to show that $|\phi_{n,3}|_{C^{0}} \to 0$ as $n \to \infty$. Let $\alpha_{n} = f_{n}(r,\theta)dr \wedge d\theta$ and observe that $|f_{n}|_{C^{0}}\to 0$ as $n \to \infty$ and $f_{n}$ vanishes in a neighborhood $\mathcal{N}$ of the boundary $\partial \D$ that is independent of $n$. We construct $\beta_{n,3}$ using Banyaga's construction explained in \S\ref{sec:const_of_phi_two} where the defining terms of $\beta_{n,3}$ only involve integrals of $f_{n}$ and some auxiliary cut-off functions $a_{i}(r)$ that are independent of $n$. It is then easy to see that $|\beta_{n,3}|_{C^{0}(\D)} \to 0$ as $n \to \infty$; see also \cite[Lemma 3.0.4]{ian_thesis}. To bound the $C^{0}$-norm of $\phi_{n,3}$ it is enough to bound the norm of $X_{n,t}$. To that end,
\begin{equation}\phantomsection\label{eq:bound_norm_Xnt}|X_{n,t}|_{C^{0}(\D)} = |X_{n,t}|_{C^{0}(\D\backslash \mathcal{N})} \leq |T_{\Omega_{n,t}}^{-1}|_{\mathrm{op}(\D \backslash\mathcal{N})}|\beta_{n,3}|_{C^{0}(\D \backslash\mathcal{N})},
\end{equation}
where $T_{\Omega_{n,t}}:T\D \rightarrow T^{*}\D$ is defined by $X\mapsto \Omega_{n,t}(X,-)$ and $|T_{\Omega_{n,t}}^{-1}|_{\mathrm{op}(\D \backslash \mathcal{N})}$ is the operator norm of $T_{\Omega_{n,t}}^{-1}: T^{*}\D \rightarrow T\D$ restricted to $T^{*}\D|_{\D \backslash\mathcal{N}}$. The quantity on the right side of \eqref{eq:bound_norm_Xnt} converges to zero because $$\abs{T^{-1}_{\Omega_{n,t}}}_{\mathrm{op}(\D\backslash \mathcal{N})}\leq C\abs{T^{-1}_{\Omega}}_{\mathrm{op}(\D \backslash \mathcal{N})} < \infty,$$
for some constant $C>0$ and for all large enough $n$, as $|\overline{\eta}_n-1|_{C^0} \to 0,$ and $\abs{\beta_{n,3}}_{\mathrm{op}(\D \backslash \mathcal{N})} \to 0$ as $n \to \infty$. Since the set $\D \backslash \mathcal{N}$ is compact we have finished the proof. 

\section{Proof of Theorem \ref{thm:main_theorem_introduction}}\phantomsection\label{sec:proof_main_theorem} 
Here we present the proof of the main theorem of the paper, see Theorem \ref{thm:main_theorem_introduction}. Given a Hamiltonian diffeomorphism $\phi\in \Ham(\Sigma, \omega)$ where $\Sigma$ is a closed oriented surface of genus $g \geq 1$, Theorem \ref{thm:main_theorem_introduction} studies the behavior of the sequence $\{\gamma(\phi^{k})\}_{k \geq 1}$, in particular, whether it converges to zero for some subsequence. The proof of the theorem is split into four main steps. First we assume that the topological entropy is positive. This case is well studied and we will shortly explain it in \S\ref{sec:pos_top_ent_case}. For the following discussion we assume that the topological entropy is zero. In \S\ref{sec:pseudo_rot_torsion} we briefly comment on the case of pseudo-rotation and torsion Hamiltonian diffeomorphisms, which are known not to exist in $\Ham(\Sigma_{g})$ with genus $g\geq 1$. In step three, we assume that there is a \emph{smooth closed contractible} annulus $A$ that is invariant under $\phi$ with some technical properties which will be discussed in details in \S\ref{sec:smooth_non_ess_inv_ann}. In step four, we will consider one of the Franks-Handel $\phi$-invariant annuli given by Theorem \ref{thm:franks_handel_theorem}, and depending on whether it is \emph{essential} or \emph{inessential}, \emph{regular} or \emph{irregular}, we will divide the proof into few steps; this will be accomplished in \S\ref{sec:non_smooth_inv_ann} and \S\ref{sec:smooth_ess_inv_ann}.

\subsection{Positive topological entropy}\phantomsection\label{sec:pos_top_ent_case}
Let $\{\phi_{t}\}$ be a fixed Hamiltonian isotopy with $\phi_{1} = \phi$. If $h_{top}(\phi)>0$ then $\phi$ has infinitely many hyperbolic periodic points, up to replacing $\phi$ with its iteration we might assume that $\phi$ has many hyperbolic fixed points; see \cite{katok-89}. Suppose $\mathcal{P}$ is a finite set of hyperbolic fixed points of $\phi$ with $\abs{\mathcal{P}}>\dim \H_{*}(\Sigma,\K)$. It follows from \cite[Proposition 3.8]{cgg-barcode-entropy} that for every $k\geq 1$ there is at least one finite bar $J_{k}\in \mathcal{B}(\phi^{k})$ in the barcode of $\phi^{k}$ one of whose ends corresponds to a generator of the complex $\CF(\phi^{k})$ whose leading term is a hyperbolic fixed point $p\in \mathcal{P}\subset \Fix(\phi^{k})$ and the length of the bar $\ell(J_{k})\geq \eps(\mathcal{P})>0$; this means that the bar corresponds to the equation,
$$dx = y, \text{ for some } x,y\in \CF(\phi),$$ 
and $\{\phi_{t}(p)\}$ has the highest action among the orbits appearing in either $x$ or $y$. Indeed, the length of $J$ is bounded from below by the energy of a Floer trajectory $u$ that is asymptotic to the Hamiltonian orbit $\{\phi^{k}_{t}(p)\}$ at one of its ends. It was shown in \cite[Theorem 3.1]{ginzburg-gurel-hyperbolic} that any Floer cylinder asymptotic to the hyperbolic periodic orbit $\{\phi^{k}_{t}(p)\}$ has energy at least $c = c(\{\phi_{t}(p)\}) > 0$ that depends only on the orbit $\{\phi_{t}(p)\}$ and not $k$; see also \cite[Theorem 6.1]{gg-pseudo-rotations}. The idea is that, first, one reduces to the case where $\{\phi^{k}_{t}(p)\}$ is a constant orbit $\{\phi^{k}_{t}(p)\} \equiv p$, see \cite[\S9]{salamon-zehnder-92}. If there is a sequence of Floer cylinders, for possibly different iterations of the Hamiltonian, asymptotic to $p$ with energy converging to zero, then, via a Gromov-type convergence \cite{fish-11-gromov-comp}, one can find an integral curve $\Gamma:[\tau-L, \tau +L]\rightarrow M$, for arbitrary large enough number $L$ and $\tau\in[0,1]$, of the Hamiltonian flow $\{\phi_{t}\}$ that stay inside a fixed small neighborhood $B$ of $p$ and $\Gamma(\tau)\in \partial B$, this is in contradiction with hyperbolicity of $p$. This is because for a small enough fixed neighborhood $B$ of $p$ there is $L_{0}>0$ depending only on $p$ and $B$ so that for every $x\in M$ and $\tau\in [0,1]$ with $\phi_{\tau}(x)\in \partial B$ there exists $t\in [\tau - L_{0}, \tau + L_{0}]$ so that $\phi_{t}(x)\not\in B$. Therefore $\gamma(\phi^{k})\geq \beta(\phi^{k})\geq \ell(J_{k})\geq \eps(\mathcal{P}) >0$ for all $k\geq 1$.

\subsection{Torsions and pseudo-rotations}\phantomsection\label{sec:pseudo_rot_torsion}
Let $\phi \in \Ham(\Sigma)$ be a Hamiltonian diffeomorphism of a closed oriented surface. When $\Sigma$ is a surface of genus $g \geq 1$ it is known that there is no Hamiltonian diffeomorphism $\phi\neq id$ that is torsion, i.e., there is no $n >1$ so that $\phi^{n} = id$; see \cite{Polterovich-02, Atallah-Shelukhin} and references therein. On the case of $S^{2}$, there are torsion Hamiltonian diffeomorphisms, e.g., rational rotations. However, in this case the sequence $\{\gamma(\phi^{k})\}_{k \geq 1}$ is not interesting to study, it is just a finite set of numbers including zero. A Hamiltonian diffeomorphism $\phi$ is called pseudo-rotation if $\#\mathrm{Per}(\phi) < \infty$. It is known that there are no Hamiltonian pseudo-rotations on surfaces of genus $g\geq 1$; this is known as the Conley conjecture that is proved in many cases; see \cite{ginzburg-10-conley-conj} and references therein. Pseudo-rotations exist on the sphere, e.g., irrational rotations. Let $\phi$ be a pseudo-rotation on the sphere $S^{2}$. It has been proved in \cite{gg-pseudo-rotations} that $\liminf_{k \geq 1}\gamma(\phi^{k}) = 0.$ In \S\ref{sec:complex_proj_spc} we provide a complementary result to the latter; see Theorem \ref{thm:full_int_case_sphere}.

\subsection{Inessential smooth invariant annuli}\phantomsection\label{sec:smooth_non_ess_inv_ann}
In this section we prove Theorem \ref{thm:main_theorem_introduction} in the presence of a certain invariant smooth annulus. This section contains the input ideas and methods from quantitative Lagrangian Floer theory. The non-smooth cases studied in the subsequent sections will eventually be reduced to the smooth setting of the current section. The reduction process uses ideas from $C^{0}$-symplectic geometry \cite{MR2884036,sikorav-01} techniques of low dimensional conservative dynamics \cite{franks-handel-zero-ent, lecalvez-tal-22}.

Let us fix the setting of this section. Here $\Sigma$ will be a closed oriented symplectic surface of genus $g \geq 1$ and $\phi \in \Ham(\Sigma)$ will be a Hamiltonian diffeomorphism. Let $A \subset \Sigma$ be a smooth closed contractible annulus with boundaries $L, L'$ that is $\phi$-invariant, i.e., $\phi(A) = A$. Denote by $D$ the disk bounded by $L$ and $D'$ the disk bounded by $L'$. Let $\hbar'>\hbar>0$ be the areas of the $D'$ and $D$ respectively.

\begin{theorem}\phantomsection\label{mainthm:smth_cont_ann_case}
Suppose $\phi$ possesses contractible fixed points on $L$ and $L'$, and it has no fixed points in the interior of $A$. Then,
$$\liminf_{k\geq 1}\gamma(\phi^{k})\geq \hbar/2 > 0.$$
\end{theorem}
\begin{proof}
We will argue by contradiction and assume that for some $\delta >0$,
$$\liminf_{k\geq 1}\gamma(\phi^{k}) \leq \hbar/2 - 3\delta.$$
Fix some $k$ so that $\gamma(\phi^{k})\leq \hbar/2 - 2\delta$ and let $\psi := \phi^{k}$. In \S\ref{sec:non_van_lfh_double_mod_act} we prove that for the interval $I = [-\hbar/2+\delta, \delta]$ and $c := c([M], H^{(k)})$ the spectral invariant of the unit $[M]$ with respect to $H$, see \S\ref{sec:spectral_invariants}, one has,
\begin{equation}\phantomsection\label{eq:non_van_floer_hom_l_lprime}
\HF^{I+c}_{1}(L,\psi) \neq 0 \text{ and } \HF^{I+c}_{1}(L', \psi)\neq 0,
\end{equation}
where the latter are the Lagrangian Floer homologies of the pairs $(L,\psi)$ and $(L',\psi)$ in the action window $I + c$ and degree $1$ and $H$ is mean normalized and generates an isotopy $\{\phi_{t}\}$ with $\phi_{1} = \phi$; see \S\ref{sec:lag_floer_theory}. Given that $L$ and $L'$ are invariant under $\phi$, all points on $L$ and $L'$ define a Hamiltonian chord and one can show that all chords with index one have the same actions; indeed there is a unique index one chord on $L$ and $L'$ after non-degenerate perturbations, see the computations in proof of Theorem \ref{thm:lag_act_of_ind_one_chords_eq_ham_act}. Let $a(\phi,1), a'(\phi,1)$ be the action of index $1$ chords of $\{\phi_{t}\}$ from $L$ to itself and $L'$ to itself respectively, and define 
$$\mathrm{D}(\phi):= a'(\phi,1) - a(\phi,1).$$
In \S\ref{sec:diff_act_index_one_chords} we prove the following key identity whose proof uses the assumption on the existence of contractible fixed points on $L$ and $L'$,
\begin{equation}\phantomsection\label{eq:diff_acts_mean_rot_iden}
\mathrm{D}(\phi) = \int_{A}\rho_{\t{\phi}}(x)\omega(x) + (\hbar' - \hbar)( \frac{\mu_{\{\phi_{t}\}}(L)}{2} - \rho_{\t{\phi}}(L)),
\end{equation}
for every lift $\t{\phi}$ of $\phi$ to the universal cover $\t{A}$ of $A$ where $\rho_{\t{\phi}}$ is the rotation number \S\ref{sec:rot_numb_ann} and $\mu_{\{\phi_{t}\}}(L) := \mu_{RS}(\{d\phi_{t}(T_{q}L)\})$ with $\mu_{RS}$ being the Robbin-Salamon index; see Theorem \ref{mainthm:diff_act_chords_vs_ave_rot_num}. It is shown in Corollary \ref{cor:scaling_prop_diff_act_ind_ones} that,
\begin{equation}\phantomsection\label{eq:lin_growth_diff_acts}
\mathrm{D}(\phi^{k}) = k\mathrm{D}(\phi),
\end{equation}
and since $\phi$ has no fixed point in the interior of $A$, and it is isotopic to identity by Remark \ref{rem:rest_are_isotop_to_id}, the result \cite[Proposition 2.6]{franks-handel-zero-ent}, which follows from a celebrated theorem of Franks \cite{franks1988generalizations} (see also \cite{franks-recurrence-88}), together with \eqref{eq:diff_acts_mean_rot_iden} implies that 
\begin{equation}\phantomsection\label{eq:diff_act_non_zero}
\mathrm{D}(\phi)\neq 0,
\end{equation}
see Corollary \ref{cor:scaling_prop_diff_act_ind_ones}; here $\mathrm{D}(\phi^{k})$ is the difference of the action of index $1$ chords of $\{\phi_{t}^{k}\}$ on $L$ and $L'$. By assumption there are infinitely many iterations $k$ for which $\gamma(\phi^{k})\leq \hbar/2 - 2\delta$ and the corresponding filtered Floer homology invariants in \eqref{eq:non_van_floer_hom_l_lprime} are non-zero in degree $1$ and in an interval of length equal to $\hbar/2$. The latter implies that $D(\phi^{k}) \leq \hbar/2$ for infinitely many $k$ while it follows from \eqref{eq:lin_growth_diff_acts} and \eqref{eq:diff_act_non_zero} that $D(\phi^{k})$ grows linearly as $k \to \infty$ and this is a contradiction.
\end{proof}

\begin{remark}\phantomsection\label{rem:theorem_B_with_mean_rot_ass}
The assumption on fixed point freeness of $\phi|_{A}$ in the statement of Theorem \ref{mainthm:smth_cont_ann_case} was only used to show that the mean rotation vector is never zero with respect to any lift $\t{\phi}$ of $\phi|_{A}$ to the universal cover $\t{A}$ of $A$. Indeed, Theorem \ref{mainthm:smth_cont_ann_case} holds replacing the fixed point freeness of $\phi|_{A}$ with the following assumption:
$$\int_{A}\rho_{\t{\phi}}(x)\omega(x)\neq 0,$$
for all lifts $\t{\phi}$ of $\phi|_{A}$ to $\t{A}$.
\end{remark}

\subsubsection{Non-vanishing of the Lagrangian Floer homologies}\phantomsection\label{sec:non_van_lfh_double_mod_act}
This section explains that the filtered Lagrangian Floer homology of a pair $(L,\psi)$ where $\gamma(\psi) < \hbar/2$ does not vanish in some small action windows; $\hbar =\hbar(L,J)$ where $J$ is some generic almost complex structure and $\hbar(L,J)$ is defined in \S\ref{sec:minimal_enr_inv}. In the following theorem $H$ is a mean normalized Hamiltonian generating $\psi\in \Ham(\Sigma)$ and $c := c([\Sigma], H)$.
\begin{theorem}\phantomsection\label{mainthm:non_van_floer_hom}
Let $L\subset \Sigma$ be an embedded contractible closed curve bounding a disk $D$ of area $\hbar > 0$. Assume that $\gamma(\psi) \leq \hbar/2 - 2\delta$ for some small $0 < \delta \ll \hbar$. Then, for $I = [-\hbar/2+\delta, \delta]$ we have
$$\HF^{I + c}_{1}(L, \psi) \neq 0.$$
\end{theorem}
\begin{proof}
This follows from the discussion in \S\ref{sec:module_action}, Proposition \ref{prop:pss_modul_incl_comm_diag} and the fact that the double module action $\m_{2}$ introduced in \S\ref{sec:twin_tower_map} does not vanish when restricted to the action window $I$. We choose a representative $\alpha\in \CF_{1}(H)$ of $\pss_{H}([\Sigma])$ and a representative $\beta\in \CF_{1}(\bar{H})$ of $\pss_{\bar{H}}([\Sigma])$ so that $c(\alpha):=c(\alpha, H) \leq c + \eps$ and $c(\beta):=c(\beta, H) \leq \bar{c} + \eps$ for some arbitrary small number $\eps> 0$, where $\bar{c} :=c([M], \bar{H})$; here we omit the choice of the almost complex structures in the Floer complex as they have no effect on the argument. Then, it follows from the discussion in \S\ref{sec:twin_tower_map} and \eqref{eq:module_act_vs_pp} that for the double module action $\m_{2}(\xi):=\m(\m(\xi\otimes \alpha)\otimes \beta)$ we have
$$\m_{2}(\xi) = \m(\xi\otimes \p\p(\alpha \otimes \beta))\in \HF^{I+c+\bar{c}}(L,0), \forall \xi\in \CF^{I}(L,0).$$
Notice that $\mathcal{A}_{0}(\p\p(\alpha\otimes \beta)) \leq c+\bar{c} + 2\eps$; see \S\ref{sec:pair_of_pants_unit_times_unit}. Let $\zeta\in \mathrm{CM}(f)$ be a representative of the fundamental class $[M]$ in the Morse complex of some $C^{2}$-small Morse function $f$ on $\Sigma$. It is well-known that $\beta(\psi)\leq \gamma(\psi)$ where $\beta(\psi)$ is the boundary depth of $\psi$; see \cite{kislev-shelukhin-21}. So, given that $\p\p(\alpha \otimes\beta)$ represents $\pss_{0}([M])$, see \S\ref{sec:pair_of_pants_unit_times_unit}, we have
$$\p\p(\alpha\otimes \beta) - \pss_{0}(\zeta) = d\eta,$$
for some $\eta\in \CF_{2}(0)$ with action,
$$\begin{aligned}
\mathcal{A}_{0}(\eta)&\leq \max\{\mathcal{A}_{0}(\p\p(\alpha \otimes \beta)), \mathcal{A}_{0}(\pss_{0}(\zeta))\} + \beta(\psi)\\
&\leq \max\{\gamma(\psi) + 2\eps, \eps\} + \gamma(\psi)\leq \hbar-4\delta+2\eps.
\end{aligned}$$
The above action estimate for $\eta$ together with Gromov-Floer compactness and gluing analysis proves that,
$$\m_{2}(-) = \m(-\otimes \p\p(\alpha\otimes \beta)) = \m(- \otimes \pss_{0}(\zeta)),$$
as maps $\HF^{I}_{*}(L,0) \rightarrow \HF_{*}^{I + \gamma(\psi)}(L,0)$. Since $|I|< \hbar$ and $0\in I \cap (I + \gamma(\psi))$, conditions of Proposition \ref{prop:pss_modul_incl_comm_diag} are met, so the following diagram commutes,
\[
\begin{tikzcd}[column sep=3em]
\HF_{1-*}^{I}(L,0) \arrow[r, "\m_{2}"] & \HF_{1-*}^{I+\gamma(\psi)}(L,0) \\
\H_{*}(L) \arrow[r, "id"] \arrow[u, "\pss_{L,0,I}"] & \H_{*}(L). \arrow[u, "\pss_{L,0,I+\gamma(\psi)}", swap]
\end{tikzcd}
\]
Note that the $\pss_{L,0,I}(\mathrm{min})$ represents a non-zero class in $\HF_{1}^{I}(L,0)$ where $\mathrm{min}$ is the class of the unique minimum of a $C^{2}$-small Morse function on $L$, and so does $\pss_{L,0,I+\gamma(\psi)}(\mathrm{min})$ in the Floer homology $\HF_{1}^{I+\gamma(\psi)}(L,0)$. Hence, $\m_{2}\neq 0$ consequently $\m(-\otimes \alpha)\neq 0$ and $\HF_{1}^{I + c}(L,\psi)\neq 0$.
\end{proof}

\subsubsection{Difference of actions of index one chords}\phantomsection\label{sec:diff_act_index_one_chords}
This section proves a key identity that relates the actions of Hamiltonian chords of $\{\phi_{t}\}$, generated by a mean normalized Hamiltonian $H$, on the smooth Lagrangians $L$ and $L'$ that form the boundary of $A$ to the mean rotation number over $A$ with respect to lifts $\t{\phi}$ of $\phi$ to the universal cover $\t{A}$ of $A$. Suppose $q\in L$ and $q'\in L'$ are contractible fixed points of $\phi = \phi_{1}$ and define
$$\mu_{\{\phi_{t}\}}(L) := \mu_{RS}(\{d\phi_{t}(T_{q}L)\}) \text{ and } \mu_{\{\phi_{t}\}}(L') = \mu_{RS}(\{d\phi_{t}(T_{q'}L')\}),$$
which are even numbers if $\phi$ preserves $q,q',L,L'$, and are independent of the choices of $\{\phi_{t}\}$, where we choose an arbitrary symplectic trivialization of the unique disks that bound $\{\phi_{t}(q)\}$ and $\{\phi_{t}(q')\}$ to define the Robbin-Salamon indices \S\ref{sec:indices}. Let $\rho_{\t{\phi}}(L),\rho_{\t{\phi}}(L')$ be the rotation numbers of points on $L,L'$ with respect to $\t{\phi}$ and define
$$\mathrm{D}(\phi):= a'(\phi,1) - a(\phi,1),$$ 
to be the difference of actions of index $1$ chords on $L$ and $L'$; there are unique capped chords with index one on $L$ and $L'$ after non-degenerate perturbations, see \S\ref{sec:comp_act_index_one_chords}. Before stating Theorem \ref{mainthm:diff_act_chords_vs_ave_rot_num}, the main theorem of this section, we prove a lemma that plays a key role in the proof of Theorem \ref{mainthm:diff_act_chords_vs_ave_rot_num}.
\begin{lemma}\phantomsection\label{lem:diff_index_diff_rot}
$\mu_{\{\phi_{t}\}}(L') - \mu_{\{\phi_{t}\}}(L) = 2(\rho_{\t{\phi}}(L') - \rho_{\t{\phi}}(L))$.
\end{lemma}
\begin{proof}
Let $\gamma: [0,1] \rightarrow A$ be a smooth curve with $\gamma(0)=q$ and $\gamma(1) = q'$ so that it is tangent to $L$ and $L'$ at the end points $q$ and $q'$ and $\mu_{RS}(\Gamma(T_{\gamma(s)}\gamma)) = 0$ for some symplectic trivialization $\Gamma: \gamma^{*}T\Sigma \rightarrow [0,1]\times \R^{2}.$ Define $v: [0,1]^{2} \rightarrow \Sigma$ by $v(s,t) = \phi_{t}(\gamma(s))$. Extend the trivialization $\Gamma$ to a symplectic trivialization, still denoted by $\Gamma: v^{*}T\Sigma \rightarrow [0,1]^{2}\times \R^{2}$. Then, $\Gamma\big(d_{\gamma(s)}\phi_{t}(T_{\gamma(s)}\gamma)\big)$ defines a contraction of the loop of Lagrangians defined on the boundary of $[0,1]^{2}$, so the sum of the following indices must vanish,
\begin{enumerate}
\item $\mu_{RS}\big(\Gamma\big(T_{\gamma(s)}\gamma\big)\big) + \mu_{RS}\big(\Gamma\big(\{d\phi_{t}(T_{q'}L')\}\big)\big),$
\item $-\mu_{RS}\big(\Gamma\big(d_{\gamma(s)}\phi_{1}(T_{\gamma(s)}\gamma)\big)\big) -\mu_{RS}\big(\Gamma\big(\{d\phi_{t}(T_{q}L)\}\big)\big),$
\end{enumerate}
implying that the difference $\mu_{\{\phi_{t}\}}(L') - \mu_{\{\phi_{t}\}}(L)$ is equal to,
$$\mu_{\{\phi_{t}\}}(L') - \mu_{\{\phi_{t}\}}(L)= \mu_{RS}\big(\Gamma\big(d_{\gamma(s)}\phi_{1}(T_{\gamma(s)}\gamma)\big)\big).$$
It is clear that $d_{\gamma(s)}\phi_{1}(T_{\gamma(s)}\gamma) = T_{\phi_{1}(\gamma(s))}(\phi_{1}\circ \gamma)$ and the following holds,
$$\mu_{RS}\big(\Gamma\big(T_{\phi_{1}(\gamma(s))}(\phi_{1}\circ \gamma)\big)\big) = 2(\rho_{\t{\phi}}(L') - \rho_{\t{\phi}}(L));$$
this is basically because the number of times $\phi_{1}\circ \gamma$ winds around the disk $D$ is equal to the difference of the rotation numbers of ends $q$ and $q'$; note that $\phi_{1}\circ \gamma$ is an embedded curve.
\end{proof}

\begin{theorem}\phantomsection\label{mainthm:diff_act_chords_vs_ave_rot_num}
Suppose $A$ is $\phi$-invariant, $\phi$ possesses contractible fixed points $q\in L$ and $q'\in L'$. Then, for every lift $\t{\phi}$ of $\phi$ to $\t{A}$ one has,
\begin{equation}\phantomsection\label{eq:act_diff_for_all_lifts}
\mathrm{D}(\phi) = \int_{A}\rho_{\t{\phi}}(x)\omega(x) + (\hbar' - \hbar)( \frac{\mu_{\{\phi_{t}\}}(L)}{2} - \rho_{\t{\phi}}(L)),
\end{equation} 
\end{theorem}
\begin{proof}[Proof of Theorem \ref{mainthm:diff_act_chords_vs_ave_rot_num}]
For brevity, in the course of the proof we denote $\rho_{\{\phi_{t}\}}(L)$ and $\rho_{\{\phi_{t}\}}(L')$ by $\rho$ and $\rho'$ respectively and we denote $\mu_{\{\phi_{t}\}}(L)$ and $\mu_{\{\phi_{t}\}}(L)$ by $\mu$ and $\mu'$ respectively. By Theorem \ref{thm:lag_act_of_ind_one_chords_eq_ham_act} we know that,
$$\mathrm{D} = \mu'\hbar'/2 - \mu\hbar/2 + \mathcal{A}_{H}(q') - \mathcal{A}_{H}(q).$$
Let $\lambda$ be a primitive of $\omega$ on $D'$. Let $a_{\phi,\lambda}$ be a smooth function on $D'$ so that $\phi^{*}\lambda - \lambda = da_{\phi,\lambda};$
we do not necessarily normalize the function $a_{\phi, \lambda}$. It is not hard to show that,
$$\mathcal{A}_{H}(q') - \mathcal{A}_{H}(q) = -a_{\phi,\lambda}(q') + a_{\phi,\lambda}(q);$$
we do not claim that \eqref{eq:two_acts_compared} holds in the current situation, however the differences of the actions in the above identity are still equal and the proof is similar to the proof of \eqref{eq:two_acts_compared}. Note that both sides of the above equation do not depend on the choice of the primitive $\lambda$. Let $\lambda$ be the primitive induced by the one-form $\t{\lambda} = xdy$ on the universal cover $\t{A} = [\hbar,\hbar']_{x}\times \R_{y}$ with fundamental domain $\mathrm{FD} = [\hbar, \hbar']\times [0,1]$. Then, for the smooth curve $\gamma\subset A$ from $q$ to $q'$ that is the image of the curve $\t{\gamma} = [\hbar, \hbar'] \times \{0\}$ under the projection $\t{A}\rightarrow A$ we have
$$\begin{aligned}
-a_{\phi,\lambda}(q') + a_{\phi,\lambda}(q) &= -\int_{\gamma}da_{\phi,\lambda} = \int_{\gamma}\lambda - \phi^{*}\lambda\\
&= \int_{\t{\gamma}}\t{\lambda} - \t{\phi}^{*}\t{\lambda}= \int_{\t{\gamma}}xdy - \int_{\t{\phi}(\t{\gamma})}xdy \\
&=\int_{Q}dx \wedge dy - \int_{\ell'}xdy + \int_{\ell}xdy =: X,
\end{aligned}$$
where here $Q$ is the region enclosed by the curves $\t{\gamma}, \t{\phi}(\t{\gamma})$ and the vertical segments $\ell$ and $\ell'$ where $\ell$ is the vertical segment on $\{\hbar\} \times \R$ connecting $(\hbar,0)$ and $(\hbar, \rho)$ and $\ell'$ is the vertical segment on $\{\hbar'\} \times \R$ connecting $(\hbar',0)$ to $(\hbar', \rho')$. Here we have used the fact that $\t{\phi}(\hbar,0) = (\hbar, \rho)$ and $\t{\phi}(\hbar',0) = (\hbar', \rho')$, these follow from the fact that $q$ and $q'$ are fixed points of $\phi$. Then,
$$\begin{aligned}
X &= \int_{Q}dx \wedge dy -  \rho'\hbar' + \rho\hbar = \int_{\partial Q}-ydx -  \rho'\hbar' + \rho\hbar\\
&= \int_{\t{\phi}(\t{\gamma})} ydx -  \rho'\hbar' + \rho\hbar.
\end{aligned}$$
Therefore, so far we have shown the following identity,
$$\mathrm{D} = \mu'\hbar'/2 - \mu\hbar/2 -\rho'\hbar' + \rho\hbar + \int_{\t{\phi}(\t{\gamma})} ydx.$$
On the other hand, applying the Birkhoff's ergodic theorem \S\ref{sec:birkhoff_erg_thm} to the measure preserving transformation $T := \phi$ on $A$ and the function $f(p):= y(\t{\phi}(\t{p}))- y(\t{p})$ where $\t{p}$ is a choice of a lift for $p$ to $\t{A}$, we have
$$\begin{aligned}
\int_{A}\rho_{\t{\phi}}(p)\omega(p) &= \int_{A}(y(\t{\phi}(\t{p})) - y(\t{p}))\omega(p) = \int_{\mathrm{FD}}(y(\t{\phi}(\t{p})) - y(\t{p}))\t{\omega}(\t{p})\\
&= \int_{\t{\phi}(\mathrm{FD})}yd\t{\lambda} - \int_{\mathrm{FD}}yd\t{\lambda}=\int_{\t{\phi}(\mathrm{FD})}ydx\wedge dy - \int_{\mathrm{FD}}ydx\wedge dy\\
&= \int_{\partial\t{\phi}(\mathrm{FD})}-y^{2}dx/2 - \int_{\partial(\mathrm{FD})}- y^{2}dx/2=: Y,\\
\end{aligned}$$
the second term is non-zero only on the curve $[\hbar, \hbar']\times \{1\}$, with reverse orientation, where it is equal to $(\hbar'-\hbar)/2$, and the first term is non-zero only on the curves $\t{\phi}([\hbar,\hbar']\times\{0\})$ and $\t{\phi}([\hbar,\hbar']\times\{1\})$ where the latter is actually the same as $\t{\phi}([\hbar,\hbar']\times\{0\}) + (0,1)$, so we have
$$\begin{aligned}
Y &= \int_{\t{\phi}(\t{\gamma})}-y^{2}dx/2 - \int_{\t{\phi}(\t{\gamma}) + (0,1)}-y^{2}dx/2 - (\hbar' - \hbar)/2\\
&= \int_{\t{\phi}(\t{\gamma})}((y+1)^{2}-y^{2})dx/2 - (\hbar' - \hbar)/2= \int_{\t{\phi}(\t{\gamma})}ydx\\
\end{aligned}$$
Putting all the identities above together one obtains, 
$$\begin{aligned}
\mathrm{D} &= \int_{A}\rho_{\t{\phi}}(x)\omega(x) + \mu'\hbar'/2 - \mu\hbar/2 - \rho'\hbar' + \rho \hbar,
\end{aligned}$$
which together with Lemma \ref{lem:diff_index_diff_rot} we get the following formula,
\begin{equation}
\mathrm{D} = \int_{A}\rho_{\t{\phi}}(x)\omega(x) + (\hbar' - \hbar)(\mu/2 - \rho),
\end{equation}
which finishes the proof of the theorem.
\end{proof}

\begin{remark}
In this remark we present a different proof of a particular case of Theorem \ref{mainthm:diff_act_chords_vs_ave_rot_num} which maybe of independent interest. Let $\phi_{t}\in \Ham(\Sigma)$ be a Hamiltonian isotopy generated by $H$ with time-one map $\phi= \phi_{1}$. Suppose that the smooth closed annulus $A$ is $\phi_{t}$-invariant for all $t\in [0,1]$, and $\{\phi_{t}\}$ possesses contractible fixed points $q\in L$ and $q'\in L'$. Then, the isotopy $\{\phi_{t}\}$ naturally defines a lift of $\phi$ to the universal cover $\t{A}$ of $A$ and we denote its corresponding rotation function by $\rho_{\{\phi_{t}\}}$. Then,
$$\mathrm{D}(\phi) = \mathcal{F}([\{\phi_{t}|_{A}\}]) = \int_{A}\rho_{\{\phi_{t}\}}(x)\omega(x);$$
see \S\ref{sec:flux_average_rot_vec}. To prove this, it follows from Theorem \ref{thm:lag_act_of_ind_one_chords_eq_ham_act} and an easy computation of the Robbin-Salamon indices $\mu_{\{\phi_{t}\}}(L)$ and $\mu_{\{\phi_{t}\}}(L')$ that,
$$\mathrm{D}(\phi) = \rho_{\{\phi_{t}\}}(L')\hbar' - \rho_{\{\phi_{t}\}}(L)\hbar + \mathcal{A}_{H}(q') - \mathcal{A}_{H}(q).$$
Let us normalize $H$ so that it vanishes on $L'$; note that $H_{t}$ is constant on $L'$ for every $t\in [0,1]$. We use the following formula from \S\ref{sec:awr_identities},
$$a^{\infty}_{\{\phi_{t}\}}(q) = \int_{D'} W^{\infty}_{\{\phi_{t}\}}(q,y)\omega(y) = \rho_{\{\phi_{t}\}}(L)\hbar + \int_{A}\rho_{\{\phi_{t}\}}(y)\omega(y),$$
and similarly one has,
$$a^{\infty}_{\{\phi_{t}\}}(q') = \int_{D'} W^{\infty}_{\{\phi_{t}\}}(q,y)\omega(y) = \rho_{\{\phi_{t}\}}(L')\hbar'.$$
Therefore, it follows that,
$$\begin{aligned}
\mathrm{D}(\phi) &= \rho_{\{\phi_{t}\}}(L')\hbar' - \rho_{\{\phi_{t}\}}(L)\hbar + \mathcal{A}_{H}(q') - \mathcal{A}_{H}(q)\\
&= \rho_{\{\phi_{t}\}}(L')\hbar' - \rho_{\{\phi_{t}\}}(L)\hbar -a^{\infty}_{\{\phi_{t}\}}(q') + a^{\infty}_{\{\phi_{t}\}}(q)= \int_{A}\rho_{\{\phi_{t}\}}(y)\omega(y),
\end{aligned}$$
where the second equality is by \eqref{eq:two_acts_compared}; note that \eqref{eq:two_acts_compared} holds up to a constant in general but this does not affect the differences of two actions.
\end{remark}

\begin{corollary}\phantomsection\label{cor:scaling_prop_diff_act_ind_ones}
Suppose that $\phi$ has no fixed points in the interior of $A$. Then $\mathrm{D}\neq 0$. Moreover, one has $\mathrm{D}(\phi^{k}) = k\mathrm{D}(\phi)$ for all $k \geq 1$.
\end{corollary}
\begin{proof}
Because by assumption $\phi$ possesses fixed points on $L$ we must have $\rho\in \Z$. If $\mathrm{D} = 0$ then we must have,
$$\int_{A}\rho_{\t{\phi}}(x)\omega(x)\in (\hbar' - \hbar)\Z,$$
for every lift $\t{\phi}$ of $\phi$. This implies that one can pick a lift $\t{\phi}$ so that the mean rotation number over $A$ vanishes. Then,  Lemma \ref{lem:zero_ave_rot_num_fixed_point} implies the existence of a fixed point of $\phi$ in the interior of $A$ which is a contradiction; $\phi|_{A}$ is isotopic to the identity by Remark \ref{rem:rest_are_isotop_to_id}. To prove $\mathrm{D}(\phi^{k}) =k\mathrm{D}(\phi)$, observe that the identity \eqref{eq:act_diff_for_all_lifts} holds for every lift $\t{\phi}$ of $\phi$, hence, we pick any lift $\t{\phi}$ of $\phi$ and for $\phi^{k}$ we choose $(\t{\phi})^{k}$ and we observe that, for almost every $x$ in $A,$ 
$$\rho_{(\t{\phi})^{k}}(x) = k\rho_{\t{\phi}}(x).$$
Moreover, the index $\mu = \mu_{\{\phi_t\}}(L) = \mu_{RS}(\{d\phi_{t}(T_{q}L)\})$ does not depend on the choice of the isotopy $\{\phi_{t}\}$, see \S\ref{sec:indices}, so to compute the corresponding index for $\phi^{k}$ we can use the isotopy $\phi_{kt}$,
$$\begin{aligned}
\mu_{\{\phi_{kt}\}}(L)= \mu_{RS}(\{d\phi_{kt}(T_{q}L)\}) &= \sum_{j=0}^{k-1}\mu_{RS}(\{d\phi_{t}(d\phi_{1}^{j}(T_{q}L))\})\\
&=\sum_{j=0}^{k-1}\mu_{RS}(\{d\phi_{t}(T_{q}L)\}) = k\mu,
\end{aligned}$$
where the last equality follows from $\phi(L) = L$ and $\phi(q) = q$.
\end{proof}

\subsubsection{Computation of actions of index one chords}\phantomsection\label{sec:comp_act_index_one_chords}
Let $\phi\in \Ham(\Sigma)$ and $L, D,\hbar$ be as before. Suppose that $\phi(L) = L$. Moreover, assume that there is at least one contractible fixed point $q\in L$ for $\phi$. Let $H$ be a mean normalized Hamiltonian generating $\{\phi_{t}\}$ with $\phi_{1} = \phi$. Every point $x\in L$ defines a Hamiltonian chord from $L$ to itself. Consider the Floer homology $\HF(L, H)$ which is defined by
$$\HF^{[a,b]}(L,H; J) = \lim_{n\to \infty}\HF^{[a,b]}(L,H_{n}; J)$$
where the limit is over a sequence of Hamiltonians $H_{n}$ converging to $H$ in $C^{\infty}$ topology and all their Hamiltonian chords from $L$ to itself are non-degenerate, i.e. $\phi_{H_{n},1}(L) \pitchfork L$. In the particular case of the above, one can take a perfect Morse function $g: L \rightarrow \R$ with critical points on $p,q$, for some $p\in L$ distinct from $q$, with $\mathrm{ind}(p) = 0, \mathrm{ind}(q) =1$, and extend it to a small tubular neighborhood of $L$ so that it is independent of the normal co-ordinate to $L$ and finally extend to $\Sigma$ by a cutting off procedure; by abuse of notation we denote the extension by $g$.
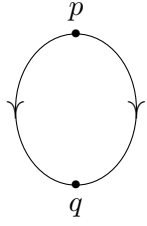
\begin{figure}[ht]
  \centering
  \begin{tikzpicture}
   \coordinate (A) at (0,-1);
   \draw (0,0) circle (0.8 and 1) coordinate(X);
   \node at (0.8, 0) {$\curlyvee$}; \node at (-0.8, 0) {$\curlyvee$};
   \node at (0, 1) {\tiny{$\bullet$}};\node at (0, -1) {\tiny{$\bullet$}};
   \node at (0, 1.3) {$p$};\node at (0, -1.3) {$q$};
  \end{tikzpicture}
  \caption{\emph{Positive} gradient flow of $g$ on $L$.}
  \phantomsection\label{fig:grad_g}
\end{figure}
Take $H' = H \# \delta g$ as a perturbation of $H$ with flow $\phi_{\delta, t} = \phi_{t}\psi_{\delta,t}$ with transverse intersections $\phi_{1}\psi_{\delta,1}(L) \pitchfork L$ where $\psi_{\delta,t}$ is the flow of $\delta g$. One can pick a sequence $\delta_{n}\to 0$ and work with $H_{n} = H\#\delta_{n}g$ to define $\HF(L, H)$. Below, we denote $\psi_{\delta,1}$ by $\psi_{\delta}$. In the following theorem by $\u{L}$ we mean the anchored Lagrangian $\u{L} = (L,\alpha,\Lambda)$ with constant anchor $\alpha \equiv q$ and $\Lambda \equiv T_{q}L$, see \S\ref{sec:anchored_lag_floer_hom}, and $\mu(\{d\phi_{t}(T_{q}L)\})$ is computed using a constant trivialization $\R^{2} \rightarrow T_{q}L$.

\begin{theorem}\phantomsection\label{thm:action_of_index_one_chords}
Suppose $\phi_{t}(q) = q$ for all $t\in [0,1]$. For every non-degenerate perturbation $H\# \delta g$ as above, if $u: [0,1] \times [0,1] \rightarrow \Sigma$ is a Lagrangian capping of a Hamiltonian chord $x$ with $\mathrm{CZ}$-index equal to $1$ then,
$$\mathcal{A}_{\u{L},H\#\delta g}(x,u) = \frac{\mu_{\{\phi_{t}\}}(L)\hbar}{2} + \int_{0}^{1}H_{t}(q)dt + e(\delta),$$
where $e(\delta) \to 0$ as $\delta \to 0$ and $\mu_{\{\phi_{t}\}}(L) := \mu_{RS}(\{d\phi_{t}(T_{q}L)\})$.
\end{theorem}
\begin{proof}
For brevity we denote $\mu_{\{\phi_{t}\}}(L)$ by $\mu$ in the course of the proof. Let $p = x(0)$ if $x(0)\neq q$ and choose $p\in L\backslash\{q\}$ arbitrarily if $x(0) = q$. By assumption we have $\phi_{t}(q) = q$ for all $t\in [0,1]$, and $p,q$ are critical points of $g$, so we also have $\psi_{\delta,t}(p) = p$ and $\psi_{\delta,t}(q) = q$ for all $t\in [0,1]$. Defining $\t{L}:= \phi_{\delta,1}^{-1}(L)$ one has
$$\t{L}\cap L= \psi_{\delta}^{-1}\phi^{-1}(L) \cap L = \psi_{\delta}^{-1}(L)\cap L =  \{p,q\};$$ 
see Figure \ref{fig:two_circles_intersect}. Given that $\phi(L) = L$, and so $d\phi(T_{z}L) = T_{\phi(z)}L$ for all $z\in L$, and the fact that the linearization of $\psi_{\delta,t}$ is given by 
$$d\psi_{\delta,t}(p) = e^{\delta it(\nabla^{2}g)(p)},$$
similarly for $q$, one can check that $\phi\psi_{\delta}(L)$ intersects $L$ transversely at $p,q$ for all $\delta > 0$; so we can use the perturbations $H \# \delta g$ to define $\HF(\u{L}, H)$. 
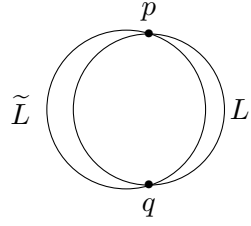
\begin{figure}[ht]
  \centering
  \begin{tikzpicture}
   \coordinate (A) at (0,-1);
   \draw (0,0) circle (1 and 1) coordinate(X);
   \draw (-0.3,0) circle (1.05 and 1.05);
   \node at (0, 1) {\tiny{$\bullet$}};\node at (0, -1) {\tiny{$\bullet$}};
   \node at (0, 1.3) {$p$};\node at (0, -1.3) {$q$};
   \node at (-1.7, 0) {$\t{L}$};\node at (1.2, 0) {$L$};
  \end{tikzpicture}
  \caption{Graph of $L$ and $\t{L} = \phi_{\delta,1}^{-1}(L) = \psi_{\delta}^{-1}(L)$ intersecting transversely. The strips on the sides have equal areas.}
  \phantomsection\label{fig:two_circles_intersect}
\end{figure}
By the discussion in \S\ref{sec:transform_map} it is enough to prove that any capped intersection point $(x,v)$ in the complex $\CF(\u{L}, \u{\t{L}})$ with Conley-Zehnder index $1$ has action near $r\hbar$ and its action converges to $r\hbar$ as $\delta \to 0$, see \eqref{eq:action_shift_int_chord} and the discussion afterwards, where $\u{\t{L}} = (\overline{H_{\delta}})_{*}\u{L}$ and  $H_{\delta} = H \# \delta g$. From now on we transfer the situation to the Euclidean plane $\mathbb{C}$ by picking a Darboux chart so that $q = -i$ and $p = i$ and $L = S^{1}(r)$; let us assume $\hbar = \pi$ for simplicity so that we can assume $r = 1$. The Robbin-Salamon indices are calculated with respect to the standard trivialization $T\R^{2} = \R^{2} \times \R^{2}.$ The anchor $\alpha_{\u{L}, \u{\t{L}}}$ is equal to $\phi_{\delta,t}^{-1}(q) = \psi_{\delta,t}^{-1}\phi_{t}^{-1}(q) \equiv q$ and $\Lambda_{\u{L}, \u{\t{L}}}(t) = d\phi_{\delta,t}^{-1}(T_{q}L)$. We now try to find the cappings of the intersection points $p,q$ with index $1$. Here is a particular capping $v: [0,1]\times [0,1] \rightarrow \Sigma$ for the intersection point $p$:
\begin{enumerate}
\item $v_{p}(s,0) = e^{-\pi is}q$,
\item $v_{p}(1,t) \equiv p,$
\item $v_{p}(s,1) = \phi_{\delta,1}^{-1}e^{-\pi is}q,$
\item $v_{p}(0, t) = \alpha_{\u{L}, \u{\t{L}}}(t) \equiv q$;
\end{enumerate}
this is the strip on the left side of Figure \ref{fig:two_circles_intersect}. Here are the computations of the Robbin-Salamon indices \S\ref{sec:indices}:
\begin{enumerate}
\item $\mu_{RS}(T_{v_{p}(s,0)}L) = -\frac{1}{2} - \frac{1}{2} = -1,$,
\item $\mu_{RS}(\Lambda_{can}(T_{p}L, T_{p}\t{L})) = \frac{1}{2}$,
\item $\mu_{RS}(T_{v_{p}(1-s,1)}\t{L}) = 1 + 1 = 2$,
\item $\mu_{RS}(\overline{\Lambda}_{\u{L}, \u{\t{L}}}) = -\mu_{RS}(\{d\psi_{\delta,t}^{-1}d\phi_{1}^{-1}(T_{q}L)\}) -\mu_{RS}(\{d\phi_{t}^{-1}(T_{q}L)\}) = -\frac{1}{2} + 2r,$
\end{enumerate}
where $\mu = 2r$; note that $d\phi_{1}^{-1}(T_{q}L) = T_{q}L$. By the discussion in \S\ref{sec:transform_map} the index of the corresponding capped chord $(\phi_{\delta,t}(p), u_{p})$ is,
$$\mathrm{CZ}(\phi_{\delta,t}(p),u_{p}) = 1 + \mu_{RS}(\{d\phi_{t}(T_{q}L)\}) =1+2r.$$
Here is a particular capping $v_{q}: [0,1] \times [0,1] \rightarrow \Sigma$, described in the Darboux chart, for the intersection point $q$:
\begin{enumerate}
\item $v_{q}(s,t) \equiv q$.
\end{enumerate}
The Robbin-Salamon index of each Lagrangian path involved in the definition of the Conley-Zehnder index of $v_{q}$ is as follows,
\begin{enumerate}
\item $\mu_{RS}(T_{v_{q}(s,0)}L) = 0$,
\item $\mu_{RS}(\Lambda_{can}(T_{q}L, T_{q}\t{L})) = \frac{1}{2}$,
\item $\mu_{RS}(T_{v_{p}(1-s,1)}\t{L}) =0$,
\item $\mu_{RS}(\overline{\Lambda}_{\u{L}, \u{\t{L}}}) = -\mu_{RS}(\{d\psi_{\delta,t}^{-1}d\phi_{1}^{-1}(T_{q}L)\}) -\mu_{RS}(\{d\phi_{t}^{-1}(T_{q}L)\})= -\frac{1}{2} +2r.$
\end{enumerate}
Hence, the Conley-Zehnder index of $(q,v_{q})$ is equal to $- \mu_{RS}(\{d\phi_{t}^{-1}(T_{q}L)\})$ which by the discussion in \S\ref{sec:transform_map} we conclude that the corresponding capped chord $(\phi_{\delta,t}(q), u_{q})$ has index,
$$\mathrm{CZ}(\phi_{\delta,t}(q),u_{q}) = \mu_{RS}(\{d\phi_{t}(T_{q}L)\}) =2r.$$
The re-cappings of these capped chords are concatenations $w * u_{p}$ or $w * u_{q}$ where $w: [0,1] \times [0,1] \rightarrow \Sigma$ with $w(s,0)$ rotating $m$-times around $L$ and $w(s,1)$ rotating $n$-times around $L$ in which case it will shift the index by $\pm 2(m - n)$. Hence, only the capped chords $(\phi_{\delta,t}(p),w*u_{p})$ can have index $1$ and only a unique such cappings. Let,
$$e(\delta):= \mathcal{A}_{\u{L}, \u{\t{L}}}(p,v_{p}) = \text{ area of the left strip in Figure \ref{fig:two_circles_intersect}},$$
which obviously satisfies $e(\delta) \to 0$ as $\delta \to 0$. We conclude that the action of the capped chord $(\phi_{\delta,t}(p), u_{p}')$ that has index $1$, denoted by $a(\phi,1)$, satisfies,
$$a(\phi,1)= r\hbar + \int_{0}^{1} H_{t}(q) dt  + e(\delta),$$
which finishes the proof of the theorem.
\end{proof}
In the following theorem $\mu_{RS}(\{d\phi_{t}(T_{q}L)\})$ is computed using a fixed trivialization $\R^{2} \times \D \rightarrow u_{q}^{*}T\Sigma$ where $u_{q}$ is the unique capping of $\phi_{t}(q)$; the index does not depend on the trivialization.
\begin{theorem}\phantomsection\label{thm:lag_act_of_ind_one_chords_eq_ham_act}
Let $L \subset \Sigma$ be as before and $H$ be a Hamiltonian generating $\{\phi_{t}\}$ with $\phi_{1}(L) = L$. Suppose $q\in L$ is a contractible fixed point of $\phi_{1}$. Then, the Lagrangian action of any generator $(x,u)$ of $\CF_{1}(\u{L}, H\#\delta g)$ equals,
$$\mathcal{A}_{\u{L},H}(x,u) = \frac{\mu_{\{\phi_{t}\}}(L)\hbar}{2} + \mathcal{A}_{H}(q) + e(\delta),$$ 
where $e(\delta) \to 0$ as $\delta \to 0$; $\mathcal{A}_{H}(q)$ is the Hamiltonian action of $q$ and $\mu_{\{\phi_{t}\}}(L):= \mu_{RS}(\{d\phi_{t}(T_{q}L)\})$.
\end{theorem}
\begin{proof}
Let $G$ be a mean-normalized Hamiltonian generating a contractible loop of Hamiltonian diffeomorphisms $\eta_{t}$ so that $\eta_{t}(q) = \phi_{t}(q)$; see \cite[\S9]{salamon-zehnder-92}. Let $K = \overline{G}\# H$ generating the flow $\varphi_{t} = \eta_{t}^{-1}\phi_{t}$. Then, the flow $\varphi_{t}$ preserves $q$ for all times. Define $K_{\delta} = \overline{G}\# H \# \delta g.$ By Theorem \ref{thm:action_of_index_one_chords} we have that the Lagrangian action of any generator of $\CF_{1}(\u{L}, K_{\delta})$ equals,
\begin{equation}\phantomsection\label{eq:lag_act_formula}
r\hbar + \int_{0}^{1}K_{t}(q)dt + e(\delta),
\end{equation}
where $r$ satisfies $\mu_{RS}(\{d\varphi_{t}(T_{q}L)\}) = 2r$ and $e(\delta) \to 0$ as $\delta \to 0$. Since $G$ is a mean-normalized Hamiltonian the transformation $\mathfrak{n}_{G,w}: \CF(\u{L}, K_{\delta}) \rightarrow \CF(\u{L}, H_{\delta})$ preserves both degrees and actions; see \S\ref{sec:naturality_transformation}. So, the action of every generator of $\CF_{1}(\u{L}, H_{\delta})$ equals,
$$\begin{aligned}
r\hbar + \int_{0}^{1}K_{t}(q)dt + e(\delta) &= r\hbar + \int_{0}^{1}(\overline{G}\# H)_{t}(q)dt + e(\delta)\\
&=r \hbar - \int_{0}^{1}G_{t}\circ\eta_{t}(q)dt + \int_{0}^{1}H_{t}\circ\eta_{t}(q)dt + e(\delta)\\
&= r\hbar - \mathcal{A}_{G}(q) + \mathcal{A}_{H}(q) + e(\delta)= r \hbar + \mathcal{A}_{H}(q) + e(\delta),
\end{aligned}$$
where the last two equalities are by \cite[\S4.3]{schwarz-00} and the fact that $\Sigma$ has genus $g \geq 1$ and is a symplectically aspherical manifold. Because $\eta_{t}$ is a contractible loop of Hamiltonian diffeomorphisms we have
$$\mu_{RS}(\{d\varphi_{t}(T_{q}L)\}) = \mu_{RS}(\{d\phi_{t}(T_{q}L)\}).$$
Indeed, fix a trivialization $\Phi(1): \R^{2} \rightarrow T_{q}L$. Pick a path of loops of Hamiltonian diffeomorphisms $\eta_{s,t}$ where $\eta_{0,t} \equiv id$ and $\eta_{1,t} = \eta_{t}$. Consider the trivialization 
$$\Phi: \D \times \R^{2} \rightarrow u_{q}^{*}T\Sigma; se^{2\pi it}\mapsto d\eta_{s,t} \circ \Phi(1),$$
where $u_{q}:\D \rightarrow \Sigma$ is defined by $u_{q}(se^{2\pi it}) = \eta_{s,t}(q)$. Then 
$$\begin{aligned}
\mu_{RS}(\{d\phi_{t}(T_{q}L)\}) &= \mu_{RS}(\big\{\Phi(e^{2\pi it})^{-1} d\phi_{t}(T_{q}L)\big\})\\
&= \mu_{RS}(\big\{\Phi(1)^{-1}d\eta_{t}^{-1}d\phi_{t}(T_{q}L)\big\})\\
&= \mu_{RS}(\big\{\Phi(1)^{-1}d\varphi_{t}(T_{q}L)\big\}) = \mu_{RS}(\{d\varphi_{t}(T_{q}L)\}).
\end{aligned}$$
This completes the the proof of the theorem.
\end{proof}

\subsection{Inessential Franks-Handel invariant annuli}\phantomsection\label{sec:non_smooth_inv_ann} In this section we prove Theorem \ref{thm:main_theorem_introduction} in the case where there is an invariant Frank-Handel type annulus that is not necessarily smooth. Depending on the type of the annulus we divide the proof into two cases. In the current section we prove the theorem in the case when the annulus is inessential, i.e., contractible, and in \S\ref{sec:smooth_ess_inv_ann} we deal with the case where the annulus is essential. The main theorem of this section is the following where $\Sigma$ is a closed oriented surface of genus at least $1$ and $\phi_{c}$ is the prime extension of $\phi|_{A}$ to the prime end compactification $A_{c}$ of $A$, see Lemma \ref{lem:carath_ext_of_maps} and Remark \ref{rem:prime_end_on_annulus}, and $\t{\phi}_{c}$ is some lift of $\phi_{c}$ to the universal cover $\t{A}_{c}$ of $A_{c}$; see \S\ref{sec:theory_prime_ends}.

\begin{theorem}\phantomsection\label{mainthm:iness_franks_handel_inv_ann}
Let $\phi\in \Ham(\Sigma)$ be a Hamiltonian diffeomorphism with zero topological entropy. Suppose $A$ is an open \emph{inessential} annulus that is $\phi$-invariant and the rotation interval $\rho_{\t{\phi}_{c}}(A_{c})$ is non-trivial. Then, there exists a constant $\hbar > 0$ so that $\liminf_{k\geq 1}\gamma(\phi^{k}) \geq \hbar/2 > 0.$
\end{theorem}

\begin{remark}\phantomsection\label{rem:exist_of_such_ann}
In fact, for every Hamiltonian diffeomorphism $\phi \neq id$ of a closed oriented surface $\Sigma$ with genus at least $1$ that has zero topological entropy, there exists an open topological annulus $A\subset \Sigma$ such that it is invariant under $\phi$ and the rotation interval $\rho_{\t{\phi}_{c}}(A_{c})$ is a non-trivial segment. This is a result proved in a work in progress by Guih\'eneuf-Le Calvez-Passeggi-Tal; see Theorem \ref{thm:up_coming_work_of_LCT}. In the case where such an invariant annulus is essential we show in \S\ref{sec:smooth_ess_inv_ann} that $\liminf_{k\geq }\gamma(\phi^{k}) = \infty$. And the case where $A$ is inessential is covered by Theorem \ref{mainthm:iness_franks_handel_inv_ann} above.
\end{remark}

Before starting the proofs, we introduce the following terminology.
\begin{definition}\phantomsection\label{def:strong_reg}
Let $A$ be an open contractible annulus inside $\Sigma$. Let $\phi: S^{1} \times (0,1) \rightarrow \Sigma$ be a homeomorphism onto $A$ and $\gamma := \phi(S^{1}\times \{1/2\})$. Since $A$ is contractible so must be $\gamma$. Let $D$ be the unique embedded disk in $\Sigma$ whose boundary is $\gamma$, clearly $D\not\subset A$ as $\gamma$ is essential in $A$. Since $A$ is homotopy equivalent to $\gamma$ the set $\Sigma \backslash A$ has two closed connected components. Let $C_{A}$ be the connected component $D \backslash A$. It is obvious that $D_{A}:= \mathrm{int}(C_{A}) = D\backslash \bar{A}$. We say that $A$ is regular if $D_{A} \neq \emptyset$. In particular, $\hbar_{A}:=\mathrm{area}(D_{A}) > 0$. 
\end{definition}
\begin{remark}
Note that if $A$ is not regular, then the closure of $A$ covers $D$ and so $\bar{A}$ becomes an open disk. For an open annulus we let $\partial\bar{A} = \mathrm{Fr}_{1}(A) \sqcup \mathrm{Fr}_{2}(A)$ where $\mathrm{Fr}_{i}(A)$ are the two connected components of $\partial{\bar{A}}$ and $\mathrm{Fr}_{1}(A) = D \cap \partial{\bar{A}}$. Note that $\mathrm{Fr}_{1}(A)$ is empty unless $A$ is regular in which case $\mathrm{Fr}_{1}(A) = \partial\bar{D}_{A}\neq \emptyset$.
\end{remark}
\begin{remark}\phantomsection\label{rem:regularity_to_some_people}
Note that to some specialists regularity means $\mathrm{int}(\bar{A}) = A$. In this case one can show that $D_{A}$ is a non-empty open simply connected domain as follows. First we have $D_{A}\neq \emptyset$ since otherwise we would have $D\subset \bar{A}$ which then one picks a point $x\in D\backslash A$ and an open neighborhood $U_{x} \subset D \subset \bar{A}$ of $x$ and conclude that $x\in \mathrm{int}(\bar{A})\backslash A$. Assuming that $D_{A}$ is not simply connected there must exist a closed essential loop $\ell$ in $D_{A}$ which implies that $\bar{A}$ is disconnected and this is a contradiction as the closure of any connected set is connected. Define $\mathrm{Fr}_{1}(A) = D \cap \partial\bar{A} = \partial\bar{D}_{A}$, which is connected as the boundary of a simply connected domain. Since $D \cup A$ is also simply connected open domain, $\mathrm{Fr}_{2}(A) := \partial\overline{D \cup A}$ is another connected component of $\partial\bar{A}$ which is connected. So, $\partial{\bar{A}} =\mathrm{Fr}_{1}(A) \sqcup \mathrm{Fr}_{2}(A)$ so that $\mathrm{Fr}_{1}(A)$ bounds a non-empty open simply connected domain $D_{A}$, i.e., an open disk. This notion is also different from Franks-Handel regularity, their annuli satisfy $\mathrm{int}(\bar{A}) = A$ \emph{only inside} $\Sigma \backslash \Fix(\phi)$ where $A$ is $\phi$-invariant fixed point free annulus and $\phi$ is some area-preserving diffeomorphism of $S^{2}$. 
\end{remark}

\subsubsection{Inessential regular Franks-Handel invariant annuli}\phantomsection\label{sec:non_smooth_non_ess_inv_ann} In this section we study the case where the annulus $A$ from the statement of Theorem \ref{mainthm:iness_franks_handel_inv_ann} is regular in the sense of Definition \ref{def:strong_reg}. The main techniques of this section are the theory of prime ends developed by Carathéodory, see \cite{caratheodory-original-paper}, \cite[\S5.1]{MR2167675} and \S\ref{sec:theory_prime_ends}, and $C^{0}$-approximation techniques of Lefeuvre-Sikorav and Entov-Polterovich-Py explained in details in \S\ref{sec:fragmentation_c0_approximation}. These techniques allow us to reduce the problem to the smooth setting of \S\ref{sec:smooth_non_ess_inv_ann}.

Fix a conformal diffeomorphism $\Gamma:\A \rightarrow A$ where $\A =\{a <|z| <b\}$ for some numbers $0 < a <b <\infty$; this uses the regularity assumption on $A$. We denote by $\Gamma_{c}: \bar{\A} \rightarrow A_{c}$ the prime extension of $\Gamma$ given by the celebrated theorem of Carathéodory \cite{caratheodory-original-paper, mather-top-proof-82}, (see Theorem \ref{thm:ext_to_closure_vs_carath_compl}) and Remark \ref{rem:prime_end_on_annulus}, which is a homeomorphism, where on the right-hand side we consider the prime end topology. 

Throughout this section, by $C^{0}$-norm of a map $\phi\in \Ham(\Sigma)$ on a subset $K\subset \Sigma$ we shall mean $\max_{x\in K} d(x, \phi(x))$ where $d$ is a distance function induced by some apriori fixed metric on $\Sigma$.

\begin{proof}[Proof of Theorem \ref{mainthm:iness_franks_handel_inv_ann} (regular case)] In addition to the properties assumed for $A$ in the statement, suppose also that $A$ is regular in the sense of Definition \ref{def:strong_reg}. Suppose by contradiction that,
\begin{equation}\phantomsection\label{eq:contrary_assumption}
\liminf_{k\geq 1}\gamma(\phi^{k}) = 0.
\end{equation}
Let $\hbar >0$ be the area of the disk $D_{A}$ introduced in Definition \ref{def:strong_reg} for $A$. In the course of the proof we are free to work with any iteration of $\phi$ as if one proves that $\liminf \gamma(\phi^{Nk}) \geq \eps > 0$ for some $\eps>0$ and $N >0$ then by triangle inequality $\liminf \gamma(\phi^{k}) \geq \eps/N > 0$ which will be a contradiction with the assumption \eqref{eq:contrary_assumption}.

Step $1$. One can assume that $A$ is fixed point free. Since the rotation interval $\rho_{\t{\phi}_{c}}(A_{c})$ is non-trivial it follows from \cite[Theorem 2.3]{franks-handel-zero-ent} that for every rational number $p/q\in \rho_{\t{\phi}_{c}}(A_{c})$, in its reduced form, $\phi$ has a periodic point of period $q$ with rotation number $p/q$. This in particular implies that $\phi|_{A}$ is not torsion and up to passing to an iteration it has at least one fixed point in $A$. Let $S^{2}_{A} = A/\sim$ be obtained as above by collapsing the frontiers of $A$ to points $p_{1}$ and $p_{2}$ and denoted the induced homeomorphism by $\phi_{A}$. It is clear that $\phi_{A}$ has entropy zero, infinite order and at least three fixed points, and it is smooth away from a finite set of fixed points. Hence, by Remark \ref{rem:franks_handel_for_homeos} we can apply Theorem \ref{thm:franks_handel_theorem} to $\phi_{A}$. Up to passing to an iteration of $\phi$ and so an iteration of $\phi_{A}$ it follows from \cite[Lemma 15.9]{franks-handel-zero-ent} that there exists a fixed point free $\phi_{A}$-invariant annulus $U\in \mathcal{A}(\phi_{A},1)$ so that both of its frontiers are non-singular, they contain more than one points, $U$ is regular in the sense of Definition \ref{def:strong_reg}, and the rotation interval of $U$ is non-trivial; see the proof of \cite[Lemma 15.9]{franks-handel-zero-ent}. So, from now on we assume that $A$ is $\phi$-invariant, inessential, has non-trivial rotation interval, and is fixed point free for $\phi$. Note that $\phi|_{A}$ is also isotopic to the identity by Remark \ref{rem:rest_are_isotop_to_id}. In particular, fixed point free-ness together with the result of Franks \cite{franks1988generalizations, franks-recurrence-88} and \cite[Proposition 2.6]{franks-handel-zero-ent} we have
\begin{equation}\phantomsection\label{eq:key_ave_non_zero_thmE}
\int_{A_{c}}\rho_{\t{\phi}_{c}}(x)\omega(x) \neq 0.
\end{equation}
Step $2$. We show that up to passing to an iteration one can assume that $\gamma(\phi) < \hbar/2$. To see this, by assumption \eqref{eq:contrary_assumption} there is a large number $N$ for which $\gamma(\phi^{N}) < \hbar/2$. Let $p/q$ be a rational number in the rotation interval of $A$, which is assumed to be non-trivial, so that $q$ is a prime number with $q\nmid N$ and let $z\in A$ be a periodic point with rotation number $p/q$. Let $V\in \mathcal{A}(\phi,N)$ be an invariant annulus that contains $z$; see \S\ref{sec:franks_handel_thm_subsec}. It follows from \cite[Lemma 15.7 (2)]{franks-handel-zero-ent} that $V$ is essential in $A$. Hence, if we let $\hbar_{V}$ be the area of the disk $D_{V}$, then $V$ is an open annulus invariant under $\phi^{N}$, it contains no fixed points of $\phi^{N}$, and moreover, $\gamma(\phi^{N})<\hbar_{U}/2 \leq \hbar_{V}/2$ where the latter inequality is because $V$ is essential in $U$. So, for the rest of the proof we assume that $\gamma(\phi) < \hbar/2$. 

Step $3$. The prime extension $\phi_{c}$ of $\phi$ to the prime end compactification $A_{c}$ of $A$ possesses fixed points $x$ and $y$ on the two connected components of the boundary of $A_{c}$. This is proved in item $(\mathrm{iv})$ of Theorem \ref{thm:regopen}. Let $A_{p}$ be the sequence of nested annuli inside $A_{c}$ constructed in Theorem \ref{mainthm:approx_smth_inv_annulus} and let $\partial A_{p} = L_{p}\sqcup L_{p}'$. Pick some arbitrary points $x_{p}\in L_{p}$ and $y_{p}\in L_{p}'$ so that $x_{p} \to x$ and $y_{p}\to y$ as $p \to \infty$.

Step $4$. Let $\phi_{p} = \psi_{p}^{-1}\phi$ be the Hamiltonian diffeomorphisms that are small perturbations of $\phi$ by $\psi_{p}$ constructed in Theorem \ref{mainthm:approx_smth_inv_annulus} so that $\phi_{p}(A_{p}) = A_{p}$ and they satisfy the properties listed in Theorem \ref{mainthm:approx_smth_inv_annulus}. We perturb $\phi_{p}$ so that they will have contractible fixed points on $L_{p}$ and $L_{p}'$ for all large enough $p$. Since $\gamma(\psi_{p}) \to 0$ as $p \to \infty$, for every large enough $p$ we have $\gamma(\phi_{p}) < \hbar/2 <\hbar_{p}/2$ where $\hbar_{p}$ is the area of the disk $D_{p}$ with boundary $L_{p}$; see proof of Theorem \ref{mainthm:approx_smth_inv_annulus}. It follows from Theorem \ref{mainthm:non_van_floer_hom} that for the interval $I_{p} = [-\hbar_{p}+\delta, \delta]$ and $c_{p} = c([\Sigma], \phi_{p})$,
$$\HF^{I_{p} + c_{p}}_{1}(L_{p}, \phi_{p}) \neq 0, \HF^{I_{p} + c_{p}}_{1}(L_{p}', \phi_{p}) \neq 0.$$ 
This in particular implies that there is at least one contractible chord for $\phi_{p}$ on both $L_{p}$ and $L_{p}'$ for large enough $p$. Since $L_{p}$ and $L_{p}'$ are preserved under $\phi_{p}$ we conclude that all chords are contractible. In particular, the chords $\gamma_{x_{p}}$ with ends $x_{p}$ and $\phi_{p}(x_{p})$ are contractible, similarly the chords $\gamma_{y_{p}}$ with ends $y_{p}$ and $\phi_{p}(y_{p})$ are contractible. However, the end points $x_{p},\phi_{p}(x_{p})$ get arbitrarily close to each other in prime end topology as $p \to \infty$. In Lemma \ref{lem:pert_to_make_fixed_points} we construct Hamiltonian diffeomorphisms $\psi_{p}'\in \Ham(\Sigma)$ so that $\gamma(\psi_{p}')$ is arbitrarily small for any given $p$ and $|\Gamma^{-1}_{c}\psi_{p,c}'\Gamma_{c}|_{C^{0}(\A)} \to 0$ as $p \to \infty$ where $\psi_{p,c}'$ is the prime end extension of $\psi_{p}'$, and moreover, defining \[\phi_{p}' = \psi_{p}'\phi_{p} = \psi_{p}'\psi_p^{-1}\phi,\] we have that $\phi_{p}'$ has a contractible fixed point on $L_{p}$ and $L_{p}'$ for every large enough $p$. 

Step $5$. For every large enough $p$ one has 
$$\int_{A_{p}}\rho_{\t{\phi}_{p}'}(x)\omega(x)\neq 0,$$
for any lift $\t{\phi}_{p}'$ of $\phi_{p}'$ to the universal cover $\t{A}_{p}$ of $A_{p}$. This will be shown in Lemma \ref{lem:mean_rot_vec_non_zero_forallp} and the proof will use $C^{0}$-smallness of $\psi_{p}$ and $\psi_{p}'$ in prime end topology and continuity of the rotation number, see Remark \ref{rem:rot_numb_of_zero_ent_on_closed_ann}, and that mean rotation number for $\t{\phi}_{c}$ over $A_{c}$ is never zero \eqref{eq:key_ave_non_zero_thmE} for any lift $\t{\phi}_{c}$ of $\phi_{c}$.

Step $6$. It follows from Theorem \ref{mainthm:smth_cont_ann_case} and Remark \ref{rem:theorem_B_with_mean_rot_ass} that,
\begin{equation}\phantomsection\label{eq:liminf_gamma_for_phip}
\liminf_{k\geq 1} \gamma((\phi_{p}')^{k}) \geq \hbar_{p}/2 > 0,
\end{equation}
Hence, for a fixed $\delta >0$, for every $p$ there exists $N_{p}$ so that for every $n \geq N_{p}$,
$$\gamma((\phi_{p}')^{n}) \geq \hbar_{p}/2 - \delta/2.$$
In Lemma \ref{lem:bounded_nps} we show that $\sup_{p}N_{p} = N<\infty$, hence, there exists a universal number $N$ such that for every $n \geq N$ and every $p\gg 1$ we have $\gamma((\phi_{p}')^{n}) \geq \hbar_{p}/2 - \delta/2$. On the other hand by assumption \eqref{eq:contrary_assumption} there exists a sequence $n_{k}\to \infty$ so that $\gamma(\phi^{n_{k}}) \leq \hbar/2 - 2\delta.$ Hence, for every large enough $p$ and $n_{k} \geq N$ we have, by Lemma \ref{lemma:biinv},
$$\begin{aligned}
\hbar_{p}/2 - \delta/2 \leq \gamma((\phi_{p}')^{n_{k}}) &\leq n_{k}(\gamma(\psi_{p})+\gamma(\psi_{p}')) + \gamma(\phi^{n_{k}})\\
&\leq n_{k}(\gamma(\psi_{p}) + \gamma(\psi_{p}')) + \hbar/2 - 2\delta,
\end{aligned}$$
which implies that for every large enough $p$ so that $\hbar_{p}\geq \hbar -\delta$ one has $n_{k}(\gamma(\psi_{p})+ \gamma(\psi_{p}')) \geq \delta > 0$ whenever $n_{k} \geq N$. This is a contradiction because fixing $n_{k} \geq N$ and increasing $p$ and using the fact that  $\gamma(\psi_{p}), \gamma(\psi_{p}')\to 0$ as $p\to \infty$ we will have $n_{k}(\gamma(\psi_{p}) + \gamma(\psi_{p}')) < \delta$.
\end{proof}

\begin{theorem}\phantomsection\label{mainthm:approx_smth_inv_annulus}
Let $A\subset \Sigma$ be an open regular topological annulus that is $\phi$-invariant for some $\phi \in \Ham(\Sigma)$. There exist a sequence of nested closed smooth annuli $A_{1}\subset A_{2}\subset A_{3}\subset \dots \subset A$ that are essential in $A$ with $\cup A_{p} = A$ and Hamiltonian diffeomorphisms $\psi_{p}\in \Ham(\Sigma)$ satisfying,
\begin{enumerate}
\item $\psi_{p}(A_{p}) = \phi(A_{p})$ for all $p$,
\item $\mathrm{supp}(\psi_{p}) \Subset A$ for all $p$,
\item $\gamma(\psi_{p}) \to 0$ as $p \to \infty$,
\item $|\Gamma^{-1}_{c}\psi_{p,c}\Gamma_{c}|_{C^{0}(\A_{p} \cup \ell_{0})}\to 0$ as $p\to \infty$,
\end{enumerate}
where $\gamma$ is the spectral norm \S\ref{sec:spectral_invariants}, $\psi_{p,c}$ is the prime extension  \S\ref{sec:theory_prime_ends} of $\psi_{p}$ to the prime end compactification $A_{c}$ of $A$, $\A_{p} =\Gamma_{c}^{-1}(A_{p})$ and $\ell_{0} =\bar\A \cap \R^{+}.$
\end{theorem}
\begin{proof}
Let $\phi_{c}:A_{c} \rightarrow A_{c}$ be the prime extension  of $\phi$ to the prime end compactification $A_{c}$ of $A$; see \cite{caratheodory-original-paper, mather-top-proof-82}, Lemma \ref{lem:carath_ext_of_maps} and Remark \ref{rem:prime_end_on_annulus}. Let $\Omega := \Gamma^{*}(\omega|_{A})$ where $\omega$ is the symplectic form on $\Sigma$. Note that the $\Omega$-area of $\A$ is finite. Consider the diffeomorphism $\Phi:=\Gamma^{-1}\phi\Gamma: \A \rightarrow \A$ and observe that $\Phi^{*}\Omega =\Omega$. If we define $\Phi_{c}:= \Gamma_{c}^{-1}\phi_{c}\Gamma_{c}$ then $\Phi_{c}$ is a homeomorphism of $\bar{\A}$ that extends $\Phi$. First let $\Lambda^{b}_{\eps} = \{|z| = b - \eps\}$. Let $\phi^{b}_{n}=\phi^{b}_{\kh,n}$ be the smooth diffeomorphisms constructed in \S\ref{sec:smooth_pert_inv_lag} satisfying $\phi^{b}_{n}(\Lambda_{b,n}) = \Phi(\Lambda_{b,n})$ where $\Lambda_{b,n} = \Lambda^{b}_{\eps_{n}}$ for some sequence of numbers $\eps_{n}\to 0$ as $n \to \infty$; see Remark \ref{rem:adjustment_from_disk_to_annulus}. Moreover, we have $|\phi^{b}_{n}|_{C^{0}} \to 0$ as $n \to \infty$ and,
$$\mathrm{supp}(\phi^{b}_{n})\Subset V_{n}':=\{\frac{(b - R_{n})^{2}}{b - r_{n}}<|z|<b - r_{n}\},$$
where $r_{n},R_{n} > 0$ are some numbers converging to zero as $n \to \infty$. Let,
$$\{\Psi_{b,p} = \phi^{b}_{n_{p}}\circ \phi^{b}_{n_{p},1}\circ \phi^{b}_{n_{p},2} \circ \phi^{b}_{n_{p},3}\}_{p\geq 1}$$
be the sequence of diffeomorphisms of $\A$ from \S\ref{sec:smooth_pert_inv_lag_symp} satisfying
$$\Psi_{b,p}^{*}\Omega = \Omega,\ |\Psi_{b,p}|_{C^{0}(\A_{b,p}\cup \ell_{0})}\leq 1/p,\ \Psi_{b,p} (\Lambda_{b,n_{p}}) = \Phi(\Lambda_{b,n_{p}}),$$
where $n_{p}$ is a sequence of numbers diverging to infinity as $p \to \infty$, $\A_{b,p} = \{a < |z| \leq b-\eps_{n_{p}}\}$, $\ell_{0} = \bar\A\cap \R^{+}$, and,
$$\mathrm{supp}(\Psi_{b,p})\Subset U_{b,p}:=\{b-1/p<|z|<b\}.$$
Repeat the process for $\Lambda^{a}_{\eps} = \{|z| = a + \eps\}$ and $\Phi$ to find diffeomorphisms
$$\{\Psi_{a,p} = \phi^{a}_{n_{p}}\circ \phi^{a}_{n_{p},1}\circ \phi^{a}_{n_{p},2} \circ \phi^{a}_{n_{p},3}\}_{p\geq 1}$$
so that they satisfy the following properties:
$$\Psi_{a,p}^{*}\Omega = \Omega,\ |\Psi_{a,p}|_{C^{0}(\A_{a,p} \cup \ell_{0})}\leq 1/p,\ \Psi_{a,p} (\Lambda_{a,n_{p}}) = \Phi(\Lambda_{a,n_{p}}),$$
where $n_{p}$ is a sequence of numbers diverging to infinity, $\Lambda_{a,n_{p}} = \Lambda^{a}_{\eps_{n_{p}}}$ for some sequence $\eps_{n_{p}}\to 0$ as $p \to \infty$, $\A_{a,p} = \{a + \eps_{n_{p}} \leq |z| < b\}$, and,
$$\mathrm{supp}(\Psi_{a,p})\Subset U_{a,p}:=\{a<|z|< a + 1/p\}.$$
We define $\Psi_{p,c}:= \Psi_{a,p}\circ \Psi_{b,p}$; note that for large enough $p$ the supports of $\Psi_{a,p}$ and $\Psi_{b,p}$ are disjoint and near the two boundaries of $\A$ and so
$$|\Psi_{p,c}|_{C^{0}(\A_{p}\cup \ell_{0})} \to 0 \text{ as } p \to \infty,$$
where $\A_{p}:= \A_{a,p}\cap \A_{b,p}$. We relabel the sequence so that it starts from a large enough $p$. Define $\psi_{p,c}:= \Gamma_{c}\Psi_{p,c}\Gamma^{-1}_{c}$ and let $\psi_{p}= \psi_{p,c}|_{A}$. Note that $\A_{p}\subset \A$ is the annulus whose boundaries are $\Lambda_{a,n_{p}}$ and $\Lambda_{b,n_{p}}$, we define $A_{p}:= \Gamma_{c}(\A_{p})$ whose boundary is $L_{p}\sqcup L_{p'}$ where $L_{p} := \Gamma_{c}(\Lambda_{a,p})$ and $L_{p}' := \Gamma_{c}(\Lambda_{b,p})$. 

Let us check that the sequences $\psi_{p}$ and $A_{p}$  satisfy all the required properties. First, it is obvious that $\cup A_{p} = A$ and that $\psi_{p}$ are Hamiltonian diffeomorphisms of $\Sigma$. By construction $\mathrm{supp}(\psi_{p}) \Subset A$ and we have
$$\begin{aligned}
\psi_{p}(A_{p}) &= \psi_{p,c}(A_{p})= \Gamma_{c}\Psi_{p,c}\Gamma^{-1}_{c}\Gamma_{c}(\A_{p}) = \Gamma_{c}\Psi_{p,c}(\A_{p})\\ &= \Gamma_{c} \Phi_{c}(\A_{p}) = \Gamma_{c} \Gamma^{-1}_{c}\phi_{c}\Gamma_{c}(\A_{p}) = \phi(A_{p}).
\end{aligned}$$
As we pointed out above for large enough $p$ the supports of $\Psi_{b,p}$ and $\Psi_{a,p}$ are disjoint and since their $C^{0}$-norm converge to zero on $\A_{p} \cup \ell_{0}$ as $p \to \infty$ so the $C^{0}$-norm of $\Gamma_{c}^{-1}\psi_{p,c}\Gamma_{c} = \Psi_{p,c}$ which are the prime end compactifications of $\psi_{p}$ must go to zero on $\A_{p} \cup \ell_{0}$. 

It remains to prove that $\gamma(\psi_{p}) \to 0$ as $p \to \infty$. To prove this we use a Hamiltonian extension lemma of Entov-Polterovich-Py \cite{MR2884036}; see also \cite{MR5031842}. Let $Q_{b,p} \subset Q_{b,p}' \subset Q_{b,p}'' \subset U_{b,p}$ be a nested sequence of arbitrary small rectangles near the support of $\Psi_{b,p}$ containing $\ell_{0}\cap U_{b,p}$ so that $\Psi_{b,p}(Q_{b,p}) \subset Q_{b,p}'$. 
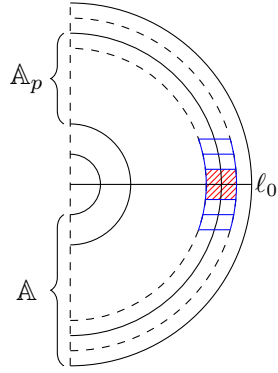
\begin{figure}[ht]
\centering
\begin{tikzpicture}[scale=0.4]
  \begin{scope}
    \clip ({sqrt(4.5^2-0.5^2)},0.5) -- ({sqrt(5.5^2-0.5^2)},0.5)
      arc ({asin(0.5/5.5)}:{-asin(0.5/5.5)}:5.5)
      -- ({sqrt(4.5^2-0.5^2)},-0.5)
      arc ({-asin(0.5/4.5)}:{asin(0.5/4.5)}:4.5)
      -- cycle;
    \foreach \c in {3.8,4.0,...,6.2}
      \draw[red] ({\c-2},-2) -- ({\c+2},2);
  \end{scope}
  \draw[dashed] (0,-6) -- (0,6);
  \draw (0,-1) arc (-90:90:1);
  \draw (0,-2) arc (-90:90:2);
  \draw (0,-5) arc (-90:90:5);
  \draw (0,-6) arc (-90:90:6);
  \draw[dashed] (0,-4.5) arc (-90:90:4.5);
  \draw[dashed] (0,-5.5) arc (-90:90:5.5);
  \draw (0,0) -- (6,0);
  \draw[blue] ({sqrt(4.5^2-0.5^2)},0.5) -- ({sqrt(5.5^2-0.5^2)},0.5);
  \draw[blue] ({sqrt(4.5^2-0.5^2)},-0.5) -- ({sqrt(5.5^2-0.5^2)},-0.5);
  \draw[blue] ({sqrt(4.5^2-1^2)},1) -- ({sqrt(5.5^2-1^2)},1);
  \draw[blue] ({sqrt(4.5^2-1^2)},-1) -- ({sqrt(5.5^2-1^2)},-1);
  \draw[blue] ({sqrt(4.5^2-1.5^2)},1.5) -- ({sqrt(5.5^2-1.5^2)},1.5);
  \draw[blue] ({sqrt(4.5^2-1.5^2)},-1.5) -- ({sqrt(5.5^2-1.5^2)},-1.5);
  \draw[blue] ({sqrt(5.5^2-1.5^2)},1.5) arc ({asin(1.5/5.5)}:{-asin(1.5/5.5)}:5.5);
  \draw[blue] ({sqrt(4.5^2-1.5^2)},1.5) arc ({asin(1.5/4.5)}:{-asin(1.5/4.5)}:4.5);
  \node at (6.5, 0) {\small{$\ell_{0}$}};
  \draw [decorate,decoration={brace,amplitude=5pt,mirror,raise=0.5ex}] (0,5) -- (0,2) node[midway,xshift=-1.5em]{$\A_{p}$};
  \draw [decorate,decoration={brace,amplitude=5pt,mirror,raise=0.5ex}] (0,-1) -- (0,-6) node[midway,xshift=-1.5em]{$\A$};
\end{tikzpicture}
\caption{The nested sequence of cubes $Q_{b,p}\subset Q_{b,p}'\subset Q_{b,p}''$. The red dashed cube is $Q_{b,p}$. The region between the two dashed half circles is $U_{b,p}$.}
\label{fig:the_cubes}
\end{figure}
This is possible because $\Psi_{b,p} \equiv id$ on $\ell_{0}$; see Figure \ref{fig:the_cubes}. Then, by \cite[Lemma 6.5]{MR2884036} we can find a Hamiltonian diffeomorphism $\xi_{b,p}$ compactly supported in $Q_{b,p}''$ such that $\xi_{b,p}|_{Q_{b,p}} = \Psi_{b,p}|_{Q_{b,p}}$. Letting $\eta_{b,p} = \xi_{b,p}^{-1}\Psi_{b,p}$ we observe that $\eta_{b,p}$ is compactly supported inside $U_{b,p}\backslash Q_{b,p}^{sh}$ where $Q_{b,p}^{sh} \subset Q_{b,p}$ is a slightly smaller rectangle. Constructing $\xi_{a,p}$ and $\eta_{a,p}$ similarly for some rectangles $Q_{a,p}\subset Q_{a,p}' \subset Q_{a,p}'' \subset U_{a,p}$ we can fragment $\Psi_{p,c}$ as follows,
$$\Psi_{p,c} = \xi_{a,p}\circ \xi_{b,p}\circ \eta_{a,p}\circ \eta_{b,p}.$$
An important observation is that we have 
$$\psi_{p} = (\Gamma\xi_{a,p}\Gamma^{-1})\circ (\Gamma\xi_{b,p}\Gamma^{-1})\circ (\Gamma\eta_{a,p}\Gamma^{-1}) \circ (\Gamma\eta_{b,p}\Gamma^{-1}),$$
where each one of the fragments above are supported in some topological disk whose area converge to zero as $p \to \infty$. Hence we have
$$\begin{aligned}
\gamma(\psi_{p}) &\leq \gamma(\Gamma\xi_{a,p}\Gamma^{-1})+ \gamma(\Gamma\xi_{b,p}\Gamma^{-1})+ \gamma(\Gamma\eta_{a,p}\Gamma^{-1}) + \gamma(\Gamma\eta_{b,p}\Gamma^{-1})\\
&\leq \mathrm{area}(\Gamma(Q_{a,p}'')) + \mathrm{area}(\Gamma(Q_{b,p}'')) \\
&+ \mathrm{area}(\Gamma(U_{a,p}\backslash Q^{sh}_{a,p})) + \mathrm{area}(\Gamma(U_{b,p}\backslash Q^{sh}_{b,p}));\end{aligned}$$
which implies that $\gamma(\psi_{p})\to 0$ as $p \to \infty$; see \cite[Theorem 1]{MR5018404}.
\end{proof}

\begin{lemma}\phantomsection\label{lem:pert_to_make_fixed_points}
Let $\psi_{p}$ and $\phi$ be as in the proof of Theorem \ref{mainthm:iness_franks_handel_inv_ann}. There is a Hamiltonian diffeomorphism $\psi_{p}'$ that is compactly supported in $A$, it preserves $A_{p}$, it has arbitrary small spectral norm and $|\Gamma^{-1}_{c}\psi_{p,c}'\Gamma_{c}|_{C^{0}(\A)} \to 0$ as $p \to \infty$ where $\psi_{p,c}'$ is its prime extension to $A_{c}$. Moreover, $\psi_{p}'\psi_{p}^{-1}\phi$ possesses a contractible fixed point on $L_{p}$ and $L_{p}'$.
\end{lemma}
\begin{proof}
Let $x\in \{|z| = a\}$ and $y\in \{|z| = b\}$ be the fixed points of $\Phi_{c} = \Gamma_{c}^{-1}\phi_{c}\Gamma_{c}$; their existence was shown in Step $3$ of the proof of Theorem \ref{mainthm:iness_franks_handel_inv_ann}. Let $x_{p}\in \Lambda_{a,n_{p}}$ and $y_{p}\in \Lambda_{b,n_{p}}$ be some arbitrary points so that $x_{p} \to x$ and $y_{p} \to y$ as $p\to \infty$; recall from Theorem \ref{mainthm:approx_smth_inv_annulus} that $\partial \A_{p} = \Lambda_{a,n_{p}} \sqcup \Lambda_{b,n_{p}}$. Since
$$\abs{\Psi_{p,c}^{-1}}_{C^{0}(\Phi_{c}(\A_{p}))} \to 0 \text{ as } p \to \infty,$$
and $\Phi_{c}(x) = x$ we must have,
$$d(x_{p}, \Psi_{p,c}^{-1}\Phi_{c}(x_{p})) \to 0 \text{ as } p\to \infty.$$
Let $h_{a,p}: \A \rightarrow \R$ be a radial Hamiltonian function so that on $\Lambda_{a,n_{p}}$ we have $h_{a,p} = 0$ and $\partial_{r}h_{p}=\theta_{p}$ is some small number so that $R_{\theta_{p}}\Psi_{p,c}^{-1}\Phi_{c}(x_{p}) = x_{p}$ where $R_{\theta_{p}}$ is the rotation with angle $\theta_{p}$, and $h_{a,p} = 0$ outside of a small neighborhood of $\Lambda_{a,n_{p}}$ so that $|h_{a,p}|_{C^{1}} \leq \theta_{p}$ and $|h_{a,p}|$ is arbitrary small where the latter is the Hofer norm of $h_{a,p}$. Define $\Psi_{a,p}'$ to be the time-one map of $h_{a,p}$. Then, clearly $|\Psi_{a,p}'|_{C^{0}} \to 0$ as $p \to\infty$ and $\Psi_{a,p}'\Psi_{p,c}^{-1}\Phi_{c}(x_{p}) = x_{p}$. 

Construction of $\Psi_{b,p}'$ is similar. It is supported in a small neighborhood of $\Lambda_{b,n_{p}}$ and $\Psi_{b,p}'\Psi_{p,c}^{-1}\Phi_{c}(y_{p}) = y_{p}$. Define
$$\Psi_{p,c}' = \Psi_{a,p}' \circ \Psi_{b,p}',$$ 
and let $\psi_{p}'$ be the restriction of $\Gamma_{c}\Psi_{p,c}'\Gamma_{c}^{-1}$ to $A$. Then, it is clear that 
$$\psi_{p}'\psi_{p}^{-1}\phi(\Gamma(x_{p})) = \Gamma(x_{p}),\  \psi_{p}'\psi_{p}^{-1}\phi(\Gamma(y_{p})) = \Gamma(y_{p}).$$
Moreover, $\psi_{p}'$ is a Hamiltonian generated by $(h_{a,p} + h_{b,p})\circ \Gamma^{-1}$ which has arbitrary small Hofer norm and hence $\gamma(\psi_{p}')$ is arbitrary small by the well known inequality $\gamma(G)\leq |G|$. It remains to show that the fixed points $\Gamma(x_{p})$ and $\Gamma(y_{p})$ are contractible fixed points of $\psi_{p}'\psi_{p}^{-1}\phi$. It was shown in Step $4$ of the proof of Theorem \ref{mainthm:iness_franks_handel_inv_ann} that all the Hamiltonian chords of $\phi_{p}$ from $L_{p}$ to itself and $L_{p}'$ to itself are contractible. Let $\psi_{p,t}'$ be the Hamiltonian isotopy from $id$ to $\psi_{p}'$ generated by $(h_{a,p} + h_{b,p})\circ \Gamma^{-1}$. Since the latter is supported in $A$, for large enough $p$ one finds a Hamiltonian isotopy $\phi_{p,t}'$ from $id$ to $\phi_{p}'$ so that $\{\phi_{p,t}'(\Gamma(x_{p}))\}$ and $\{\phi_{p,t}'(\Gamma(y_{p}))\}$ are contractible.
\end{proof}

\begin{lemma}\phantomsection\label{lem:mean_rot_vec_non_zero_forallp}
For every large enough $p$ one has $$\int_{A_{p}}\rho_{\t{\phi}_{p}'}(x)\omega(x)\neq 0.$$
\end{lemma}
\begin{proof}
For a given $p$, let $\phi_{c}, \psi_{p,c}, \psi_{p,c}'$ be the prime extensions of $\phi, \psi_{p}, \psi_{p}'$ to the prime end compactification $A_{c}$ of $A$. Define $\Phi_{c}, \Psi_{p,c}, \Psi_{p,c}'$ to be the conjugations of $\phi_{c}, \psi_{p,c}, \psi_{p,c}'$ by $\Gamma_{c}: \bar{\A} \rightarrow A_{c}$ respectively. Let $\t{\phi}_{c}, \t{\psi}_{p,c}, \t{\psi}_{p,c}'$ be some lifts of their corresponding maps on $A_{c}$ to the universal cover $\t{A}_{c}$ of $A_{c}$. Set $\t{\Phi}_{c}, \t{\Psi}_{p,c}, \t{\Psi}_{p,c}'$ to be the conjugation of the latter maps by $\Gamma_{c}$ and
$$\Phi_{p,c}':= \Psi_{p,c}'\Psi_{p,c}^{-1}\Phi_{c},\ \t{\Phi}_{p,c}':= \t{\Psi}_{p,c}'\t{\Psi}_{p,c}^{-1}\t{\Phi}_{c}.$$
It is clear that $\t{\Phi}_{p,c}'|_{\t{\A}_{p}}$ is a lift of $\Phi_{p,c}'|_{\A_{p}}$. Recall from Theorem \ref{mainthm:approx_smth_inv_annulus} and Lemma \ref{lem:pert_to_make_fixed_points} that 
$$\abs{\Psi_{p,c}'}_{C^{0}(\A)}, \abs{\Psi_{p,c}^{-1}}_{C^{0}(\Phi_{c}(\A_{p}))} \to 0 \text{ as } p \to \infty.$$
We make particular choices of lifts $\t{\Psi}_{p,c}'$ and $\t{\Psi}_{p,c}$. Let $\Psi_{p,c,t}'$ be the symplectic isotopy generated by $h_{a,p} + h_{b,p}$ from $id$ to $\Psi_{p,c}'$ with support inside $\A$ which satisfies,
$$\abs{\Psi_{p,c,t}'}_{C^{0}(\A)} \leq \abs{\Psi_{p,c}'}_{C^{0}(\A)}, \forall t\in [0,1].$$
Recall that $\Psi_{p,c} = \Psi_{a,p}\Psi_{b,p}$ as constructed in Theorem \ref{mainthm:approx_smth_inv_annulus}. The constructions of $\Psi_{a,p}$ and $\Psi_{b,p}$ have been performed in \S\ref{sec:smooth_pert_inv_lag} and  \S\ref{sec:smooth_pert_inv_lag_symp} where each one of these maps are compositions of four types of maps,
$$\Psi_{a,p} = \phi^{a,p}_{n,0}\circ \phi^{a,p}_{n,1}\circ \phi^{a,p}_{n,2}\circ \phi^{a,p}_{n,3},\ \Psi_{b,p} = \phi^{b,p}_{n,0}\circ \phi^{b,p}_{n,1}\circ \phi^{b,p}_{n,2}\circ \phi^{b,p}_{n,3},$$
where as explained in \S\ref{sec:smooth_pert_inv_lag} one also has 
$$\phi^{a,p}_{n,0} = (sh^{a,p}_{n,3}\circ sh^{a,p}_{n,2} \circ sh^{a,p}_{n,1})^{-1},\ \phi^{b,p}_{n,0} = (sh^{b,p}_{n,3}\circ sh^{b,p}_{n,2} \circ sh^{b,p}_{n,1})^{-1}.$$
It is shown in \S\ref{sec:smooth_pert_inv_lag} that,
$$\abs{\phi^{a,p}_{n,i}}_{C^{0}(\A)}, \abs{\phi^{b,p}_{n,i}}_{C^{0}(\A)} \to 0 \text{ as } p\to \infty; \forall i\in\{0,1,3\},$$
and $|\phi^{a,p}_{n,2}|_{C^{0}(\A_{p})}, |\phi^{b,p}_{n,2}|_{C^{0}(\A_{p})}\to 0$ as $p\to \infty$. 

It is evident from the constructions of the maps $sh^{a,p}_{n,i}, sh^{b,p}_{n,i}, \phi^{a,p}_{n,j}, \phi^{b,p}_{n,j}$ for $i\in \{1,2,3\}$ and $j\in \{1,3\}$ that they are all time-one maps of some flows whose isotopy-$C^{0}$-norms converge to zero as $p\to \infty$ where by isotopy-$C^{0}$-norm of $\{f_{t}\}_{t\in [0,1]}$ we mean $\max_{t\in [0,1]}|f_{t}|_{C^{0}}$. For a similar reason, as pointed out in \S\ref{sec:const_of_phi_two}, there are smooth isotopies with time-one maps $\phi^{a,p}_{n,2}$ and $\phi^{b,p}_{n,2}$ whose isotopy-$C^{0}$-norms on the subset $\A_{p}\cup \ell_{0}$ converge to zero as $p \to \infty$. Hence, there exists an isotopy $\Psi_{p,c,t}$ from $id$ to $\Psi_{p,c}$ supported in $\A$ so that,
$$\max_{t\in [0,1]}\abs{\Psi_{p,c,t}^{-1}}_{C^{0}(\Phi_{c}(\A_{p}))}\to 0 \text{ as } p\to \infty.$$
We pick the lifts $\t{\Psi}_{p,c}'$ and $\t{\Psi}_{p,c}^{-1}$ to be given by the isotopies $\Psi_{p,c,t}'$ and $\Psi_{p,c,t}$ respectively. Then, the key conclusion is that if we let $\t{\Psi}_{p,c}'':=\t{\Psi}_{p,c}'\circ \t{\Psi}_{p,c}^{-1}$ then
\begin{equation}\phantomsection\label{eq:key_estimate_on_univ_cover}
\abs{\t{\Psi}_{p,c}''}_{C^{0}(\t{\Phi}_{c}(\t{\A}_{p}))} \to 0 \text{ as } p\to \infty.
\end{equation}
We apply Birkhoff's ergodic theorem \S\ref{sec:birkhoff_erg_thm} to the measure preserving transformations $T(\t{x}) = \t{\Phi}_{c}(\t{x})$ and $T(\t{x}) = \t{\Phi}_{p,c}'|_{\t{\A}_{p}}(\t{x})$ and the functions $f(x) = y(T(\t{x})) - y(\t{x})$ for $x\in \A$ and $x \in \A_{p}$ respectively where $\t{x}$ is some lift of $x$ to the universal covers $\t{\A}$ and $\t{\A}_{p}$, and we estimate as follows:
$$\begin{aligned}
&\int_{\A_{p}}\rho_{\t{\Phi}'_{p,c}}(x)\omega(x)-\int_{\bar{\A}}\rho_{\t{\Phi}_{c}}(x)\omega(x)
=\int_{\A_{p}}\rho_{\t{\Psi}_{p,c}''\t{\Phi}_{c}}(x)\omega(x) - \int_{\bar{\A}}\rho_{\t{\Phi}_{c}}(x)\omega(x)\\
&=\int_{\mathrm{FD}_{p}}\big(y(\t{\Psi}_{p,c}''(\t{\Phi}_{c}(\t{x}))) - y(\t{x})\big)\omega(\t{x}) - \int_{\mathrm{FD}}\big(y(\t{\Phi}_{c}(\t{x})) - y(\t{x})\big)\omega(\t{x})\\
&=\int_{\mathrm{FD}_{p}}\big(y(\t{\Psi}_{p,c}''(\t{\Phi}_{c}(\t{x}))) - y(\t{\Phi}_{c}(\t{x}))\big)\omega(\t{x}) - \int_{\mathrm{FD}\backslash \mathrm{FD}_{p}}\rho_{\t{\Phi}_{c}}(x)\omega(x),
\end{aligned}$$
where $\mathrm{FD}_{p}\subset \mathrm{FD}$ are some fundamental domains of $\t{\A}_{p}$ and $\t{\bar{\A}}$. By the key estimate \eqref{eq:key_estimate_on_univ_cover}, for $p$ large enough the $C^{0}$-norm of $\t{\Psi}_{p,c}''$ will be small enough and since $\phi$ has zero entropy the rotation number $\rho_{\t{\Phi}_{c}}: \bar{\A} \rightarrow \R$ is a continuous function, see Remark \ref{rem:rot_numb_of_zero_ent_on_closed_ann} and Lemma \ref{lem:rotation_number_conjugation}, so the last two terms become arbitrarily small for large enough numbers $p$. The above discussion together with \eqref{eq:key_ave_non_zero_thmE} and Lemma \ref{lem:rotation_number_conjugation} imply,
\begin{equation}\phantomsection\label{eq:ave_of_bigger_rot_funct_non_zero}
\int_{\A_{p}}\rho_{\t{\Phi}_{p,c}'}(x)\omega(x)\not\in a_{p}\Z,
\end{equation}
for the above particular lift, and so for all lifts, and for any large enough $p$; here $a_{p} = \mathrm{area}_{\omega}(A_{p}) \to a$ as $p \to \infty$ since $\cup A_{p} = A$. We emphasize that the proof is not finished yet because the quantity in \eqref{eq:ave_of_bigger_rot_funct_non_zero} is the integral of the rotation function $\rho_{\t{\Phi}_{p,c}'}: \bar{\A} \rightarrow \R$ over the sub annulus $\A_{p}$. What we need to show is that the integral of the rotation function $\rho_{\t{f}}:\A_{p} \rightarrow \R$ is not in $a_{p}\Z$ where $\t{f}$ is the lift of $\Phi_{p,c}'|_{\A_{p}}$ to $\t{\A}_{p}$ obtained by restricting $\t{\Phi}_{p,c}'$ to $\t{\A}_{p}$. However, this immediately follows from Lemma \ref{lem:rot_num_int_ann_vs_ext_ann} and Remark \ref{rem:equality_of_rot_num_nested_ann} which imply that the integral in \eqref{eq:ave_of_bigger_rot_funct_non_zero} coincides with the integral $\int_{\A_{p}}\rho_{\t{f}}(x)\omega(x)$ modulo $a_{p}\Z$.
\end{proof}

\begin{lemma}\phantomsection\label{lem:bounded_nps}
The numbers $\{N_{p}\}_{p\geq 1}$ are uniformly bounded.
\end{lemma}
\begin{proof}
Recall that for large enough $p$ we have $D(\phi_{p}) \neq 0$ where $\phi_{p}$'s are as in the proof of Theorem \ref{mainthm:iness_franks_handel_inv_ann}. Note that if $\gamma(\phi_{p}^{N})\leq \hbar_{p}/2 - 2\delta$ for some $\delta >0$ and $N>0$ then by Theorem \ref{mainthm:non_van_floer_hom},
$$\HF^{I+c}_{1}(L_{p}, \phi_{p}^{N})\neq 0 \text{ and } \HF^{I+c}_{1}(L_{p}', \phi_{p}^{N})\neq 0,$$
where $I = [-\hbar_{p}/2 + \delta, \delta]$, $c = c([\Sigma], \phi_{p}^{N})$ and $L_{p},L_{p}'$ are the connected components of $\partial A_{p}$. It follows that if $D(\phi_{p}^{N}) = ND(\phi_{p}) > \hbar_{p}/2$ then $\gamma(\phi_{p}^{N}) \geq \hbar_{p}/2$. Therefore, the numbers $N_{p}$ only need to be large enough so that,
$$N_{p}\mathrm{D}(\phi_{p}) > \hbar_{p}/2.$$
So we only need to show that the following invariant is bounded from below,
$$\mathrm{D}(\phi_{p})= \int_{A_{p}}\rho_{\t{\phi}_{p}}(x)\omega(x) + a_{p}( \frac{\mu_{\{\phi_{p,t}\}}(L_{p})}{2} - \rho_{\t{\phi}_{p}}(L_{p})),$$
where $a_{p} = \mathrm{area}_{\omega}(A_{p})$. Note that the first term, the mean rotation number, converges to a fixed number that depends only on $\t{\phi}_{c}$ and $A_{c}$, namely the mean rotation number over $A_{c}$ with respect to $\t{\phi}_{c}$, see the proof of Lemma \ref{lem:mean_rot_vec_non_zero_forallp}, which was shown that it is not an integer multiple of $a = \mathrm{area}_{\omega}(A)$. Given that the rest of the terms are inside $a_{p}\Z$ for all $p$ we conclude that for a fixed small $\eps > 0$ once $p$ is large enough,
$$\mathrm{D}(\phi_{p}) \geq d(\bar{\rho}, a\Z) - \eps >0,$$
where $\bar{\rho}$ is the integral of $\rho_{\t{\phi}_{c}}$ over $A_{c}$ and $d$ is the Euclidean distance.
\end{proof}

\subsubsection{Inessential irregular Franks-Handel invariant annuli}\phantomsection\label{sec:non_smooth_non_reg_non_ess_inv_ann}
This case is reduced to the case of \S\ref{sec:non_smooth_non_ess_inv_ann} by passing to an iteration of $\phi$. 
\begin{proof}[Proof of Theorem \ref{mainthm:iness_franks_handel_inv_ann} (irregular case)] Suppose that $A$ is inessential open irregular annulus that is $\phi$-invariant and the rotation interval $\rho_{\t{\phi}_{c}}(A_{c})$ is non-trivial. This in particular implies that $\phi$ has periodic points in the interior of $A$. Let $S^{2}_{A}:= A/\sim$ where we have collapsed the frontiers $\mathrm{Fr}_{1}(A), \mathrm{Fr}_{2}(A)$ of $A$ into single points $p_{1}:= [\mathrm{Fr}_{1}(A)]$ and $p_{2}:= [\mathrm{Fr}_{2}(A)]$ and let $\phi_{A}: S^{2}_{A} \rightarrow S^{2}_{A}$ be the induced homeomorphism by $\phi$. The resulting map $\phi_{A}$ has zero entropy, it has at least three fixed points up to passing to an iteration, and has infinite order as $\phi|_{A}$ is not torsion since it has non-trivial rotation interval $J$, i.e., by \cite{franks1988generalizations} there are periodic points of $\phi$ of period $q$ and rotation number $p/q\in J$ where $q$ is any arbitrary large prime number, this clearly implies that $\phi$ is not torsion. By Remark \ref{rem:franks_handel_for_homeos} we can still apply the structure theorem of Franks and Handel, Theorem \ref{thm:franks_handel_theorem}, to the homeomorphism $\phi_{A}$ and all of its iterates to find invariant open topological annuli. It follows from \cite[Lemma 15.9]{franks-handel-zero-ent} that for a sufficiently large prime number $q$ there is an open annulus $V\in \mathcal{A}(\phi,q)$ that is a proper and essential subset of $A$ and it has a non-trivial rotation interval $\rho_{\t{f}}(V_{c})$ where $\t{f}$ is a lift of $f := \phi^{q}|_{V}$ to the universal cover $\t{V}$ of $V$. In particular, $V$ is an open regular $\phi^{q}$-invariant annulus that has non-trivial rotation interval. Replacing $\phi$ with $\phi^{q}$ we can apply the result of \S\ref{sec:non_smooth_non_ess_inv_ann} to $\phi^{q}$ and find a constant $\hbar > 0$ so that
$$\liminf_{k\geq 1}\gamma(\phi^{qk}) \geq \hbar/2 > 0.$$
It follows from the triangle inequality that $\liminf_{k\geq 1}\gamma(\phi^{k}) \geq \hbar/2q >0$.
\end{proof}
\subsection{Essential invariant annuli}\phantomsection\label{sec:smooth_ess_inv_ann}
In this section we prove Theorem \ref{thm:main_theorem_introduction} in the case where there is an essential invariant annulus and the Hamiltonian diffeomorphism has zero topological entropy. Namely, we show the following theorem where $\Sigma$ is a closed orientable surface of genus $g\geq 1$,

\begin{theorem}\phantomsection\label{mainthm:ess_smth_inv_ann}
Suppose $\phi \in \Ham(\Sigma)$ is a Hamiltonian diffeomorphism and $A\subset \Sigma$ is an essential open annulus that is $\phi$-invariant. Assume that $\phi$ has zero topological entropy and $A$ has a non-trivial rotation interval with respect to some, and hence all, lifts of $\phi$ to $\t{A}$. Then, $\liminf_{k\geq 1}\gamma(\phi^{k}) = \infty.$
\end{theorem}
See \S\ref{sec:proof_of_ess_smth_case} for the proof. The idea of the proof is to show that the lifts of high iterations of $\phi$ to $n$-covering spaces $\Sigma_{n} \rightarrow \Sigma$ for large integers $n \gg 1$ displace balls of arbitrary large capacities and use the energy capacity inequality to bound the spectral norm of iterations from below. For this we need to study the relation between Floer theory of $(\Sigma, H,J)$ and the lifted data $(\Sigma_{n}, \t{H}_{n}, \t{J}_{n})$; see \S\ref{sec:floer_hom_covering_spcs} below.
\subsubsection{Floer homology and covering spaces}\phantomsection\label{sec:floer_hom_covering_spcs}
Let $\Sigma$ be a closed orientable surface of genus $r \geq 1$ equipped with a symplectic form $\omega$. Pick an almost complex structure $J: T\Sigma \rightarrow T\Sigma$ compatible with $\omega$, let $\t{J}_{n} = \pi^{*}_{n}J$ and $\t{\omega}_{n}= \pi^{*}_{n}\omega$, where $\pi_{n}: \Sigma_{n} \rightarrow \Sigma$ is an $n$-covering of $\Sigma$ from a closed surface $\Sigma_{n}$ of genus $rn - n+1$. Suppose $H: \Sigma\times \R\rightarrow \R$ is a smooth generic Hamiltonian generating the Hamiltonian flow $\{\phi_{t}\}$ with time-one map $\phi_{1} = \phi$; here generic means that the data is Floer regular, i.e., all one periodic orbits of $\{\phi_{t}\}$ are non-degenerate and transversality is achieved for the moduli spaces of Floer cylinders; see \S\ref{sec:abs_ham_floer_theory} and \cite{MR1702944} and references therein. Let $\t{H}_{n}: \Sigma_{n} \times \R \rightarrow \R$ be defined by $\t{H}_{n} = H\circ \pi_{n}$ where $\pi_{n}: \Sigma_{n} \rightarrow \Sigma$ is the covering map and let $\t{\phi}_{n,t}$ be the Hamiltonian flow of $\t{H}_{n}$.
\begin{remark}
The data $(\t{H}_{n}, \t{J}_{n})$ is Floer regular as the suitable linearized operators are isomorphic to the corresponding operators for $(H_n, J_n)$.
\end{remark}
Denote by $\CF(H,J)$ the Floer complex of $H$ and $J$ that is a free $\Z/2\Z$-module generated by the set of \emph{contractible} fixed points $\Fix_{c}(\phi)$ of $\phi$, equipped with the Floer differential $d$. Denote also by $\CF(\t{H}_{n}, \t{J}_{n})$ the corresponding Floer complex for the lifted data equipped with the Floer differential $\t{d}_{n}$. We define the following linear lift map,
$$L_{n}: \CF(H,J) \rightarrow \CF(\t{H}_{n},\t{J}_{n});$$
on a generator $x\in \Fix_{c}(\phi)$ it is defined by $L_{n}(x) = \t{x}_{1} + \dots + \t{x}_{n}$ where $\t{x}_{1}, \dots, \t{x}_{n}$ are the lifts of $x$ to the covering space $\Sigma_{n}$. Clearly this is a well-defined linear map because we only work with the contractible fixed points. As a vector space one has
$$\CF(\t{H}_{n}, \t{J}_{n}) = \CF(H,J)\oplus \dots \oplus \CF(H,J),$$ 
where the number of factors are $n$. Let $u: \R\times S^{1} \rightarrow \Sigma$ be a smooth map that solves Floer's equation for $H$ and $J$, i.e.,
$$\partial_{s}u(s,t) + J(u(s,t))(\partial_{t}u(s,t) - X_{H}(u(s,t))) = 0,$$
and suppose that $u(-\infty,t) = \phi_{t}(x)$ and $u(+\infty, t) = \phi_{t}(y)$ for some fixed points $x,y\in \Fix_{c}(\phi)$. Let $\t{z}_{1}, \dots, \t{z}_{n}$ be the lifts of $z = u(0,0)$ to $\Sigma_{n}$ and let $\t{u}_{k}: \R \times S^{1} \rightarrow \Sigma_{n}$ be the smooth lifts of $u$ so that $\t{u}_{k}(0,0) = \t{z}_{k}$. we have
$$2E(\t{u}_{k}) = \int_{\R\times S^{1}}|d\t{u}_{k}|_{\t{J}_{n}}^{2} ds\wedge dt = \int_{\R\times S^{1}}|du_{k}|_{J}^{2} ds\wedge dt < \infty,$$
where the second equality is by $|d\pi \circ d\t{u}|_{J} = |d\t{u}|_{\t{J}_{n}}$. Moreover,
$$\partial_{s}\t{u}_{k}(s,t) + \t{J}_{n}(\t{u}_{k}(s,t))(\partial_{t}\t{u}_{k}(s,t) - X_{\t{H}_{n}}(\t{u}_{k}(s,t))) = 0,$$
because the equation above holds after applying $d\pi_{n}$ to both sides and $d\pi_{n}$ is pointwise an isomorphism. Hence, it follows from \cite[Proposition 1.21]{MR1702944} that there are some fixed points $\t{x}_{k},\t{y}_{k}\in \Fix_{c}(\t{\phi})$ so that $\t{u}_{k}(-\infty,t) = \t{\phi}_{t}(\t{x}_{k})$ and $\t{u}_{k}(+\infty,t) = \t{\phi}_{t}(\t{y}_{k})$. It is not hard to show that $\t{x}_{k}$ is a lift of $x$ and $\t{y}_{k}$ is a lift of $y$. Since $u(s,t)$ converges uniformly to $x,y$ as $s\to \pm\infty$ one can extend it to a continuous map $\bar{\R} \times S^{1}\rightarrow \Sigma$ where $\bar{\R} = \R \cup\{\pm\infty\}$ and so one can show that the asymptotics $\t{x}_{k},\t{y}_{k}$ for the $k$-th lift $\t{u}_{k}$ of $u$ are all pairwise distinct, hence, for each lift $\t{x}_{k}$ of $x$ there is a unique lift $\t{y}_{k}$ of $y$ and a unique lift $\t{u}_{k}$ of $u$ so that the latter solves the Floer's equation for $(\t{H}_{n}, \t{J}_{n})$ and has asymptotics $\t{\phi}_{t}(\t{x}_{k})$ and $\t{\phi}_{t}(\t{y}_{k})$. On the other hand, every finite energy solution of the Floer's equation for $(\t{H}_{n}, \t{J}_{n})$ will be a solution of the Floer's equation for $(H,J)$ after composing it with $\pi_{n}$ and so it is one of the lifted solutions constructed above. We have
$$\begin{aligned}
\dim\mathcal{M}(\t{x}_{k}, \t{y}_{\ell}, \t{H}_{n}, \t{J}_{n}) &= \mathrm{CZ}(\t{x}_{k}) - \mathrm{CZ}(\t{y}_{\ell})\\
&= \mathrm{CZ}(x) - \mathrm{CZ}(y)=\dim\mathcal{M}(x, y, H, J),
\end{aligned}$$
for every lifts $\t{x}_{k}, \t{y}_{\ell}$ of $x,y\in \Fix_{c}(\phi)$, where $\mathrm{CZ}(x)$ is the Conley-Zehnder index of $x$; see \S\ref{sec:indices}. Let us check that $L_{n}$ is a chain map, i.e., $L_{n}d = \t{d}_{n}L_{n}$. Let $\t{x}_{1}, \dots, \t{x}_{n}$ be all lifts of $x\in \Fix_{c}(\phi)$ and suppose $dx = y_{1} + \dots + y_{\ell}$. First, one has
\begin{equation}\phantomsection\label{eq:lnd}
\begin{aligned}
L_{n}d(x) &= L_{n}\big(\sum_{i=1}^{\ell}y_{i}\big)=\sum_{i=1}^{\ell}\sum_{j=1}^{n}\t{y}_{ij},
\end{aligned}
\end{equation}
where $\t{y}_{ij}$ is the $j$-th lift of $y_{i}$. In the following, for a fixed point $y\in \Fix_{c}(\phi)$ let $\{\t{y}_{i}\}_{i=1}^{n}$ be all the lifts of $y$. On the other hand one has
\begin{equation}\phantomsection\label{eq:dnln}
\begin{aligned}
\t{d}_{n}L_{n}(x) &= \sum_{i=1}^{n}\t{d}_{n}(\t{x}_{i}) = \sum_{i=1}^{n}\sum_{y\in \Fix_{c}(\phi)}\sum_{j=1}^{n}\#\mathcal{M}(\t{x}_{i}, \t{y}_{j},\t{H}_{n},\t{J}_{n})\t{y}_{j}\\
&=\sum_{y\in \Fix_{c}(\phi)}\sum_{j=1}^{n}\big(\sum_{i=1}^{n}\#\mathcal{M}(\t{x}_{i}, \t{y}_{j},\t{H}_{n},\t{J}_{n})\big)\t{y}_{j}\\
&=\sum_{y\in \Fix_{c}(\phi)}\sum_{j=1}^{n}\#\mathcal{M}(x, y,H,J)\t{y}_{j} = \sum_{i=1}^{\ell}\sum_{j=1}^{n}\t{y}_{ij},
\end{aligned}
\end{equation}
which together with \eqref{eq:lnd} proves $L_{n}d = \t{d}_{n}L_{n}$. Next we show that the unit is mapped to the unit under the lift map $L_{n}$, i.e.,
\begin{equation}\phantomsection\label{eq:lnpss_eq_pssln}
L_{n}(\pss_{H}([\Sigma])) = \pss_{\t{H}_{n}}([\Sigma_{n}]);
\end{equation}
see \S\ref{sec:pss_isomorphisms} for the definitions of $\pss_{H}, \pss_{\t{H}_{n}}$. Pick a generic interpolation data $G: \Sigma\times \R \times \R \rightarrow \R$, in the sense of \S\ref{sec:pss_isomorphisms}, so that $G(x,t,s) = H(x,t)$ for $s\geq 1$ and $G(x,t,s) = f(x)$ for $s\leq 0$ where $f: \Sigma \rightarrow \R$ is a $C^{2}$-small perfect Morse function with a unique maximum $\zeta\in \Sigma$. Fix a metric $g$ on $\Sigma$ and lift it to the metric $\t{g}_{n}$ on $\t{\Sigma}_{n}$. Let $\t{f}_{n} = f\circ \pi_{n}$ be the lift of $f$ to $\Sigma_{n}$ that is also a Morse function, and similarly let $\t{G}_{n} = G\circ \pi_{n}$. One can show that the data $(\t{G}_{n}, \t{J}_{n}, \t{f}_{n})$ is also generic in the sense of \S\ref{sec:pss_isomorphisms}. Take the representative $\zeta \in \mathrm{CM}(f,g)$ of the fundamental class $[\Sigma]$. Let $\t{\zeta}_{1}, \dots, \t{\zeta}_{n}$ be the lifts of $\zeta$ and observe that $\t{\zeta}_{1} + \dots + \t{\zeta}_{n}$ is a representative of $[\Sigma_{n}]$ in $\mathrm{CM}(\t{f}_{n},\t{g}_{n})$. We show that $L_{n}(\pss_{H}(\zeta)) = \pss_{\t{H}_{n}}(\t{\zeta}_{1} + \dots + \t{\zeta}_{n})$. First, note that,
$$\begin{aligned}
\dim\mathcal{M}(\t{\zeta}_{k}, \t{x}_{\ell}, \t{G}_{n}, \t{J}_{n}) &= \mathrm{ind}(\t{\zeta}_{k}) - n - \mathrm{CZ}(\t{x}_{\ell})\\
&= \mathrm{ind}(\zeta) - n - \mathrm{CZ}(x)=\dim\mathcal{M}(\zeta, x, G, J),
\end{aligned}$$
where $\t{\zeta}_{k}, \t{x}_{\ell}$ are any lifts of $\zeta,x$ respectively, $\mathrm{ind}(-)$ is the Morse index and $\mathcal{M}(\zeta, x, G, J)$ is the moduli space of $\pss$-configurations as in \S\ref{sec:pss_isomorphisms}. Similar to the above analysis, one can show that every solution $u \in \mathcal{M}(\zeta, x, G, J)$ has exactly $n$ lifts, $$\t{u}_{1}, \dots, \t{u}_{n} \text{ with asymptotics }\{\t{\zeta}_{k_{1}}, \t{x}_{\ell_{1}}\}, \dots, \{\t{\zeta}_{k_{n}}, \t{x}_{\ell_{n}}\},$$ 
so that $\{k_{j}\}_{i=1}^{n} = \{\ell_{j}\}_{i=1}^{n} = \{1,\dots, n\}$ and for every $j$ one has, $$\t{u}_{j}\in \mathcal{M}(\t{\zeta}_{k_{j}}, \t{x}_{\ell_{j}}, \t{G}_{n}, \t{J}_{n}).$$ 
Write $\pss_{H}(\zeta) = x_{1} + \dots + x_{\ell}$ where $x_{j}\in \Fix_{c}(\phi)$ for $1\leq j\leq \ell$. Then,
$$L_{n}(\pss_{H}(\zeta)) = \sum_{i=1}^{\ell}\sum_{j=1}^{n}\t{x}_{ij},$$
where $\t{x}_{ij}$ is the $j$-th lift of $x_{i}$. On the other hand we have
$$\begin{aligned}
\pss_{\t{H}_{n}}(\sum_{i=1}^{n}\t{\zeta}_{i}) &= \sum_{i=1}^{n}\sum_{\t{x}\in \Fix_{c}(\t{\phi})}\#\mathcal{M}(\t{\zeta}_{i}, \t{x}, \t{G}_{n}, \t{J}_{n})\t{x}\\
&=\sum_{i=1}^{n}\sum_{x\in \Fix_{c}(\phi)}\sum_{j=1}^{n}\#\mathcal{M}(\t{\zeta}_{i}, \t{x}_{j}, \t{G}_{n}, \t{J}_{n})\t{x}_{j}\\
&=\sum_{x\in \Fix_{c}(\phi)}\sum_{j=1}^{n}\big(\sum_{i=1}^{n}\#\mathcal{M}(\t{\zeta}_{i}, \t{x}_{j}, \t{G}_{n}, \t{J}_{n})\big)\t{x}_{j}\\
&=\sum_{x\in \Fix_{c}(\phi)}\sum_{j=1}^{n}\#\mathcal{M}(\zeta, x, G, J)\t{x}_{j} = \sum_{i=1}^{\ell}\sum_{j=1}^{n}\t{x}_{j},
\end{aligned}$$
which finishes the proof of \eqref{eq:lnpss_eq_pssln}. As a corollary we have the following result where $\t{\phi}$ is the lift of $\phi$ given by a choice of a Hamiltonian isotopy $\phi_{t}\in \Ham(\Sigma)$ with $\phi_{1} = \phi$; note that this does not depend on the choice of the isotopy because $\pi_{1}(\Ham(\Sigma)) = 0$.
\begin{corollary}\phantomsection\label{cor:spec_norm_covering_spc}
Let $\phi \in \Ham(\Sigma, \omega)$ be a Hamiltonian diffeomorphism and $\t{\phi}_{n}\in \Ham(\Sigma_{n}, \t{\omega}_{n})$ be the lift of $\phi$ given by a Hamiltonian isotopy. Then, one has $\gamma(\t{\phi}_{n})\leq \gamma(\phi)$.
\end{corollary}
\begin{proof}
Let $H: \Sigma \times \R \rightarrow \R$ be a Hamiltonian function generating $\phi$ and let $\t{H}_{n} = H\circ \pi_{n}$ be the lift of $H$.  Pick a representative $\xi\in \CF(H,J)$ of $\pss_{H}([\Sigma])$ so that $\mathcal{A}_{H}(\xi) \leq c([\Sigma], H) + \eps$ for some arbitrary small $\eps > 0$. Then, as was shown above the chain $L_{n}(\xi) \in \CF(\t{H}_{n}, \t{J}_{n})$ is a representative of $\pss_{\t{H}_{n}}([\Sigma_{n}])$ and clearly one has $\mathcal{A}_{\t{H}_{n}}(L_{n}(\xi)) = \mathcal{A}_{H}(\xi)$. Hence, $$c([\Sigma], \t{H}_{n}) \leq \mathcal{A}_{\t{H}_{n}}(L_{n}(\xi)) = \mathcal{A}_{H}(\xi)\leq c([\Sigma], H) + \eps.$$
By taking $\eps \to 0$ one shows that $c([\Sigma_{n}], \t{H}_{n}) \leq c([\Sigma], H)$.
The result follows via a similar analysis for $\bar{H}$ and obtaining $c([\Sigma_{n}], \t{\bar{H}}_{n}) \leq c([\Sigma], \bar{H})$.
\end{proof}

\subsubsection{Proof of Theorem \ref{mainthm:ess_smth_inv_ann} (genus \texorpdfstring{$1$}{})}\phantomsection\label{sec:proof_of_ess_smth_case}
We first write the proof in the case of the torus $T^{2} = \R^{2}/\Z^{2}$ as it is easier to visualize. The proof for higher genus surfaces is similar and will be discussed in \S\ref{sec:proof_high_gen_essent}. Let $\phi\in \Ham(T^{2})$ and $A\subset T^{2}$ be as in Theorem \ref{mainthm:ess_smth_inv_ann}. 

Step $1$. We define a type of rotation number for the points of $A$ which uses the universal cover of the surface $T^{2}$ and we show that it coincides with the standard rotation number function $\rho_{\t{\phi}}$ from \S\ref{sec:rot_numb_ann} where $\t{\phi}_{c}$ is a lift of $(\phi|_{A})_{c}$ to the universal cover $\t{A}_{c}$ of $A_{c}$. Pick a Hamiltonian function $H: T^{2} \times \R \rightarrow \R$ generating the Hamiltonian flow $\{\phi_{t}\}$ with $\phi_{1} = \phi$. Let $\R^{2}$ be the universal cover of $T^{2}$, $\t{A}$ be a connected component of $\pi^{-1}(A)$ which is a copy of the universal cover of $A$, $\FD_{A}$ be a fundamental domain for $\t{A}$, and finally define $\tau\in \Z\oplus \Z$ be the unique element so that $\t{A} = \sqcup_{\ell\in \Z}\tau^{\ell}(\FD_{A})$. Consider the lift $\{\t{\phi}_{t}\}$ of the isotopy $\{\phi_{t}\}$ where the starting points are chosen in the fundamental domain $\mathrm{FD} = [0,1] \times [0,1]$. The isotopy $\{\t{\phi}_{t}\}$ is generated by the Hamiltonian $\t{H} = H \circ \pi$ via the Hamilton's equation $\t{\omega}(\t{X}, - ) = -d\t{H}$ where $\pi: \R^{2} \rightarrow T^{2}$ is the covering map and $\t{\omega} = \pi^{*}\omega$. Since the flux of the isotopy $\{\phi_{t}\}$ vanishes, for every point $\t{x}\in \t{A}$ one has $\t{\phi}(\t{x})\in \t{A}$ where $\t{\phi} = \t{\phi}_{1}$. Let $A_{c}$ be the prime end compactification of $A$. To define the rotation number for points of $A_{c}$ relative to $\t{\phi}|_{\t{A}}:\t{A} \rightarrow \t{A}$ we consider the prime extensions $(\phi|_{A})_{c}:A_{c} \rightarrow A_{c}$ and $(\t{\phi}|_{\t{A}})_{c}: \t{A}_{c} \rightarrow \t{A}_{c}$ to the prime end compactifications and observe that $(\t{\phi}|_{\t{A}})_{c}$ is a lift of $(\phi|_{A})_{c}$. Then, we define the rotation function $\rho_{(\t{\phi}|_{\t{A}})_{c}}: A_{c} \rightarrow \R$ using the lift $(\t{\phi}|_{A})_{c}$ by
\begin{equation}\phantomsection\label{eq:rot_func_new}
\rho_{(\t{\phi}|_{\t{A}})_{c}}(x) = \lim_{k\to \infty} \frac{\nu((\t{\phi}|_{\t{A}})_{c}^{k}(\t{x}))}{k},
\end{equation}
where for a point $\t{y}\in \t{A}$ we define $\nu(\t{y})$ to be the unique integer $\ell$ so that $\tau_{c}^{-\ell}(\t{y})\in \FD_{A,c}$ where $\tau_{c}$ is the prime extension of $\tau$ to $\t{A}_{c}$. One can show that there is a homeomorphism $\Gamma_{c}: [0,1]\times \R \rightarrow \t{A}_{c}$ so that $\Gamma_{c} \circ T = \tau_{c} \circ \Gamma_{c}$ where $T(\t{x}) = \t{x} + (0,1)$, and using which, on observes that $\Gamma_{c}^{-1}(\t{\phi}|_{\t{A}})_{c}\Gamma_{c}$ is a lift of $(\phi|_{A})_{c}$ to the universal cover $[0,1]\times \R$. It is easy to check that $\rho_{(\t{\phi}|_{\t{A}})_{c}}(\Gamma_{c}(\t{x})) = \rho_{\Gamma_{c}^{-1}(\t{\phi}|_{\t{A}})_{c}\Gamma_{c}}(\t{x})$ for all points $\t{x}\in [0,1]\times \R$ for which both sides are defined; the latter is the standard rotation number, see Lemma \ref{lem:rotation_number_conjugation} and \S\ref{sec:rot_numb_ann}. This in particular implies that the rotation function $\rho_{(\t{\phi}|_{\t{A}})_{c}}$ defined by \eqref{eq:rot_func_new} is well-defined everywhere and it is continuous since so is the standard rotation function $\rho_{\Gamma_{c}^{-1}(\t{\phi}|_{\t{A}})_{c}\Gamma_{c}}$; see Remark \ref{rem:rot_numb_of_zero_ent_on_closed_ann}.

Step $2$. Let us reduce the problem to the case where the rotation numbers of points of $A$ are either all positive or all negative. For simplicity let us denote $\Gamma_{c}^{-1}(\t{\phi}|_{\t{A}})_{c}\Gamma_{c}$ by $\t{\phi}_{c}$. We know that the rotation interval $\rho_{\t{\phi}_{c}}(A_{c})$ is non-trivial. Without loss of generality assume that there is a positive rational number $p/q$ in the rotation interval with $q$ being a prime number, other wise pick a negative number and proceed. Let $x\in A$ be a periodic point with rotation number $p/q$; see \cite[Theorem 2.3]{franks-handel-zero-ent}. Let $N$ be a large integer so that $q \nmid N$ and let $V\in \mathcal{A}(\phi,N)$ be an $\phi^{N}$-invariant annulus so that $V\subset A$ and $x\in V$. Then, it follows from \cite[Lemma 15.7 (2)]{franks-handel-zero-ent} that $V$ is essential in $A$, hence it is $\phi$-invariant, and has non-trivial rotation interval because it contains $x$ and some $N$-periodic orbit on its frontiers; here we used Lemma \ref{lem:rot_num_int_ann_vs_ext_ann}. Moreover, for large enough $N$ all the rotation numbers are positive and near $p/q$. Indeed, for every two points $y,z\in V$ one has,
$$|\rho_{\t{\phi}_{c}}(y) - \rho_{\t{\phi}_{c}}(z)| \leq 1/N,$$
where this time $\t{\phi}_{c}$ refers to a lift of $(\phi|_{V})_{c}$ to $\t{V}_{c}\subset \t{A}_{c}$ where $\t{V}_{c}$ is the prime end compactification of $\t{V} = \pi^{-1}(V)\cap \t{A}$, however, due to Lemma \ref{lem:rot_num_int_ann_vs_ext_ann} and continuity of rotation numbers we freely use the same notation $\rho_{\t{\phi}_{c}}$; the inequality holds since otherwise one can apply \cite{franks1988generalizations, franks-recurrence-88}, see \cite[Proposition 2.6]{franks-handel-zero-ent}, to conclude the existence of a fixed point of $\phi^{N}$ inside $V$ which is a contradiction; $\phi^{N}|_{V}$ is isotopic to the identity by Remark \ref{rem:rest_are_isotop_to_id}. See also \cite[Lemma 3.6]{acgps-26} for a similar argument.

Step $3$. So far we have an $\phi$-invariant open annulus $V$ so that the rotation interval $\rho(V_{c})$ contains only positive numbers where $\rho := \rho_{(\t{\phi}|_{\t{V}})_{c}}$. Define $\FD_{V} := \t{V} \cap \FD_{A}$ and $\FD_{V,\ell} = \sqcup_{j=0}^{\ell-1}\ \tau^{j}(\FD_{V})$. For every integer $\ell>0$ pick a large enough number $k$ and $N_{k,\ell}$ so that, $$\cup_{t\in[0,1]}\t{\phi}_{t}^{k}(\FD_{V,\ell}) \subset Q_{N_{k,\ell}} := [-N_{k,\ell}, N_{k,\ell}]^{2},$$
and $\t{\phi}_{1}^{k}(\FD_{V,\ell}) \cap \FD_{V,\ell} = \emptyset$; the latter is possible because the rotation numbers of all points in $\FD_{V}$ are positive and rotation numbers of points in $\tau^{j}(\FD_{V})$ are equal to the rotation numbers of their corresponding points in $\FD_{V}$. Let $T^{2}_{k,\ell} = \R^2/(2N_{k,\ell}\Z^{2})$ where the fundamental domain is taken to be $Q_{N_{k,\ell}}$. Let $\t{\phi}_{k,\ell}\in \Ham(T^{2}_{k,\ell})$ be the lift of $\phi^{k}$ to $T^{2}_{k,\ell}$ given by the isotopy $\{\phi_{t}^{k}\}$. By construction and the energy capacity inequality, see Remark \ref{rem:spec_cap_of_disks_in_high_genus_surf}, for every $\ell > 0$ and every large enough $k$ depending on $\ell$ we have $\gamma(\t{\phi}_{k,\ell}) \geq \mathrm{area}(\FD_{V,\ell})$ which together with Corollary \ref{cor:spec_norm_covering_spc} implies $\gamma(\phi^{k}) \geq \gamma(\t{\phi}_{k,\ell}) \geq \mathrm{area}(\FD_{V,\ell})$ for every fixed $\ell$ and large enough $k = k(\ell)$. Hence, we have $\liminf_{k\geq 1} \gamma(\phi^{k}) = \infty,$ which finishes the proof of the Theorem \ref{mainthm:ess_smth_inv_ann} in the case of the torus $T^{2}$.

\subsubsection{Proof of Theorem \ref{mainthm:ess_smth_inv_ann} (genus \texorpdfstring{$\geq 2$}{})}\phantomsection\label{sec:proof_high_gen_essent} 
Let $\Sigma$ be a closed oriented surface of genus at least $2$. The proof in this case is similar to the proof presented in \S\ref{sec:proof_of_ess_smth_case}. Let $A\subset \Sigma$ be an essential annulus. Let $V \Subset A$ be the annulus found in Step $2$ of \S\ref{sec:proof_of_ess_smth_case} so that the rotation interval $\rho(V_{c})$ contains only positive numbers (or only negative numbers) where $\rho := \rho_{(\t{\phi}|_{\t{V}})_{c}}$ and $(\t{\phi}|_{\t{V}})_{c}$ is defined as in Step $1$ of \S\ref{sec:proof_of_ess_smth_case}. Define $\tau$ and $\FD_{V}$ as in Step $3$ of \S\ref{sec:proof_of_ess_smth_case} and let $\FD_{V,\ell} = \sqcup_{j=1}^{\ell-1}\tau^{j}(\FD_{V})$. Let $k >0$ be a sufficiently large number so that $\t{\phi}^{k}(\FD_{V,\ell}) \cap \FD_{V,\ell} = \emptyset$ which is possible due to the fact that all points of $V$ have positive rotation numbers. 

We now discuss the only step that is different from \S\ref{sec:proof_of_ess_smth_case} which is how to construct a finite covering $\Sigma_{N} \rightarrow \Sigma$ so that the corresponding covering map $\t{\Sigma} \rightarrow \Sigma_{N}$ from the universal cover $\t{\Sigma}$ of $\Sigma$ is injective on the set
$$K_{k,\ell}:=\cup_{t\in[0,1]}\t{\phi}_{t}^{k}(\FD_{V,\ell}).$$

Step $*$. Let $\pi_{1}(\Sigma) = \langle x_{1},y_{1}, \dots, x_{r}, y_{r} | \Pi_{j=1}^{r}[x_{j},y_{j}] =1\rangle$ be the fundamental group of $\Sigma$ where $r$ is the genus of $\Sigma$. Given a compact subset $K \subset \t{\Sigma}$ we would like to find a finite covering $\Sigma_{N}\rightarrow \Sigma$ so that the covering projection $\t{\Sigma}\rightarrow \Sigma_{N}$ is injective on $K$. This follows from the following claim where $\mathrm{wl}(\gamma)$ is the word length of $\gamma$ with respect to the generating set $\{x^{\pm 1}_i, y^{\pm 1}_i\}_{1 \leq i \leq r}$: 

Claim. For every integer $N >0$ there exists a finite index normal subgroup $G_{N}\subset \pi_{1}(\Sigma)$ that does not contain any element $\gamma \in \pi_{1}(\Sigma)$ with $\mathrm{wl}(\gamma) \leq N$.

Proof of Claim. A group $G$ is called residually finite if for every element $g\in G$ there is a normal subgroup $G'\subset G$ with finite index so that $g\not\in G'$. It is known that the fundamental group of compact orientable surfaces is residually finite, see \cite{hempel-72} for a very short and beautiful proof. The number of elements $\gamma \in \pi_{1}(\Sigma)$ with word length $\mathrm{wl}(\gamma)\leq N$ is bounded, let 
$$\mathrm{wl}^{-1}([0,N]) = \{\gamma_{1}, \dots, \gamma_{L}\},$$ 
be all the elements with $\mathrm{wl}\leq N$. Let $H_{1}, \dots, H_{L}$ be the normal subgroups with finite index so that $\gamma_{i}\not\in H_{i}$. Consider the map
$$\pi_{1}\oplus \dots \oplus \pi_{L}: \pi_{1}(\Sigma) \rightarrow \frac{\pi_{1}(\Sigma)}{H_{1}} \oplus \dots \oplus \frac{\pi_{1}(\Sigma)}{H_{L}},$$
where $\pi_{i}: \pi_{1}(\Sigma)\rightarrow \pi_{1}(\Sigma)/H_{i}$ is the quotient map for $i = 1,\dots, L$. Define 
$$G_{N}:= \ker(\pi_{1}\oplus \dots \oplus \pi_{L}).$$
Then, it is clear that $G_{N}$ is a normal subgroup of $\pi_{1}(\Sigma)$ with finite index and $\gamma_{i}\not\in G_{N}$ for all $i=1,\dots, L$. \qed 

We continue Step $*$. Let $N_{K}$ be a large enough number so that 
$$K \subset \sqcup_{\mathrm{wl}(\gamma)\leq N_{K}} \gamma(\FD),$$
where $\FD\subset \t{\Sigma}$ is a fundamental domain. Define $G_{K} := G_{N_{K}}$ where the latter is given by the Claim above. Let $\Sigma_{K} := \t{\Sigma}/G_{K}$. Then, it is clear that $\Sigma_{K}\rightarrow \Sigma$ is a finite covering, of degree $[\pi_{1}(\Sigma) : G_{K}]$, hence $\Sigma_{K}$ is a closed orientable surface of genus $[\pi_{1}(\Sigma) : G_{K}](r-1) + 1$, and it is clear that the projection map $\t{\Sigma}\rightarrow \Sigma_{K}$ is injective on $K$ which finishes Step $*$. 

The rest of the proof is similar to Step $3$ of \S\ref{sec:proof_of_ess_smth_case}. Indeed, let $\Sigma_{k,\ell}$ be the finite covering $\Sigma_{K}$ for $K = K_{k,\ell}$ given by the claim above and let $\t{\phi}_{k,\ell}$ be the lift of $\phi$ to $\Sigma_{k,\ell}$ obtained by the isotopy $\{\phi_{t}^{k}\}$. By the energy-capacity inequality, the choice of $\Sigma_{k,\ell}$, and Corollary \ref{cor:spec_norm_covering_spc} we have
$$\gamma(\phi^{k}) \geq \gamma(\t{\phi}_{k,\ell}) \geq \mathrm{area}(\FD_{V,\ell}), \forall \ell>0 \text{ and } k = k(\ell) \text{ large enough}.$$
Therefore, choosing $\ell$ arbitrarily large we conclude $\liminf_{k\geq 1}\gamma(\phi^{k}) = \infty.$

\section{On rigidity of autonomous Hamiltonians in higher dimensions}\phantomsection\label{sec:autonomous_ham_higher_dim} 
This section studies the problem of spectral rigidity of Hamiltonian diffeomorphisms in higher dimensions. In particular, we make some progress in the case of autonomous Hamiltonian diffeomorphisms on e.g. product of surfaces. The main theorem of this section is Theorem \ref{mainthm:auto_on_symp_hyperb} below. 

Let us first introduce the setting. Let $(M,\omega)$ be a closed symplectic manifold and let $\pi: \t{M} \rightarrow M$ be the universal covering space equipped with the form $\t{\omega} = \pi^{*}\omega$. Pick a Riemannian metric $g$ on $M$ and let $\t{g}$ be the lifted metric on $\t{M}$. We say that $(M,\omega)$ is \emph{symplectically hyperbolic} if $\t{\omega} = d\t{\lambda}$ for some smooth one form $\t{\lambda}\in \Omega^{1}(\t{M},\R)$ with $|\t{\lambda}|$ bounded, i.e., there exists a constant $C_{M} > 0$ so that $|\t{\lambda}(\xi)| \leq C_{M}|\xi|$ for all $\xi\in T\t{M}$ where the norms are with respect to $\t{g}$. See the work of Polterovich \cite{Polterovich-02}. 

One particular important property of symplectically hyperbolic manifolds is that they are symplectically atoroidal and hence aspherical, $\omega$ vanishes on tori and in particular on spheres inside $M$, this follows from the combination of the following facts: if $\omega$ is symplectically hyperbolic then so is $\omega + df$ for every smooth function $f$, pull back of a symplectically hyperbolic form is also symplectically hyperbolic, and that $S^{2}$ and $T^{2}$ with their standard forms are not symplectically hyperbolic. Symplectic hyperbolicity is equivalent to $(M,\omega)$ satisfying a symplectic isoperimetric inequality property;
$$\exists C_{M}>0 \text{ s.t. }\forall f:\bar\D \rightarrow M; \abs{\int_{\D}f^{*}\omega} \leq C_{M}\ell(f|_{\partial\bar\D}),$$
where $\ell$ is the length functional; see the work of Sikorav \cite{sikorav-01} and \cite{MR2547825}; the above isoperimetric inequality is the key to the proof of Theorem \ref{mainthm:auto_on_symp_hyperb}. 

Examples of such manifolds are closed surfaces of genus at least $2$, since the symplectic form $\t{\omega} = dx\wedge dy/y^{2}$ on the upper half plane $\mathbb{H}$ has the bounded primitive $\t{\lambda} = dx/y$ where the metric is $g = (dx^{2} + dy^{2})/y^{2}$. Every symplectic surface bundle $E \rightarrow B$ with fiber being a closed surface of genus at least $2$ is symplectically hyperbolic. Any symplectic manifold $(M,\omega)$ that is symplectically aspherical and $\pi_{1}(M)$ is hyperbolic, in particular $M$ admits a metric of negative sectional curvature, is symplectically hyperbolic; see \cite[Corollary 1.12]{MR2547825}. Symplectic hyperbolicity is preserved under products. Not every atoroidal manifold is symplectically hyperbolic; see \cite[Corollary 4.3]{MR4696270}.

\begin{theorem}\phantomsection\label{mainthm:auto_on_symp_hyperb}
Let $(M,\omega)$ be a closed symplectically hyperbolic manifold and let $(\phi^t)_{t \in \R}$ be an autonomous Hamiltonian flow. Then there exists a constant $\delta > 0$ so that $\liminf_{t\geq 1}\gamma(\phi^{t}) \geq \delta$.
\end{theorem}

\begin{proof}
Let $H: M \rightarrow \R$ be an autonomous Hamiltonian generating the Hamiltonian diffeomorphism $\phi$. By to shifting $H$ by a constant assume that
$$m := \max H = -\min H > 0.$$ 
Suppose that there is a sequence of numbers $\{n_{k}\}_{k\geq 1},$ $n_k \in [1,\infty),$ with $n_{k}\to \infty$ and $\gamma(\phi^{n_{k}}) \to 0$ as $k\to \infty$. This implies, without loss of generality, that $c([M], n_{k}H) \leq c_{0}$ for some fixed constant $c_{0} >0$ and all large enough $k$; here we have used the fact that symplectic hyperbolic manifolds are symplectically aspherical, and so the spectral norm $\gamma(H) = \gamma(\phi)$ for every Hamiltonian $H$ with time-one map $\phi$. Let $\pi: \t{M} \rightarrow M$ be the universal covering of $M$ and $\t{\omega} = \pi^{*}\omega$. Pick a bounded primitive $\t{\lambda}$ of $\t{\omega}$ that exists by the definition of symplectic hyperbolicity. Fix a metric $g$ on $M$ and define $\widetilde{g}=\pi^{*}g$. Let $\Sigma := H^{-1}(m)$ be the maximum set of $H$ and define 
$$\Sigma_{\delta} = H^{-1}((m -\delta, m]),$$ 
to be a small neighborhood of $\Sigma$. Let $\gamma$ be a $T$-periodic Hamiltonian orbit inside $\Sigma_{\delta}$ and $u$ be a capping of $\gamma$. Let $\widetilde{\gamma}$ and $\widetilde{u}$ be some lifts of $\gamma$ and $u$ to the universal cover $\widetilde{M}$ of $M$. Then, it is easy to check that $\widetilde{\gamma}$ is an $T$-periodic Hamiltonian orbit of $\widetilde{H} = H\circ \pi$. We estimate the action of $\widetilde{\gamma}$ as follows. First, letting $X_{H}$ be the Hamiltonian vector field of $H$ one has $|X_{H}|_{C^{0}(\Sigma_{\delta})} \leq C_{H} \delta$ for some constant $C_{H}$ depending on $H$ and the metric $g$, and $|\widetilde{\lambda}(\xi)| \leq C_{M}|\xi|_{\widetilde{g}}$ for some constant $C_{M}$ depending only on $M$. Then,
$$\begin{aligned}
\mathcal{A}_{H}(\gamma) = \mathcal{A}_{\widetilde{H}}(\widetilde{\gamma}) &= \int_{0}^{T} \widetilde{H}\circ \widetilde{\gamma}(t) dt - \widetilde{\lambda}(\dot{\widetilde{\gamma}}(t))dt\\
&\geq T(\min_{\Sigma_{\delta}} H - C_{M}|X_{H}|_{C^{0}(\Sigma_{\delta})})\\
&\geq T(m - \delta - C_{M}C_{H}\delta).
\end{aligned}$$
By the estimates above, the actions of $T$-periodic Hamiltonian orbits lying inside $\Sigma_{\delta}$, for a fixed small $\delta$, is higher than $c_{0}$ for $T = n_{k}$ and $k$ large enough, in particular their actions are strictly larger than $c([M], n_{k}H)\leq c_{0}$ for every large enough $k$. Let $\theta: \mathbb{R} \rightarrow \mathbb{R}$ be a smooth function so that it is identity on $\{x \leq m -\delta\}$ and it is the constant $m-\delta/2$ on $[m-\delta/2, m]$ and $0 \leq \theta'\leq 1$. Consider the linear interpolation $\theta_{r}(x)= (1-r)x + r\theta$ for $r\in [0,1]$. Similar estimates as above show that for every Hamiltonian $T$-periodic Hamiltonian orbit $\gamma$ of the autonomous function $\theta_{r}\circ H$ inside $\Sigma_{\delta}$ one has,
$$\mathcal{A}_{\theta_{r} \circ H}(\gamma) \geq T(m - \delta - C_{M}C_{\theta_{r}\circ H}\delta).$$ 
Fix a small symplectic ball $B \subset \Sigma_{\delta/2}$. Let $F$ be a radial Hamiltonian supported inside $B$ so that it displaces a ball of capacity $a >0$ for some positive constant $a$ which depends only on $H$ and $\delta$. Define the following parametrized Hamiltonian,
$$G_{T,r,s} = (\theta_{r} \circ TH) \# sF, r,s\in [0,1].$$ 
Let $T = n_{k}$ for some large enough $k$. First fix $s = 0$ and vary $r\in [0,1]$. Since $c([\mathrm{pt}], n_{k}H)\leq c([M], n_{k}H) \leq c_{0}$ for all large enough $k$ and the actions of each $1$-periodic Hamiltonian orbits of $n_{k}H$ inside $\Sigma_{\delta}$, where the function $\theta_{r}\circ n_{k}H$ is deforming, are strictly larger than $c_{0}$, we conclude that the action spectrum $\mathrm{spec}(\theta_{r}\circ n_{k}H)$ changes only away from the spectral invariants $c([\mathrm{pt}], n_{k}H)$ and $c([M], n_{k}H)$ as $r$ varies and so by continuity of the spectral invariants, 
$$c([\mathrm{pt}], n_{k}H) = c([\mathrm{pt}], \theta \circ n_{k}H), c([M], n_{k}H) = c([M], \theta \circ n_{k}H), \forall k\gg 1.$$
Now fix $r =1$ and let $s$ vary in the interval $[0,1]$. As $s$ varies the spectrum of $G_{n_{k},1,s}$ changes only above $n_{k}(m - \delta) - c(B)$ which is strictly larger than $c_{0}$ for large enough $k$; this is because the only new periodic orbits born as $s$ varies from $0$ to $1$ are the orbits inside $B$ whose actions are close to the value of $\theta\circ n_{k}H$ on $\Sigma_{\delta/2}$ minus the capacity $c(B)$ of $B$. Hence, by continuity of spectral invariants,
$$\begin{aligned}
c([\mathrm{pt}], \theta\circ n_{k}H) &= c([\mathrm{pt}], (\theta\circ n_{k}H)\# F),\\
c([M], \theta\circ n_{k}H) &= c([M], (\theta\circ n_{k}H)\# F), \forall k\gg 1.
\end{aligned}$$
In particular we have $\gamma((\theta\circ n_{k}H)\# F)= \gamma(n_{k}H)$ for all $k\gg 1$; here we used the fact that $M$ is symplectically aspherical and so by Poincaré duality it holds that $c([M],-H) = -c([\mathrm{pt}],H)$; see \S\ref{sec:spectral_invariants}. However, by energy-capacity inequality,
$$a \leq \gamma((\theta\circ n_{k}H)\# F)= \gamma(n_{k}H), \forall k\gg 1,$$
which is a contradiction as the right hand side converges to zero as $k\to \infty$.
\end{proof}
The following is an immediate corollary of Theorem \ref{mainthm:auto_on_symp_hyperb} since symplectic hyperbolicity is preserved under products and surfaces of genus at least $2$ are symplectic hyperbolic.
\begin{corollary}
Let $id \neq \phi\in \Ham(\Sigma_{1}\times \dots \times \Sigma_{\ell})$ be an autonomous Hamiltonian diffeomorphism where $\Sigma_{j}$ is a closed oriented surface of genus at least $2$ for all $j=1,\dots, \ell$. Then, there exists a constant $\delta(\phi)>0$ so that $\liminf_{k\geq 1}\gamma(\phi^{k}) \geq \delta(\phi)> 0$.
\end{corollary}

\subsection{Non-autonomous somewhere flat Hamiltonians}\phantomsection\label{sec:non_aut_some_flat}
The following is a result on the spectral rigidity of Hofer-Zehnder type Hamiltonian diffeomorphisms.
\begin{prop}\phantomsection\label{lem:just_const}
Let $(M,\omega)$ be any closed symplectically aspherical manifold. Let $U$ and $V$ be two disjoint open subsets of $M$. Suppose $H: M\times \R \rightarrow \R$ be a Hamiltonian that satisfies,
$$H|_{U\times \R} \equiv c_{1}, H|_{V\times \R}\equiv c_{2}, \text{ for some constants } c_{1} \neq c_{2}.$$ 
Then, there exists a constant $\delta >0$ so that $\liminf_{k\geq 1}\gamma(\phi_{H}^{k}) \geq \delta >0$.
\end{prop}
\begin{proof}
Assume the contrary so that $\liminf_{k\geq 1}\gamma(\phi_{H}^{k}) = 0$. Up to shifting $H$ by a constant assume that $c_{1} = -c_{2} = c > 0$. Let $\{n_{k}\}_{k\geq 1}$ be a sequence of number so that $n_{k} \to \infty$ and $\gamma(\phi_{H}^{n_{k}})\to 0$ as $k\to \infty$. Up to passing to a subsequence and without loss of generality assume that
$$c([M], H^{(n_{k})}) \leq c_{0}, \forall k\gg1,$$
where $c_{0} >0$ is some constant. Let $B\subset U$ be a small fixed symplectic ball and let $F$ be a small radial map whose time-one map displaces a ball of capacity $a>0$ where $a$ is some independent constant. Consider the interpolating Hamiltonian $G_{k,s} = H^{(n_{k})}\#sF$. Since $H^{(n_{k})}|_{U\times \R} = n_{k}c > c_{0}$ for large enough $k$, we conclude that the new orbits appearing in the ball $B$ as $s$ varies from $0$ to $1$ have actions bigger than $c([\mathrm{pt}], H^{(n_{k})})$ and $ c([M], H^{(n_{k})})$ and hence by continuity of spectral invariant,
$$c([\mathrm{pt}], H^{(n_{k})}) = c([\mathrm{pt}], H^{(n_{k})}\#F), c([M], H^{(n_{k})}) = c([M], H^{(n_{k})}\#F),$$
which implies that $a \leq \gamma(H^{(n_{k})}\#F) = \gamma(H^{(n_{k})})$ that is a contradiction for large enough $k$ as the right hand side converges to zero.
\end{proof}
\subsection{Completely integrable Hamiltonian systems}\phantomsection\label{sec:comp_integ_ham_sys}
A Hamiltonian system $$(M,\omega, H)$$ consisting of a closed symplectic manifold and a Hamiltonian function $H: M \rightarrow \R$ is called completely integrable if it admits $n = \dim M/2$ commuting integrals of motion $f_{1} = H, f_{2}, \dots, f_{n}: M \rightarrow \R$ that are independent in some open dense subset of $M$; commuting means that they pairwise commute with respect to the Poisson bracket, $\{f_{i},f_{j}\} = 0$ for $i,j\in \{1,2,\dots, n\}$. For a completely integrable system $(M,\omega, H, \{f_{i}\}_{i=1}^{n})$ define $F: M \rightarrow \R^{n}$ by $F = (f_{1}, \dots, f_{n})$; see \cite{da2003symplectic, MR4605050} for details on integrable systems. For a regular value $c\in \R^{n}$ of $F$ the level set $F^{-1}(c)$ is a Lagrangian submanifold of $M$. A rich collection of examples are the symplectic toric manifolds that are shortly discussed in \S\ref{sec:symp_toric_mnfld}. The goal in this section is to study the behavior of the spectral norm under iterations of the particular class of Hamiltonian functions of the form $h\circ F$ where $h: \R^{n} \rightarrow \R$ is any smooth function. We study this problem for monotone symplectic toric manifolds in \S\ref{sec:complex_proj_spc} and \S\ref{sec:ext_to_symp_toric}, and for more general completely Hamiltonian integrable systems in \S\ref{sec:more_gen_int_ham_sys}.

\subsubsection{Symplectic toric manifolds}\label{sec:symp_toric_mnfld}
Let $(M,\omega)$ be a symplectic manifold of dimension $2n$ and let  $\rho: T^{n} \rightarrow \Ham(M,\omega)$ be an effective Hamiltonian action on $M$; see \cite[\S1.4]{da2003symplectic}. Let $\mu: M \rightarrow \R^{n}$ be a moment map of the action $\rho$. We call $(M,\omega, \rho, \mu)$ a symplectic toric manifold. The set of symplectic toric manifolds of dimension $n$, up to an equivalence relation \cite[Definition 1.6.2]{da2003symplectic}, is in one-to-one correspondence with the set of Delzant's polytopes in $\R^{n}$ via the map that takes $(M,\omega, \rho, \mu)$ to $\mu(M)$; see \cite[\S2.1]{da2003symplectic} and \cite{MR984900}. In the following, for a given Delzant polytope $\Delta\subset \R^{n}$, we will sketch a construction of a symplectic toric manifold $(M_{\Delta}, \omega_{\Delta}, \rho_{\Delta}, \mu_{\Delta})$ due to Delzant so that $\mu_{\Delta}(M_{\Delta}) = \Delta$. This construction will be used to study spectral rigidity of Hamiltonian functions of the form $h\circ \mu_{\Delta}$ where $h: \R^{n}\rightarrow \R$ is any smooth function. Note that by Delzant's theorem \cite{MR984900} any symplectic toric manifold will be equivalent to those constructed below. 

Let $\Delta\subset \R^{n}$ be a Delzant polytope, i.e.,
\begin{enumerate}
\item each vertex is adjacent to $n$ edges,
\item there are outward normal vectors to each facet of $\Delta$ that are primitive and integral,
\item for each vertex $v$ the primitive integral vectors $e_{i}(v)$ pointing along the edges of $\Delta$ from $v$ form a $\Z$-basis for $\Z^{n}\subset \R^{n}$.
\end{enumerate}
A Delzant polytope is called \emph{monotone} with monotonicity constant $\kappa >0$ if in addition it satisfies,
\begin{enumerate}[resume]
\item the vertices of $\Delta$ are integral and there is a unique integral point $\kappa^{-1}\xi$ in the interior of $\kappa^{-1}\Delta$,
\item for each vertex $v$, $v + \kappa\sum_{i=1}^{n} e_{i}(v) = \xi$; indeed, $\xi = \frac{1}{n}\sum_{v\in V(\Delta)} v.$ 
\end{enumerate}
Let $d$ be the number of facets of $\Delta$ and $v_{1},\dots, v_{d}$ be the outward primitive and integral normal vectors to its facets. Define the projection $\pi: \R^{d}\rightarrow \R^{n}$ by $\pi(e_{i}) =v_{i}$ and let $\pi:T^{d} \rightarrow T^{n}$ be induced by $\pi$ where $T^{i} = \R^{i}/\Z^{i}, i=d,n$. Let $N = \ker\pi$ and $\iota:N\rightarrow T^{d}$ be the inclusion. Let $\sigma: T^{n}\rightarrow T^{d}$ be a section of $\pi$, i.e., $\pi \circ \sigma = id$; see \cite[\S2.5]{da2003symplectic} for the construction of $\sigma$. So the following short exact sequence splits
$$0\rightarrow N \overset{\iota}\rightarrow T^{d} \underset{\sigma}{\overset{\pi}\rightleftarrows} T^{n} \rightarrow 0,$$
and $T^{d} = \iota(N)\oplus \sigma(T^{n})$. One can describe the polytope $\Delta$ in the dual space,
$$\Delta = \{x\in (\R^{n})^{*} : \langle x, v_{i}\rangle \leq \lambda_{i}, 1\leq i\leq n\},$$
for some vector $\lambda = (\lambda_{1},\dots, \lambda_{n})$. Let $\Delta' := \pi^{*}(\Delta)\subset (\R^{d})^{*}$. Consider the action of $T^{d}$ on $\C^{d}$ defined by $(t_{1},\dots, t_{d})z = (e^{-2\pi it_{1}}z, \dots,e^{-2\pi it_{n}}z_{n})$. This is a Hamiltonian action with moment map
$$\phi: \C^{d} \rightarrow (\R^{d})^{*}; z\mapsto -(\pi \abs{z_{1}}^{2}, \dots, \pi \abs{z_{d}}^{2}) + \lambda.$$
Letting $\mathfrak{n}$ be the Lie algebra of $N$ we have $(\R^{d})^{*} = \mathfrak{n}^{*}\oplus (\R^{n})^{*}$ with projection maps $\iota^{*}$ and $\sigma^{*}$ induced by $\iota$ and $\sigma$. Let $Z = (\iota^{*}\circ\phi)^{-1}(0)$. One has
$$\phi(Z) = \Delta' = \pi^{*}(\Delta).$$
One can show that $N$ acts freely on $Z$ and $M_{\Delta} = Z/N$ is a smooth manifold that admits a symplectic form $\omega_{\Delta}$ so that $\pr^{*}\omega_{\Delta} = \iota^{*}\omega_{\mathrm{st}}$ where $\pr:Z \rightarrow Z/N$ is the projection map. One also observes that $\sigma(T^{n})$ acts on $Z$ and its moment map $\sigma^{*}\circ\phi$ is preserved under the action of $N$ and hence $\sigma(T^{n})$ acts Hamiltonian on $M_{\Delta}$ with moment map $\mu_{\Delta}:M_{\Delta} \rightarrow (\R^{n})^{*}$ which satisfies,
$$\mu_{\Delta}\circ \pr = \sigma^{*}\circ \phi \circ \iota;$$
see \cite[\S2.4(1)]{da2003symplectic}. Defining $\rho_{\Delta}$ to be the action of $\sigma(T^{n})$ on $M_{\Delta}$ we claim that $(M_{\Delta}, \omega_{\Delta}, \rho_{\Delta}, \mu_{\Delta})$ is the desired symplectic toric manifold. One last thing to check is,
$$\mu_{\Delta}(M_{\Delta}) = \sigma^{*}\circ\phi(\iota(Z))=\sigma^{*}\pi^{*}(\Delta) = \Delta.$$
We now describe the group $\pi_{2}(M_{\Delta}, L_{\eta})$ for a an interior point $\eta \in \Delta$ where $L_{\eta} = \mu_{\Delta}^{-1}(\eta)$; this is used in \S\ref{sec:ext_to_symp_toric}. First, it is easy to show that $L_{\eta} = L_{\pi^{*}\eta}/N$ where $L_{\pi^{*}\eta} = \phi^{-1}(\pi^{*}\eta)$. We have $\langle\pi^{*}\eta, e_{i}\rangle = \langle \eta, v_{i}\rangle.$ So,
$$L_{\pi^{*}\eta} = \{z\in Z : -\pi\abs{z_{i}}^{2}+\lambda_{i} = \langle \eta, v_{i}\rangle, 1\leq i\leq d\},$$
is an $d$-dimensional torus and each circle $S_{\pi^{*}\eta,i}:=\{-\pi \abs{z_{i}}^{2} + \lambda_{i}= \langle \eta,v_{i}\rangle, z_{j}=0; j\neq i\}$ inside $L_{\pi^{*}\eta}$ has a unique bounding disk $$D_{\pi^{*}\eta,i} = \{\langle \eta,v_{i}\rangle\leq -\pi \abs{z_{i}}^{2}+\lambda_{i} \leq \lambda_{i}, z_{j}=0; j\neq i\}\subset Z.$$
Hence, it follows that
$$\pi_{2}(Z, L_{\pi^{*}\eta}) = \Z\langle D_{\pi^{*}\eta,1},\dots, D_{\pi^{*}\eta,d}\rangle.$$
To compute the group $\pi_{2}(M_{\Delta}, L_{\eta})$ consider the long exact sequence 
$$\dots \rightarrow \pi_{2}(L_{\eta}) \rightarrow \pi_{2}(M_{\Delta})\overset{a}\rightarrow\pi_{2}(M_{\Delta},L_{\eta}) \overset{c}\rightarrow \pi_{1}(L_{\eta}) \overset{b}\rightarrow \pi_{1}(M) \rightarrow \dots$$
and observe that $\pi_{2}(M_{\Delta}, L_{\eta}) = \mathrm{ker}(c)\oplus \ker(b).$ By a diagram chasing argument involving homotopy long exact sequences for pairs $(L_{\pi^{*}\eta}, Z),(L_{\eta}, M_{\Delta})$ and the fibrations $N\rightarrow Z\rightarrow M_{\Delta}$ and $N\rightarrow L_{\pi^{*}\eta}\rightarrow L_{\eta}$ one shows that $b = 0$ and $\pi_{2}(L_{\eta}) = 0$ which implies $a$ is injective and so
$$\pi_{2}(M_{\Delta}, L_{\eta}) = \ker(c)\oplus \ker(b) =\pi_{2}(M_{\Delta})\oplus \pi_{1}(L_{\eta}),$$
where one also has $\pi_{1}(L_{\eta}) = \pi_{1}(L_{\pi^{*}\eta})/\pi_{1}(N)$. Denote by $D_{\eta,i}$ the projection of the disk $D_{\pi^{*}\eta, i}$ by the map $\pr:Z\rightarrow M_{\Delta}$ for $i=1,\dots, d$. It is clear that $\omega_{\Delta}(D_{\eta,i}) = \omega_{\mathrm{st}}(D_{\pi^{*}\eta,i}) = \lambda_{i} - \langle \eta, v_{i}\rangle$ given that $\pr^{*}\omega_{\Delta} = \iota^{*}\omega_{\mathrm{st}}$. The spheres in $\pi_{2}(M_{\Delta})\subset\pi_{2}(M_{\Delta},L_{\eta})$ are of the form $\sum a_{i}[D_{\pi^{*}\eta, i}]$ where $\sum a_{i}[S_{\eta,i}] = 0$ where $S_{\eta,i}$ is the projection of $S_{\pi^{*}\eta,i}$. So,
$$\begin{aligned}
\omega_{\Delta}(\pi_{2}(M_{\Delta}, L_{\eta})) &= \Z\langle \{\omega_{\Delta}(D_{\eta, i})\}_{i=1}^{d}\rangle\\
&= \Z\langle \{\omega_{\mathrm{st}}(D_{\pi^{*}\eta, i})\}_{i=1}^{d}\rangle,\\
&=\Z\langle \{\lambda_{i} - \langle\eta,v_{i}\rangle\}_{i=1}^{d}\rangle,
\end{aligned}$$
The number $\lambda_{i} - \langle \eta, v_{i}\rangle$ is the affine distance of $\eta$ to the facet of $\Delta$ corresponding to $v_{i}$.

Recall that $(M_{\Delta}, \omega_{\Delta})$ is called monotone if $[\omega_{\Delta}] = \kappa_{\Delta}c_{1}(TM_{\Delta}, \omega_{\Delta})$ for some constant $\kappa_{\Delta}>0$. It is shown in \cite[Proposition 1.10]{MR2507748} that $(M_{\Delta},\omega_{\Delta})$ is monotone if and only if $\Delta$ is monotone with monotonicity constant $\kappa_{\Delta}$ in the sense above; see also \cite[Lemma 3.3]{MR2768658}. In this case one easily checks that $\lambda_{i} - \langle \eta, v_{i}\rangle = \kappa_{\Delta}$ for all $1\leq i\leq d$. So, for a monotone symplectic toric manifold $(M_{\Delta}, \omega_{\Delta})$ and the integral point $\kappa_{\Delta}^{-1}\xi\in \kappa_{\Delta}^{-1}\Delta$ 
$$\omega_{\Delta}(\pi_{2}(M_{\Delta}, L_{\xi}))\subset \kappa_{\Delta}\Z.$$
It follows from the above discussion that $L_{\xi}$ is a monotone Lagrangian torus in $M_{\Delta}$. We will use this special fiber $L_{\xi}$ in the subsequent section to prove spectral rigidity of Hamiltonian functions of the form $h \circ \mu_{\Delta}$. One useful change of coordinate, in the monotone case, is to work with a basis of $\R^{n}$ that is consist of the primitive integral vectors $\{e_{1}(v), \dots, e_{n}(v)\}$ that point along the edges of $\Delta$ from a vertex $v$ and also shift $\Delta$ so that $v$ lies at the point $-(\kappa_{\Delta}, \dots, \kappa_{\Delta})$. In the new set of coordinates, for the shifted $\Delta$, we have $\lambda_{i} = \kappa_{\Delta}$ for all $1\leq i\leq d$, and for any point $\eta\in \Delta$ we have
$$\omega_{\Delta}(D_{\eta,i}) = \kappa_{\Delta} - \langle \eta,v_{i}\rangle, 1\leq i\leq d.$$
Moreover, by item $(\text{v})$ in the definition of a monotone polytope above, in the new coordinate one has $\xi = 0.$

\subsubsection{The case of the sphere}\label{sec:complex_proj_spc}
We first study the case of the two dimensional sphere to show the ideas without getting too technical and in \S\ref{sec:ext_to_symp_toric} we extend the results to all monotone symplectic toric manifolds. In the following let $S^{2}\subset \R^{3}$ be the standard sphere of radius $1/2$, shift up in $z$-axis by $+1/2$, and equipped with the symplectic form
$$\omega:= \frac{1}{\pi}\iota_{\nu}(dx\wedge dy \wedge dz) = \frac{1}{4\pi}\sin(\theta)d\theta \wedge d\phi$$ 
for $\theta\in [0,\pi]$ and $\phi \in [0,2\pi)$, where $\nu$ is the unit normal vector field to $S^{2}$. Let also $z = \frac{1}{2}\cos(\theta)$ be the height function on $S^{2}$. Then, for a smooth function $h: \R \rightarrow \R$ the solution to $\omega(X_{h},-) = -d(h\circ z)$ is of the form  $X = -2\pi h'(z)\partial_{\phi}$. So if $h'(a) \in \Z$ one has closed Hamiltonian orbits on $L_{a} := z^{-1}(a)$ where $a\in[0,1]$.

\begin{remark}
A Hamiltonian diffeomorphism of $\C P^{n}$ is called pseudo-rotation if it has finitely many periodic points, or equivalently, exactly $n+1$ periodic points \cite{MR4413744}. A Hamiltonian pseudo-rotation $\phi$ of $\C P^{n}$ is known \cite{gg-pseudo-rotations} to satisfy the following rigidity property
$$\liminf_{k\geq 1}\gamma(\phi^{k}) = 0.$$
In particular, whenever $\phi$ is the time-one map of $H = h\circ \mu$ where $\mu:\C P^{n}\rightarrow \R^{n}$ is the moment map of the standard Hamiltonian $T^{n}$-action, see below, and $h: \R^{n}\rightarrow \R$ is a linear function, i.e., $\nabla h\equiv \mathrm{const}$, then $\liminf_{k\geq 1}\gamma(\phi^{k}) = 0$. This is because either $\phi$ will be a pseudo-rotation or it will be torsion depending on rationality of the vector $\nabla h$; see \S\ref{sec:pseudo_rot_torsion}.
\end{remark}
The following theorem and the results of \S\ref{sec:ext_to_symp_toric} complement the above remark in the case of $\C P^{n}$, concerning functions of the moment polytope, and study the general case of monotone symplectic toric manifolds.

\begin{theorem}\phantomsection\label{thm:full_int_case_sphere}
Suppose $h:\R\rightarrow \R$ is a smooth non-linear function. Let $h\circ z:S^{2} \rightarrow \R$ generate the autonomous Hamiltonian flow $(\phi^{t})_{t\in \R}.$ Then $$\liminf_{t\geq 1}\gamma(\phi^{t}) >0.$$ 
\end{theorem}
\begin{proof}
It is not hard to show that the area of $S^{2}\cap \{z\leq a\}$ is equal to $a$ for every $a\in [0,1]$. Observe that $L_{1/2}$ is a monotone Lagrangian submanifold. Without loss of generality assume that $h'(1/2)\geq 0$, otherwise we would work with $-h$ which has the same spectral norm as $h$.

Let $g:L_{a} \rightarrow \R$ be a small perfect Morse function with critical points $p$ of index zero and $q$ of index one and extend $g$ naturally to a neighborhood of $L$, we denote the extension by $g$ as well and its Hamiltonian flow by $\psi_{t}$. Let $\delta >0$ be some arbitrary small number. Denote $\psi_{\delta,t} = \psi_{\delta t}.$

Claim $1$. For every $a\in (0,1)$ there exist an index one chord on $L_{a}$ for the Hamiltonian $H\# \delta g$ with action equal to
$$h(a) + e(\delta),$$ 
where $e(\delta) \to 0$ as $\delta \to 0$.

Proof of Claim $1$. Let $H_{\delta} = H\# \delta g$ be a perturbation of $H$ generating $\phi_{\delta, t}:= \phi_{t}\psi_{\delta, t}$ so that $\phi_{\delta}^{-1}(L_{a})$ intersects $L_{a}$ transversely at $\{p,q\}$. Consider the following capped chords
\begin{enumerate}
\item $[p,u_{p}]$ where $u_{p}(s,t) = \phi_{t}\psi_{\delta,t}(p)$.
\item $[q,u_{q}]$ where $u_{q}(s,t) = \phi_{t}\psi_{\delta,t}(q)$.
\end{enumerate}
The index of these capped chords are either $0$ or $1$ depending on the index of $p$ and $q$ and one of them is equal to $1$. The action of both capped chords are equal to $h(a) + e(\delta)$ for some error term $e(\delta) \to 0$ as $\delta \to 0.$ This finishes the proof of Claim $1$.

Applying Claim $1$ above to $\ell h$ for integers $\ell \geq 1$ we observe $\ell h(a) + e(\delta)$ is the action of an index one chord in $\CF(L, (\ell h\circ z)\# \delta g)$.

Claim $2$. There exists $\eps >0$ and an integer $N>1$ depending only on $h$ so that for every $\ell \geq N$ there exists $0<a<1$ so that 
$$d\big(\ell h(1/2), \ell h(a) + (1-2a)\Z\big) \geq \eps,$$
where $d(A,B)$ is the Hausdorff distance between two sets $A,B \subset \R$.

Proof of Claim $2$. By the assumption that $h'$ is not constant there must be a small connected interval $[r,s]\subset (0,1)\backslash \{\frac{1}{2}\}$ so that
\begin{enumerate}
\item $h' \neq 0$ on $[r,s]$.
\item $h'$ is non-constant on $[r,s]$.
\end{enumerate}
Consider the following function,
$$\Phi_{h}:[r,s] \rightarrow \R; a\mapsto \frac{h(1/2) - h(a)}{1 - 2a}.$$
Obviously by the choice of the interval $[r,s]$ the function $\Phi$ is not constant. Observe that $\Phi([r,s])$ is a non-trivial interval. Define 
$$\eps:= \frac{1}{2} - s,$$
and observe that it only depends on $h$. Let $N$ be a large enough constant so that for all $\ell \geq N,$ $\Phi_{\ell h}([r,s])$ contains a number $k_{\ell} + 1/2$ for some $k_{\ell}\in \Z$. In other words, for every $\ell \geq N$ there exists $a\in [r,s]$ so that $\Phi_{\ell h}(a) = k_{\ell} + 1/2.$ This implies
$$h(1/2) = h(a) + k_{\ell}(1 -2a) + \frac{1}{2} - a.$$
Therefore, the distance of $h(1/2)$ from the set $h(a) + (1-2a)\Z$ is exactly $\frac{1}{2} - a \geq \frac{1}{2} - s = \eps$ which finishes the proof of Claim $2$.

We continue proof of the theorem. Let $[r,s]$, $\eps$ and $N$ be as in claim $2$ above. Suppose by contradiction that $\liminf_{k \geq 1}\gamma(\phi^{k}) =0$. Hence, there are infinitely many iterations $k$ for which 
$$\gamma(\phi^{k})<\min\{r/2,\eps/2\}.$$ 
Let $\ell >N$ be such an iteration and define $\psi := \phi^{\ell}$. By Theorem \ref{mainthm:non_van_floer_hom}
$$\HF^{I+c}_{1}(L_{a}, \psi) \neq 0 \text{ and } \HF^{I+c}_{1}(L_{1/2}, \psi) \neq 0$$
where $I = [\gamma(\psi)-\delta'/2,\delta'/2]$ for some arbitrary small $\delta' >0$, $c = c([S^{2}], \psi)$, and $a\in [r,s]$ is given by Claim $2$ above for $\ell$. Hence we must have
$$h(1/2)\in [\gamma(\psi)-\delta'/2, \delta'/2],$$
because $L_{1/2}$ is monotone and the action of any index one chord is equal to $h(1/2)$. The action of index one chords of $\ell h$ on the Lagrangian $L_{a}$ lie in the set $h(a) + (1-2a)\Z$. Thus there is a point of this set in the interval $[\gamma(\psi)-\delta'/2, \delta'/2]$ which implies
$$\eps \leq d(\ell h(1/2), \ell h(a) + (1-2a)\Z)\leq \gamma(\psi) + \delta'<\eps/2+\delta'$$
which is a contradiction once choosing $\delta'$ small enough.
\end{proof}

\begin{remark}
It is interesting to see why the above proof does not apply to rotations about the $z$-axis, that is, when $h$ is linear in $a.$ A direct inspection shows that in this case Claim 2 fails, since if $h(a)= \theta a$ then it is easy to see by breaking into the case where $\theta$ is rational and the case where it is irrational that \[\inf_{\ell\geq 1} d(\ell h(1/2),\ell h(a)+(1-2a)\Z) = d(\frac{\theta}{2} (1-2a)\Z, (1-2a)\Z) = 0.\]
\end{remark}


\subsubsection{Extension to monotone symplectic toric manifolds}\label{sec:ext_to_symp_toric}
Here we prove spectral rigidity of Hamiltonian functions of the form $h \circ \mu$ where $\mu$ is the moment map of a monotone symplectic toric manifold and $h:\R^{n}\rightarrow \R$ is non-linear. Let $(M,\omega)$ be a monotone symplectic toric manifold with monotonicity constant $\kappa$, the moment map $\mu: M \rightarrow \R^{n}$ and the monotone moment polytope $\Delta$; see \S\ref{sec:symp_toric_mnfld}. The following lemma is a higher dimensional analogue of Theorem \ref{thm:action_of_index_one_chords} where $H = h\circ\mu$, $H_{\delta} = H\# \delta g$, $g$ is the extension of a $C^{2}$-small Morse function on $L_{\eta}$, for some $\eta\in \mathrm{int}(\Delta)$, to $M$ and $\delta>0$ is an arbitrary small number. As before $\u{L}_{\eta} = (L_{\eta}, \alpha \equiv q, \Lambda \equiv T_{q}L_{\eta})$ where $q\in L_{\eta}$ is a fixed arbitrary point. We will always work with a set of coordinates in $\R^{n}$, introduced at the end of \S\ref{sec:symp_toric_mnfld}, with respect to which the unique integral point in $\kappa^{-1}\Delta$ is the origin which we denote it by $\xi = 0$. Recall that $L_{\xi} = \mu^{-1}(\xi)$ is a monotone Lagrangian torus. In the following lemma we denote
$$\eta_{i} :=\kappa -\langle \eta, v_{i}\rangle, 1\leq i \leq d$$
where $v_{i}$ is the primitive integral vector normal to the $i$-th facet of $\Delta$. In particular $\xi_{i}=\kappa$ for all $1\leq i\leq d$. The set of the actions of index one chords in the complex $\CF(\u{L}_{\eta}, H_{\delta})$ will be denoted by $\mathcal{S}_{h,\delta}(\eta)$.


\begin{lemma}\phantomsection\label{lem:spec_form}
The spectrum $\mathcal{S}_{h,\delta}(\eta)$ is a subset of the following set,
$$\left\{h(\eta) + 
\big\lfloor\frac{n - \mathrm{ind}(p)}{2}\big\rfloor\eta_{i_{0}} + e(\delta)+\sum_{i=1}^{d-1} k_{i}(\eta_{i+1}-\eta_{i})
\right\}_{p,k_{i}},$$
where $p\in \mathrm{Crit}(g), k\in \Z$, $i_{0}\in \{1,\dots, d\}$ is fixed and arbitrary, $k_{i}\in \Z$ and $e(\delta)\to 0$ as $\delta \to 0$.
\end{lemma}
\begin{proof}
The proof of this lemma is very similar to that of Claim $1$ in the proof of Theorem \ref{thm:full_int_case_sphere}. The only new input are, first to observe that $L_{\eta}$ and $\phi_{\delta}^{-1}(L_{\eta})$ intersect transversely at the critical points of $g:L_{\eta} \rightarrow \R$ whose Morse indices vary in the set $\{0,1,\dots, n\}$, and second, to consider all the re-cappings with Maslov index zero and non-zero symplectic areas. According to the discussion in \S\ref{sec:symp_toric_mnfld} the period group $\omega(\pi_{2}(M, L_{\eta}))$ is generated by the areas of the disks $\{D_{\eta, i}\}_{i=1}^{d}$, where $d$ is the number of facets of $\Delta$, since $\omega(D_{\eta,i}) = \eta_{i}$ for $1\leq i\leq d$,
$$\omega(\pi_{2}(M,L_{\eta})) = \Z\langle \eta_{1}, \dots, \eta_{d}\rangle.$$
Each $D_{\eta, i}$ has Maslov index $2$. Therefore, the kernel of the Maslov class is generated by $\{D_{\eta,i+1}-D_{\eta,i}\}_{i=1}^{d-1}$ and the area of these generators are $\eta_{i+1} - \eta_{i}$. Hence, the lemma is proved.
\end{proof}

\begin{theorem}\phantomsection\label{thm:full_int_case_proj_space1}
Let $(M,\omega)$ be a monotone symplectic toric manifold with moment map $\mu:M \rightarrow \R^{n}$. Let $h: \R^{n} \rightarrow \R$ be a smooth function which is not affine on the image of $\mu.$ If $h\circ \mu : M\rightarrow \R$ generates the Hamiltonian flow $(\phi^{t})_{t\in\R}$ then $\liminf_{t\geq 1}\gamma(\phi^{t}) >0.$
\end{theorem}
\begin{proof} 
The case of $n=1$, i.e., $M = S^{2}$, is proved in Theorem \ref{thm:full_int_case_sphere}, so assume $n\geq 2$. Since $\nabla h$ is not constant one can find a vector $\zeta\in \Z^{n}$ so that $h(a\zeta)$ is not linear as a function of $a\in \R$. Indeed, it is enough to find any direction $0\neq \zeta\in \R^{n}$ so that $h(a\zeta)$ is not linear in $a$, then by a small perturbation and scaling one can assume $\zeta\in \Z^{n}$. Suppose that for all $\zeta\in \R^{n}$ there exist constants $s_{\zeta},b_{\zeta}\in \R$ so that $h(a\zeta) = s_{\zeta}a + b_{\zeta}$. By evaluating at $a = 0,1,$ we obtain 
$$h(a\zeta) = a(h(\zeta) - h(0)) + h(0).$$
Differentiating at $a = 0$ one finds that $h(\zeta) = \langle dh(0), \zeta\rangle + h(0)$ for every $\zeta \in \R^{n}$ which is a contradiction to $\nabla h$ being non-constant. 

Denote by $\kappa$ the monotonicity constant of $M$. Let $[r,s]\subset (0,1)$ be a small enough interval for which
\begin{enumerate}
\item $h(a\zeta)$ is not linear on every $[r',s']\subset[r,s]$,
\item $[r,s] \cap \{\frac{\kappa}{k}, \frac{2\kappa}{k}, \dots, \frac{n\kappa}{k}\} = \emptyset$ for all $k\in \Z$.
\end{enumerate}
Let $\mathcal{S}_{h}(\eta) = \lim_{\delta \to 0}\mathcal{S}_{h,\delta}.$ For every $\eta\in \Delta := \mu(M)$ define
$$\mathcal{B}_{h}(\eta) := \{h(\eta) + i\eta_{i_{0}}, 0\leq i \leq n\},$$
where $i_{0}\in \{1,\dots, d\}$ is arbitrary and fixed. Observe that $\mathcal{B}_{h}(\eta)$ is a finite set of numbers and by Lemma \ref{lem:spec_form}
\begin{equation}\phantomsection\label{eq:spect_eq}
\mathcal{S}_{h}(\eta) \subset \mathcal{B}_{h}(\eta) + \sum_{1\leq i\leq d-1}(\eta_{i+1} - \eta_{i})\Z.
\end{equation}
For the particular case $\eta = 0, a\zeta$ we have
\begin{equation}\phantomsection\label{eq:spect_eq2}
\mathcal{S}_{h}(0) \subset \mathcal{B}_{h}(0) \text{ and } \mathcal{S}_{h}(a\zeta) \subset \mathcal{B}_{h}(a\zeta) + a\Z.
\end{equation}
The latter is by \eqref{eq:spect_eq} and that 
$$(a\zeta)_{i+1} - (a\zeta)_{i}= (\kappa -a\langle\zeta,v_{i+1}\rangle) - (\kappa -a\langle\zeta,v_{i}\rangle)\in a\Z, \forall i \in \{1,\dots, d-1\}.$$
Claim $1$. Let $[r',s'] \subset [r,s].$ There exists $\eps >0$ and $N>1$ so that for every $\ell \geq N$ there is $a\in [r',s']$ such that for every $A\in \mathcal{S}_{\ell h}(0)$
$$d(A, \mathcal{S}_{\ell h}(a\zeta))\geq \eps.$$
We postpone the proof of Claim $1$ to the end and explain how it implies the theorem. Suppose by contradiction $\liminf_{t\geq 1}\gamma(\phi^{t})=0$. 


In view of Lemma \ref{lem:fukaya_trick} we may pick an interval $[r',s'] \subset [r,s]$ such that \[\hbar_{\min} = \min_{a\in [r',s']} \hbar(L_{a\zeta}, J_{a}) \geq \delta(r',s') > 0\] for the certain lower bound depending continuously on the complex structure in the $C^2$ topology as in Remark \ref{rem:C2 top}. 



By the contrapositive assumption there exists $\ell \geq N$ so that
\begin{equation}\label{eq:small_gamma}
\gamma(\phi^{\ell})<\min\{\delta(r',s')/4,\eps/2\}.
\end{equation}
Note that for a given $a \in [r',s']$ and a sufficiently $C^2$-small perturbation $J$ of $J_a$ in view of \cite[Remark 4.3.2]{mcduff-salamon-12} we have \begin{equation}\label{ineq:hbar-delta}\hbar(L,J) \geq \delta(r',s')/2.\end{equation} 

Now define $\psi := \phi^{\ell}$. The following claim is a higher dimensional version of Theorem \ref{mainthm:non_van_floer_hom}.

Claim $2$. Let $L$ be a closed Lagrangian submanifold in a closed symplectic manifold $M$. Let $\psi$ be a Hamiltonian diffeomorphism of $M$, $J$ be a generic $\omega$-tamed almost complex structure on $M$, and $\hbar:=\hbar(L,J)$ be defined as in \S\ref{sec:minimal_enr_inv}. Suppose that $\gamma(\psi)\leq \hbar/2  - 2\delta'$ for some small $0 <\delta' \ll \hbar$. Then
$$\HF_{*}^{I+c}(L,\psi) \neq 0$$
where $I = [-\hbar/2 + \delta', \delta']$ and $c = c([M], \psi)$ and $*\in \{0,1,\dots, n\}$.

Proof of Claim $2$. This is shown similarly as in the proof of Theorem \ref{mainthm:non_van_floer_hom}; see also \cite[\S5]{kislev-shelukhin-21} and \cite{MR4879080}. The proof of claim $2$ is finished.

It follows from \eqref{eq:small_gamma} and Claim $2$ that
$$\HF^{I+c}_{1}(L_{a\zeta}, \psi) \neq 0 \text{ and } \HF^{I+c}_{1}(L_{0}, \psi) \neq 0,$$
where $I = [\gamma(\psi)-\delta'/2,\delta'/2]$ for some arbitrary small $\delta' >0$, $c = c([M], \psi)$, and $a\in [r',s']$ is given by Claim $1$ above for the given $\ell,$ while the generic almost complex structure is chosen sufficiently $C^2$-close to $J_a$ to satisfy \eqref{ineq:hbar-delta}. It follows that the interval $[\gamma(\psi)-\delta'/2, \delta'/2]$ contains elements from both $\mathcal{S}_{\ell h}(0)$ and $\mathcal{S}_{\ell h}(a\zeta)$. So there must exists $A\in \mathcal{S}_{\ell h}(0)$ so that
$$\eps \leq d(A, \mathcal{S}_{\ell h}(a\zeta)) \leq \gamma(\psi) + \delta'<\eps$$
which is a contradiction and this finishes the proof of the theorem. 

Proof of Claim $1$. First recall that
$$\mathcal{B}_{h}(\eta) = \{h(\eta) +i\eta_{i_{0}}:0\leq i \leq n\},$$
where $i_{0}\in \{1,\dots, d\}$ is arbitrary. Let $i_{0} = 1$ so that 
\begin{equation}\label{eq:xi_a_one}
(a\zeta)_{1} = \kappa - a\langle \zeta,v_{1}\rangle.
\end{equation}
Define the following function
$$\Phi_{h}:[r',s'] \rightarrow\R;\ a\mapsto \frac{h(0) -h(a\zeta)}{a}.$$
The minimal gap in the set $\Sigma_{a} :=\cup_{i=0}^{n} (a\Z + i\kappa)$ is equal to 
$$\mathrm{mg}(\Sigma_{a}) = \min_{k\in \Z, 0\leq \ell\leq n} \abs{ka - \ell\kappa}.$$
 
Define $2\rho := \min_{a\in [r',s']}\mathrm{mg}(\Sigma_{a})$. By the property $(iv)$ of the interval $[r,s]$ which also applies to $[r',s']$ we have $\rho >0$. Let $N$ be a large enough integer so that for every $\ell \geq N$ the image $\Phi_{\ell h}([r',s']) = \ell\Phi_{h}([r',s'])$ contains a number of the form $k + \rho$ where $k\in \Z$; this is where we used property $(\mathrm{i})$ of $[r,s]$ which implies that $\Phi_{h}$ is not constant on $[r',s']$. For an integer $\ell \geq N$ let $a\in [r',s']$ so that $\Phi_{\ell h}(a) = k + \rho$ which implies
$$\ell h(0) = \ell h(a\zeta) + ka + \rho a.$$
Let $A\in \mathcal{S}_{\ell h}(0)$. So, by \eqref{eq:spect_eq2} there is $p\in \{0,\dots, n\}$ so that
$$A=\ell h(0)+ p\kappa = \ell h(a\zeta) + ka + p\kappa + \rho a.$$
Because, according to \eqref{eq:spect_eq2} and \eqref{eq:xi_a_one}, we have $\mathcal{S}_{\ell h}(a\zeta)\subset \ell h(a\zeta) + \Sigma_{a}$ it follows that
$$\rho a \leq \rho s<\rho\leq \mathrm{mg}(\Sigma_{a})/2\leq \mathrm{mg}\big(\mathcal{S}_{\ell h}(a\zeta)\big)/2,$$ 
we get that
$$d(A, \mathcal{S}_{\ell h}(a\zeta)) = \rho a\geq \rho r'.$$
This finishes the proof of the claim for $\eps:= \rho r'$ which only depends on $[r',s']$ and so it only depends on $h$.
\end{proof}

\begin{remark}
The general case of not necessarily monotone symplectic toric manifolds seem to be not accessible via the current techniques used above. A seemingly serious issue is that when the manifold is not monotone we do not have a monotone Lagrangian fiber, therefore, even when $M$ is rational, working with two non-monotone fibers $L_{0}$ and $L_{a\zeta}$ one has to separate the finite set $\mathcal{B}_{\ell h}(0)$ from the infinite set $\mathcal{B}_{\ell h}(a\zeta) + a\Z + b\Z$ where $b$ is a fixed positive number depending on $0$, and $a$ is allowed to vary in a small interval $[r,s]$. However, for any infinite set of numbers $A\subset [r,s]$
$$\inf_{a\in A} \mathrm{mg}(\mathcal{B}_{h}(a\zeta) + a\Z + b\Z) = 0,$$
which is problematic when trying to find a uniform lower bound for $\gamma(\phi^{\ell})$ for all $\ell \geq 1$. This means that there are only finitely many choices of $a$ available for all $\ell$. This issue maybe resolvable in a similar fashion as above, however, we leave it for now and study this together with the irrational case in a future work.
\end{remark}

\subsection{Integrable systems beyond symplectic toric}\label{sec:more_gen_int_ham_sys}
In this section we provide a different proof for spectral rigidity of certain completely integrable Hamiltonian systems, not necessarily coming from symplectic toric structures, that is based on a module action of Hamiltonian Floer homology on the Lagrangian Floer homology developed by Albers \cite{albers-10}, Biran and Cornea \cite{biran-cornea-09}; the formulation we use here is a quantitative version of the latter studied by Cant and Zhang \cite{cant-zhang-24}, see also \cite{MR4879080}. The techniques here apply to symplectic manifolds beyond the symplectic toric manifolds such as certain complex Grassmannians and $T^{2n}$, and can be applied to certain non-autonomous Hamiltonian diffeomorphisms. For the definition of a completely integrable Hamiltonian system $(M,\omega, H,F)$ see \S\ref{sec:comp_integ_ham_sys}.

\begin{theorem}\phantomsection\label{thm:comp_int_sys}
Let $(M,\omega, H, F)$ be a completely integrable Hamiltonian system so that $(M,\omega)$ satisfies either of the following:
\begin{enumerate}
\item $(M,\omega)$ is symplectically aspherical,
\item or the spectral diameter of $\t{\Ham}(M,\omega)$ is finite,
\item or $\#\pi_{1}(\Ham(M,\omega)) < \infty$.
\end{enumerate}
Let $h: \R^{n}\rightarrow \R$ be a smooth non-constant function. Suppose that $\max h$ and $\min h$ are realized by some regular values $c_{\pm}$ of $F$. Then, there exist a constant $\hbar >0$ so that $\liminf_{k\geq 1}\gamma(\phi^{k}) \geq \hbar >0$ where $\phi$ is the time-one map of $h\circ F$.
\end{theorem}
\begin{proof}
Up to shifting $h$ by a constant assume that $\int_{M}(h\circ F) \omega^{n} = 0$ which in particular implies $\max h >0$ and $\min h <0$. Let $\hbar = \min\{\hbar(L_{+}), \hbar(L_{-})\}>0$ where $L_{\pm} = F^{-1}(c_{\pm})$. Suppose that $\liminf_{k\geq 1} \gamma(\phi^{k}) < \hbar$ and let $\{n_{k}\}_{k \geq 1}$ be a sequence of numbers so that $n_{k} \to \infty$ and $\gamma(\phi^{n_{k}})< \hbar$ for all $k\geq 1$. It follows from \cite{kislev-shelukhin-21} that $\beta(\phi^{n_{k}}) <\hbar$ for all $k\geq 1$ where $\beta(\phi)$ is the boundary depth of the barcode $\mathcal{B}(\phi)$ of $\phi$; see \S\ref{sec:pers_mod_and_barcodes}. For a Hamiltonian function $G$ on $M$ and some closed Lagrangian submanifold $L\subset M$ define
$$e_{G,L}:= \beta(\phi_{G}) + \int_{0}^{1}\max_{M}G_{t}dt - \int_{0}^{1}\min_{L}G_{t} dt.$$
One can show that if $e_{G,L} < \hbar(L)$ then $c([M],G) \geq \int_{0}^{1}\min_{L}G_{t}dt$; see \cite{biran-cornea-09, cant-zhang-24}. Applying this fact to $G = \pm n_{k}(h\circ F)$ and $L = L_{\pm}$, and recalling that the boundary depth only depends on the time-one map, one finds that, 
$$c([M],\pm n_{k}(h\circ F))\geq n_{k}\max \pm h.$$ 
It follows that $\gamma(n_{k}(h\circ F)) = |n_{k}(h\circ F)|$ where the latter is the Hofer norm of $n_{k}(h\circ F)$ which obviously diverges to infinity. If $(M,\omega)$ satisfies the second item, i.e., the spectral diameter of the universal cover of $\Ham(M,\omega)$ is finite, then we have reached a contradiction and the proof is finished in this case. If $(M,\omega)$ is either aspherical or $\pi_{1}(\Ham(M,\omega))$ is a finite group then it is well-known that the spectral norm of a mean-normalized Hamiltonian $G$ only depends on the time-one map $\phi_{G}$ of $G$. Hence,
$$\abs{n_{k}(h\circ F)} = \gamma(n_{k} (h\circ F)) = \gamma(\phi^{n_{k}}),$$
which is again a contradiction. So $\liminf_{k\geq 1} \gamma(\phi^{k}) \geq \hbar > 0$.
\end{proof}
The theorem applies to completely integrable Hamiltonian systems on 
\begin{itemize}
\item $T^{2n}$, which satisfies $(i)$,
\item $\C P^{1} \times \C P^{1}$, which satisfies $(iii)$; see \cite{MR809718} and \cite[Chapter 9]{mcduff-salamon-12},
\item $\mathrm{Gr}_{\C}(2,n)$, the complex Grassmannians of $2$-planes in $\C^{n}$, whenever $n$ is a prime number, which satisfies $(ii)$ as shown in \cite{alizadeh2025spectral}.
\end{itemize}

\begin{remark}\phantomsection\label{rem:lag_in_max_set}
The above result holds in the following more general setting. Suppose $(M,\omega)$ is a closed symplectic manifold satisfying either of the properties listed in Theorem \ref{thm:comp_int_sys}. Let $H: M \times \R \rightarrow \R$ be a possibly non-autonomous Hamiltonian function with time-one map $\phi_{H}\neq id$. Assume that there are closed Lagrangian submanifolds $L_{\pm}$ so that $L_{+}\subset H_{t}^{-1}(\max H_{t})$ and $L_{-}\subset H_{t}^{-1}(\min H_{t})$ for all $t\in [0,1]$. Then, $\liminf_{k\geq 1}\gamma(\phi_{H}^{k}) \geq \hbar > 0$ where $\hbar = \min\{\hbar(L_{+}), \hbar(L_{-})\}$. Note that Proposition \ref{lem:just_const} does not follow from this remark as $c_{1}$ and $c_{2}$ are not necessarily the maximum or the minimum of the Hamiltonian function $H$, neither the results of this remark are obtained by the techniques used in the proofs of Theorem \ref{mainthm:auto_on_symp_hyperb} and Proposition \ref{lem:just_const} as the symplectic hyperbolicity is crucial in those proofs.
\end{remark}

\bibliographystyle{plain}
\bibliography{citations}

@article {Atallah-Shelukhin,
    AUTHOR = {Atallah, M. S. and Shelukhin, E.},
     TITLE = {Hamiltonian no-torsion},
   JOURNAL = {Geom. Topol.},
  FJOURNAL = {Geometry \& Topology},
    VOLUME = {27},
      YEAR = {2023},
    NUMBER = {7},
     PAGES = {2833--2897},
      ISSN = {1465-3060},
   MRCLASS = {37J11 (53D22 53D40 57R17 57S17)},
  MRNUMBER = {4645487},
       DOI = {10.2140/gt.2023.27.2833},
       URL = {https://doi.org/10.2140/gt.2023.27.2833},
}

@article {bochner-montgomery,
    AUTHOR = {Bochner, Salomon and Montgomery, Deane},
     TITLE = {Locally compact groups of differentiable transformations},
   JOURNAL = {Ann. of Math. (2)},
  FJOURNAL = {Annals of Mathematics. Second Series},
    VOLUME = {47},
      YEAR = {1946},
     PAGES = {639--653},}

@article {leroux-seyfaddini,
    AUTHOR = {Le Roux, F. and Seyfaddini, S.},
     TITLE = {The {A}nosov-{K}atok method and pseudo-rotations in symplectic
              dynamics},
   JOURNAL = {J. Fixed Point Theory Appl.},
  FJOURNAL = {Journal of Fixed Point Theory and Applications},
    VOLUME = {24},
      YEAR = {2022},
    NUMBER = {2},
     PAGES = {Paper No. 36, 39},}

@article {joksimovic-seyfaddini,
    AUTHOR = {Joksimovi\'{c}, D. and Seyfaddini, S.},
     TITLE = {A {H}\"{o}lder-type inequality for the {$C^0$} distance and
              {A}nosov-{K}atok pseudo-rotations},
   JOURNAL = {Int. Math. Res. Not. IMRN},
  FJOURNAL = {International Mathematics Research Notices. IMRN},
      YEAR = {2024},
    NUMBER = {8},
     PAGES = {6303--6324},}

@article {Polterovich-02,
    AUTHOR = {Polterovich, L.},
     TITLE = {Growth of maps, distortion in groups and symplectic geometry},
   JOURNAL = {Invent. Math.},
  FJOURNAL = {Inventiones Mathematicae},
    VOLUME = {150},
      YEAR = {2002},
    NUMBER = {3},
     PAGES = {655--686},
      ISSN = {0020-9910},
   MRCLASS = {53D35 (53D40)},
  MRNUMBER = {1946555},
MRREVIEWER = {Karl Friedrich Siburg},
       DOI = {10.1007/s00222-002-0251-x},
       URL = {https://doi.org/10.1007/s00222-002-0251-x},
}

@article {Sugimoto-23,
    AUTHOR = {Sugimoto, Y.},
     TITLE = {No {$C^1$}-recurrence of iterations of symplectomorphisms},
   JOURNAL = {Geom. Dedicata},
  FJOURNAL = {Geometriae Dedicata},
    VOLUME = {219},
      YEAR = {2025},
    NUMBER = {1},
     PAGES = {Paper No. 9, 13},
      ISSN = {0046-5755},
   MRCLASS = {53D05 (53D35)},
  MRNUMBER = {4840557},
       DOI = {10.1007/s10711-024-00969-1},
       URL = {https://doi.org/10.1007/s10711-024-00969-1},
}

@article {gg-aid,
    AUTHOR = {Ginzburg, V. L. and G\"{u}rel, Ba\c{s}ak Z.},
     TITLE = {Approximate identities and {L}agrangian {P}oincar\'{e}
              recurrence},
   JOURNAL = {Arnold Math. J.},
  FJOURNAL = {Arnold Mathematical Journal},
    VOLUME = {5},
      YEAR = {2019},
    NUMBER = {1},
     PAGES = {5--14},}

@article {cgg-growth-of-floer-barcode,
    AUTHOR = {\c{C}ineli, E. and Ginzburg, V. L. and G\"{u}rel, B. Z.},
     TITLE = {On the growth of the {F}loer barcode},
   JOURNAL = {J. Mod. Dyn.},
  FJOURNAL = {Journal of Modern Dynamics},
    VOLUME = {20},
      YEAR = {2024},
     PAGES = {275--298},
      ISSN = {1930-5311},
   MRCLASS = {53D40 (37J11 37J46)},
  MRNUMBER = {4799465},
       DOI = {10.3934/jmd.2024007},
       URL = {https://doi.org/10.3934/jmd.2024007},
}

@article {cgg-barcode-entropy,
    AUTHOR = {\c{C}ineli, E. and Ginzburg, V. L. and G\"{u}rel, B. Z.},
     TITLE = {Topological entropy of {H}amiltonian diffeomorphisms: a
              persistence homology and {F}loer theory perspective},
   JOURNAL = {Math. Z.},
  FJOURNAL = {Mathematische Zeitschrift},
    VOLUME = {308},
      YEAR = {2024},
    NUMBER = {4},
     PAGES = {Paper No. 73, 38},
      ISSN = {0025-5874},
   MRCLASS = {37J11 (37B40 37J46 53D40)},
  MRNUMBER = {4822772},
       DOI = {10.1007/s00209-024-03627-0},
       URL = {https://doi.org/10.1007/s00209-024-03627-0},
}

@article {cgg-generic-behaviour,
    AUTHOR = {\c{C}ineli, E. and Ginzburg, V. L. and G\"{u}rel, B. Z.},
     TITLE = {On the generic behavior of the spectral norm},
   JOURNAL = {Pacific J. Math.},
  FJOURNAL = {Pacific Journal of Mathematics},
    VOLUME = {328},
      YEAR = {2024},
    NUMBER = {1},
     PAGES = {119--135},
      ISSN = {0030-8730},
   MRCLASS = {37J11 (37J46 53D40)},
  MRNUMBER = {4732759},
MRREVIEWER = {Cristian L\u{a}zureanu},
       DOI = {10.2140/pjm.2024.328.119},
       URL = {https://doi.org/10.2140/pjm.2024.328.119},
}

@article {kp-recurrent,
    AUTHOR = {Kolev, B. and P\'{e}rou\`eme, M.-C.},
     TITLE = {Recurrent surface homeomorphisms},
   JOURNAL = {Math. Proc. Cambridge Philos. Soc.},
  FJOURNAL = {Mathematical Proceedings of the Cambridge Philosophical
              Society},
    VOLUME = {124},
      YEAR = {1998},
    NUMBER = {1},
     PAGES = {161--168},}

@article {hofer-metric,
    AUTHOR = {Hofer, H.},
     TITLE = {On the topological properties of symplectic maps},
   JOURNAL = {Proc. Roy. Soc. Edinburgh Sect. A},
  FJOURNAL = {Proceedings of the Royal Society of Edinburgh. Section A.
              Mathematics},
    VOLUME = {115},
      YEAR = {1990},
    NUMBER = {1-2},
     PAGES = {25--38},
}

@article {lalonde-mcduff,
    AUTHOR = {Lalonde, F. and McDuff, D.},
     TITLE = {The geometry of symplectic energy},
   JOURNAL = {Ann. of Math. (2)},
  FJOURNAL = {Annals of Mathematics. Second Series},
    VOLUME = {141},
      YEAR = {1995},
    NUMBER = {2},
     PAGES = {349--371},
}

@article{polterovich-lagrangian,
	AUTHOR = {Polterovich, L.},
	TITLE = {Symplectic displacement energy for {L}agrangian submanifolds},
	JOURNAL = {Ergodic Theory Dynam. Systems},
	FJOURNAL = {Ergodic Theory and Dynamical Systems},
	VOLUME = {13},
	YEAR = {1993},
	NUMBER = {2},
	PAGES = {357--367},
}

@book{polterovich-book,
    AUTHOR = {Polterovich, L.},
     TITLE = {The geometry of the group of symplectic diffeomorphisms},
    SERIES = {Lectures in Mathematics ETH Z\"{u}rich},
 PUBLISHER = {Birkh\"{a}user Verlag, Basel},
      YEAR = {2001},
       URL = {https://doi.org/10.1007/978-3-0348-8299-6},
}

@article{cant-zhang-24,
  title={On the spectral capacity of submanifolds},
  author={Cant, D. and Zhang, J.},
  journal={arXiv preprint arXiv:2409.14142},
  year={2024}
}

@article{SW-unpublished,
  title={Continuous {H}amiltonian no-torsion},
  author={Shelukhin, E. and Wilkins, N.},
  journal={Forthcoming},
  year = {2026},
}

@article {biran-cornea-09,
    AUTHOR = {Biran, P. and Cornea, O.},
     TITLE = {Rigidity and uniruling for {L}agrangian submanifolds},
   JOURNAL = {Geom. Topol.},
  FJOURNAL = {Geometry \& Topology},
    VOLUME = {13},
      YEAR = {2009},
    NUMBER = {5},
     PAGES = {2881--2989},
      ISSN = {1465-3060},
   MRCLASS = {53D12 (53D35 53D40)},
  MRNUMBER = {2546618},
MRREVIEWER = {Tobias Ekholm},
       DOI = {10.2140/gt.2009.13.2881},
       URL = {https://doi.org/10.2140/gt.2009.13.2881},
}

@article {albers-10,
    AUTHOR = {Albers, P.},
     TITLE = {Erratum for ``{O}n the extrinsic topology of {L}agrangian
              submanifolds'' [MR2180810]},
   JOURNAL = {Int. Math. Res. Not. IMRN},
  FJOURNAL = {International Mathematics Research Notices. IMRN},
      YEAR = {2010},
    NUMBER = {7},
    volume = {2},
     PAGES = {1363--1369},
      ISSN = {1073-7928},
   MRCLASS = {53D12 (53D35 53D40)},
  MRNUMBER = {2609024},
       DOI = {10.1093/imrn/rnq047},
       URL = {https://doi.org/10.1093/imrn/rnq047},
}

@article {gg-pseudo-rotations,
    AUTHOR = {Ginzburg, V. L. and G\"{u}rel, B. Z.},
     TITLE = {Hamiltonian pseudo-rotations of projective spaces},
   JOURNAL = {Invent. Math.},
  FJOURNAL = {Inventiones Mathematicae},
    VOLUME = {214},
      YEAR = {2018},
    NUMBER = {3},
     PAGES = {1081--1130},
      ISSN = {0020-9910},
   MRCLASS = {53D40 (37J10 37J45)},
  MRNUMBER = {3878727},
       DOI = {10.1007/s00222-018-0818-9},
       URL = {https://doi.org/10.1007/s00222-018-0818-9},
}

@article {Bramham-15,
    AUTHOR = {Bramham, B.},
     TITLE = {Pseudo-rotations with sufficiently {L}iouvillean rotation
              number are {$C^0$}-rigid},
   JOURNAL = {Invent. Math.},
  FJOURNAL = {Inventiones Mathematicae},
    VOLUME = {199},
      YEAR = {2015},
    NUMBER = {2},
     PAGES = {561--580},
      ISSN = {0020-9910},
   MRCLASS = {37A05 (37A25 37C40 37E45 37J10)},
  MRNUMBER = {3302121},
MRREVIEWER = {Xueting Tian},
       DOI = {10.1007/s00222-014-0525-0},
       URL = {https://doi.org/10.1007/s00222-014-0525-0},
}

@article {ginzburg-10-conley-conj,
    AUTHOR = {Ginzburg, V. L.},
     TITLE = {The {C}onley conjecture},
   JOURNAL = {Ann. of Math. (2)},
  FJOURNAL = {Annals of Mathematics. Second Series},
    VOLUME = {172},
      YEAR = {2010},
    NUMBER = {2},
     PAGES = {1127--1180},
      ISSN = {0003-486X},
   MRCLASS = {53D40 (37J05 53D35)},
  MRNUMBER = {2680488},
MRREVIEWER = {Hai-Long Her},
       DOI = {10.4007/annals.2010.172.1129},
       URL = {https://doi.org/10.4007/annals.2010.172.1129},
}

@book{KH-book,
    AUTHOR = {Katok, A. and Hasselblatt, B.},
     TITLE = {Introduction to the modern theory of dynamical systems},
    SERIES = {Encyclopedia of Mathematics and its Applications},
    VOLUME = {54},
      NOTE = {With a supplementary chapter by Katok and Leonardo Mendoza},
 PUBLISHER = {Cambridge University Press, Cambridge},
      YEAR = {1995},}

@article {MR4879080,
    AUTHOR = {Alizadeh, H. and Atallah, M. S. and Cant, D.},
     TITLE = {Lagrangian intersections and the spectral norm in
              convex-at-infinity symplectic manifolds},
   JOURNAL = {Math. Z.},
  FJOURNAL = {Mathematische Zeitschrift},
    VOLUME = {310},
      YEAR = {2025},
    NUMBER = {1},
     PAGES = {Paper No. 2, 48},
      ISSN = {0025-5874},
   MRCLASS = {53D12 (37J06 53D40)},
  MRNUMBER = {4879080},
       DOI = {10.1007/s00209-025-03710-0},
       URL = {https://doi.org/10.1007/s00209-025-03710-0},
}

@book{mcduff-salamon-12,
  title={J-holomorphic curves and symplectic topology},
  author={McDuff, D. and Salamon, D.},
  volume={52},
  year={2012},
  publisher={American Mathematical Soc.}
}

@article {sikorav-01,
    AUTHOR = {Sikorav, J.-C.},
     TITLE = {Growth of a primitive of a differential form},
   JOURNAL = {Bull. Soc. Math. France},
  FJOURNAL = {Bulletin de la Soci\'{e}t\'{e} Math\'{e}matique de France},
    VOLUME = {129},
      YEAR = {2001},
    NUMBER = {2},
     PAGES = {159--168},
      ISSN = {0037-9484},
   MRCLASS = {58A12 (53C20 53C65 58A10)},
  MRNUMBER = {1871291},
MRREVIEWER = {J. Eichhorn},
       DOI = {10.24033/bsmf.2390},
       URL = {https://doi.org/10.24033/bsmf.2390},
}

@phdthesis{bechara2023asymptotic,
  title={The asymptotic action and generalizations of a theorem of Fathi},
  author={Bechara Senior, D.},
  year={2023},
  school={Bochum, Ruhr-Universit{\"a}t Bochum, 2022}
}

@article {lecalvez-tal-22,
    AUTHOR = {Le Calvez, P. and Tal, F.},
     TITLE = {Topological horseshoes for surface homeomorphisms},
   JOURNAL = {Duke Math. J.},
  FJOURNAL = {Duke Mathematical Journal},
    VOLUME = {171},
      YEAR = {2022},
    NUMBER = {12},
     PAGES = {2519--2626},
      ISSN = {0012-7094,1547-7398},
   MRCLASS = {37E30 (37B40)},
  MRNUMBER = {4474474},
       DOI = {10.1215/00127094-2022-0057},
       URL = {https://doi.org/10.1215/00127094-2022-0057},
}

@article {hu-lalonde-10,
    AUTHOR = {Hu, S. and L., Fran\c{c}ois},
     TITLE = {A relative {S}eidel morphism and the {A}lbers map},
   JOURNAL = {Trans. Amer. Math. Soc.},
  FJOURNAL = {Transactions of the American Mathematical Society},
    VOLUME = {362},
      YEAR = {2010},
    NUMBER = {3},
     PAGES = {1135--1168},
      ISSN = {0002-9947,1088-6850},
   MRCLASS = {53D40 (53D12 53D45)},
  MRNUMBER = {2563724},
MRREVIEWER = {Umberto\ Leone\ Hryniewicz},
       DOI = {10.1090/S0002-9947-09-04986-1},
       URL = {https://doi.org/10.1090/S0002-9947-09-04986-1},
}

@incollection {biran-cornear-lag-quan-hom-09,
    AUTHOR = {Biran, P. and Cornea, O.},
     TITLE = {A {L}agrangian quantum homology},
 BOOKTITLE = {New perspectives and challenges in symplectic field theory},
    SERIES = {CRM Proc. Lecture Notes},
    VOLUME = {49},
     PAGES = {1--44},
 PUBLISHER = {Amer. Math. Soc., Providence, RI},
      YEAR = {2009},
      ISBN = {978-0-8218-4356-7},
   MRCLASS = {53D40 (53D12 53D42)},
  MRNUMBER = {2555932},
MRREVIEWER = {Timothy\ Perutz},
       DOI = {10.1090/crmp/049/01},
       URL = {https://doi.org/10.1090/crmp/049/01},
}

@article {schwarz-00,
    AUTHOR = {Schwarz, M.},
     TITLE = {On the action spectrum for closed symplectically aspherical
              manifolds},
   JOURNAL = {Pacific J. Math.},
  FJOURNAL = {Pacific Journal of Mathematics},
    VOLUME = {193},
      YEAR = {2000},
    NUMBER = {2},
     PAGES = {419--461},
      ISSN = {0030-8730,1945-5844},
   MRCLASS = {53D40 (57R58)},
  MRNUMBER = {1755825},
MRREVIEWER = {David\ E.\ Hurtubise},
       DOI = {10.2140/pjm.2000.193.419},
       URL = {https://doi.org/10.2140/pjm.2000.193.419},
}

@article {kislev-shelukhin-21,
    AUTHOR = {Kislev, A. and Shelukhin, E.},
     TITLE = {Bounds on spectral norms and barcodes},
   JOURNAL = {Geom. Topol.},
  FJOURNAL = {Geometry \& Topology},
    VOLUME = {25},
      YEAR = {2021},
    NUMBER = {7},
     PAGES = {3257--3350},
      ISSN = {1465-3060,1364-0380},
   MRCLASS = {57R17 (53D12 53D40 55N31)},
  MRNUMBER = {4372632},
MRREVIEWER = {J.\ Zhang},
       DOI = {10.2140/gt.2021.25.3257},
       URL = {https://doi.org/10.2140/gt.2021.25.3257},
}

@article {robbin-salamon-maslov-for-paths-93,
    AUTHOR = {Robbin, J. and Salamon, D.},
     TITLE = {The {M}aslov index for paths},
   JOURNAL = {Topology},
  FJOURNAL = {Topology. An International Journal of Mathematics},
    VOLUME = {32},
      YEAR = {1993},
    NUMBER = {4},
     PAGES = {827--844},
      ISSN = {0040-9383},
   MRCLASS = {58F05 (58E05)},
  MRNUMBER = {1241874},
MRREVIEWER = {J.\ S.\ J.},
       DOI = {10.1016/0040-9383(93)90052-W},
       URL = {https://doi.org/10.1016/0040-9383(93)90052-W},
}

@article {salamon-zehnder-92,
    AUTHOR = {Salamon, D. and Zehnder, E.},
     TITLE = {Morse theory for periodic solutions of {H}amiltonian systems
              and the {M}aslov index},
   JOURNAL = {Comm. Pure Appl. Math.},
  FJOURNAL = {Communications on Pure and Applied Mathematics},
    VOLUME = {45},
      YEAR = {1992},
    NUMBER = {10},
     PAGES = {1303--1360},
      ISSN = {0010-3640,1097-0312},
   MRCLASS = {58E05 (34C25 58F22 70H05)},
  MRNUMBER = {1181727},
MRREVIEWER = {Y-G.\ Oh},
       DOI = {10.1002/cpa.3160451004},
       URL = {https://doi.org/10.1002/cpa.3160451004},
}

@article {Gromov-kahlerhyp,
    AUTHOR = {Gromov, M.},
     TITLE = {K\"{a}hler hyperbolicity and {$L_2$}-{H}odge theory},
   JOURNAL = {J. Differential Geom.},
  FJOURNAL = {Journal of Differential Geometry},
    VOLUME = {33},
      YEAR = {1991},
    NUMBER = {1},
     PAGES = {263--292},}

@article{S-HS,
      title={A symplectic {H}ilbert-{S}mith conjecture}, 
      author={E. Shelukhin},
      note = {Preprint arXiv:2403.07195, 2024},
      journal = {arXiv:2403.07195},
      year = {2024},
}

@article{franks1988generalizations,
  title={Generalizations of the {P}oincar{\'e}-{B}irkhoff theorem},
  author={Franks, J.},
   JOURNAL = {Ann. of Math. (2)},
  FJOURNAL = {Annals of Mathematics. Second Series},

  volume={128},
  number={1},
  pages={139--151},
  year={1988},
  publisher={JSTOR}
}

@book {MR3674984,
    AUTHOR = {McDuff, D. and Salamon, D.},
     TITLE = {Introduction to symplectic topology},
    SERIES = {Oxford Graduate Texts in Mathematics},
   EDITION = {Third},
 PUBLISHER = {Oxford University Press, Oxford},
      YEAR = {2017},
     PAGES = {xi+623},
      ISBN = {978-0-19-879490-5; 978-0-19-879489-9},
   MRCLASS = {53D35 (53D40 57R17 57R57 57R58)},
  MRNUMBER = {3674984},
MRREVIEWER = {Hansj\"{o}rg\ Geiges},
       DOI = {10.1093/oso/9780198794899.001.0001},
       URL = {https://doi.org/10.1093/oso/9780198794899.001.0001},
}

@article {franks-handle-pp-ham-03,
    AUTHOR = {Franks, J. and Handel, M.},
     TITLE = {Periodic points of {H}amiltonian surface diffeomorphisms},
   JOURNAL = {Geom. Topol.},
  FJOURNAL = {Geometry and Topology},
    VOLUME = {7},
      YEAR = {2003},
     PAGES = {713--756},
      ISSN = {1465-3060,1364-0380},
   MRCLASS = {37J10 (37C25 37E30 37E45 57M60)},
  MRNUMBER = {2026545},
MRREVIEWER = {Eijirou\ Hayakawa},
       DOI = {10.2140/gt.2003.7.713},
       URL = {https://doi.org/10.2140/gt.2003.7.713},
}

@article {franks-handle-rot-vector-96,
    AUTHOR = {Franks, J.},
     TITLE = {Rotation vectors and fixed points of area preserving surface
              diffeomorphisms},
   JOURNAL = {Trans. Amer. Math. Soc.},
  FJOURNAL = {Transactions of the American Mathematical Society},
    VOLUME = {348},
      YEAR = {1996},
    NUMBER = {7},
     PAGES = {2637--2662},
      ISSN = {0002-9947,1088-6850},
   MRCLASS = {58F20},
  MRNUMBER = {1325916},
MRREVIEWER = {M.\ Hurley},
       DOI = {10.1090/S0002-9947-96-01502-4},
       URL = {https://doi.org/10.1090/S0002-9947-96-01502-4},
}

@article {qu-xia-24,
    AUTHOR = {Qu, H. and Xia, Z.},
     TITLE = {Action and periodic orbits of area-preserving diffeomorphisms},
   JOURNAL = {J. Differential Equations},
  FJOURNAL = {Journal of Differential Equations},
    VOLUME = {391},
      YEAR = {2024},
     PAGES = {246--264},
      ISSN = {0022-0396,1090-2732},
   MRCLASS = {37J40 (37E30 53D05)},
  MRNUMBER = {4703363},
MRREVIEWER = {Jo\~{a}o\ Lopes Dias},
       DOI = {10.1016/j.jde.2024.01.026},
       URL = {https://doi.org/10.1016/j.jde.2024.01.026},
}

@article {oh-duke,
    AUTHOR = {Oh, Y-G.},
     TITLE = {Spectral invariants, analysis of the {F}loer moduli space, and
              geometry of the {H}amiltonian diffeomorphism group},
   JOURNAL = {Duke Math. J.},
  FJOURNAL = {Duke Mathematical Journal},
    VOLUME = {130},
      YEAR = {2005},
    NUMBER = {2},
     PAGES = {199--295},}

@incollection {MR2884036,
    AUTHOR = {Entov, M. and Polterovich, L. and Py, P.},
     TITLE = {On continuity of quasimorphisms for symplectic maps},
 BOOKTITLE = {Perspectives in analysis, geometry, and topology},
    SERIES = {Progr. Math.},
    VOLUME = {296},
     PAGES = {169--197},
      NOTE = {With an appendix by M. Khanevsky},
 PUBLISHER = {Birkh\"{a}user/Springer, New York},
      YEAR = {2012},
      ISBN = {978-0-8176-8276-7},
   MRCLASS = {53D35},
  MRNUMBER = {2884036},
MRREVIEWER = {David\ F.\ Mart\'{\i}nez-Torres},
       DOI = {10.1007/978-0-8176-8277-4\{_}8}

@book {MR2167675,
    AUTHOR = {Krantz, S. G.},
     TITLE = {Geometric function theory},
    SERIES = {Cornerstones},
      NOTE = {Explorations in complex analysis},
 PUBLISHER = {Birkh\"{a}user Boston, Inc., Boston, MA},
      YEAR = {2006},
     PAGES = {xiv+314},
      ISBN = {978-0-8176-4339-3; 0-8176-4339-7},
   MRCLASS = {30-01 (46J15)},
  MRNUMBER = {2167675},
}

@article {MR614728,
    AUTHOR = {Epstein, D. B. A.},
     TITLE = {Prime ends},
   JOURNAL = {Proc. London Math. Soc. (3)},
  FJOURNAL = {Proceedings of the London Mathematical Society. Third Series},
    VOLUME = {42},
      YEAR = {1981},
    NUMBER = {3},
     PAGES = {385--414},
      ISSN = {0024-6115,1460-244X},
   MRCLASS = {30D40 (30C60 30G15 57N05)},
  MRNUMBER = {614728},
MRREVIEWER = {Yu.\ Yu.\ Trokhimchuk},
       DOI = {10.1112/plms/s3-42.3.385},
       URL = {https://doi.org/10.1112/plms/s3-42.3.385},
}

@book {MR2193309,
    AUTHOR = {Milnor, J.},
     TITLE = {Dynamics in one complex variable},
    SERIES = {Annals of Mathematics Studies},
    VOLUME = {160},
   EDITION = {Third},
 PUBLISHER = {Princeton University Press, Princeton, NJ},
      YEAR = {2006},
     PAGES = {viii+304},
      ISBN = {978-0-691-12488-9; 0-691-12488-4},
   MRCLASS = {37Fxx (30-01 30D05 37-01)},
  MRNUMBER = {2193309},
}

@article {caratheodory-original-paper,
    AUTHOR = {Carathéodory, C.},
     TITLE = {Über die gegenseitige Beziehung der Ränder bei der konformen Abbildung des Inneren einer Jordanschen Kurve auf einen Kreis},
   JOURNAL = {Mathematische Annalen},
  FJOURNAL = {Proceedings of the London Mathematical Society. Third Series},
    VOLUME = {73},
      YEAR = {1913},
    NUMBER = {2},
     PAGES = {305--320},
       DOI = {10.1007/BF01456720},
       URL = {https://doi.org/10.1007/BF01456720},
}

@article {MR3850107,
    AUTHOR = {Leclercq, R. and Zapolsky, F.},
     TITLE = {Spectral invariants for monotone {L}agrangians},
   JOURNAL = {J. Topol. Anal.},
  FJOURNAL = {Journal of Topology and Analysis},
    VOLUME = {10},
      YEAR = {2018},
    NUMBER = {3},
     PAGES = {627--700},
      ISSN = {1793-5253,1793-7167},
   MRCLASS = {57R17 (53D12 53D40)},
  MRNUMBER = {3850107},
MRREVIEWER = {Roman\ Golovko},
       DOI = {10.1142/S1793525318500267},
       URL = {https://doi.org/10.1142/S1793525318500267},
}

@book {MR3524783,
    AUTHOR = {Oh, Y-G.},
     TITLE = {Symplectic topology and {F}loer homology. {V}ol. 2},
    SERIES = {New Mathematical Monographs},
    VOLUME = {29},
      NOTE = {Floer homology and its applications},
 PUBLISHER = {Cambridge University Press, Cambridge},
      YEAR = {2015},
     PAGES = {xxiii+446},
      ISBN = {978-1-107-10967-4},
   MRCLASS = {53D40 (53Dxx 57R58)},
  MRNUMBER = {3524783},
MRREVIEWER = {Hansj\"{o}rg\ Geiges},
}

@article {MR4480149,
    AUTHOR = {Shelukhin, E.},
     TITLE = {Viterbo conjecture for {Z}oll symmetric spaces},
   JOURNAL = {Invent. Math.},
  FJOURNAL = {Inventiones Mathematicae},
    VOLUME = {230},
      YEAR = {2022},
    NUMBER = {1},
     PAGES = {321--373},
      ISSN = {0020-9910,1432-1297},
   MRCLASS = {53D40 (57R17)},
  MRNUMBER = {4480149},
MRREVIEWER = {Olga\ Bernardi},
       DOI = {10.1007/s00222-022-01124-x},
       URL = {https://doi.org/10.1007/s00222-022-01124-x},
}

@incollection {MR3220941,
    AUTHOR = {Auroux, D.},
     TITLE = {A beginner's introduction to {F}ukaya categories},
 BOOKTITLE = {Contact and symplectic topology},
    SERIES = {Bolyai Soc. Math. Stud.},
    VOLUME = {26},
     PAGES = {85--136},
 PUBLISHER = {J\'{a}nos Bolyai Math. Soc., Budapest},
      YEAR = {2014},
      ISBN = {978-3-319-02035-8; 978-3-319-02036-5},
   MRCLASS = {53D37 (53D40)},
  MRNUMBER = {3220941},
MRREVIEWER = {Nick\ Sheridan},
       DOI = {10.1007/978-3-319-02036-5\{_}3}

@article {conley-zehnder-83,
    AUTHOR = {Conley, C. C. and Zehnder, E.},
     TITLE = {The {B}irkhoff-{L}ewis fixed point theorem and a conjecture of
              {V}. {I}. {A}rnold},
   JOURNAL = {Invent. Math.},
  FJOURNAL = {Inventiones Mathematicae},
    VOLUME = {73},
      YEAR = {1983},
    NUMBER = {1},
     PAGES = {33--49},
      ISSN = {0020-9910,1432-1297},
   MRCLASS = {58F05 (58C30)},
  MRNUMBER = {707347},
MRREVIEWER = {Yu.\ E.\ Gliklikh},
       DOI = {10.1007/BF01393824},
       URL = {https://doi.org/10.1007/BF01393824},
}

@book {MR2441780,
    AUTHOR = {Seidel, P.},
     TITLE = {Fukaya categories and {P}icard-{L}efschetz theory},
    SERIES = {Zurich Lectures in Advanced Mathematics},
 PUBLISHER = {European Mathematical Society (EMS), Z\"{u}rich},
      YEAR = {2008},
     PAGES = {viii+326},
      ISBN = {978-3-03719-063-0},
   MRCLASS = {53D40 (16E45 32Q65 53D12)},
  MRNUMBER = {2441780},
MRREVIEWER = {Timothy\ Perutz},
       DOI = {10.4171/063},
       URL = {https://doi.org/10.4171/063},
}

@article {floer-morese-th-lag-int,
    AUTHOR = {Floer, A.},
     TITLE = {Morse theory for {L}agrangian intersections},
   JOURNAL = {J. Differential Geom.},
  FJOURNAL = {Journal of Differential Geometry},
    VOLUME = {28},
      YEAR = {1988},
    NUMBER = {3},
     PAGES = {513--547},
      ISSN = {0022-040X,1945-743X},
   MRCLASS = {58F05 (35J65 58E05)},
  MRNUMBER = {965228},
MRREVIEWER = {Jean-C.\ Sikorav},
       URL = {http://projecteuclid.org/euclid.jdg/1214442477},
}

@article {floer-symp-fix-pt,
    AUTHOR = {Floer, A.},
     TITLE = {Symplectic fixed points and holomorphic spheres},
   JOURNAL = {Comm. Math. Phys.},
  FJOURNAL = {Communications in Mathematical Physics},
    VOLUME = {120},
      YEAR = {1989},
    NUMBER = {4},
     PAGES = {575--611},
      ISSN = {0010-3616,1432-0916},
   MRCLASS = {58F05 (70H05)},
  MRNUMBER = {987770},
MRREVIEWER = {Y-G.\ Oh},
       URL = {http://projecteuclid.org/euclid.cmp/1104177909},
}

@incollection {MR1432464,
    AUTHOR = {Piunikhin, S. and Salamon, D. and Schwarz, M.},
     TITLE = {Symplectic {F}loer-{D}onaldson theory and quantum cohomology},
 BOOKTITLE = {Contact and symplectic geometry ({C}ambridge, 1994)},
    SERIES = {Publ. Newton Inst.},
    VOLUME = {8},
     PAGES = {171--200},
 PUBLISHER = {Cambridge Univ. Press, Cambridge},
      YEAR = {1996},
      ISBN = {0-521-57086-7},
   MRCLASS = {57R57 (58E05 58F05)},
  MRNUMBER = {1432464},
MRREVIEWER = {Vsevolod\ V.\ Shevchishin},
}

@article {MR2424172,
    AUTHOR = {Albers, P.},
     TITLE = {A {L}agrangian {P}iunikhin-{S}alamon-{S}chwarz morphism and
              two comparison homomorphisms in {F}loer homology},
   JOURNAL = {Int. Math. Res. Not. IMRN},
  FJOURNAL = {International Mathematics Research Notices. IMRN},
      YEAR = {2008},
     PAGES = {Art. ID rnm134, 56},
      ISSN = {1073-7928,1687-0247},
   MRCLASS = {53D40 (53D12)},
  MRNUMBER = {2424172},
MRREVIEWER = {M.\ J.\ Usher},
       DOI = {10.1093/imrn/rnm134},
       URL = {https://doi.org/10.1093/imrn/rnm134},
}

@article {MR2122744,
    AUTHOR = {Kati\'{c}, J. and Milinkovi\'{c}, D.},
     TITLE = {Piunikhin-{S}alamon-{S}chwarz isomorphisms for {L}agrangian
              intersections},
   JOURNAL = {Differential Geom. Appl.},
  FJOURNAL = {Differential Geometry and its Applications},
    VOLUME = {22},
      YEAR = {2005},
    NUMBER = {2},
     PAGES = {215--227},
      ISSN = {0926-2245,1872-6984},
   MRCLASS = {53D40 (53D12)},
  MRNUMBER = {2122744},
MRREVIEWER = {Cheol-Hyun\ Cho},
       DOI = {10.1016/j.difgeo.2004.10.008},
       URL = {https://doi.org/10.1016/j.difgeo.2004.10.008},
}

@article {MR1157321,
    AUTHOR = {Viterbo, C.},
     TITLE = {Symplectic topology as the geometry of generating functions},
   JOURNAL = {Math. Ann.},
  FJOURNAL = {Mathematische Annalen},
    VOLUME = {292},
      YEAR = {1992},
    NUMBER = {4},
     PAGES = {685--710},
      ISSN = {0025-5831,1432-1807},
   MRCLASS = {58F05 (53C15 53C23 57R50 58E05)},
  MRNUMBER = {1157321},
MRREVIEWER = {Nikolai\ K.\ Smolentsev},
       DOI = {10.1007/BF01444643},
       URL = {https://doi.org/10.1007/BF01444643},
}

@incollection {MR2103018,
    AUTHOR = {Oh, Y-G.},
     TITLE = {Construction of spectral invariants of {H}amiltonian paths on
              closed symplectic manifolds},
 BOOKTITLE = {The breadth of symplectic and {P}oisson geometry},
    SERIES = {Progr. Math.},
    VOLUME = {232},
     PAGES = {525--570},
 PUBLISHER = {Birkh\"{a}user Boston, Boston, MA},
      YEAR = {2005},
      ISBN = {0-8176-3565-3},
   MRCLASS = {53D40 (53D45 58E30)},
  MRNUMBER = {2103018},
MRREVIEWER = {D.\ Milinkovi\'{c}},
       DOI = {10.1007/0-8176-4419-9\{_}18}

@article {MR1979584,
    AUTHOR = {Entov, M. and Polterovich, L.},
     TITLE = {Calabi quasimorphism and quantum homology},
   JOURNAL = {Int. Math. Res. Not.},
  FJOURNAL = {International Mathematics Research Notices},
      YEAR = {2003},
     PAGES = {1635--1676},
      ISSN = {1073-7928,1687-0247},
   MRCLASS = {53D35 (53D40 53D45)},
  MRNUMBER = {1979584},
MRREVIEWER = {Ignasi\ Mundet-Riera},
       DOI = {10.1155/S1073792803210011},
       URL = {https://doi.org/10.1155/S1073792803210011},
}

@article{usher-bound-depth-13,
    AUTHOR = {Usher, M.},
     TITLE = {Hofer's metrics and boundary depth},
   JOURNAL = {Ann. Sci. \'{E}c. Norm. Sup\'{e}r. (4)},
  FJOURNAL = {Annales Scientifiques de l'\'{E}cole Normale Sup\'{e}rieure.
              Quatri\`eme S\'{e}rie},
    VOLUME = {46},
      YEAR = {2013},
    NUMBER = {1},
     PAGES = {57--128 (2013)},
      ISSN = {0012-9593,1873-2151},
   MRCLASS = {53D40 (53D12 53D35)},
  MRNUMBER = {3087390},
MRREVIEWER = {David\ E.\ Hurtubise},
       DOI = {10.24033/asens.2185},
       URL = {https://doi.org/10.24033/asens.2185},
}

@article{polterovich-shelukhin-16,
  title={Autonomous Hamiltonian flows, {H}ofer’s geometry and persistence modules},
  author={Polterovich, L. and Shelukhin, E.},
  journal={Selecta Mathematica},
  volume={22},
  number={1},
  pages={227--296},
  year={2016},
  publisher={Springer}
}

@article{polterovich-shelukhin-Stojisavljevic-17,
    AUTHOR = {Polterovich, L. and Shelukhin, E. and
              Stojisavljevi\'{c}, V.},
     TITLE = {Persistence modules with operators in {M}orse and {F}loer
              theory},
   JOURNAL = {Mosc. Math. J.},
  FJOURNAL = {Moscow Mathematical Journal},
    VOLUME = {17},
      YEAR = {2017},
    NUMBER = {4},
     PAGES = {757--786},}

@article {Zomorodian-Carlsson-05,
    AUTHOR = {Zomorodian, A. and Carlsson, G.},
     TITLE = {Computing persistent homology},
   JOURNAL = {Discrete Comput. Geom.},
  FJOURNAL = {Discrete \& Computational Geometry. An International Journal
              of Mathematics and Computer Science},
    VOLUME = {33},
      YEAR = {2005},
    NUMBER = {2},
     PAGES = {249--274},
      ISSN = {0179-5376},
   MRCLASS = {55N99 (55-04 55U10)},
  MRNUMBER = {2121296},
MRREVIEWER = {Donald M. Davis},
       DOI = {10.1007/s00454-004-1146-y},
       URL = {https://doi.org/10.1007/s00454-004-1146-y},
}

@article {franks-handel-zero-ent,
    AUTHOR = {Franks, J. and Handel, M.},
     TITLE = {Entropy zero area preserving diffeomorphisms of {$S^2$}},
   JOURNAL = {Geom. Topol.},
  FJOURNAL = {Geometry \& Topology},
    VOLUME = {16},
      YEAR = {2012},
    NUMBER = {4},
     PAGES = {2187--2284},
      ISSN = {1465-3060,1364-0380},
   MRCLASS = {37C05 (37C85)},
  MRNUMBER = {3033517},
       DOI = {10.2140/gt.2012.16.2187},
       URL = {https://doi.org/10.2140/gt.2012.16.2187},
}

@article {MR5026815,
    AUTHOR = {Senior, D. B. and Le Calvez, P. and Pirnapasov,
              A.},
     TITLE = {The asymptotic mean action and the asymptotic linking number
              for pseudo-rotations},
   JOURNAL = {Math. Z.},
  FJOURNAL = {Mathematische Zeitschrift},
    VOLUME = {312},
      YEAR = {2026},
    NUMBER = {2},
     PAGES = {Paper No. 65, 23},
      ISSN = {0025-5874,1432-1823},
   MRCLASS = {37E30 (37A25 37C86 37E45)},
  MRNUMBER = {5026815},
       DOI = {10.1007/s00209-026-03944-6},
       URL = {https://doi.org/10.1007/s00209-026-03944-6},
}

@article {MR3548467,
    AUTHOR = {Humili\`ere, V. and Le Roux, F. and
              Seyfaddini, S.},
     TITLE = {Towards a dynamical interpretation of {H}amiltonian spectral
              invariants on surfaces},
   JOURNAL = {Geom. Topol.},
  FJOURNAL = {Geometry \& Topology},
    VOLUME = {20},
      YEAR = {2016},
    NUMBER = {4},
     PAGES = {2253--2334},
      ISSN = {1465-3060,1364-0380},
   MRCLASS = {53D40 (37E30 53Dxx)},
  MRNUMBER = {3548467},
MRREVIEWER = {Chun-Gen\ Liu},
       DOI = {10.2140/gt.2016.20.2253},
       URL = {https://doi.org/10.2140/gt.2016.20.2253},
}

@article {MR5018404,
    AUTHOR = {Alizadeh, H. and Atallah, M. S. and Cant, D.},
     TITLE = {The spectral diameter of a symplectic ellipsoid},
   JOURNAL = {J. Symplectic Geom.},
  FJOURNAL = {The Journal of Symplectic Geometry},
    VOLUME = {23},
      YEAR = {2025},
    NUMBER = {5},
     PAGES = {995--1024},
      ISSN = {1527-5256,1540-2347},
   MRCLASS = {53D05},
  MRNUMBER = {5018404},
       DOI = {10.4310/jsg.251217110546},
       URL = {https://doi.org/10.4310/jsg.251217110546},
}

@book {MR648108,
    AUTHOR = {Walters, P.},
     TITLE = {An introduction to ergodic theory},
    SERIES = {Graduate Texts in Mathematics},
    VOLUME = {79},
 PUBLISHER = {Springer-Verlag, New York-Berlin},
      YEAR = {1982},
     PAGES = {ix+250},
      ISBN = {0-387-90599-5},
   MRCLASS = {28Dxx (54H20 58F11)},
  MRNUMBER = {648108},
MRREVIEWER = {M.\ A.\ Akcoglu},
}

@book{mirzakhani-ergodic-theory-14,
    AUTHOR = {Mirzakhani, M.},
     TITLE = {Introduction to ergodic theory},
 note ={Available at \url{https://math.berkeley.edu/~fengt/ergodic_theory.pdf}},
      YEAR = {2014},
     PAGES = {1-56},
     publisher = {notes by T. Feng},
}

@article {ginzburg-gurel-hyperbolic,
    AUTHOR = {Ginzburg, V. L. and G\"{u}rel, B. Z.},
     TITLE = {Hyperbolic fixed points and periodic orbits of {H}amiltonian
              diffeomorphisms},
   JOURNAL = {Duke Math. J.},
  FJOURNAL = {Duke Mathematical Journal},
    VOLUME = {163},
      YEAR = {2014},
    NUMBER = {3},
     PAGES = {565--590},
      ISSN = {0012-7094,1547-7398},
   MRCLASS = {53D40 (37J10 70H12)},
  MRNUMBER = {3165423},
MRREVIEWER = {Cristian\ L\u{a}zureanu},
       DOI = {10.1215/00127094-2410433},
       URL = {https://doi.org/10.1215/00127094-2410433},
}

@article {fish-11-gromov-comp,
    AUTHOR = {Fish, J. W.},
     TITLE = {Target-local {G}romov compactness},
   JOURNAL = {Geom. Topol.},
  FJOURNAL = {Geometry \& Topology},
    VOLUME = {15},
      YEAR = {2011},
    NUMBER = {2},
     PAGES = {765--826},
      ISSN = {1465-3060,1364-0380},
   MRCLASS = {53D45 (32Q65 53C15)},
  MRNUMBER = {2800366},
MRREVIEWER = {Antoine\ F. R. Gournay},
       DOI = {10.2140/gt.2011.15.765},
       URL = {https://doi.org/10.2140/gt.2011.15.765},
}

@article {franks-recurrence-88,
    AUTHOR = {Franks, J.},
     TITLE = {Recurrence and fixed points of surface homeomorphisms},
   JOURNAL = {Ergodic Theory Dynam. Systems},
  FJOURNAL = {Ergodic Theory and Dynamical Systems},
    VOLUME = {8$^*$},
      YEAR = {1988},
    NUMBER = {Charles Conley Memorial Issue},
     PAGES = {99--107},
      ISSN = {0143-3857,1469-4417},
   MRCLASS = {58F20},
  MRNUMBER = {967632},
MRREVIEWER = {Morton\ Brown},
       DOI = {10.1017/S0143385700009366},
       URL = {https://doi.org/10.1017/S0143385700009366},
}

@incollection {mather-top-proof-82,
    AUTHOR = {Mather, J. N.},
     TITLE = {Topological proofs of some purely topological consequences of {C}arath\'{e}odory's theory of prime ends},
 BOOKTITLE = {Selected studies: physics-astrophysics, mathematics, history
              of science},
     PAGES = {225--255},
 PUBLISHER = {North-Holland, Amsterdam-New York},
      YEAR = {1982},
      ISBN = {0-444-86161-0},
   MRCLASS = {57N05 (30C80 57N50 58F11)},
  MRNUMBER = {662863},
}

@article{ian_thesis,
  title={Uniform approximation of a volume-preserving homeomorphisms by volume-preserving diffeomorphisms},
  author={Leclaire, I.},
  journal={M.Sc. thesis, Universit\'{e} de Montr\'{e}al},
  year={2026},
}

@article {banyaga-moser-comp-mnflds,
    AUTHOR = {Banyaga, A.},
     TITLE = {Formes-volume sur les vari\'{e}t\'{e}s \`a bord},
   JOURNAL = {Enseign. Math. (2)},
  FJOURNAL = {L'Enseignement Math\'{e}matique. Revue Internationale. 2e
              S\'{e}rie},
    VOLUME = {20},
      YEAR = {1974},
     PAGES = {127--131},
      ISSN = {0013-8584},
   MRCLASS = {53C65 (58A10)},
  MRNUMBER = {358649},
MRREVIEWER = {R.\ S.\ Palais},
}

@incollection {MR1702944,
    AUTHOR = {Salamon, D.},
     TITLE = {Lectures on {F}loer homology},
 BOOKTITLE = {Symplectic geometry and topology ({P}ark {C}ity, {UT}, 1997)},
    SERIES = {IAS/Park City Math. Ser.},
    VOLUME = {7},
     PAGES = {143--229},
 PUBLISHER = {Amer. Math. Soc., Providence, RI},
      YEAR = {1999},
      ISBN = {0-8218-0838-9},
   MRCLASS = {53D40 (37J45 53D45 57R17 57R58)},
  MRNUMBER = {1702944},
MRREVIEWER = {David\ E.\ Hurtubise},
       DOI = {10.1016/S0165-2427(99)00127-0},
       URL = {https://doi.org/10.1016/S0165-2427(99)00127-0},
}

@article {MR4587314,
    AUTHOR = {Ganor, Y. and Tanny, S.},
     TITLE = {Floer theory of disjointly supported {H}amiltonians on
              symplectically aspherical manifolds},
   JOURNAL = {Algebr. Geom. Topol.},
  FJOURNAL = {Algebraic \& Geometric Topology},
    VOLUME = {23},
      YEAR = {2023},
    NUMBER = {2},
     PAGES = {645--732},
      ISSN = {1472-2747,1472-2739},
   MRCLASS = {53D40 (57R17 57R58)},
  MRNUMBER = {4587314},
       DOI = {10.2140/agt.2023.23.645},
       URL = {https://doi.org/10.2140/agt.2023.23.645},
}

@article{acgps-26,
  title={On the topology of the mean action spectrum},
  author={Alizadeh, H. and Cristofaro-Gardiner, D. and Prinapasov, A. and Shelukhin, E.},
  journal={In preparation},
  year={2026}
}

@article {MR2547825,
    AUTHOR = {K\c{e}dra, J.},
     TITLE = {Symplectically hyperbolic manifolds},
   JOURNAL = {Differential Geom. Appl.},
  FJOURNAL = {Differential Geometry and its Applications},
    VOLUME = {27},
      YEAR = {2009},
    NUMBER = {4},
     PAGES = {455--463},
      ISSN = {0926-2245,1872-6984},
   MRCLASS = {53D35 (53D05)},
  MRNUMBER = {2547825},
MRREVIEWER = {Tobias\ Ekholm},
       DOI = {10.1016/j.difgeo.2009.01.006},
       URL = {https://doi.org/10.1016/j.difgeo.2009.01.006},
}

@article {MR4696270,
    AUTHOR = {Brunnbauer, M. and Kotschick, D. and Sch\"{o}nlinner, L.},
     TITLE = {On atoroidal and hyperbolic cohomology classes},
   JOURNAL = {Topology Appl.},
  FJOURNAL = {Topology and its Applications},
    VOLUME = {344},
      YEAR = {2024},
     PAGES = {Paper No. 108830, 9},
      ISSN = {0166-8641,1879-3207},
   MRCLASS = {53C23 (20F65 57K43 58A12)},
  MRNUMBER = {4696270},
       DOI = {10.1016/j.topol.2024.108830},
       URL = {https://doi.org/10.1016/j.topol.2024.108830},
}

@article{alizadeh2025spectral,
  title={On the spectral diameter of the {G}rassmannians},
  author={Alizadeh, H. and Atallah, M. S and Cant, D. and Shang, Jianqiao},
  journal={arXiv preprint arXiv:2512.20125},
  year={2025}
}

@article{da2003symplectic,
  title={Symplectic toric manifolds},
  author={Da Silva, A Cannas},
  journal={Symplectic geometry of integrable Hamiltonian systems (Barcelona, 2001), Adv. Courses Math. CRM Barcelona},
  pages={85--173},
  year={2003}
}

@book {MR4605050,
    AUTHOR = {Evans, J.},
     TITLE = {Lectures on {L}agrangian torus fibrations},
    SERIES = {London Mathematical Society Student Texts},
    VOLUME = {105},
 PUBLISHER = {Cambridge University Press, Cambridge},
      YEAR = {2023},
     PAGES = {xiii+225},
      ISBN = {978-1-009-37262-6; 978-1-009-37263-3},
   MRCLASS = {53D12 (53D20 57S12)},
  MRNUMBER = {4605050},
       DOI = {10.1017/9781009372671},
       URL = {https://doi.org/10.1017/9781009372671},
}

@article {MR809718,
    AUTHOR = {Gromov, M.},
     TITLE = {Pseudo holomorphic curves in symplectic manifolds},
   JOURNAL = {Invent. Math.},
  FJOURNAL = {Inventiones Mathematicae},
    VOLUME = {82},
      YEAR = {1985},
    NUMBER = {2},
     PAGES = {307--347},
      ISSN = {0020-9910,1432-1297},
   MRCLASS = {53C15 (32F25 53C57 57R15)},
  MRNUMBER = {809718},
MRREVIEWER = {Yakov\ Eliashberg},
       DOI = {10.1007/BF01388806},
       URL = {https://doi.org/10.1007/BF01388806},
}

@article {MR3123670,
    AUTHOR = {Seyfaddini, S.},
     TITLE = {{$C^0$}-limits of {H}amiltonian paths and the {O}h-{S}chwarz
              spectral invariants},
   JOURNAL = {Int. Math. Res. Not. IMRN},
  FJOURNAL = {International Mathematics Research Notices. IMRN},
      YEAR = {2013},
     PAGES = {4920--4960},
      ISSN = {1073-7928,1687-0247},
   MRCLASS = {53D40 (53D45 58E30)},
  MRNUMBER = {3123670},
MRREVIEWER = {Karl\ Friedrich\ Siburg},
       DOI = {10.1093/imrn/rns191},
       URL = {https://doi.org/10.1093/imrn/rns191},
}

@article {MR5031842,
    AUTHOR = {Alizadeh, H.},
     TITLE = {Hamiltonian fragmentation in dimension four with application
              to spectral estimators},
   JOURNAL = {J. Topol. Anal.},
  FJOURNAL = {Journal of Topology and Analysis},
    VOLUME = {18},
      YEAR = {2026},
    NUMBER = {1},
     PAGES = {221--241},
      ISSN = {1793-5253,1793-7167},
   MRCLASS = {53D05},
  MRNUMBER = {5031842},
       DOI = {10.1142/S1793525324500298},
       URL = {https://doi.org/10.1142/S1793525324500298},
}

@article {MR682312,
    AUTHOR = {Katznelson, Y. and Weiss, B.},
     TITLE = {A simple proof of some ergodic theorems},
   JOURNAL = {Israel J. Math.},
  FJOURNAL = {Israel Journal of Mathematics},
    VOLUME = {42},
      YEAR = {1982},
    NUMBER = {4},
     PAGES = {291--296},}

@article {MR1567747,
    AUTHOR = {Franks, J.},
     TITLE = {Book {R}eview: {E}rgodic theory and differentiable dynamics},
   JOURNAL = {Bull. Amer. Math. Soc. (N.S.)},
  FJOURNAL = {American Mathematical Society. Bulletin. New Series},
    VOLUME = {20},
      YEAR = {1989},
    NUMBER = {2},
     PAGES = {190--193},
      ISSN = {0273-0979,1088-9485},
   MRCLASS = {DML},
  MRNUMBER = {1567747},
       DOI = {10.1090/S0273-0979-1989-15759-5},
       URL = {https://doi.org/10.1090/S0273-0979-1989-15759-5},
}

@book {MR832433,
    AUTHOR = {Cornfeld, I. P. and Fomin, S. V. and Sina\u{\i}, Ya. G.},
     TITLE = {Ergodic theory},
    SERIES = {Grundlehren der mathematischen Wissenschaften [Fundamental
              Principles of Mathematical Sciences]},
    VOLUME = {245},
      NOTE = {Translated from the Russian by A. B. Sosinski\u{\i}},
 PUBLISHER = {Springer-Verlag, New York},
      YEAR = {1982},
     PAGES = {x+486},
      ISBN = {0-387-90580-4},
   MRCLASS = {28D05 (54H20 58F11)},
  MRNUMBER = {832433},
MRREVIEWER = {D.\ Newton},
       DOI = {10.1007/978-1-4615-6927-5},
       URL = {https://doi.org/10.1007/978-1-4615-6927-5},
}

@article {poincare-85,
    AUTHOR = {Poincaré, H.},
     TITLE = {Sur les courbes définies par les équations différentielles},
   JOURNAL = {Journal de Mathématiques Pures et Appliquées, Serie 4},
    VOLUME = {1},
      YEAR = {1885},
     PAGES = {167--244},
       URL = {https://www.numdam.org/item/JMPA_1885_4_1__167_0},
}

@article {katok-89,
    AUTHOR = {Katok, A.},
     TITLE = {Lyapunov exponents, entropy and periodic orbits for
              diffeomorphisms},
   JOURNAL = {Inst. Hautes \'{E}tudes Sci. Publ. Math.},
  FJOURNAL = {Institut des Hautes \'{E}tudes Scientifiques. Publications
              Math\'{e}matiques},
      YEAR = {1980},
     PAGES = {137--173},
      ISSN = {0073-8301,1618-1913},
   MRCLASS = {28D20 (58F11 58F15)},
  MRNUMBER = {573822},
MRREVIEWER = {R.\ L.\ Adler},
       URL = {http://www.numdam.org/item?id=PMIHES_1980__51__137_0},
}

@inproceedings {smale-68,
    AUTHOR = {Smale, S.},
     TITLE = {Differentiable dynamical systems},
 BOOKTITLE = {Proc. {I}nternat. {C}ongr. {M}ath. ({M}oscow, 1966)},
     PAGES = {p. 139},
 PUBLISHER = {Izdat. ``Mir'', Moscow},
      YEAR = {1968},
   MRCLASS = {57.50},
  MRNUMBER = {233380},
}

@article {MR2339286,
    AUTHOR = {B\'{e}guin, F. and Crovisier, S. and Le Roux,
              F.},
     TITLE = {Construction of curious minimal uniquely ergodic homeomorphisms on manifolds: the {D}enjoy-{R}ees technique},
   JOURNAL = {Ann. Sci. \'{E}cole Norm. Sup. (4)},
  FJOURNAL = {Annales Scientifiques de l'\'{E}cole Normale Sup\'{e}rieure.
              Quatri\`eme S\'{e}rie},
    VOLUME = {40},
      YEAR = {2007},
    NUMBER = {2},
     PAGES = {251--308},
      ISSN = {0012-9593},
   MRCLASS = {37B05 (37A25 37B40)},
  MRNUMBER = {2339286},
MRREVIEWER = {Konstantin\ Medynets},
       DOI = {10.1016/j.ansens.2007.01.001},
       URL = {https://doi.org/10.1016/j.ansens.2007.01.001},
}

@article {MR616561,
    AUTHOR = {Rees, M.},
     TITLE = {A minimal positive entropy homeomorphism of the {$2$}-torus},
   JOURNAL = {J. London Math. Soc. (2)},
  FJOURNAL = {Journal of the London Mathematical Society. Second Series},
    VOLUME = {23},
      YEAR = {1981},
    NUMBER = {3},
     PAGES = {537--550},
      ISSN = {0024-6107,1469-7750},
   MRCLASS = {58F99 (28D20 58F11)},
  MRNUMBER = {616561},
MRREVIEWER = {Carmen\ Chicone},
       DOI = {10.1112/jlms/s2-23.3.537},
       URL = {https://doi.org/10.1112/jlms/s2-23.3.537},
}

@article {LeCalvez-Tal-forcing-theory-18,
    AUTHOR = {Le Calvez, P. and Tal, F. A.},
     TITLE = {Forcing theory for transverse trajectories of surface
              homeomorphisms},
   JOURNAL = {Invent. Math.},
  FJOURNAL = {Inventiones Mathematicae},
    VOLUME = {212},
      YEAR = {2018},
    NUMBER = {2},
     PAGES = {619--729},
      ISSN = {0020-9910,1432-1297},
   MRCLASS = {37E30 (37C85 55M20)},
  MRNUMBER = {3787834},
MRREVIEWER = {Olga\ Lukina},
       DOI = {10.1007/s00222-017-0773-x},
       URL = {https://doi.org/10.1007/s00222-017-0773-x},
}

@article {shelukhin-integ-wind-15,
    AUTHOR = {Shelukhin, E.},
     TITLE = {``{E}nlacements asymptotiques'' revisited},
   JOURNAL = {Ann. Math. Qu\'{e}.},
  FJOURNAL = {Annales Math\'{e}matiques du Qu\'{e}bec},
    VOLUME = {39},
      YEAR = {2015},
    NUMBER = {2},
     PAGES = {205--208},
      ISSN = {2195-4755,2195-4763},
   MRCLASS = {53D22},
  MRNUMBER = {3418805},
MRREVIEWER = {Pierre\ Py},
       DOI = {10.1007/s40316-015-0035-5},
       URL = {https://doi.org/10.1007/s40316-015-0035-5},
}

@article {MR1043272,
    AUTHOR = {Newhouse, S. E.},
     TITLE = {Corrections to: ``{C}ontinuity properties of entropy'' [{A}nn.
              of {M}ath. (2) {\bf 129} (1989), no. 2, 215--235; {MR}0986792
              (90f:58108)]},
   JOURNAL = {Ann. of Math. (2)},
  FJOURNAL = {Annals of Mathematics. Second Series},
    VOLUME = {131},
      YEAR = {1990},
    NUMBER = {2},
     PAGES = {409--410},
      ISSN = {0003-486X,1939-8980},
   MRCLASS = {58F11 (28D20 54C70)},
  MRNUMBER = {1043272},
       DOI = {10.2307/1971497},
       URL = {https://doi.org/10.2307/1971497},
}

@article {GLCPT-25,
    AUTHOR = {Guihéneuf, P-A and Le Calvez, P. and Passeggi, A. and Tal, F. A.},
     TITLE = {Zero entropy measure preserving surface homeomorphisms},
   JOURNAL = {In preparation},
      YEAR = {2026},
}

@article{alizadeh-shelukhin-locality-26,
    AUTHOR = {Alizadeh, H. and Shelukhin, E.},
     TITLE = {Locality of spectral invariants in monotone symplectic manifolds},
   JOURNAL = {In preparation},
      YEAR = {2026},
}

@article {hempel-72,
    AUTHOR = {Hempel, J.},
     TITLE = {Residual finiteness of surface groups},
   JOURNAL = {Proc. Amer. Math. Soc.},
  FJOURNAL = {Proceedings of the American Mathematical Society},
    VOLUME = {32},
      YEAR = {1972},
     PAGES = {323},
      ISSN = {0002-9939,1088-6826},
   MRCLASS = {55A05 (20J99)},
  MRNUMBER = {295352},
MRREVIEWER = {B.\ Maskit},
       DOI = {10.2307/2038357},
       URL = {https://doi.org/10.2307/2038357},
}

@article {MR4413744,
    AUTHOR = {Shelukhin, E.},
     TITLE = {On the {H}ofer-{Z}ehnder conjecture},
   JOURNAL = {Ann. of Math. (2)},
  FJOURNAL = {Annals of Mathematics. Second Series},
    VOLUME = {195},
      YEAR = {2022},
    NUMBER = {3},
     PAGES = {775--839},
      ISSN = {0003-486X,1939-8980},
   MRCLASS = {53D40},
  MRNUMBER = {4413744},
MRREVIEWER = {Roman\ Golovko},
       DOI = {10.4007/annals.2022.195.3.1},
       URL = {https://doi.org/10.4007/annals.2022.195.3.1},
}

@article {MR4822930,
    AUTHOR = {Shelukhin, E. and Tonkonog, D. and Vianna, R.},
     TITLE = {Geometry of symplectic flux and {L}agrangian torus fibrations},
   JOURNAL = {J. Topol.},
  FJOURNAL = {Journal of Topology},
    VOLUME = {17},
      YEAR = {2024},
    NUMBER = {4},
     PAGES = {Paper No. e70002, 56},
      ISSN = {1753-8416,1753-8424},
   MRCLASS = {53D12 (53D40)},
  MRNUMBER = {4822930},
MRREVIEWER = {Alexander\ Fel\cprime shtyn},
       DOI = {10.1112/topo.70002},
       URL = {https://doi.org/10.1112/topo.70002},
}

@article {MR984900,
    AUTHOR = {Delzant, T.},
     TITLE = {Hamiltoniens p\'{e}riodiques et images convexes de
              l'application moment},
   JOURNAL = {Bull. Soc. Math. France},
  FJOURNAL = {Bulletin de la Soci\'{e}t\'{e} Math\'{e}matique de France},
    VOLUME = {116},
      YEAR = {1988},
    NUMBER = {3},
     PAGES = {315--339},
      ISSN = {0037-9484},
   MRCLASS = {58F05},
  MRNUMBER = {984900},
MRREVIEWER = {J.\ J.\ Duistermaat},
       URL = {http://www.numdam.org/item?id=BSMF_1988__116_3_315_0},
}

@article {MR2507748,
    AUTHOR = {Entov, M. and Polterovich, L.},
     TITLE = {Rigid subsets of symplectic manifolds},
   JOURNAL = {Compos. Math.},
  FJOURNAL = {Compositio Mathematica},
    VOLUME = {145},
      YEAR = {2009},
    NUMBER = {3},
     PAGES = {773--826},
      ISSN = {0010-437X,1570-5846},
   MRCLASS = {53D40 (53D12 53D35)},
  MRNUMBER = {2507748},
MRREVIEWER = {Martin\ Pinsonnault},
       DOI = {10.1112/S0010437X0900400X},
       URL = {https://doi.org/10.1112/S0010437X0900400X},
}

@incollection {MR2768658,
    AUTHOR = {McDuff, D.},
     TITLE = {Displacing {L}agrangian toric fibers via probes},
 BOOKTITLE = {Low-dimensional and symplectic topology},
    SERIES = {Proc. Sympos. Pure Math.},
    VOLUME = {82},
     PAGES = {131--160},
 PUBLISHER = {Amer. Math. Soc., Providence, RI},
      YEAR = {2011},
      ISBN = {978-0-8218-5235-4},
   MRCLASS = {53D12 (52B12 53D20 53D35)},
  MRNUMBER = {2768658},
MRREVIEWER = {Mark\ Alan\ Branson},
       DOI = {10.1090/pspum/082/2768658},
       URL = {https://doi.org/10.1090/pspum/082/2768658},
}

@article {lecalvez-tal-forcing-theory,
    AUTHOR = {Le Calvez, P. and Tal, F. A.},
     TITLE = {Forcing theory for transverse trajectories of surface
              homeomorphisms},
   JOURNAL = {Invent. Math.},
  FJOURNAL = {Inventiones Mathematicae},
    VOLUME = {212},
      YEAR = {2018},
    NUMBER = {2},
     PAGES = {619--729},
      ISSN = {0020-9910,1432-1297},
   MRCLASS = {37E30 (37C85 55M20)},
  MRNUMBER = {3787834},
MRREVIEWER = {Olga\ Lukina},
       DOI = {10.1007/s00222-017-0773-x},
       URL = {https://doi.org/10.1007/s00222-017-0773-x},
}

@article{cggm-topics,
      title={Topics in Symplectic Dynamics: Barcode Entropy}, 
      author={E. \c{C}ineli and V. L. Ginzburg and B. Z. G\"{u}rel and M. Mazzucchelli},
      note={Preprint arXiv:2605.25965, 2026},
      JOURNAL = {arXiv:2605.25965},
      YEAR = {2026},
}

@article {MR1101087,
    AUTHOR = {Llibre, J. and MacKay, R. S.},
     TITLE = {Rotation vectors and entropy for homeomorphisms of the torus
              isotopic to the identity},
   JOURNAL = {Ergodic Theory Dynam. Systems},
  FJOURNAL = {Ergodic Theory and Dynamical Systems},
    VOLUME = {11},
      YEAR = {1991},
    NUMBER = {1},
     PAGES = {115--128},
      ISSN = {0143-3857,1469-4417},
   MRCLASS = {58F20},
  MRNUMBER = {1101087},
MRREVIEWER = {Hu\ Yi\ Hu},
       DOI = {10.1017/S0143385700006040},
       URL = {https://doi.org/10.1017/S0143385700006040},
}

@article {Guelman-Koropecki-Tal-14,
    AUTHOR = {Guelman, N. and Koropecki, A. and Tal, F. A.},
     TITLE = {A characterization of annularity for area-preserving toral
              homeomorphisms},
   JOURNAL = {Math. Z.},
  FJOURNAL = {Mathematische Zeitschrift},
    VOLUME = {276},
      YEAR = {2014},
    NUMBER = {3-4},
     PAGES = {673--689},
      ISSN = {0025-5874,1432-1823},
   MRCLASS = {37E30 (57N05)},
  MRNUMBER = {3175156},
MRREVIEWER = {Saadet\ \"{O}yk\"{u}\ Yurtta\c{s}},
       DOI = {10.1007/s00209-013-1218-x},
       URL = {https://doi.org/10.1007/s00209-013-1218-x},
}

@article {MR1507223,
    AUTHOR = {Montgomery, D.},
     TITLE = {Pointwise {P}eriodic {H}omeomorphisms},
   JOURNAL = {Amer. J. Math.},
  FJOURNAL = {American Journal of Mathematics},
    VOLUME = {59},
      YEAR = {1937},
    NUMBER = {1},
     PAGES = {118--120},
      ISSN = {0002-9327,1080-6377},
   MRCLASS = {DML},
  MRNUMBER = {1507223},
       DOI = {10.2307/2371565},
       URL = {https://doi.org/10.2307/2371565},
}

@article {MR4650154,
    AUTHOR = {Conejeros, J. and Tal, F. A.},
     TITLE = {Applications of forcing theory to homeomorphisms of the closed
              annulus},
   JOURNAL = {Ann. Sci. \'{E}c. Norm. Sup\'{e}r. (4)},
  FJOURNAL = {Annales Scientifiques de l'\'{E}cole Normale Sup\'{e}rieure.
              Quatri\`eme S\'{e}rie},
    VOLUME = {56},
      YEAR = {2023},
    NUMBER = {4},
     PAGES = {1155--1197},
      ISSN = {0012-9593,1873-2151},
   MRCLASS = {37E30 (37C86 37E45)},
  MRNUMBER = {4650154},
MRREVIEWER = {Luis\ Hern\'{a}ndez-Corbato},
       DOI = {10.24033/asens.2552},
       URL = {https://doi.org/10.24033/asens.2552},
}

@book {MR2850125,
    AUTHOR = {Farb, Benson and Margalit, Dan},
     TITLE = {A primer on mapping class groups},
    SERIES = {Princeton Mathematical Series},
    VOLUME = {49},
 PUBLISHER = {Princeton University Press, Princeton, NJ},
      YEAR = {2012},
     PAGES = {xiv+472},
      ISBN = {978-0-691-14794-9},
   MRCLASS = {57M50 (20F36 20F65 57M07 57N05)},
  MRNUMBER = {2850125},
MRREVIEWER = {Stephen\ P.\ Humphries},
}

@article {MR214087,
    AUTHOR = {Epstein, D. B. A.},
     TITLE = {Curves on {$2$}-manifolds and isotopies},
   JOURNAL = {Acta Math.},
  FJOURNAL = {Acta Mathematica},
    VOLUME = {115},
      YEAR = {1966},
     PAGES = {83--107},
      ISSN = {0001-5962,1871-2509},
   MRCLASS = {57.20},
  MRNUMBER = {214087},
MRREVIEWER = {P.\ Dedecker},
       DOI = {10.1007/BF02392203},
       URL = {https://doi.org/10.1007/BF02392203},
}

@article {MR1581099,
    AUTHOR = {Baer, R.},
     TITLE = {Kurventypen auf {F}l\"{a}chen},
   JOURNAL = {J. Reine Angew. Math.},
  FJOURNAL = {Journal f\"{u}r die Reine und Angewandte Mathematik. [Crelle's
              Journal]},
    VOLUME = {156},
      YEAR = {1927},
     PAGES = {231--246},
      ISSN = {0075-4102,1435-5345},
   MRCLASS = {DML},
  MRNUMBER = {1581099},
       DOI = {10.1515/crll.1927.156.231},
       URL = {https://doi.org/10.1515/crll.1927.156.231},
}

@article {MR1581156,
    AUTHOR = {Baer, R.},
     TITLE = {Isotopie von {K}urven auf orientierbaren, geschlossenen
              {F}l\"{a}chen und ihr {Z}usammenhang mit der topologischen
              {D}eformation der {F}l\"{a}chen},
   JOURNAL = {J. Reine Angew. Math.},
  FJOURNAL = {Journal f\"{u}r die Reine und Angewandte Mathematik. [Crelle's
              Journal]},
    VOLUME = {159},
      YEAR = {1928},
     PAGES = {101--116},
      ISSN = {0075-4102,1435-5345},
   MRCLASS = {DML},
  MRNUMBER = {1581156},
       DOI = {10.1515/crll.1928.159.101},
       URL = {https://doi.org/10.1515/crll.1928.159.101},
}

@article {MR4263685,
    AUTHOR = {Buhovsky, L. and Humili\`ere, V. and Seyfaddini, S.},
     TITLE = {The action spectrum and {$C^0$} symplectic topology},
   JOURNAL = {Math. Ann.},
  FJOURNAL = {Mathematische Annalen},
    VOLUME = {380},
      YEAR = {2021},
    NUMBER = {1-2},
     PAGES = {293--316},
      ISSN = {0025-5831,1432-1807},
   MRCLASS = {53D05 (37J06)},
  MRNUMBER = {4263685},
MRREVIEWER = {St\'{e}phane\ Tchuiaga},
       DOI = {10.1007/s00208-021-02183-w},
       URL = {https://doi.org/10.1007/s00208-021-02183-w},
}

@article {MR4504915,
    AUTHOR = {Kawamoto, Y.},
     TITLE = {On {$C^0$}-continuity of the spectral norm for symplectically
              non-aspherical manifolds},
   JOURNAL = {Int. Math. Res. Not. IMRN},
  FJOURNAL = {International Mathematics Research Notices. IMRN},
      YEAR = {2022},
     PAGES = {17187--17230},
      ISSN = {1073-7928,1687-0247},
   MRCLASS = {53D05 (37J06 58D05)},
  MRNUMBER = {4504915},
MRREVIEWER = {Jakob\ Hedicke},
       DOI = {10.1093/imrn/rnab206},
       URL = {https://doi.org/10.1093/imrn/rnab206},
}

@article {MR543339,
    AUTHOR = {Mather, J. N.},
     TITLE = {Area preserving twist homeomorphism of the annulus},
   JOURNAL = {Comment. Math. Helv.},
  FJOURNAL = {Commentarii Mathematici Helvetici},
    VOLUME = {54},
      YEAR = {1979},
    NUMBER = {3},
     PAGES = {397--404},
      ISSN = {0010-2571,1420-8946},
   MRCLASS = {58F10 (54H25 55M99)},
  MRNUMBER = {543339},
MRREVIEWER = {Carl\ P.\ Simon},
       DOI = {10.1007/BF02566283},
       URL = {https://doi.org/10.1007/BF02566283},
}

@article {MR3820002,
    AUTHOR = {Koropecki, Andres and Tal, Fabio Armando},
     TITLE = {Fully essential dynamics for area-preserving surface
              homeomorphisms},
   JOURNAL = {Ergodic Theory Dynam. Systems},
  FJOURNAL = {Ergodic Theory and Dynamical Systems},
    VOLUME = {38},
      YEAR = {2018},
    NUMBER = {5},
     PAGES = {1791--1836},
      ISSN = {0143-3857,1469-4417},
   MRCLASS = {37E30 (37E45)},
  MRNUMBER = {3820002},
MRREVIEWER = {Jingzhi\ Yan},
       DOI = {10.1017/etds.2016.110},
       URL = {https://doi.org/10.1017/etds.2016.110},
}

@article {MR3193751,
    AUTHOR = {Koropecki, Andres and Tal, Fabio Armando},
     TITLE = {Strictly toral dynamics},
   JOURNAL = {Invent. Math.},
  FJOURNAL = {Inventiones Mathematicae},
    VOLUME = {196},
      YEAR = {2014},
    NUMBER = {2},
     PAGES = {339--381},
      ISSN = {0020-9910,1432-1297},
   MRCLASS = {37E45 (37E10)},
  MRNUMBER = {3193751},
       DOI = {10.1007/s00222-013-0470-3},
       URL = {https://doi.org/10.1007/s00222-013-0470-3},
}
\end{document}